\documentclass[11pt]{article}

\usepackage[a4paper,margin=1in]{geometry}
\usepackage{lmodern}
\usepackage{amsmath,amssymb,amsthm,mathtools}
\usepackage{bm}
\usepackage{booktabs}
\usepackage{array}
\usepackage{tikz}
\usepackage{microtype}
\usepackage{hyperref}

\hypersetup{colorlinks=true,linkcolor=blue,citecolor=blue,urlcolor=blue,
  pdfauthor={Ziqian Lin and Huangang Lu},
  pdftitle={A Tensor Critical Threshold in Cubic Spherical Kernel Graphs}}

\newtheorem{theorem}{Theorem}[section]
\newtheorem{proposition}[theorem]{Proposition}
\newtheorem{lemma}[theorem]{Lemma}
\newtheorem{corollary}[theorem]{Corollary}
\newtheorem{remark}[theorem]{Remark}

\numberwithin{equation}{section}
\numberwithin{table}{section}
\numberwithin{figure}{section}

\newcommand{\Sph}{\mathbb S}
\newcommand{\E}{\mathbb E}
\newcommand{\Pp}{\mathsf P}
\newcommand{\Qq}{\mathsf Q}
\newcommand{\TV}{\operatorname{TV}}
\newcommand{\tr}{\operatorname{tr}}
\newcommand{\HS}{\mathrm{HS}}
\newcommand{\1}{\mathbf 1}
\newcommand{\Kop}{\mathsf K}
\newcommand{\Adj}{\bm A}
\newcommand{\cH}{\mathcal H}
\newcommand{\Normal}{\mathcal N}

\title{A Tensor Critical Threshold in Cubic Spherical Kernel Graphs}
\author{Ziqian Lin\thanks{Corresponding author. School of Mathematics and Statistics,
Fuzhou University, Fuzhou, China.
Email: \href{mailto:linziqian@fzu.edu.cn}{linziqian@fzu.edu.cn}.}
\and Huangang Lu\thanks{School of Physics, Peking University,
Beijing, China.
Email: \href{mailto:luhuangang@pku.edu.cn}{luhuangang@pku.edu.cn}.}}
\date{}

\begin{document}
\maketitle

\begin{abstract}
We identify a strong-detection transition within the critical scale for
testing cubic spherical kernel graphs against \(G(n,1/2)\).
Let \(X_1,\ldots,X_n\) be independent uniform points on
\(\mathbb S^{d-1}\).  Conditional on these points, edges are independent,
with connection probability
\[
 \frac{1+\beta Z_{3,d}(\langle X_i,X_j\rangle)}2.
\]
Here \(Z_{3,d}\) is the normalized cubic zonal harmonic, with multiplicity
\(N_{3,d}\sim d^3/6\).  Assume \(n\to\infty\), \(d=d_n\to\infty\),
\(\beta=\beta_n\) with \(|\beta|\le1\), and \(n\le d^A\), with
\(A<\infty\) fixed.
Write \(r_n=n\beta^2/(N_{3,d}d)\).
On the critical scale \(r_n\to r\in(0,\infty)\), we identify the tensor
variational threshold
\[
 r_{\mathrm{tens}}
 :=\min_{0<q<1}\frac{-2\{\log(1-q)+q\}}{q^6}.
\]
For \(0<r<r_{\mathrm{tens}}\), the log-likelihood ratio under the null admits
an asymptotically Gaussian expansion in the signed triangle, and the
total-variation distance converges to \(2\Phi(r^{3/2}/2)-1\), where \(\Phi\)
is the standard normal distribution function.  For \(r>r_{\mathrm{tens}}\),
the likelihood-ratio test is strongly consistent, although the centered
signed triangle still has Gaussian fluctuations.  In the strict regimes
\(r_n\to0\) and \(r_n\to\infty\), total variation tends to zero and one,
respectively, even when \(\beta\to0\).
\end{abstract}

\section{Introduction}
\label{sec:introduction}

\subsection{Why the cubic layer is a separate critical problem}

We observe a graph whose edges are conditionally independent given latent
points on a high-dimensional sphere, and test it against an Erd\H{o}s--R\'enyi
graph with identical one-edge marginals.  Such matched-marginal tests ask
whether higher-order edge dependence retains enough information to certify
the latent geometry.  They arose for hard spherical graphs in
\cite{BubeckDingEldanRacz2016} and now form a central problem in
high-dimensional random graph comparison
\cite{BrennanBreslerNagaraj2020,LiuRacz2023,
LiuMohantySchrammYang2024,MaoWuXu2026}.

This paper studies the pure degree-three connection kernel
\[
 K^{(3)}_{d,\beta}(t)=\frac{1+\beta Z_{3,d}(t)}2,
\]
where \(Z_{3,d}\) is the normalized cubic zonal harmonic and
\(N_{3,d}=d(d-1)(d+4)/6\) is its multiplicity.  The matched null is
\(G(n,1/2)\).  The centered kernel has a single spectral block whose
nonzero eigenvalues all have the same sign, so the cube trace does not
cancel.

A natural statistic for detecting geometric dependence is the signed
triangle.  Let \(a_{ij}\in\{0,1\}\) indicate whether vertices \(i\) and \(j\)
are adjacent, and write \(Y_{ij}=2a_{ij}-1\) for the centered edge signs.
Summing the products around all vertex triples gives
\[
 T_n=\sum_{i<j<k}Y_{ij}Y_{jk}Y_{ki}.
\]

Under \(G(n,1/2)\), it has mean zero and standard deviation
\(\sqrt{\binom n3}\).  Under the geometric model, its mean is
\(\binom n3\beta^3/N_{3,d}^2\).  The absolute mean shift divided by the
null standard deviation is therefore
\[
 \sqrt{\binom n3}\,\frac{|\beta|^3}{N_{3,d}^2}
 \asymp \frac{n^{3/2}|\beta|^3}{N_{3,d}^2}.
\]
Signal and null fluctuations are comparable at the cube-trace scale
\(n\beta^2\asymp N_{3,d}^{4/3}\).

For degree three, \(N_{3,d}\asymp d^3\), so this scale is also
\(n\beta^2\asymp N_{3,d}d\).  The variational analysis compares the signal
and entropy contributions on the respective scales
\(n^2\beta^2/N_{3,d}\) and \(nd\).
These two contributions are therefore comparable in the cubic critical
regime.

The triangle's mean shift measures the strength of one statistic, whereas
optimal testing can use the entire graph.  A distributional limit for the
triangle does not by itself identify the limiting likelihood or the optimal
testing errors.  The coincidence of the two scales therefore raises the
central question: does the signed triangle describe the full critical likelihood,
or can the full graph provide additional information for testing?

\subsection{Main results}

Throughout the results below, \(n\to\infty\), \(d=d_n\to\infty\),
\(\beta=\beta_n\) with \(|\beta|\le1\), and \(n\le d^A\), with
\(A<\infty\) fixed.  Write
\[
 r_n=\frac{n\beta^2}{N_{3,d}d}.
\]

At the critical scale \(r_n\to r\in(0,\infty)\), the threshold
\[
 r_{\mathrm{tens}}
 :=\min_{0<q<1}\frac{-2\{\log(1-q)+q\}}{q^6}
\]
separates two distinct testing behaviors.  For
\(0<r<r_{\mathrm{tens}}\), the log-likelihood ratio under the null has an
asymptotically Gaussian expansion determined by the signed triangle.
The total-variation distance therefore satisfies
\[
 \TV\longrightarrow2\Phi(r^{3/2}/2)-1,
\]
where \(\Phi\) is the standard normal distribution function.  For
\(r>r_{\mathrm{tens}}\), the likelihood-ratio test is strongly consistent
and \(\TV\to1\).  Throughout the critical scale, the normalized signed
triangle under the alternative has a unit-variance Gaussian limit after
subtracting its mean,
whose absolute value tends to \(r^{3/2}\).  The transition in optimal
testing thus occurs while the triangle retains a finite normalized mean shift.

When \(A>4\), there are parameter sequences with
\(r_n\to r_{\mathrm{tens}}\) along which total variation tends to one, and
others along which it tends to
\(2\Phi(r_{\mathrm{tens}}^{3/2}/2)-1\).  The limiting critical parameter
alone therefore does not determine the total-variation limit at the threshold.

Outside the critical scale, we prove that \(r_n\to0\) implies \(\TV\to0\), whereas
\(r_n\to\infty\) implies \(\TV\to1\).  In the latter regime, the signed
triangle gives a strongly consistent test.  Both conclusions allow
\(\beta\to0\).

The critical-scale and endpoint results are stated in
Theorem~\ref{thm:pure-cubic-critical}.
Theorem~\ref{thm:pure-cubic-triangle-clt} gives the Gaussian limits for
the signed triangle throughout the critical scale, and
Theorem~\ref{thm:pure-cubic-dichotomy} treats the regimes outside it.

\subsection{Proof strategy}

The proof combines variational estimates at order \(nd\) with a sharp
second-moment calculation at order one.  The former locate the threshold
and control pairs of latent samples with nonvanishing feature overlap.
The second-moment bound and triangle limits then identify the limiting
likelihood below the threshold.

The spherical tensor variational formula
\cite[Appendix~A.3]{JagannathLopattoMiolane2020} and a comparison with the
cubic harmonic features determine the limiting Gaussian potential on
scalar matrices.  Optimizing over the Gaussian observation
strength gives the \(q^6\) objective defining \(r_{\mathrm{tens}}\).
To justify this threshold for the graph, we must also control nonscalar
matrix directions as the harmonic dimension \(N_{3,d}\) grows.
Proposition~\ref{prop:cubic-matrix-branch-gap} extends the scalar bound
to every positive-semidefinite matrix, giving a uniform negative bound
away from zero overlap when \(r<r_{\mathrm{tens}}\).
Gaussian interpolation and Bernoulli--Gaussian comparison then control
the graph likelihood integrated over these overlaps.  Concentration
allows an observation-dependent truncation that discards \(o(1)\)
likelihood mass, as shown in
Section~\ref{subsec:cubic-variational-mechanism}.

For the remaining small overlaps, expansion of the Bernoulli likelihood
factors and entropy estimates preserve the exact second-moment constant.
Writing \(L_n\) for the graph likelihood ratio relative to
\(\Qq_n=G(n,1/2)\), the normalized truncation \(\widetilde L_n\) satisfies
\[
 \|L_n-\widetilde L_n\|_{L^1(\Qq_n)}\to0,
 \qquad \limsup\E_{\Qq_n}\widetilde L_n^2\le e^{r^3}.
\]
The triangle limits under the null and alternative identify the
conditional mean of any subsequential likelihood limit given the
limiting triangle.  This conditional mean is a Gaussian likelihood ratio
with second moment \(e^{r^3}\).  It therefore attains the upper bound,
and equality in conditional Jensen's inequality identifies the full
likelihood limit.  Section~\ref{subsec:cubic-likelihood-mechanism}
completes this argument.

Above the threshold, a scalar trial matrix gives a positive Gaussian
free-energy lower bound.  Bernoulli--Gaussian comparison and concentration
of the log-likelihood under the alternative yield strong detection.
Diagonal sequences of the fixed-parameter limits give the two distinct
total-variation limits on approach to \(r_{\mathrm{tens}}\).
Outside the critical scale, triangle moments give strong detection,
while an amplitude-dependent second-moment bound gives
indistinguishability, including when \(\beta\to0\).

\subsection{Relation to previous work}

Latent-space, random dot-product, and random connection models describe
network dependence through unobserved vertex positions
\cite{HoffRafteryHandcock2002,AthreyaEtAl2018,Penrose2003,DettmannGeorgiou2016}.
For hard spherical graphs, Bubeck, Ding, Eldan, and R\'acz
\cite{BubeckDingEldanRacz2016} introduced matched-marginal testing.
Subsequent work refined the dependence on edge density
\cite{BrennanBreslerNagaraj2020}, treated sparse graphs
\cite{LiuMohantySchrammYang2024}, improved bounds when \(d>n\)
\cite{DuMaoSunWuXu2026}, and developed sandwiching couplings
\cite{BangachevBreslerSandwiching2025}.
For soft graphs, Liu and R\'acz study noisy spherical models
\cite{LiuRacz2023} and monotone connection functions with Gaussian latent
vectors \cite{LiuRaczLatent2023}.

Smooth connection functions bring the role of the kernel's spectrum into
focus.  For fixed smooth spherical kernels bounded away from zero and one
and satisfying \(K'(0)\ne0\), Mao, Wu, and Xu \cite{MaoWuXu2026} identify
the critical dimension \(d\asymp n^{3/4}\).  They develop a posterior-overlap
analysis and propose a more general spectral conjecture based on the cube
trace of the standardized kernel.

The relative performance of cycle statistics also depends on the model.
Fourier methods analyze low-degree detection in spherical and random
algebraic graphs \cite{BangachevBreslerFourier2024,BangachevBresler2025}.
For \(L_\infty\) threshold graphs on a torus, Bangachev and Bresler
\cite[Theorems~5--6]{BangachevBresler2024} show that signed four-cycles
outperform signed triangles and establish lower bounds for low-degree
polynomial tests.

Exact critical testing limits have precedents in several related models.
For Wishart--Gaussian comparison, entropy estimates establish
total-variation convergence above the \(n^3\) dimension scale
\cite{BubeckGanguly2018}, and R\'acz and Richey \cite{RaczRichey2019}
compute the limiting total variation within that scale.
Related likelihood limits hold below the reconstruction threshold in
spiked Wigner models \cite{ElAlaouiKrzakalaJordan2020}, while signed
cycles determine the asymptotic likelihood-ratio test in stochastic
block models with growing average degree \cite{BanerjeeMa2017}.

Our variational analysis uses the replica-symmetric (RS) formula for
spherical tensor PCA of Jagannath, Lopatto, and Miolane
\cite[Appendix~A.3]{JagannathLopattoMiolane2020}.  Matrix variational
formulas also arise in finite-rank inference \cite{LelargeMiolane2017},
and Guerra interpolation \cite{Guerra2001} bounds free energies subject
to overlap constraints, as in the Franz--Parisi framework
\cite{FranzParisi1995,ElAlaouiKrzakala2018}.
Observation-dependent truncation has a precedent in the refined
second-moment analysis of spiked tensors by Perry, Wein, and Bandeira
\cite{PerryWeinBandeira2020}.

\subsection{Organization}

Section~\ref{sec:model} introduces the cubic harmonic kernel and its
feature representation, defines the graph model and testing notation, and
records the spectral and cycle identities used in the proofs.
Section~\ref{sec:results} states the critical testing theorem, the
triangle central limit theorem, and the strict dichotomy.
Section~\ref{sec:cubic-strict} proves the strict dichotomy, and
Section~\ref{sec:cubic} proves the triangle limits and critical testing
theorem.
Appendix~\ref{app:cubic-fixed-amplitude} proves the partition-function
bound used in the second-moment argument, with explicit dependence on the
amplitude.
Appendix~\ref{app:cubic-triangle-clt} gives the
triangle moment estimates.
Appendix~\ref{app:cubic-franz-parisi} proves the matrix and Franz--Parisi
estimates, the likelihood truncation, and the supercritical and endpoint
results.
The regularity and local partition-function bound used in that
likelihood argument is stated in
Proposition~\ref{prop:cubic-raw-local-partition} and proved in
Appendix~\ref{app:cubic-local-partition}.

\section{The cubic model and notation}
\label{sec:model}

\subsection{The cubic harmonic kernel}
\label{subsec:cubic-harmonic-kernel}

For \(d\ge2\), let \(\sigma_d\) be uniform probability measure on
\(\Sph^{d-1}\subset\mathbb R^d\).  The normalized cubic zonal harmonic is
\begin{equation}
 Z_{3,d}(t)=\frac{(d+2)t^3-3t}{d-1},\qquad -1\le t\le1.
 \label{eq:cubic-zonal-polynomial}
\end{equation}
To describe its harmonic representation, let \(\cH_{3,d}\) be the space
of real spherical harmonics of degree three, with dimension
\(N_{3,d}:=\dim\cH_{3,d}=d(d-1)(d+4)/6\).
For any orthonormal basis \((\psi_{3,m})_{m=1}^{N_{3,d}}\) of
\(\cH_{3,d}\) in \(L^2(\sigma_d)\), the spherical harmonic addition formula gives
\begin{equation}
 Z_{3,d}(\langle x,y\rangle)
 =\frac1{N_{3,d}}\sum_{m=1}^{N_{3,d}}\psi_{3,m}(x)\psi_{3,m}(y).
 \label{eq:cubic-addition-formula}
\end{equation}
The dimension and addition formulas are standard facts from spherical
harmonic theory \cite[Chapter~1]{DaiXu2013}.
For points \(x,y\in\Sph^{d-1}\), we use \(Z_{3,d}(x,y)\) as shorthand
for \(Z_{3,d}(\langle x,y\rangle)\).

For the spectral and likelihood calculations, we introduce the normalized
feature map and express the kernel as an inner product:
\[
 q_{3,d}(x)
 :=\frac1{\sqrt{N_{3,d}}}
   \bigl(\psi_{3,1}(x),\ldots,\psi_{3,N_{3,d}}(x)\bigr),
 \qquad
 Z_{3,d}(x,y)=\langle q_{3,d}(x),q_{3,d}(y)\rangle.
\]
Taking \(y=x\) in \eqref{eq:cubic-addition-formula} gives
\(N_{3,d}^{-1}\sum_m\psi_{3,m}(x)^2=1\).
The definition of \(q_{3,d}\), the
orthonormality of the basis, and the oddness of degree-three harmonics
then yield, for \(X\sim\sigma_d\),
\[
 \|q_{3,d}(x)\|=1,\qquad
 \E q_{3,d}(X)q_{3,d}(X)^{\mathsf T}=N_{3,d}^{-1}I_{N_{3,d}},\qquad
 q_{3,d}(-x)=-q_{3,d}(x).
\]
Consequently, the polynomial \(Z_{3,d}\) satisfies \(Z_{3,d}(1)=1\) and
\(Z_{3,d}(-t)=-Z_{3,d}(t)\).  The feature representation and the
Cauchy--Schwarz inequality imply that \(|Z_{3,d}(t)|\le1\) for every
\(t\in[-1,1]\).

\subsection{Graph model and testing notation}
\label{subsec:testing-convention}

We study the connection kernel
\begin{equation*}
 K^{(3)}_{d,\beta}(t)=\frac{1+\beta Z_{3,d}(t)}2,
 \qquad |\beta|\le1.
\end{equation*}
The bound \(|Z_{3,d}|\le1\) ensures that this kernel takes values in
\([0,1]\).  Sample
\(X_1,\ldots,X_n\stackrel{\mathrm{iid}}{\sim}\sigma_d\).
Conditional on these points, the upper-triangular entries of the symmetric
adjacency matrix \(\Adj=(a_{ij})_{i,j\le n}\) are independent, with
\[
 a_{ij}\mid X_1,\ldots,X_n
 \sim\operatorname{Bernoulli}
       \bigl(K^{(3)}_{d,\beta}(\langle X_i,X_j\rangle)\bigr),
 \qquad i<j,
\]
and \(a_{ii}=0\).  For a measurable connection function
\(K:[-1,1]\to[0,1]\), write \(\Pp_{n,d,K}\)
for the law of this construction with \(K\) in place of \(K^{(3)}_{d,\beta}\).
Our alternative and null are
\[
 \Pp_n^{(3)}:=\Pp_{n,d,K^{(3)}_{d,\beta}},
 \qquad \Qq_n:=G(n,1/2).
\]
This is a spherical soft random geometric graph
\cite{MaoWuXu2026}; under the Erd\H{o}s--R\'enyi null, all upper-triangular
entries are independent Bernoulli variables \cite{JansonLuczakRucinski2000}.
Oddness of \(Z_{3,d}\) makes the marginal edge probability exactly \(1/2\)
under the alternative as well.  Thus, any distinction between the two
graph laws must arise from joint edge dependence rather than individual
edge marginals.

For graph laws \(P,Q\) on \(\{0,1\}^{\binom n2}\), let
\[
 \TV(P,Q):=\sup_B|P(B)-Q(B)|=\frac12\int|dP-dQ|.
\]
If a test takes value one under the alternative, its optimal sum of errors
satisfies
\begin{equation}
 \inf_\varphi\{Q(\varphi=1)+P(\varphi=0)\}=1-\TV(P,Q);
 \label{eq:testing-total-variation}
\end{equation}
see \cite[Chapter~2]{Tsybakov2009}.  We say that strong detection is possible
when \(\TV(P_n,Q_n)\to1\).

\subsection{Spectral and cycle identities}
\label{subsec:cubic-spectral-identities}

We first compute the spectrum of the integral operator associated with
the centered kernel.  Traces of its powers then give the expected products
of edge signs around simple cycles.  For a bounded real zonal kernel
\(\kappa\), define
\[
 (\Kop_\kappa f)(x)
 =\int_{\Sph^{d-1}}\kappa(\langle x,y\rangle)f(y)\,d\sigma_d(y).
\]
This is a self-adjoint Hilbert--Schmidt operator.  Its powers
\(\Kop_\kappa^\ell\) are trace class for every integer \(\ell\ge2\)
\cite[Chapter~2]{Simon2005}.  In the cubic model,
\(\kappa=2K^{(3)}_{d,\beta}-1=\beta Z_{3,d}\).  For
\(f\in L^2(\sigma_d)\), equation~\eqref{eq:cubic-addition-formula} gives
\[
 \begin{aligned}
 (\Kop_{\beta Z_{3,d}}f)(x)
 &=\beta\int_{\Sph^{d-1}}Z_{3,d}(x,y)f(y)\,d\sigma_d(y)\\
 &=\frac{\beta}{N_{3,d}}\sum_{m=1}^{N_{3,d}}\psi_{3,m}(x)
   \int_{\Sph^{d-1}}\psi_{3,m}(y)f(y)\,d\sigma_d(y).
 \end{aligned}
\]
Since the \(\psi_{3,m}\) form an orthonormal basis of \(\cH_{3,d}\),
the finite sum equals \(f\) when \(f\in\cH_{3,d}\) and is zero when
\(f\in\cH_{3,d}^{\perp}\).  Thus \(\Kop_{\beta Z_{3,d}}\) is
\(\beta/N_{3,d}\) times the orthogonal projection onto \(\cH_{3,d}\), with
\begin{equation}
 \left.\Kop_{\beta Z_{3,d}}\right|_{\cH_{3,d}}
 =\frac{\beta}{N_{3,d}}I_{\cH_{3,d}},
 \qquad
 \left.\Kop_{\beta Z_{3,d}}\right|_{\cH_{3,d}^{\perp}}=0,
 \label{eq:single-layer-spectrum}
\end{equation}
Taking powers and summing over the \(N_{3,d}\) basis vectors gives
\begin{equation}
 \tr(\Kop_{\beta Z_{3,d}}^\ell)
 =N_{3,d}\left(\frac{\beta}{N_{3,d}}\right)^\ell
 =\frac{\beta^\ell}{N_{3,d}^{\ell-1}},
 \qquad \ell\ge2.
 \label{eq:single-layer-traces}
\end{equation}
For \(\beta\ne0\), the operator has rank \(N_{3,d}\), with all nonzero
eigenvalues of the same sign.  Thus the odd traces do not exhibit the
cross-block cancellation possible for mixed kernels.  At \(\beta=0\),
the centered operator is zero and the graph law equals the null.

We relate these traces to graph statistics through the edge signs
\(Y_{ij}=2a_{ij}-1\).  For a connection function \(K\), set
\(\kappa=2K-1\).
For a simple cycle of length \(\ell\ge3\) on distinct vertices
\(i_1,\ldots,i_\ell\), set \(i_{\ell+1}=i_1\) and \(x_{\ell+1}=x_1\).
Conditional independence of the cycle edges, together with
\(\E(Y_{ij}\mid X_1,\ldots,X_n)=\kappa(\langle X_i,X_j\rangle)\), gives
\begin{equation}
  \E_{\Pp_{n,d,K}}
  \prod_{r=1}^{\ell}Y_{i_ri_{r+1}}
  =
  \int
  \prod_{r=1}^{\ell}
  \kappa(\langle x_r,x_{r+1}\rangle)
  \,d\sigma_d(x_1)\cdots d\sigma_d(x_\ell)
  =
  \tr(\Kop_\kappa^\ell).
  \label{eq:cycle-trace-identity}
\end{equation}
For the last equality, write the trace as the Hilbert--Schmidt inner
product of \(\Kop_\kappa\) and \(\Kop_\kappa^{\ell-1}\), using
self-adjointness.  Composition of their kernels gives the cyclic integral;
boundedness of \(\kappa\) justifies Fubini's theorem.  Under \(\Qq_n\),
the cycle moment is zero because its distinct edges are independent
and centered.

\section{Main results}
\label{sec:results}

We now state the testing results for the cubic model.  The main result
identifies a transition within the critical scale.  Below a tensor threshold,
the signed triangle determines the asymptotic log-likelihood ratio, whereas
above it the full graph is strongly distinguishable.  We present this result
first, followed by the Gaussian limits for the signed triangle and the
testing limits outside the critical scale.

For \(n\ge3\), define the signed triangle and its normalized version by
\begin{equation}
 T_n=\sum_{i<j<k}Y_{ij}Y_{jk}Y_{ki},\qquad
 \widehat T_n=\binom n3^{-1/2}T_n.
 \label{eq:cubic-triangle-statistic}
\end{equation}
The testing regimes and the normalized triangle mean are described by
\begin{equation}
 r_n=\frac{n\beta^2}{N_{3,d}d},
 \qquad
 \Gamma_{3,n}=\frac{n^3\beta^6}{N_{3,d}^4},
 \qquad
 \gamma_n=\frac{\sqrt{\binom n3}\,\beta^3}{N_{3,d}^2}.
 \label{eq:cubic-model-parameters}
\end{equation}
These parameters satisfy
\[
 \Gamma_{3,n}=r_n^3\frac{d^3}{N_{3,d}},
 \qquad
 \gamma_n^2=\frac{\binom n3}{n^3}\Gamma_{3,n}.
\]
All limits are taken as \(n\to\infty\) and \(d=d_n\to\infty\), with
\(\beta=\beta_n\) allowed to vary.  We suppress the dependence of
\(d\) and \(\beta\) on \(n\) in the notation.
Since \(N_{3,d}\sim d^3/6\), the strict regimes \(r_n\to0\) and
\(r_n\to\infty\) are equivalent to \(\Gamma_{3,n}\to0\) and
\(\Gamma_{3,n}\to\infty\), respectively.  The critical scale is
\(r_n\to r\in(0,\infty)\), where \(\Gamma_{3,n}\to6r^3\).  We write
\(\Normal(\mu,\Sigma)\) for a normal law with mean \(\mu\) and covariance
\(\Sigma\) (variance in one dimension).

Theorem~\ref{thm:pure-cubic-critical} describes testing at the critical
scale and identifies the tensor threshold for strong detection.

\begin{theorem}[Testing at the cubic critical scale]
\label{thm:pure-cubic-critical}
Fix $A<\infty$.  Let $d=d_n\to\infty$,
$n\le d^A$, and $|\beta|\le1$.
Suppose $r_n\to r\in(0,\infty)$.  Then
\(\gamma_n^2\to r^3\) and
\(\Gamma_{3,n}=6\gamma_n^2\{1+o(1)\}\to6r^3\).  Define
\[
 r_{\mathrm{tens}}
 :=\min_{0<q<1}\frac{-2\{\log(1-q)+q\}}{q^6}.
\]
If \(0<r<r_{\mathrm{tens}}\), then, under \(\Qq_n\),
\[
 \log\frac{d\Pp_n^{(3)}}{d\Qq_n}
 =\gamma_n\widehat T_n-\frac12r^3+o_{\Qq_n}(1),
\]
and
\[
 \TV(\Pp_n^{(3)},\Qq_n)
 \longrightarrow2\Phi(r^{3/2}/2)-1,
\]
where \(\Phi\) is the standard normal distribution function.
If \(r>r_{\mathrm{tens}}\), then the likelihood-ratio test is strongly
consistent and \(\TV(\Pp_n^{(3)},\Qq_n)\to1\).

When \(A>4\), there are parameter sequences with \(d=d_n\to\infty\),
\(n\le d^A\), \(|\beta|\le1\), and \(r_n\to r_{\mathrm{tens}}\)
along which total variation tends to one.
There are also such sequences along which total variation converges to
\(2\Phi(r_{\mathrm{tens}}^{3/2}/2)-1\).
\end{theorem}

For \(0<r<r_{\mathrm{tens}}\), the signed triangle determines the asymptotic
log-likelihood ratio, but the positive total-variation limit is strictly less
than one.  The models can therefore be distinguished with a nonvanishing
advantage, although no test makes both error probabilities tend to zero.
For \(r>r_{\mathrm{tens}}\), strongly consistent testing becomes possible.
Thus \(r_{\mathrm{tens}}\) marks the onset of strong detection within the
critical scale.  Its defining variational problem has a unique minimizer,
characterized by the stationarity equation in
\eqref{eq:cubic-main-threshold-stationarity}.

For \(A>4\), the two sequence constructions show that
\(r_n\to r_{\mathrm{tens}}\) alone does not determine the total-variation
limit.
The theorem establishes the existence of these distinct limits but does not
identify a finite-size window or determine the limit for a prescribed
approach rate.

Theorem~\ref{thm:pure-cubic-triangle-clt} gives the Gaussian limits for
the signed triangle throughout the critical scale, including at
\(r_{\mathrm{tens}}\).

\begin{theorem}[Gaussian limits for the signed triangle]
\label{thm:pure-cubic-triangle-clt}
Fix \(A<\infty\).  Let \(d=d_n\to\infty\), \(n\le d^A\), and
\(|\beta|\le1\).  Suppose \(r_n\to r\in(0,\infty)\).  Then
\[
 \widehat T_n\Rightarrow\Normal(0,1)\quad\hbox{under }\Qq_n,\qquad
 \widehat T_n-\gamma_n\Rightarrow\Normal(0,1)
       \quad\hbox{under }\Pp_n^{(3)},
\]
with \(\gamma_n^2\to r^3\).
\end{theorem}

Section~\ref{subsec:cubic-triangle-critical} gives the proof, with the
detailed moment estimates in
Appendix~\ref{app:cubic-triangle-clt}.
Below \(r_{\mathrm{tens}}\), these limits combine
with a sharp second-moment bound for a truncated likelihood to prove the
expansion in Theorem~\ref{thm:pure-cubic-critical}.  Above the threshold,
they contrast with strong detection from the full graph.

Theorem~\ref{thm:pure-cubic-dichotomy} treats the regimes outside the
critical scale.  These require separate estimates because the fixed-\(r\)
limits above do not by themselves
provide uniformity as \(r\) tends to zero or infinity.
The indistinguishability conclusion tracks the amplitude throughout and
therefore remains valid when \(\beta\) tends to zero.

\begin{theorem}[Testing outside the critical scale]
\label{thm:pure-cubic-dichotomy}
Fix \(A<\infty\).  Let \(d=d_n\to\infty\),
\(n\le d^A\), and \(|\beta|\le1\).  For the graph laws
\(\Pp_n^{(3)}\) and \(\Qq_n\) from Section~\ref{subsec:testing-convention},
\[
  \Gamma_{3,n}\to\infty
  \quad\Longrightarrow\quad
  \TV(\Pp_n^{(3)},\Qq_n)\to1,
\]
whereas
\[
  \Gamma_{3,n}\to0
  \quad\Longrightarrow\quad
  \TV(\Pp_n^{(3)},\Qq_n)\to0.
\]
\end{theorem}

When \(\Gamma_{3,n}\to\infty\), the signed triangle alone gives a strongly
consistent test, as shown in Proposition~\ref{prop:single-positive}.
When \(\Gamma_{3,n}\to0\), the conclusion applies to every test based on
the full graph.  The rejection probabilities under the null and alternative
differ by \(o(1)\), so no statistic can yield a nonvanishing testing advantage.

The same cubic spectral quantity governs both regimes.  By
\eqref{eq:single-layer-traces},
\(\Gamma_{3,n}=n^3\tr^2(\Kop_{\beta Z_{3,d}}^3)\).
For nonzero \(\beta\), strong detection
holds when \(N_{3,d}\ll n^{3/4}|\beta|^{3/2}\), whereas indistinguishability
holds when \(N_{3,d}\gg n^{3/4}|\beta|^{3/2}\).
These limits identify the two sides of the critical scale.
Within that scale, Theorem~\ref{thm:pure-cubic-critical} describes the
further transition at \(r_{\mathrm{tens}}\).

\section{Strict testing regimes}
\label{sec:cycles}
\label{sec:cubic-strict}

We first establish the strict regimes \(r_n\to0\) and \(r_n\to\infty\),
proving Theorem~\ref{thm:pure-cubic-dichotomy}.  The detectable direction
follows from the exact triangle moments.  For indistinguishability, we bound
the second moment of a conditioned likelihood ratio using
Appendix~\ref{app:cubic-fixed-amplitude}, retaining the dependence on
\(\beta\) to cover vanishing amplitudes.  Throughout this section write
\(N=N_{3,d}\).

\subsection{Triangle moments and strong detection}

Recall \(T_n\) from \eqref{eq:cubic-triangle-statistic}.
Rotational invariance and the cycle--trace identity reduce its variance
calculation in Proposition~\ref{prop:cubic-triangle-moments} to counting
pairs of triangles according to their shared vertices.

\begin{proposition}[Exact cubic triangle moments]
\label{prop:cubic-triangle-moments}
Under \(\Qq_n=G(n,1/2)\),
\[
 \E_{\Qq_n}T_n=0,
 \qquad
 \operatorname{Var}_{\Qq_n}(T_n)=\binom n3.
\]
Under \(\Pp_n^{(3)}\),
\begin{align}
 \E_{\Pp_n^{(3)}}T_n
 &=\binom n3\frac{\beta^3}{N^2},
 \label{eq:cubic-triangle-mean}\\
 \operatorname{Var}_{\Pp_n^{(3)}}(T_n)
 &=\binom n3\left(1-\frac{\beta^6}{N^4}\right)
   +12\binom n4\left(\frac{\beta^4}{N^3}-\frac{\beta^6}{N^4}\right).
 \notag
\end{align}
\end{proposition}

\begin{proof}
Write \(U_{ijk}=Y_{ij}Y_{jk}Y_{ki}\).  Under \(\Qq_n\), these monomials
are centered and orthogonal, with square one, proving the null formulas.
Under \(\Pp_n^{(3)}\), disjoint triangles are independent.  Triangles sharing
one vertex are conditionally independent given its latent position, and
rotational invariance makes their conditional means constant.  Thus the
overlap covariance formula gives
\[
 \operatorname{Var}_{\Pp_n^{(3)}}(T_n)
 =\binom n3\operatorname{Var}_{\Pp_n^{(3)}}(U_{123})
 +12\binom n4\operatorname{Cov}_{\Pp_n^{(3)}}(U_{123},U_{124}),
\]
with the second term omitted when \(n=3\).
By \eqref{eq:single-layer-traces}--\eqref{eq:cycle-trace-identity} and
\(Y_{ij}^2=1\),
\[
 \E_{\Pp_n^{(3)}}U_{123}=\frac{\beta^3}{N^2},
 \qquad
 \E_{\Pp_n^{(3)}}U_{123}^2=1,
 \qquad
 \E_{\Pp_n^{(3)}}(U_{123}U_{124})=\frac{\beta^4}{N^3}.
\]
Substitution and linearity of expectation prove the remaining formulas.
\end{proof}

When \(\Gamma_{3,n}\to\infty\), the squared mean in
\eqref{eq:cubic-triangle-mean} dominates the variance under either graph
law.  Proposition~\ref{prop:single-positive} converts this separation into
a strongly consistent triangle test.

\begin{proposition}[Strong detection by the signed triangle]
\label{prop:single-positive}
If \(\Gamma_{3,n}\to\infty\), then the signed-triangle test oriented by
\(\operatorname{sgn}(\beta)\) is
strongly consistent and
\[
 \TV(\Pp_n^{(3)},\Qq_n)\longrightarrow1.
\]
\end{proposition}

\begin{proof}
Let \(\mu_n=\E_{\Pp_n^{(3)}}T_n\).  The moment formulas in
Proposition~\ref{prop:cubic-triangle-moments} give
\[
 \mu_n^2\asymp\frac{n^6\beta^6}{N^4},
 \qquad
 \operatorname{Var}_{\Qq_n}(T_n)\asymp n^3,
 \qquad
 \operatorname{Var}_{\Pp_n^{(3)}}(T_n)
 =O\!\left(n^3+\frac{n^4\beta^4}{N^3}\right).
\]
Dividing by \(\mu_n^2\), we obtain
\[
 \frac{\operatorname{Var}_{\Qq_n}(T_n)}{\mu_n^2}
 =O(\Gamma_{3,n}^{-1}),
 \qquad
 \frac{\operatorname{Var}_{\Pp_n^{(3)}}(T_n)}{\mu_n^2}
 =O\!\left(
   \Gamma_{3,n}^{-1}+\frac{N}{n^2\beta^2}
 \right).
\]
Both ratios tend to zero, since \(\Gamma_{3,n}\to\infty\) and
\[
 \frac{N}{n^2\beta^2}
 =\frac{\Gamma_{3,n}^{-1/3}}{nN^{1/3}}
 \le\frac{\Gamma_{3,n}^{-1/3}}n\longrightarrow0.
\]
Since \(\operatorname{sgn}(\mu_n)=\operatorname{sgn}(\beta)\), define the test
\[
 \varphi_n=\1_{\{\operatorname{sgn}(\mu_n)T_n>|\mu_n|/2\}}.
\]
Since \(\E_{\Qq_n}T_n=0\), Chebyshev's inequality bounds the Type I and
Type II errors by
\[
 \begin{aligned}
 \Qq_n(\varphi_n=1)
 &\le \Qq_n\bigl(|T_n|\ge|\mu_n|/2\bigr)
 \le \frac{4\operatorname{Var}_{\Qq_n}(T_n)}{\mu_n^2}=o(1),\\
 \Pp_n^{(3)}(\varphi_n=0)
 &\le \Pp_n^{(3)}\bigl(|T_n-\mu_n|\ge|\mu_n|/2\bigr)
 \le \frac{4\operatorname{Var}_{\Pp_n^{(3)}}(T_n)}{\mu_n^2}=o(1).
 \end{aligned}
\]
Thus the test is strongly consistent, and the testing identity
\eqref{eq:testing-total-variation} yields
\(\TV(\Pp_n^{(3)},\Qq_n)\to1\).
\end{proof}

\subsection{Indistinguishability via the second moment}
\label{subsec:cubic-likelihood-preliminaries}

Proposition~\ref{prop:cubic-amplitude-tracked-reverse} proves
indistinguishability when \(n\beta^2=o(Nd)\).  The proof bounds a conditioned
likelihood second moment and then removes the conditioning.

For probability measures \(P\ll Q\), write \(L=dP/dQ\)
and \(\chi^2(P\|Q):=\E_Q(L-1)^2\).  We also write
\(D(P\Vert Q)=\int\log L\,dP\) for the relative entropy, and set
\(D(P\Vert Q)=+\infty\) when \(P\not\ll Q\).
Cauchy--Schwarz gives
\begin{equation}
  \TV(P,Q)
  \le\frac12\sqrt{\chi^2(P\|Q)}
  =\frac12\sqrt{\E_QL^2-1}.
 \label{eq:tv-chi-square-bound}
\end{equation}

By \eqref{eq:tv-chi-square-bound}, to show that \(\TV(P,Q)=o(1)\), it
suffices to establish \(\E_QL^2=1+o(1)\).  We first derive a formula
for this second moment that also applies after conditioning.
For any probability measure \(\mu\)
on \((\Sph^{d-1})^n\), let \(\Pp_{\mu,\beta}^{(3)}\) denote the graph
model of Section~\ref{subsec:testing-convention} with
\((X_1,\ldots,X_n)\sim\mu\).
Its likelihood ratio with respect to \(\Qq_n\) is
\[
  L_{\mu,\beta}:=\frac{d\Pp_{\mu,\beta}^{(3)}}{d\Qq_n}
  =\mathbb E_{X\sim\mu}\prod_{i<j}
    \{1+\beta Y_{ij}Z_{3,d}(X_i,X_j)\}.
\]
Let \(X=(X_i)_{i\le n}\) and \(X'=(X_i')_{i\le n}\) be two independent
draws from \(\mu\), called prior replicas, and put
\[
  h_{ij}=Z_{3,d}(X_i,X_j)Z_{3,d}(X_i',X_j').
\]
Averaging \(L_{\mu,\beta}^2\) over the independent edge signs under
\(\Qq_n\) gives
\begin{equation}
  \mathbb E_{\Qq_n}L_{\mu,\beta}^2
  =\mathbb E_{\mu\otimes\mu}
    \prod_{i<j}(1+\beta^2h_{ij}).
  \label{eq:cubic-exact-second-moment}
\end{equation}

In the model of Section~\ref{subsec:testing-convention},
\((X_1,\ldots,X_n)\) has law \(\rho_n:=\sigma_d^{\otimes n}\).
Rare pairs of latent samples \(X,X'\sim\rho_n\) can prevent
\(\mathbb E_{\Qq_n}L_{\rho_n,\beta}^2\) from converging to one.
We therefore estimate the likelihood second moment after conditioning
each sample on the high-probability event
\(\mathcal G_n\) supplied by Lemma~\ref{lem:cubic-hereditary-envelope}
under the polynomial growth condition \(n\le d^A\).  Set
\begin{equation}
  g_n=\rho_n(\mathcal G_n),
  \qquad
  \pi_n=\rho_n(\,\cdot\mid\mathcal G_n).
  \label{eq:cubic-conditioned-prior}
\end{equation}
Lemma~\ref{lem:cubic-hereditary-envelope} gives
\(1-g_n=\rho_n(\mathcal G_n^c)=o(1)\).  We will prove
\(\mathbb E_{\Qq_n}L_{\pi_n,\beta}^2=1+o(1)\) and bound the change in
the graph law when \(\rho_n\) is replaced by \(\pi_n\).
The proof of Proposition~\ref{prop:cubic-amplitude-tracked-reverse}
establishes this estimate by expanding the logarithm of the product in
\eqref{eq:cubic-exact-second-moment} to second order and retaining its
term \(-\frac{\beta^4}{2}\sum_{i<j}h_{ij}^2\).  The Gibbs variational
formula then allows us to combine estimates for successive latent pairs
while retaining the restrictions on the samples.

\begin{proposition}[Cubic indistinguishability bound]
\label{prop:cubic-amplitude-tracked-reverse}
Fix \(A<\infty\).  Let \(d=d_n\to\infty\) and \(\beta=\beta_n\),
with \(n\le d^A\) and \(|\beta|\le1\).  If
\[
  n\beta^2=o(N_{3,d}d),
\]
then
\[
  \TV\bigl(\Pp_n^{(3)},\Qq_n\bigr)\to0.
\]
\end{proposition}

\begin{proof}
If \(\beta=0\), the two graph laws coincide.  They also coincide when
\(n\le2\), since there is at most one edge and its marginal probability
is \(1/2\).  We may therefore assume
\(n\ge3\) and \(0<|\beta|\le1\).
In this proof, expectations without a subscript are under
\(\rho_n\otimes\rho_n\).
Step 1 derives \eqref{eq:cubic-second-moment-reduction}, reducing the
problem to bounding \(\mathbb E[\1_{\mathcal E_{n,R}}e^{F_n^{(2)}(\beta)}]\).
Step 2 applies the Gibbs variational formula to laws \(Q\) of the paired
sample supported on \(\mathcal E_{n,R}\).  It combines the chain rule for
\(D(Q\Vert(\sigma_d\otimes\sigma_d)^{\otimes n})\) with conditional
exponential-moment bounds to prove \eqref{eq:cubic-main-compensated-bound}
for the logarithm of this expectation.
Step 3 obtains
\(\mathbb E_{\Qq_n}L_{\pi_n,\beta}^2=1+o(1)\) and compares the graph
laws generated from \(\pi_n\) and \(\rho_n\).

\medskip
\noindent\textit{Step 1. Restrict the second moment and expand the logarithm.}
For independent \(X,X'\sim\rho_n\), set
\(w_i=q_{3,d}(X_i)q_{3,d}(X_i')^{\mathsf T}\) and
\(S_r=\sum_{j=1}^r w_j\) for \(0\le r\le n\), with \(S_0=0\).
With the Hilbert--Schmidt inner product
\(\langle H,K\rangle_{\HS}=\tr(H^{\mathsf T}K)\),
the feature identities in Section~\ref{subsec:cubic-harmonic-kernel}
give \(h_{ij}=\langle w_i,w_j\rangle_{\HS}\) and
\(\|w_i\|_{\HS}=1\).  Thus
\(\|S_n\|_{\HS}^2=n+2\sum_{i<j}h_{ij}\).  Write
\[
  W_{n,\beta}=\prod_{i<j}(1+\beta^2h_{ij}).
\]
By \eqref{eq:cubic-exact-second-moment},
\begin{equation}
  \mathbb E_{\Qq_n}L_{\pi_n,\beta}^2
  =g_n^{-2}\mathbb E[\1_{\mathcal G_n(X)}\1_{\mathcal G_n(X')}W_{n,\beta}].
  \label{eq:cubic-restricted-second-moment}
\end{equation}
Proposition~\ref{prop:cubic-weighted-overlap} supplies a fixed \(R\)
such that
\begin{equation}
  \mathbb E[\1_{\mathcal G_n(X)}\1_{\mathcal G_n(X')}
    \1_{\{\|S_n\|_{\HS}^2>Rn\}}W_{n,\beta}]=o(1).
  \label{eq:cubic-partition-tail}
\end{equation}
We thus restrict to
\(\mathcal E_{n,R}=\mathcal G_n(X)\cap\mathcal G_n(X')\cap\{\|S_n\|_{\HS}^2\le Rn\}\).
On this event, the bound \eqref{eq:cubic-pair-separation} in
Lemma~\ref{lem:cubic-hereditary-envelope} gives
\(\max_{i<j}|h_{ij}|\le C_A\log^3 d/N=o(1)\).
Define
\begin{equation*}
  F_n^{(2)}(\beta)
  =\frac{\beta^2}{2}(\|S_n\|_{\HS}^2-n)-\frac{\beta^4}{2}\sum_{i<j}h_{ij}^2.
\end{equation*}
Since \(\max_{i<j}|\beta^2h_{ij}|=o(1)\) on \(\mathcal E_{n,R}\),
Taylor's theorem gives, for sufficiently large \(d\),
\begin{equation}
  |\log W_{n,\beta}-F_n^{(2)}(\beta)|
  \le C\beta^6\sum_{i<j}|h_{ij}|^3
  \le C_A\frac{\beta^6n^2\log^9 d}{N^3}.
  \label{eq:cubic-partition-taylor}
\end{equation}
Using \(g_n=1-o(1)\), we combine \eqref{eq:cubic-partition-tail} and
\eqref{eq:cubic-partition-taylor} with
\eqref{eq:cubic-restricted-second-moment} to obtain
\begin{equation}
  \mathbb E_{\Qq_n}L_{\pi_n,\beta}^2
  \le g_n^{-2}\exp\left(C_A\frac{\beta^6n^2\log^9 d}{N^3}\right)
    \mathbb E[\1_{\mathcal E_{n,R}}e^{F_n^{(2)}(\beta)}]+o(1).
  \label{eq:cubic-second-moment-reduction}
\end{equation}

\medskip
\noindent\textit{Step 2. Bound the truncated exponential moment by entropy duality.}
If \((\rho_n\otimes\rho_n)(\mathcal E_{n,R})=0\), the upper bound
\eqref{eq:cubic-main-compensated-bound} is immediate.
Otherwise, put \(P_0=\sigma_d\otimes\sigma_d\).
Under \(\rho_n\otimes\rho_n\), the pairs \((X_i,X_i')\), \(1\le i\le n\),
are iid with law \(P_0\).
To retain the restriction to \(\mathcal E_{n,R}\), we apply
the Gibbs variational formula \eqref{eq:cubic-restricted-variational}
with \(m=n\) and \(H=F_n^{(2)}(\beta)\), obtaining
\begin{equation}
  \log\mathbb E[\1_{\mathcal E_{n,R}}e^{F_n^{(2)}(\beta)}]
  =\sup_{Q:\,Q(\mathcal E_{n,R})=1}
    \{\mathbb E_QF_n^{(2)}(\beta)-D(Q\Vert P_0^{\otimes n})\}.
  \label{eq:cubic-main-variational}
\end{equation}
Here \(Q\) ranges over laws of the paired sample satisfying
\(D(Q\Vert P_0^{\otimes n})<\infty\).
The Gibbs law \eqref{eq:cubic-gibbs-law}, with \(m=n\) and
\(H=F_n^{(2)}(\beta)\), attains this supremum and is exchangeable:
\(P_0^{\otimes n}\), \(\mathcal E_{n,R}\), and \(F_n^{(2)}(\beta)\)
are all invariant under permutations of the pairs.
Thus it suffices to consider exchangeable \(Q\).
For such a \(Q\), let \(Q_i\) be the conditional law of \((X_i,X_i')\)
given \((X_j,X_j')_{j<i}\).  Write
\(\mathcal D_Q=D(Q\Vert P_0^{\otimes n})\) and
\(\mathcal D_{Q,i}=D(Q_i\Vert P_0)\).
The chain rule for relative entropy gives
\(\mathcal D_Q=\sum_{i=1}^n\mathbb E_Q\mathcal D_{Q,i}\).
We combine this decomposition with
\eqref{eq:cubic-linear-entropy-duality} to bound the objective in
\eqref{eq:cubic-main-variational} using exponential moments under \(P_0\),
which we estimate by Lemma~\ref{lem:cubic-compensated-row}.

To decompose \(F_n^{(2)}(\beta)\) in the same order, define, for fixed
values of \((X_j,X_j')_{j<i}\),
\begin{equation*}
  H_i(w)=\beta^2\langle S_{i-1},w\rangle_{\HS}
    -\frac{\beta^4}{2}\sum_{j<i}\langle w_j,w\rangle_{\HS}^2.
\end{equation*}
Then \(F_n^{(2)}(\beta)=\sum_{i=1}^n H_i(w_i)\).
Let \((\widetilde X,\widetilde X')\sim P_0\) be independent of
the full paired sample and set
\(w=q_{3,d}(\widetilde X)q_{3,d}(\widetilde X')^{\mathsf T}\).
Write \(\mathbb E_w\) for expectation over this independent pair.

Fix \(0<\eta<1/2\), to be chosen below.
With \((X_j,X_j')_{j<i}\) fixed, inequality
\eqref{eq:cubic-linear-entropy-duality}, applied to
\(H_i/(1-\eta)\), gives
\begin{equation}
  \mathbb E_{Q_i}H_i(w_i)-(1-\eta)\mathcal D_{Q,i}
  \le(1-\eta)\log\mathbb E_w e^{H_i(w)/(1-\eta)}.
  \label{eq:cubic-main-conditional-entropy}
\end{equation}
To apply Lemma~\ref{lem:cubic-compensated-row} to the exponential moment
in \eqref{eq:cubic-main-conditional-entropy}, take \(\rho=2\),
\(r=i-1\), \(v_j=w_j\), \(u=\beta^2/(1-\eta)\le2\beta^2\), and
\(c_0=1-\eta\in[1/2,1]\).
Lemma~\ref{lem:cubic-subset-tensor-bounds}, with \(m=n\) and
\(I=\{1,\ldots,i-1\}\), verifies these hypotheses uniformly on
\(\mathcal E_{n,R}\) for \(0<\eta<1/2\).
Applying Lemma~\ref{lem:cubic-compensated-row} and using
\((1-\eta)u^4=\beta^8/(1-\eta)^3\le8\beta^8\), we obtain
\begin{equation}
\begin{aligned}
  &(1-\eta)\log\mathbb E_w e^{H_i(w)/(1-\eta)}\\
  &\quad\le
    \frac{\beta^4}{2(1-\eta)N^2}
      \{\|S_{i-1}\|_{\HS}^2-(1-\eta)(i-1)\}
    +C\beta^8\frac{\|S_{i-1}\|_{\HS}^4+(i-1)^2}{N^4}.
\end{aligned}
  \label{eq:cubic-main-row-bound}
\end{equation}
To identify the contribution of the negative term in \(H_i\), expand
\[
  \|S_{i-1}\|_{\HS}^2
  =\sum_{j<i}\|w_j\|_{\HS}^2
    +2\sum_{j<k<i}\langle w_j,w_k\rangle_{\HS}
  =(i-1)+2\sum_{j<k<i}h_{jk}.
\]
Using \(1/(1-\eta)=1+\eta/(1-\eta)\), the quadratic term in
\eqref{eq:cubic-main-row-bound} becomes
\begin{equation}
  \frac{\beta^4\{\|S_{i-1}\|_{\HS}^2-(1-\eta)(i-1)\}}{2(1-\eta)N^2}
  =\frac{\beta^4}{N^2}\sum_{j<k<i}h_{jk}
    +\frac{\eta\beta^4}{2(1-\eta)N^2}\|S_{i-1}\|_{\HS}^2.
  \label{eq:cubic-row-cancellation}
\end{equation}
Under an exchangeable \(Q\), all \(\mathbb E_Qh_{jk}\), \(j<k\),
are equal.  Since
\(\sum_{i=1}^n\binom{i-1}{2}=\binom n3\) and
\(\|S_n\|_{\HS}^2-n=2\sum_{j<k}h_{jk}\), the first term on the right-hand side
of \eqref{eq:cubic-row-cancellation} satisfies
\begin{equation}
  \frac{\beta^4}{N^2}\sum_{i=1}^n\sum_{j<k<i}\mathbb E_Qh_{jk}
  =\frac{\beta^4(n-2)}{6N^2}\mathbb E_Q(\|S_n\|_{\HS}^2-n).
  \label{eq:cubic-main-summed-interaction}
\end{equation}
We now choose \(\eta=\beta^2(n-2)/(3N^2)\), so that the right-hand side
of \eqref{eq:cubic-main-summed-interaction} equals
\(\frac{\eta\beta^2}{2}\mathbb E_Q(\|S_n\|_{\HS}^2-n)\).
Since \(n\beta^2=o(Nd)\) and \(N\asymp d^3\),
\(0<\eta\le(n\beta^2/(Nd))\,d/(3N)=o(1)\), so this choice satisfies
\(0<\eta<1/2\) for all sufficiently large \(d\).
For the other term in \eqref{eq:cubic-row-cancellation} and the
fourth-order term in \eqref{eq:cubic-main-row-bound},
Lemma~\ref{lem:cubic-random-prefix} applies because \(Q\) is
exchangeable and \(\|S_n\|_{\HS}^2\le Rn\) holds \(Q\)-almost surely.
Taking \(r=i-1\) in the lemma gives
\[
  \mathbb E_Q\|S_{i-1}\|_{\HS}^2\le C_R(i-1),
  \qquad
  \mathbb E_Q\|S_{i-1}\|_{\HS}^4\le C_R(i-1)^2.
\]
Consequently, taking \(Q\)-expectations in
\eqref{eq:cubic-main-conditional-entropy} and
\eqref{eq:cubic-main-row-bound}, using
\eqref{eq:cubic-row-cancellation}, and summing over \(i\) yields
\begin{equation}
  \mathbb E_QF_n^{(2)}(\beta)-\mathcal D_Q
  \le\eta\left\{\frac{\beta^2}{2}\mathbb E_Q(\|S_n\|_{\HS}^2-n)-\mathcal D_Q\right\}
  +C_R\left(\eta\frac{\beta^4n^2}{N^2}
    +\frac{\beta^8n^3}{N^4}\right).
  \label{eq:cubic-main-entropy-reduction}
\end{equation}
Here we used the chain rule for relative entropy
\(\mathcal D_Q=\sum_i\mathbb E_Q\mathcal D_{Q,i}\) and
\eqref{eq:cubic-main-summed-interaction} with the chosen value of \(\eta\),
and subtracted \(\eta\mathcal D_Q\) from both sides.

Define \(\Psi_n(t)=\log\mathbb E[
\1_{\mathcal E_{n,R}}e^{t(\|S_n\|_{\HS}^2-n)/2}]\), \(t\ge0\).
By \eqref{eq:cubic-restricted-variational} with \(m=n\) and
\(H=\beta^2(\|S_n\|_{\HS}^2-n)/2\), the expression in braces in
\eqref{eq:cubic-main-entropy-reduction} is at most \(\Psi_n(\beta^2)\).
Taking the supremum over \(Q\) in
\eqref{eq:cubic-main-entropy-reduction} and applying
\eqref{eq:cubic-main-variational} gives the first inequality below.
By \(n\beta^2=o(Nd)\), the bounds \eqref{eq:cubic-scaled-prefix-source}
in Lemma~\ref{lem:cubic-subset-tensor-bounds}, with \(m=n\), verify the
hypotheses of Lemma~\ref{lem:cubic-linear-free-energy} at \(t=\beta^2\).
Applying Lemma~\ref{lem:cubic-linear-free-energy} gives
\begin{equation}
\begin{aligned}
  \log\mathbb E[\1_{\mathcal E_{n,R}}e^{F_n^{(2)}(\beta)}]
  &\le\eta\Psi_n(\beta^2)
    +C_R\left(\eta\frac{\beta^4n^2}{N^2}
      +\frac{\beta^8n^3}{N^4}\right)\\
  &\le C_R\left\{\eta\frac{\beta^4n^2}{N^2}
    +(1+\eta)\frac{\beta^8n^3}{N^4}\right\}
  \le C_R\frac{\beta^6n^3}{N^4}.
\end{aligned}
  \label{eq:cubic-main-compensated-bound}
\end{equation}
The last inequality in \eqref{eq:cubic-main-compensated-bound} uses
\(\eta\le\beta^2n/(3N^2)=o(1)\) and \(|\beta|\le1\).

\medskip
\noindent\textit{Step 3. Combine the bounds and compare the two graph laws.}
Substituting \eqref{eq:cubic-main-compensated-bound} into
\eqref{eq:cubic-second-moment-reduction} gives
\begin{equation}
  1\le\mathbb E_{\Qq_n}L_{\pi_n,\beta}^2
  \le g_n^{-2}\exp\left[C_A\left(
    \frac{n^3\beta^6}{N^4}
    +\frac{n^2\beta^6\log^9 d}{N^3}
  \right)+o(1)\right].
  \label{eq:cubic-conditioned-partition-bound}
\end{equation}
The additive \(o(1)\) in \eqref{eq:cubic-second-moment-reduction} has
been absorbed into the exponent, since the exponential upper bound is
at least one.  The lower bound follows from
\(\mathbb E_{\Qq_n}L_{\pi_n,\beta}=1\) and Cauchy--Schwarz.
The derivation of \eqref{eq:cubic-conditioned-partition-bound} is
uniform in the sample size and does not require it to diverge.  Thus
the same bound applies with \(n\) replaced by an auxiliary sample size
\(m\le d^A\) satisfying \(m\beta^2=o(Nd)\).

By hypothesis, \(r_n\to0\).  Since \(N\asymp d^3\) and \(|\beta|\le1\),
the definition in \eqref{eq:cubic-model-parameters} gives
\[
 \begin{aligned}
  \frac{n^3\beta^6}{N^4}
  &=r_n^3\frac{d^3}{N}=o(1),\\
  \frac{n^2\beta^6\log^9 d}{N^3}
  &=\beta^2r_n^2\frac{d^2\log^9 d}{N}
    =O\!\left(r_n^2\frac{\log^9 d}d\right)=o(1).
 \end{aligned}
\]
Together with \(g_n=1-o(1)\),
\eqref{eq:cubic-conditioned-partition-bound} therefore yields
\(\mathbb E_{\Qq_n}L_{\pi_n,\beta}^2=1+o(1)\).

Finally, replace the distribution \(\rho_n=\sigma_d^{\otimes n}\) of
\((X_1,\ldots,X_n)\) by
\(\pi_n=\rho_n(\,\cdot\mid\mathcal G_n)\) from
\eqref{eq:cubic-conditioned-prior}.
The conditional graph construction in Section~\ref{subsec:testing-convention}
is the same for both distributions and cannot increase total variation.
Since \(d\pi_n/d\rho_n=\1_{\mathcal G_n}/g_n\),
\begin{equation}
  \TV(\Pp_n^{(3)},\Pp_{\pi_n,\beta}^{(3)})
  \le\TV(\rho_n,\pi_n)
  =\frac12\int\left|1-\frac{\1_{\mathcal G_n}}{g_n}\right|\,d\rho_n
  =1-g_n.
  \label{eq:cubic-conditioning-comparison}
\end{equation}
The triangle inequality, \eqref{eq:cubic-conditioning-comparison},
and \eqref{eq:tv-chi-square-bound} now give
\[
 \begin{aligned}
  \TV(\Pp_n^{(3)},\Qq_n)
  &\le\TV(\Pp_n^{(3)},\Pp_{\pi_n,\beta}^{(3)})
       +\TV(\Pp_{\pi_n,\beta}^{(3)},\Qq_n)\\
  &\le1-g_n+\frac12\sqrt{\mathbb E_{\Qq_n}L_{\pi_n,\beta}^2-1}=o(1).
 \end{aligned}
\]
\end{proof}

\begin{proof}[Proof of Theorem~\ref{thm:pure-cubic-dichotomy}]
The centered cubic operator has eigenvalue \(\beta/N_{3,d}\) with multiplicity
\(N_{3,d}\), and therefore
\[
  \Gamma_{3,n}=n^3\tr^2(\Kop_\kappa^3)
  =\frac{n^3\beta^6}{N_{3,d}^4}.
\]
If \(\Gamma_{3,n}\to\infty\), Proposition~\ref{prop:single-positive}
gives \(\TV(\Pp_n^{(3)},\Qq_n)\to1\).  If
\(\Gamma_{3,n}\to0\), then
\[
  \frac{n\beta^2}{N_{3,d}^{4/3}}=\Gamma_{3,n}^{1/3}\to0.
\]
Since \(N_{3,d}\asymp d^3\), one has
\(N_{3,d}^{4/3}\asymp N_{3,d}d\).  Hence the hypothesis of
Proposition~\ref{prop:cubic-amplitude-tracked-reverse} holds.  It follows that
\(\TV(\Pp_n^{(3)},\Qq_n)\to0\), proving the second strict implication.
\end{proof}

The remaining regime is \(r_n\to r\in(0,\infty)\), in which the
normalized triangle mean stays of constant order.  The likelihood on
this scale is analyzed in Section~\ref{sec:cubic}.

\section{Testing at the cubic critical scale}
\label{sec:cubic}

Section~\ref{sec:cubic-strict} establishes indistinguishability when
\(r_n\to0\) and triangle-based strong detection when \(r_n\to\infty\).
We now prove Theorems~\ref{thm:pure-cubic-critical}
and~\ref{thm:pure-cubic-triangle-clt} when
\(r_n\to r\in(0,\infty)\).  Throughout this section,
\(n\to\infty\) and \(d=d_n\to\infty\), with \(n\le d^A\) for a fixed
\(A<\infty\) and \(|\beta_n|\le1\).  We write \(\beta=\beta_n\) and
\(N=N_{3,d}\).

The below-threshold likelihood proof uses two inputs.
Section~\ref{subsec:cubic-triangle-critical} proves
the triangle limits in Theorem~\ref{thm:pure-cubic-triangle-clt}, which
hold for every fixed \(r\in(0,\infty)\).
Section~\ref{subsec:cubic-variational-mechanism} derives the tensor
threshold and combines the variational estimates with the local bound in
Proposition~\ref{prop:cubic-raw-local-partition} to construct a truncated
likelihood satisfying \eqref{eq:cubic-main-likelihood-interface} for
\(0<r<r_{\mathrm{tens}}\).  Section~\ref{subsec:cubic-likelihood-mechanism}
combines these two inputs to identify the log-likelihood and the
total-variation limit.

Above \(r_{\mathrm{tens}}\), the same variational potential gives a
positive order-\(nd\) mean log-likelihood bound.
Section~\ref{subsec:cubic-supercritical-mechanism} combines this bound
with concentration to prove strong detection, then invokes
Proposition~\ref{prop:cubic-endpoint} to complete the theorem at
\(r_n\to r_{\mathrm{tens}}\).
The technical estimates are proved in
Appendices~\ref{app:cubic-triangle-clt}--\ref{app:cubic-local-partition}.

\subsection{Gaussian limits for the signed triangle}
\label{subsec:cubic-triangle-critical}

The proof of Theorem~\ref{thm:pure-cubic-triangle-clt} uses de Jong's
central limit theorem for degree-three multilinear forms.  Under the
alternative, we first separate the products of three conditionally
centered edge variables from the terms involving conditional edge means.
We present this decomposition and the resulting convergence argument
below, using the conditional fourth-moment calculations and remainder estimates from
the proof of Lemma~\ref{lem:cubic-critical-triangle-clt}.

\begin{proof}[Proof of Theorem~\ref{thm:pure-cubic-triangle-clt}]
Write \(M_n=\binom n3\) and recall from
\eqref{eq:cubic-triangle-statistic} that \(\widehat T_n=T_n/\sqrt{M_n}\).
Under \(\Qq_n\), the variables
\(Y_{ij}=2a_{ij}-1\), \(i<j\), are independent Rademacher variables.
Proposition~\ref{prop:cubic-triangle-moments} gives
\(\E_{\Qq_n}\widehat T_n=0\) and
\(\operatorname{Var}_{\Qq_n}(\widehat T_n)=1\).
Each edge belongs to \(n-2\) triangles, so the sum of squared coefficients
of terms containing any fixed edge is \((n-2)/M_n\to0\).
For four triangles with edge sets \(\Delta_1,\ldots,\Delta_4\), let
\(\nu_e=\sum_{a=1}^4\mathbf1_{\{e\in\Delta_a\}}\) be the number
of occurrences of edge \(e\), and write \(Y_e=Y_{ij}\) for
\(e=\{i,j\}\).  Independence gives
\[
 \E_{\Qq_n}\prod_{a=1}^4\prod_{e\in\Delta_a}Y_e
 =\prod_e\E_{\Qq_n}Y_e^{\nu_e}
 =\prod_e\mathbf1_{\{\nu_e\text{ is even}\}}.
\]
If a triangle is repeated, this parity condition forces four identical
triangles or two distinct triangles each appearing twice.  If all four
are distinct, each edge of the first must occur in a different remaining
triangle.  After relabeling, their vertex triples are
\((1,2,3),(1,2,a),(1,3,b),(2,3,c)\), with
\(a,b,c\notin\{1,2,3\}\).
Even multiplicity of the edges incident to vertices \(1\) and \(2\)
forces \(a=b=c\).  These are the four faces of a tetrahedron, for which
\[
 (Y_{12}Y_{23}Y_{31})(Y_{12}Y_{24}Y_{41})
 (Y_{13}Y_{34}Y_{41})(Y_{23}Y_{34}Y_{42})
 =\prod_{1\le i<j\le4}Y_{ij}^2=1.
\]
Counting the orderings in these three cases gives
\begin{equation}
 \begin{aligned}
 \E_{\Qq_n}T_n^4
 &=M_n+\frac{4!}{2!2!}\binom{M_n}{2}+4!\binom n4,\\
 \E_{\Qq_n}\widehat T_n^4
 &=3-\frac2{M_n}+\frac{24\binom n4}{M_n^2}\longrightarrow3.
 \end{aligned}
 \label{eq:cubic-null-triangle-fourth-moment}
\end{equation}
De Jong's central limit theorem \cite[Theorem~1]{deJong1990}, applied to
this normalized sum of products of three distinct centered edge variables,
therefore yields \(\widehat T_n\Rightarrow\Normal(0,1)\).

For the alternative model, \(\E\) below averages over both the latent
sample \(X\) and the graph.  Write
\(m_{ij}:=\E(Y_{ij}\mid X)=\beta Z_{3,d}(X_i,X_j)\) and
\(\varepsilon_{ij}:=Y_{ij}-m_{ij}\).  Conditional on \(X\), the
\(\varepsilon_{ij}\) are independent and centered, with variance
\(1-m_{ij}^2\).  Substituting \(Y_{ij}=\varepsilon_{ij}+m_{ij}\) into
\eqref{eq:cubic-triangle-statistic} gives
\begin{equation}
 T_n=\sum_{i<j<k}(\varepsilon_{ij}+m_{ij})
    (\varepsilon_{jk}+m_{jk})(\varepsilon_{ki}+m_{ki})
 =T_n^{(0)}+R_{1,n}+R_{2,n}+R_{3,n},
 \label{eq:cubic-triangle-mean-expansion}
\end{equation}
where
\[
 \begin{aligned}
 T_n^{(0)}
 &=\sum_{i<j<k}\varepsilon_{ij}\varepsilon_{jk}\varepsilon_{ki},\\
 R_{1,n}
 &=\sum_{i<j<k}\bigl(m_{ij}\varepsilon_{jk}\varepsilon_{ki}
   +\varepsilon_{ij}m_{jk}\varepsilon_{ki}
   +\varepsilon_{ij}\varepsilon_{jk}m_{ki}\bigr),\\
 R_{2,n}
 &=\sum_{i<j}\varepsilon_{ij}\sum_{k\ne i,j}m_{ik}m_{jk},\\
 R_{3,n}
 &=\sum_{i<j<k}m_{ij}m_{jk}m_{ki}.
 \end{aligned}
\]
Conditional centering gives
\(\E T_n^{(0)}=\E R_{1,n}=\E R_{2,n}=0\).
Proposition~\ref{prop:cubic-triangle-moments} therefore gives
\(\E R_{3,n}=\E T_n=M_n\beta^3/N^2=\sqrt{M_n}\,\gamma_n\), and hence
\begin{equation}
 \widehat T_n-\gamma_n
 =\frac{T_n^{(0)}}{\sqrt{M_n}}
  +\frac{R_{1,n}+R_{2,n}+R_{3,n}-\E R_{3,n}}{\sqrt{M_n}}.
 \label{eq:cubic-main-triangle-decomposition}
\end{equation}

Step 2 of the proof of Lemma~\ref{lem:cubic-critical-triangle-clt}
applies de Jong's theorem conditionally on \(X\), using the influence
and fourth-moment bounds in \eqref{eq:cubic-conditional-influence} and
\eqref{eq:cubic-conditional-fourth-error}.  It also verifies the variance
normalization and integrates the conditional characteristic functions
over \(X\sim\rho_n\), proving
\[
 \frac{\operatorname{Var}(T_n^{(0)}\mid X)}{M_n}
 \longrightarrow1\quad\text{in probability},\qquad
 \frac{T_n^{(0)}}{\sqrt{M_n}}\Rightarrow\Normal(0,1).
\]

Step 3 of the proof of Lemma~\ref{lem:cubic-critical-triangle-clt}
bounds the three remaining groups in
\eqref{eq:cubic-main-triangle-decomposition} and gives
\begin{equation}
 \frac{\E R_{1,n}^2+\E R_{2,n}^2+\operatorname{Var}(R_{3,n})}{M_n}
 \le C\left[\frac{\beta^2}{N}
  +\beta^4\left(\frac1{N^2}+\frac n{N^3}\right)
  +\beta^6\left(\frac1{N^3}+\frac n{N^4}\right)\right]=o(1).
 \label{eq:cubic-main-triangle-remainders}
\end{equation}
Indeed, \(n\beta^2=r_nNd\) gives
\(n\beta^4/N^3=r_n\beta^2d/N^2=o(1)\) and
\(n\beta^6/N^4=r_n\beta^4d/N^3=o(1)\); the other terms tend to zero
because \(N\to\infty\) and \(|\beta|\le1\).
Cauchy--Schwarz and \eqref{eq:cubic-main-triangle-remainders} show that
the second fraction in \eqref{eq:cubic-main-triangle-decomposition}
tends to zero in mean square.  Slutsky's theorem now gives
\(\widehat T_n-\gamma_n\Rightarrow\Normal(0,1)\) under \(\Pp_n^{(3)}\).
Finally, \eqref{eq:cubic-model-parameters} and \(N\sim d^3/6\) give
\(\gamma_n^2\to r^3\).
\end{proof}

Under the alternative, the mean shift \(\gamma_n\) comes from the sum
of products of three conditional edge means.
Equation~\eqref{eq:cubic-main-triangle-remainders} shows that, after
division by \(\sqrt{\binom n3}\), the centered part of this sum and
the terms containing one or two conditional means tend to zero in \(L^2\).
The alternative therefore changes the asymptotic mean of \(\widehat T_n\)
but leaves the variance of its Gaussian limit equal to one.
These Gaussian limits hold for every fixed \(r>0\), including
\(r=r_{\mathrm{tens}}\).

\subsection{The tensor threshold and likelihood truncation}
\label{subsec:cubic-variational-mechanism}

We now identify \(r_{\mathrm{tens}}\) through variational estimates
and use them to truncate the full likelihood
\(L_n=d\Pp_n^{(3)}/d\Qq_n\) below the threshold.
For \(0<r<r_{\mathrm{tens}}\), we construct
\(\widetilde L_n\ge0\), with
\(\E_{\Qq_n}\widetilde L_n=1\), satisfying
\begin{equation}
 \|L_n-\widetilde L_n\|_{L^1(\Qq_n)}\to0,\qquad
 \limsup\E_{\Qq_n}\widetilde L_n^2\le e^{r^3}.
 \label{eq:cubic-main-likelihood-interface}
\end{equation}
The \(L^1\) approximation preserves the graph law asymptotically, while
the sharp bound \(e^{r^3}\) matches the second moment of the Gaussian
likelihood ratio associated with the triangle's limiting mean shift.

For the construction in \eqref{eq:cubic-main-likelihood-interface},
choose a compact interval \(K\subset(0,r_{\mathrm{tens}})\) containing
\(r_n\) eventually, fix \(M_\star>4A+10\), and take
\(\mathcal G_n^\dagger\) from
Proposition~\ref{prop:cubic-raw-local-partition} for \(A,K,M_\star\).
Then \(\rho_n(\mathcal G_n^\dagger)=1-o(d^{-M_\star})\).
Write \(\pi_n^\dagger=\rho_n(\,\cdot\mid\mathcal G_n^\dagger)\).
For two arrays of latent vertex positions,
\(x=(x_1,\ldots,x_n)\) and \(x'=(x'_1,\ldots,x'_n)\) in
\((\Sph^{d-1})^n\), define
\[
 s_n(x,x')
 :=N\left\|\frac1n\sum_{i=1}^n
       q_{3,d}(x_i)q_{3,d}(x_i')^{\mathsf T}\right\|_{\HS}^2.
\]
Write \(L_x(\Adj)=\prod_{i<j}\{1+\beta Y_{ij}Z_{3,d}(x_i,x_j)\}\)
for the graph likelihood ratio conditional on \(X=x\), relative to
\(\Qq_n\).  Fix \(0<\delta\le\delta_0\) in the range of
Proposition~\ref{prop:cubic-raw-local-partition}.
Its local bound controls the integral of
\(\E_{\Qq_n}(L_xL_{x'})\) over \(s_n(x,x')<\delta\), with respect to
\(\pi_n^\dagger\otimes\pi_n^\dagger\).
To treat the complementary region, define
\[
 Z_{x,\delta}^\dagger(\Adj)
 :=\int_{\{s_n(x,x')\ge\delta\}}L_{x'}(\Adj)\,d\pi_n^\dagger(x').
\]
We seek \(\E_{\Adj\mid X=x}\log Z_{x,\delta}^\dagger\le-g_0nd\)
for some \(g_0>0\), outside a set of \(\pi_n^\dagger\)-probability
\(o(1)\).  Conditional concentration will then allow us to restrict
to graphs with \(Z_{x,\delta}^\dagger\le e^{-g_0nd/2}\), losing
\(o(1)\) likelihood mass.

After comparison with Gaussian observations, the interpolation in
Proposition~\ref{prop:cubic-Gaussian-variational-bound} bounds
\((nd)^{-1}\E_{\Adj\mid x}\log Z_{x,\delta}^\dagger\) using a
Gaussian observation of a single harmonic feature, with the
normalization \(n^2\beta^2/N=r_nnd\).
Let \(\mu_d\) be the law of \(q_{3,d}(X)\) for \(X\sim\sigma_d\), and
let \(\gamma_N=\Normal(0,I_N)\).
For a positive-semidefinite \(N\times N\) matrix \(B\), let \(P_B\)
be the law of \(B^{1/2}u+g\), where \(u\sim\mu_d\) and \(g\sim\gamma_N\)
are independent.  The interpolation leads to a matrix potential of the
form used in finite-rank matrix estimation \cite{LelargeMiolane2017}.
In the present normalization, it is
\[
 \mathcal V_{d,r}(Q)
 :=\frac1dD(P_{rNdQ}\Vert\gamma_N)
   -\frac{rN}{4}\|Q\|_{\HS}^2,
 \qquad Q\succeq0.
\]
Here \(Q\) specifies the strength and directions of the Gaussian
observation, and the quadratic penalty comes from the graph interpolation.

For \(Q=zI_N/N\), the Gaussian signal is
\(\sqrt{rzd}\,q_{3,d}(X)\).
Equation~\eqref{eq:cubic-scalar-channel-limit}, derived from the
RS variational formula for the spherical tensor model
\cite[Appendix~A.3]{JagannathLopattoMiolane2020}, gives
\[
 \mathcal V_{d,r}(zI_N/N)\longrightarrow
 v_r(z):=\mathcal D_3(rz)-\frac r4z^2
\]
locally uniformly for \(z\ge0\), where
\(\mathcal D_3(s):=\tfrac12\max_{0\le q<1}
\{sq^3+\log(1-q)+q\}\).
Optimizing also over the scalar strength \(z\) yields
\begin{equation}
 \begin{split}
 \sup_{z\ge0}v_r(z)
 &=\sup_{\substack{z\ge0\\0\le q<1}}
   \left[\frac r2zq^3-\frac r4z^2
                 +\frac12\{\log(1-q)+q\}\right]\\
 &=\sup_{0\le q<1}
   \left[\frac r4q^6+\frac12\{\log(1-q)+q\}\right].
 \end{split}
 \label{eq:cubic-main-graph-potential}
\end{equation}
Indeed, for fixed \(q\), completing the square gives \(z=q^3\).
The last supremum in \eqref{eq:cubic-main-graph-potential} becomes
positive exactly when \(r>r_{\mathrm{tens}}\), with
\(r_{\mathrm{tens}}\) as defined in
Theorem~\ref{thm:pure-cubic-critical}.

By Lemma~\ref{lem:cubic-scalar-threshold}, the unique minimizer
\(q_\star\in(4/5,1)\) defining \(r_{\mathrm{tens}}\) is characterized
by the first identity below.  The second follows by substitution into
the formula for \(r_{\mathrm{tens}}\).
\begin{equation}
 \frac{q_\star^2}{1-q_\star}
 =6\{-\log(1-q_\star)-q_\star\},
 \qquad
 r_{\mathrm{tens}}=\frac1{3q_\star^4(1-q_\star)}.
 \label{eq:cubic-main-threshold-stationarity}
\end{equation}
Numerically, \(r_{\mathrm{tens}}\approx5.274\).
The scalar objective in \eqref{eq:cubic-main-graph-potential} is
\(-q^2/4+O(q^3)\) as \(q\downarrow0\), so zero remains a strict
local maximizer for every fixed \(r\).  At \(r=r_{\mathrm{tens}}\),
the distinct points \(0\) and \(q_\star\) both attain the maximum zero.

For \(r>r_{\mathrm{tens}}\), a scalar matrix with positive limiting
potential provides the lower bound for strong detection in
Section~\ref{subsec:cubic-supercritical-mechanism}.
For \(0<r<r_{\mathrm{tens}}\), controlling
\(\E_{\Adj\mid x}\log Z_{x,\delta}^\dagger\) requires an upper bound
over all matrix directions.
Proposition~\ref{prop:cubic-matrix-branch-gap} extends the scalar
analysis to the uniform estimate
\begin{equation*}
 \mathcal V_{d,r}(Q)\le-\kappa_KN\|Q\|_{\HS}^2
 \quad(r\in K,\ Q\succeq0,\ d\ \hbox{sufficiently large}),
\end{equation*}
with \(\kappa_K>0\) for the compact interval \(K\) chosen above.
In particular, \(\mathcal V_{d,r}(Q)\le-\kappa_K\delta\) whenever
\(N\|Q\|_{\HS}^2\ge\delta>0\).

To bound \(\E_{\Adj\mid x}\log Z_{x,\delta}^\dagger\),
Proposition~\ref{prop:cubic-Gaussian-variational-bound} combines Gaussian
interpolation with covering and polar decomposition.
Applying the Bernoulli comparison in
Lemma~\ref{lem:cubic-constrained-edge-replacement}, as in Step~1 of the
proof of Lemma~\ref{lem:cubic-legal-truncation}, and choosing the
covering error sufficiently small, we obtain
\[
 \begin{split}
 \frac1{nd}\E_{\Adj\mid x}\log Z_{x,\delta}^\dagger
 &\le\sup_{\substack{Q\succeq0\\N\|Q\|_{\HS}^2\ge\delta/2}}
       \mathcal V_{d,r_n}(Q)+\frac{\kappa_K\delta}{4}+o(1)
 \le-\frac{\kappa_K\delta}{4}+o(1),
 \end{split}
\]
uniformly for \(x\in\mathcal G_n^\dagger\setminus\mathcal E_n\), where
\(\pi_n^\dagger(\mathcal E_n)=o(1)\).  The constraint
\(N\|Q\|_{\HS}^2\ge\delta/2\) follows from the covering construction in
Proposition~\ref{prop:cubic-Gaussian-variational-bound}.  That proposition and
Lemma~\ref{lem:cubic-legal-truncation} use the same prior
\(\pi_n^\dagger\) chosen above.
Taking \(g_0=\kappa_K\delta/8\), we obtain, for all sufficiently large \(n\),
\begin{equation}
 \E_{\Adj\mid x}\log Z_{x,\delta}^\dagger\le-g_0nd
 \quad(x\in\mathcal G_n^\dagger\setminus\mathcal E_n),
 \qquad \pi_n^\dagger(\mathcal E_n)=o(1).
 \label{eq:cubic-main-constrained-mean}
\end{equation}
We take \(\log Z_{x,\delta}^\dagger=-\infty\) when
\(\pi_n^\dagger\{x':s_n(x,x')\ge\delta\}=0\).

The negative conditional mean in \eqref{eq:cubic-main-constrained-mean},
together with concentration over the graph edges as in the proof of
Lemma~\ref{lem:cubic-legal-truncation}, yields, after averaging over
\(X\sim\pi_n^\dagger\) and setting \(g=g_0/2\),
\[
 \mathbb P_{X\sim\pi_n^\dagger,\ \Adj\mid X}
   \{Z_{X,\delta}^\dagger(\Adj)>e^{-gnd}\}=o(1).
\]
Accordingly, we truncate the conditional likelihoods \(L_x\) by setting
\[
 \mathcal A_{x,\delta}^\dagger
 :=\{\Adj:Z_{x,\delta}^\dagger(\Adj)\le e^{-gnd}\},
 \qquad
 \overline L_n(\Adj)
 :=\int L_x(\Adj)\1_{\mathcal A_{x,\delta}^\dagger}(\Adj)
       \,d\pi_n^\dagger(x).
\]
Since \(L_x\) is the likelihood ratio conditional on \(X=x\), Fubini's
theorem identifies the discarded likelihood mass with this probability,
\[
 1-\E_{\Qq_n}\overline L_n
 =\int\mathbb P_{\Adj\mid x}\!\left\{
       Z_{x,\delta}^\dagger>e^{-gnd}\right\}\,d\pi_n^\dagger(x)
 =o(1).
\]
Moreover, \(0\le\overline L_n\le\int L_x\,d\pi_n^\dagger(x)\).
Together with the prior-conditioning estimate
\eqref{eq:cubic-regular-prior-likelihood-transfer}, this gives
\(\|L_n-\overline L_n\|_{L^1(\Qq_n)}
\le o(1)+1-\E_{\Qq_n}\overline L_n=o(1)\).

The same truncation also controls \(\E_{\Qq_n}\overline L_n^2\), whose
pair-integral expansion we split at \(s_n=\delta\).
On \(s_n\ge\delta\), the restriction
\(\mathcal A_{x,\delta}^\dagger\) enforces
\(Z_{x,\delta}^\dagger\le e^{-gnd}\), so
\eqref{eq:cubic-truncation-pair} bounds that part by \(e^{-gnd}\).
On \(s_n<\delta\), discard
\(\1_{\mathcal A_{x,\delta}^\dagger}\1_{\mathcal A_{x',\delta}^\dagger}\).
We obtain
\[
 \E_{\Qq_n}\overline L_n^2\le e^{-gnd}
 +\iint_{\{s_n(x,x')<\delta\}}
   \E_{\Qq_n}(L_xL_{x'})\,d\pi_n^\dagger(x)\,d\pi_n^\dagger(x').
\]
By \eqref{eq:cubic-exact-second-moment},
Proposition~\ref{prop:cubic-raw-local-partition} bounds the last integral
by \(e^{r^3}+o(1)\), since conditioning both samples contributes the
factor \(\rho_n(\mathcal G_n^\dagger)^{-2}=1+o(1)\).  Consequently,
\(\limsup\E_{\Qq_n}\overline L_n^2\le e^{r^3}\).

Finally, the \(L^1(\Qq_n)\) approximation gives
\(c_n:=\E_{\Qq_n}\overline L_n\to1\).  Set
\(\widetilde L_n:=\overline L_n/c_n\), so that
\(\E_{\Qq_n}\widetilde L_n=1\).  The identities
\(\|\overline L_n-\widetilde L_n\|_{L^1(\Qq_n)}=|c_n-1|\) and
\(\E_{\Qq_n}\widetilde L_n^2=c_n^{-2}\E_{\Qq_n}\overline L_n^2\),
together with the triangle inequality, establish both estimates in
\eqref{eq:cubic-main-likelihood-interface}.

\subsection{Gaussian likelihood asymptotics below the threshold}
\label{subsec:cubic-likelihood-mechanism}

For \(0<r<r_{\mathrm{tens}}\),
Section~\ref{subsec:cubic-variational-mechanism} provides a normalized
truncation \(\widetilde L_n\) satisfying
\eqref{eq:cubic-main-likelihood-interface}.
The triangle limits in Theorem~\ref{thm:pure-cubic-triangle-clt} will
identify the conditional mean of any subsequential likelihood limit
given the limiting triangle.  The sharp second-moment bound \(e^{r^3}\)
then forces the likelihood limit to equal that conditional mean, and the
\(L^1\) approximation transfers the resulting expansion to \(L_n\).

Let \(\widetilde\Pp_n\) have density \(\widetilde L_n\) relative to
\(\Qq_n\).
By \eqref{eq:cubic-main-likelihood-interface},
\(\TV(\Pp_n^{(3)},\widetilde\Pp_n)
=\tfrac12\|L_n-\widetilde L_n\|_{L^1(\Qq_n)}=o(1)\),
so the alternative triangle limit in
Theorem~\ref{thm:pure-cubic-triangle-clt} also holds under
\(\widetilde\Pp_n\).
Since \(\gamma_n^2\to r^3\), every subsequence has a further one along
which \(\gamma_n\to\Delta\in\{-r^{3/2},r^{3/2}\}\).
Along this subsequence, the null limit in
Theorem~\ref{thm:pure-cubic-triangle-clt} and the \(L^2\) bound in
\eqref{eq:cubic-main-likelihood-interface} make
\((\widehat T_n,\widetilde L_n)\) jointly tight under \(\Qq_n\).
Let \((Z,L)\) be any limit in distribution of
\((\widehat T_n,\widetilde L_n)\) along a further subsequence, with
\(Z\sim\Normal(0,1)\).

The \(L^2(\Qq_n)\) bound in
\eqref{eq:cubic-main-likelihood-interface} makes
\(\{\widetilde L_n\}\) uniformly integrable under \(\Qq_n\).
For every bounded continuous \(f\), the identity
\(d\widetilde\Pp_n=\widetilde L_n\,d\Qq_n\) and the triangle limit
under \(\widetilde\Pp_n\) give
\[
 \E_{\Qq_n}\{\widetilde L_n f(\widehat T_n)\}
 =\E_{\widetilde\Pp_n}f(\widehat T_n)
 \longrightarrow\E f(Z+\Delta)
 =\E\{e^{\Delta Z-\Delta^2/2}f(Z)\}.
\]
Joint convergence and uniform integrability also show that the
left-hand side converges to \(\E\{Lf(Z)\}\).
Thus \(\E(L\mid Z)=e^{\Delta Z-\Delta^2/2}\), the likelihood ratio
of \(\Normal(\Delta,1)\) relative to \(\Normal(0,1)\).

The sharp upper bound in \eqref{eq:cubic-main-likelihood-interface}
now determines \(L\) itself.  Conditional Jensen's inequality and
lower semicontinuity give
\[
 e^{\Delta^2}=\E[\{\E(L\mid Z)\}^2]
 \le\E L^2
 \le\liminf_n\E_{\Qq_n}\widetilde L_n^2
 \le e^{r^3}=e^{\Delta^2}.
\]
Equality forces \(\operatorname{Var}(L\mid Z)=0\), and therefore
\(L=e^{\Delta Z-\Delta^2/2}\).
The law of \((Z,L)\) is therefore uniquely determined.
Along the chosen subsequence with \(\gamma_n\to\Delta\), tightness then
yields \((\widehat T_n,\widetilde L_n)
\Rightarrow(Z,e^{\Delta Z-\Delta^2/2})\) under \(\Qq_n\).
The continuous-mapping theorem gives
\(\widetilde L_n-e^{\gamma_n\widehat T_n-\gamma_n^2/2}
=o_{\Qq_n}(1)\).

The \(L^1\) approximation in \eqref{eq:cubic-main-likelihood-interface}
now transfers this relation to the original likelihood.
Since \(\widehat T_n=O_{\Qq_n}(1)\) and \(\gamma_n^2\to r^3\),
we have \(e^{-\gamma_n\widehat T_n+\gamma_n^2/2}=O_{\Qq_n}(1)\).
Consequently,
\[
 \frac{L_n}{e^{\gamma_n\widehat T_n-\gamma_n^2/2}}-1
 =\bigl(L_n-e^{\gamma_n\widehat T_n-\gamma_n^2/2}\bigr)
    e^{-\gamma_n\widehat T_n+\gamma_n^2/2}
 =o_{\Qq_n}(1).
\]
Taking logarithms and using \(\gamma_n^2\to r^3\) gives
\begin{equation}
 \log L_n=\gamma_n\widehat T_n-\frac12r^3+o_{\Qq_n}(1).
 \label{eq:cubic-main-log-likelihood}
\end{equation}
Every subsequence has a further subsequence along which
\eqref{eq:cubic-main-log-likelihood} holds, so this expansion is valid
along the original sequence, including when the sign of \(\beta\) changes.

Uniform integrability of \(\{L_n\}\) under \(\Qq_n\) follows from
the two bounds in \eqref{eq:cubic-main-likelihood-interface}.
Indeed, for \(t>0\),
\[
 \E_{\Qq_n}\{L_n\mathbf1_{\{L_n>t\}}\}
 \le 2\|L_n-\widetilde L_n\|_{L^1(\Qq_n)}
       +\frac4t\E_{\Qq_n}\widetilde L_n^2.
\]
Taking \(\limsup_{n\to\infty}\) and then \(t\to\infty\) gives
vanishing tail expectations.
Uniform integrability therefore allows us to pass from the distributional
limit given by \eqref{eq:cubic-main-log-likelihood} to the expectation
below.  Using the symmetry of the standard normal law, we obtain
\[
 \TV(\Pp_n^{(3)},\Qq_n)
 =\frac12\E_{\Qq_n}|L_n-1|
 \longrightarrow
 \frac12\E|e^{r^{3/2}Z-r^3/2}-1|
 =2\Phi(r^{3/2}/2)-1.
\]

\subsection{Strong detection and limits at the threshold}
\label{subsec:cubic-supercritical-mechanism}

The Gaussian limits for \(\widehat T_n\) under \(\Qq_n\) and for
\(\widehat T_n-\gamma_n\) under \(\Pp_n^{(3)}\) in
Theorem~\ref{thm:pure-cubic-triangle-clt} hold on both sides of
\(r_{\mathrm{tens}}\).
The full graph laws nevertheless pass from a total-variation limit
strictly between zero and one, established in
Section~\ref{subsec:cubic-likelihood-mechanism}, to strong detection
above the threshold.
For \(r>r_{\mathrm{tens}}\), we establish this separation by showing
that \(\log L_n\) exceeds a positive multiple of \(nd\) with probability
tending to one under \(\Pp_n^{(3)}\).

Let \(\Pp_n^{\rm reg}\) be the graph law with position distribution
\(\rho_n(\,\cdot\mid\mathcal R_{n,d})\), where
\(\rho_n=\sigma_d^{\otimes n}\) and \(\mathcal R_{n,d}\) is the
high-probability event from Lemma~\ref{lem:cubic-superpoly-regularity}.
Write \(L_n^{\rm reg}=d\Pp_n^{\rm reg}/d\Qq_n\).
To estimate \(D(\Pp_n^{\rm reg}\Vert\Qq_n)\), let
\(\Pp_{\rho_n}^{\rm G}\) be the law of
\((|\beta|Z_{3,d}(X_i,X_j)+g_{ij})_{i<j}\), where
\(X\sim\rho_n\) and the \(g_{ij}\) are independent
standard Gaussian variables, independent of \(X\).
Write \(\Qq^{\rm G}\) for the law of \((g_{ij})_{i<j}\).
Applying Lemma~\ref{lem:cubic-constrained-edge-replacement} and
accounting for conditioning as in
\eqref{eq:cubic-supercritical-mean-reduction}, we obtain
\begin{equation}
 \frac1{nd}\E_{\Pp_n^{\rm reg}}\log L_n^{\rm reg}
 =\frac1{nd}D(\Pp_n^{\rm reg}\Vert\Qq_n)
 \ge\frac1{nd}D(\Pp_{\rho_n}^{\rm G}\Vert\Qq^{\rm G})-o(1).
 \label{eq:cubic-main-Bernoulli-Gaussian-reduction}
\end{equation}
Here we used
\(0<\rho_n(\mathcal R_{n,d})\le1\) and
\(D(\Pp_{\rho_n}^{\rm G}\Vert\Qq^{\rm G})\ge0\) in
\eqref{eq:cubic-supercritical-mean-reduction}.

By \eqref{eq:cubic-main-Bernoulli-Gaussian-reduction}, it remains to
bound \(D(\Pp_{\rho_n}^{\rm G}\Vert\Qq^{\rm G})\) from below.
Step~2 of the proof of
Proposition~\ref{prop:cubic-supercritical-free-energy} establishes,
by Gaussian interpolation, that for every \(0\le z\le1\),
\begin{equation}
 \frac1{nd}D(\Pp_{\rho_n}^{\rm G}\Vert\Qq^{\rm G})
 \ge \mathcal V_{d,r_n}(zI_N/N)-\frac{\beta^2}{4d}.
 \label{eq:cubic-main-Gaussian-entropy-lower}
\end{equation}
Combining \eqref{eq:cubic-main-Bernoulli-Gaussian-reduction} and
\eqref{eq:cubic-main-Gaussian-entropy-lower} with the locally uniform
convergence in \eqref{eq:cubic-scalar-potential-uniform-limit} gives
\[
 \liminf_{n\to\infty}\frac1{nd}
 \E_{\Pp_n^{\rm reg}}\log L_n^{\rm reg}
 \ge\sup_{0\le z\le1}v_r(z)=:\psi(r)>0.
\]
Here \(\psi(r)>0\) follows from
\eqref{eq:cubic-main-graph-potential} and \(r>r_{\mathrm{tens}}\).
Together with the variance estimate
\eqref{eq:cubic-supercritical-regular-concentration}, this yields, for
all sufficiently large \(n\),
\[
 \E_{\Pp_n^{\rm reg}}\log L_n^{\rm reg}\ge\frac{\psi(r)}2nd,
 \qquad
 \operatorname{Var}_{\Pp_n^{\rm reg}}(\log L_n^{\rm reg})=o((nd)^2),
\]
so the standard deviation is of smaller order than the mean.
Step~4 of the proof of Proposition~\ref{prop:cubic-supercritical-free-energy} applies
Chebyshev's inequality, removes the conditioning, and uses Markov's
inequality under the null to obtain, for \(0<\eta(r)<\psi(r)/4\),
\[
 \Pp_n^{(3)}\{\log L_n\ge\eta(r)nd\}\longrightarrow1,
 \qquad
 \Qq_n\{\log L_n\ge\eta(r)nd\}\le e^{-\eta(r)nd}\longrightarrow0.
\]
Thus the likelihood-ratio test has vanishing Type I and Type II errors,
and \(\TV(\Pp_n^{(3)},\Qq_n)\to1\).

\begin{proof}[Completion of Theorem~\ref{thm:pure-cubic-critical}]
The relations in \eqref{eq:cubic-model-parameters} give
\(\gamma_n^2\to r^3\) and \(\Gamma_{3,n}\to6r^3\).
Section~\ref{subsec:cubic-likelihood-mechanism} proves the likelihood
expansion \eqref{eq:cubic-main-log-likelihood} and the total-variation
limit below the threshold.  The preceding test proves strong detection
above it.

For \(A>4\), Proposition~\ref{prop:cubic-endpoint} provides two sequences
with \(d_n\to\infty\), \(n\le d_n^A\), a fixed nonzero amplitude
\(\beta_0\in[-1,1]\), and \(r_n\to r_{\mathrm{tens}}\).
Their total-variation limits are
\(2\Phi(r_{\mathrm{tens}}^{3/2}/2)-1\) and \(1\), respectively,
which proves the endpoint assertion.
\end{proof}

\section{Discussion}

This paper identifies a transition within the critical scale of the pure
cubic model.  For \(0<r<r_{\mathrm{tens}}\), the signed triangle
determines the asymptotic log-likelihood ratio, and the total-variation
distance converges to a value strictly between zero and one.
For \(r>r_{\mathrm{tens}}\), the full graph admits strongly consistent
testing, while the centered and normalized signed triangle continues to
have a Gaussian limit.

The scaling relation \(N_{3,d}^{1/3}\asymp d\) explains why degree three is
special.  For a pure harmonic of fixed degree \(k\ge2\), its multiplicity
satisfies \(N_{k,d}\asymp d^k\) \cite[Chapter~1]{DaiXu2013}.
The same projection argument as in \eqref{eq:single-layer-spectrum}
gives a cube trace of \(\beta^3/N_{k,d}^2\).  The resulting cube-trace
scale \(n\beta^2\asymp N_{k,d}^{4/3}\) can be compared with
\(n\beta^2\asymp N_{k,d}d\).  At the cube-trace scale, the ratio
\(n\beta^2/(N_{k,d}d)\asymp d^{k/3-1}\) vanishes for \(k=2\), remains
of constant order for \(k=3\), and diverges for every fixed \(k\ge4\).
This comparison motivates a separate analysis of higher-degree kernels,
where the two scales no longer coincide.

For the cubic model itself, a finer description is needed near
\(r_{\mathrm{tens}}\).  When \(A>4\), the sequence-dependent limits
at \(r_{\mathrm{tens}}\)
raise the question of how total variation depends on the rate at which
\(r_n\) approaches the threshold.  Identifying the width and limiting
behavior of this transition window is a natural next problem.
The coexistence of the maximizers \(q=0\) and \(q=q_\star\) in
\eqref{eq:cubic-main-graph-potential} suggests comparing the corresponding
likelihood contributions beyond their leading \(nd\)-scale terms.

Another question is whether the testing limits extend beyond the
polynomial growth condition \(n\le d^A\).  Explicit bounds on the
approximation errors could help determine the admissible growth of
\(n\) relative to \(d\), while also providing convergence rates.

Computational questions remain as well.  The signed triangle gives a
polynomial-time strongly consistent test when \(r_n\to\infty\).
For \(r_n\to r>r_{\mathrm{tens}}\), the likelihood-ratio argument
establishes strong detection.  Whether this test, or another strongly
consistent test, admits a polynomial-time implementation throughout this
regime remains open.


\appendix
\section{A cubic partition-function bound}
\label{app:cubic-fixed-amplitude}

This appendix supplies the estimates used in
Proposition~\ref{prop:cubic-amplitude-tracked-reverse} to prove that the
conditioned likelihood second moment tends to one when
\(m\beta^2=o(N_{3,d}d)\).

We retain \(d=d_n\to\infty\) and \(\beta=\beta_n\), with
\(|\beta|\le1\), but state the estimates for an auxiliary sample size
\(m\le d^A\), where \(A<\infty\) is fixed.  The estimates are uniform
in \(m\), which need not tend to infinity.
Write \(N=N_{3,d}\) and \(\rho_m=\sigma_d^{\otimes m}\).
Constants may depend on the fixed growth exponent \(A\), but not on
\(m\), \(d\), or \(\beta\).

For independent samples \(X,X'\sim\rho_m\), write
\[
  h_{ij}=Z_{3,d}(X_i,X_j)Z_{3,d}(X_i',X_j'),\qquad
  W_{m,\beta}=\prod_{i<j}(1+\beta^2h_{ij}).
\]
Section~\ref{subsec:cubic-envelope} constructs an event \(\mathcal G_m\)
with \(\rho_m(\mathcal G_m)=1-o(1)\).
For \(\pi_m=\rho_m(\,\cdot\mid\mathcal G_m)\),
\eqref{eq:cubic-exact-second-moment} gives
\[
  \mathbb E_{\Qq_m}L_{\pi_m,\beta}^2
  =\frac{\mathbb E_{\rho_m\otimes\rho_m}\!
    [\1_{\mathcal G_m(X)}\1_{\mathcal G_m(X')}W_{m,\beta}]}
    {\rho_m(\mathcal G_m)^2}.
\]
Thus it suffices to bound the numerator by \(1+o(1)\).

Recall \(w_i=q_{3,d}(X_i)q_{3,d}(X_i')^{\mathsf T}\) and
\(S_r=\sum_{i=1}^r w_i\) for \(0\le r\le m\), with \(S_0=0\).
Then \(\|S_m\|_{\HS}^2=m+2\sum_{i<j}h_{ij}\).
For a sufficiently large fixed \(R\), set
\(\mathcal E_{m,R}=\mathcal G_m(X)\cap\mathcal G_m(X')\cap\{\|S_m\|_{\HS}^2\le Rm\}\).
Section~\ref{subsec:cubic-overlap-deletion} shows that the contribution from
\(\|S_m\|_{\HS}^2>Rm\) to the numerator above is \(o(1)\).
On \(\mathcal E_{m,R}\), the bound
\(\max_{i<j}|h_{ij}|\le C\log^3d/N=o(1)\) from
Section~\ref{subsec:cubic-envelope} permits the Taylor estimate
\eqref{eq:cubic-partition-taylor}, with \(n=m\) and
\(F_m^{(2)}(\beta)=\tfrac{\beta^2}{2}(\|S_m\|_{\HS}^2-m)
  -\tfrac{\beta^4}{2}\sum_{i<j}h_{ij}^2\).
Combining these estimates gives
\begin{equation}
\begin{aligned}
  &\mathbb E_{\rho_m\otimes\rho_m}\!
    [\1_{\mathcal G_m(X)}\1_{\mathcal G_m(X')}W_{m,\beta}]\\
  &\qquad\le
    \exp\left\{C\frac{\beta^6m^2\log^9 d}{N^3}\right\}
    \mathbb E_{\rho_m\otimes\rho_m}\!
      [\1_{\mathcal E_{m,R}}e^{F_m^{(2)}(\beta)}]+o(1).
\end{aligned}
\label{eq:cubic-roadmap-product-reduction}
\end{equation}
The cases \(\beta=0\) or \(m\le2\) are immediate, as noted in the proof of
Theorem~\ref{thm:cubic-amplitude-tracked-partition}.  In the remaining
case, it suffices to prove
\begin{equation}
  \log\mathbb E_{\rho_m\otimes\rho_m}\!
    [\1_{\mathcal E_{m,R}}e^{F_m^{(2)}(\beta)}]
  \le C\frac{\beta^6m^3}{N^4}.
\label{eq:cubic-roadmap-compensated-target}
\end{equation}

To prove \eqref{eq:cubic-roadmap-compensated-target},
Section~\ref{subsec:cubic-row-laplace} bounds the exponential moment in
\eqref{eq:cubic-compensated-row}, retaining the negative squared terms.
Section~\ref{subsec:cubic-prefix-moments} supplies the partial-sum moment
bounds, entropy identities, and hypothesis checks used to combine these
estimates in \eqref{eq:cubic-main-entropy-reduction}.
With \(\eta=\beta^2(m-2)/(3N^2)=o(1)\), the resulting reduction and the
linear bound from Section~\ref{subsec:cubic-linear-partition} give
\[
\begin{aligned}
 &\log\mathbb E_{\rho_m\otimes\rho_m}
   [\1_{\mathcal E_{m,R}}e^{F_m^{(2)}(\beta)}]\\
 &\;\le
   \eta\log\mathbb E_{\rho_m\otimes\rho_m}
   [\1_{\mathcal E_{m,R}}e^{\beta^2(\|S_m\|_{\HS}^2-m)/2}]
   +C\left(\eta\frac{\beta^4m^2}{N^2}
      +\frac{\beta^8m^3}{N^4}\right)
 &&\text{by Proposition~\ref{prop:cubic-linear-reduction}}\\[3pt]
 &\;\le C\eta\left(\frac{\beta^4m^2}{N^2}
                      +\frac{\beta^8m^3}{N^4}\right)
       +C\frac{\beta^8m^3}{N^4}
 &&\text{by Lemma~\ref{lem:cubic-linear-free-energy}}\\[3pt]
 &\;\le C\frac{\beta^6m^3}{N^4}
 &&\begin{gathered}
   \text{by the choice of }\eta,\\[-2pt]
   |\beta|\le1,\quad\beta^4m/N^2=o(1).
   \end{gathered}
\end{aligned}
\]
Combining
\eqref{eq:cubic-roadmap-product-reduction} and
\eqref{eq:cubic-roadmap-compensated-target} with the second-moment identity
above gives
\[
  \mathbb E_{\Qq_m}L_{\pi_m,\beta}^2
  \le\frac{\exp\left[C\left(\frac{\beta^6m^3}{N^4}
    +\frac{\beta^6m^2\log^9 d}{N^3}\right)\right]+o(1)}
    {\rho_m(\mathcal G_m)^2}
  =1+o(1),
\]
since \(m\beta^2=o(Nd)\), \(N\asymp d^3\), and
\(\rho_m(\mathcal G_m)=1-o(1)\).
Together with \(\mathbb E_{\Qq_m}L_{\pi_m,\beta}^2\ge1\), this proves the
required convergence.
Section~\ref{subsec:cubic-partition-proof} records the corresponding
bound when either latent sample is further restricted to a measurable
subset of \(\mathcal G_m\).

\subsection{Uniform bounds over sample subsets}
\label{subsec:cubic-envelope}

Lemma~\ref{lem:cubic-hereditary-envelope} gives bounds that hold
simultaneously for all subsets of an iid spherical sample.  These bounds
remain valid when points are deleted later in the proof, with \(m\)
still denoting the original sample size.

\begin{lemma}[Uniform bounds over sample subsets]
\label{lem:cubic-hereditary-envelope}
Fix \(A,M<\infty\).  Uniformly over \(m\le d^A\), there is an event
\(\mathcal G_{m,M}\subseteq(\Sph^{d-1})^m\) such that
\[
  \rho_m(\mathcal G_{m,M}^c)=o(d^{-M}).
\]
Membership in \(\mathcal G_{m,M}\) is unchanged by applying the same
orthogonal transformation to every point, by permuting the points, or
by replacing any individual point \(X_i\) with \(-X_i\).
On \(\mathcal G_{m,M}\), the following estimates hold for every coordinate set
\(\varnothing\ne I\subseteq[m]\), \(|I|=r\), and every
\(u\in\Sph^{d-1}\):
\begin{align}
  \left\|\sum_{i\in I}X_iX_i^{\mathsf T}\right\|_{\mathrm{op}}
  &\le C_{A,M}\left(1+\frac rd\log\frac{em}{r}\right),
  \label{eq:cubic-linear-hereditary}\\
  \sum_{i\in I}Z_{3,d}(u,X_i)^2
  &\le C_{A,M}\left\{1+\frac rN
    \left(\log\frac{em}{r}\right)^3\right\}.
  \label{eq:cubic-zonal-hereditary}
\end{align}
If \(\mathbf Q_m\) is the \(m\times N\) matrix whose \(i\)th row is
\(q_{3,d}(X_i)^{\mathsf T}\), then
\begin{equation}
  \|\mathbf Q_m\|_{\mathrm{op}}^2
  \le C_{A,M}\left(\frac mN+\log N\right).
  \label{eq:cubic-feature-gram}
\end{equation}
In addition, the kernel values at distinct sample points satisfy
\begin{equation}
  \max_{i\ne j}|Z_{3,d}(X_i,X_j)|
  \le C_{A,M}\frac{\log^{3/2} d}{\sqrt N}.
  \label{eq:cubic-pair-separation}
\end{equation}
\end{lemma}

\begin{proof}
First, we prove \eqref{eq:cubic-linear-hereditary}.  For \(X\sim\sigma_d\),
the sub-Gaussian bound for \(\sqrt d\,X\) in
\cite[Theorem~3.4.6]{Vershynin2018} gives, for every fixed
\(u\in\Sph^{d-1}\),
\begin{equation}
  \mathbb P\{d\langle u,X\rangle^2>t\}\le2e^{-ct},
  \qquad t\ge1.
  \label{eq:cubic-spherical-coordinate-tail}
\end{equation}
The variables \(d\langle u,X_i\rangle^2-1\) are independent and centered,
with uniformly bounded subexponential norms.  Bernstein's inequality
\cite[Theorem~2.8.1]{Vershynin2018} implies that, for a fixed
\(I\subseteq[m]\), \(|I|=r\), and every \(s\ge1\),
\[
  \mathbb P\left\{
    \sum_{i\in I}\langle u,X_i\rangle^2
    >\frac{C(r+s)}d
  \right\}
  \le2e^{-cs}.
\]
Choose a \(1/4\)-net \(\mathcal N\) of \(\Sph^{d-1}\) with
\(|\mathcal N|\le9^d\), as supplied by
\cite[Corollary~4.2.13]{Vershynin2018}, and take
  \(s=C_{A,M}\{d+r\log(em/r)\}\).  Since the number of \(r\)-subsets is at most
\((em/r)^r\), a union bound over \(u\in\mathcal N\), over all such subsets,
and then over \(1\le r\le m\) has failure probability \(o(d^{-M})\) when
\(C_{A,M}\) is chosen large enough.  The estimate for quadratic forms on a
net \cite[Exercise~4.4.3(b)]{Vershynin2018} then gives
\[
  \sup_{I:|I|=r}
  \left\|\sum_{i\in I}X_iX_i^{\mathsf T}\right\|_{\mathrm{op}}
  \le 2\sup_{I:|I|=r}\max_{u\in\mathcal N}
    \sum_{i\in I}\langle u,X_i\rangle^2
  \le C_{A,M}\left(1+\frac rd\log\frac{em}{r}\right),
\]
which is \eqref{eq:cubic-linear-hereditary}.

Next, we deduce \eqref{eq:cubic-zonal-hereditary}.  For each
\(u\in\Sph^{d-1}\), arrange the values \(|\langle u,X_i\rangle|^2\)
in nonincreasing order \(b_{(1)}(u)\ge\cdots\ge b_{(m)}(u)\).
Let \(X_{(j)}\) denote the corresponding sample point, so that
\(b_{(j)}(u)=|\langle u,X_{(j)}\rangle|^2\), with ties resolved
arbitrarily.  Since \eqref{eq:cubic-linear-hereditary} holds for every
subset, it applies to the first \(k\) points in this ordering.  Thus,
uniformly in \(u\in\Sph^{d-1}\) and \(1\le k\le m\),
\[
\begin{aligned}
  b_{(k)}(u)
  &\le\frac1k\sum_{j=1}^k b_{(j)}(u)
   =\frac1k u^{\mathsf T}\left(\sum_{j=1}^k
     X_{(j)}X_{(j)}^{\mathsf T}\right)u\\
  &\le\frac1k\left\|\sum_{j=1}^k
     X_{(j)}X_{(j)}^{\mathsf T}\right\|_{\mathrm{op}}
   \le C_{A,M}\left(\frac1k+\frac1d\log\frac{em}{k}\right).
\end{aligned}
\]
The formula \(Z_{3,d}(t)=\{(d+2)t^3-3t\}/(d-1)\) from
Section~\ref{subsec:cubic-harmonic-kernel} implies
\(|Z_{3,d}(t)|^2\le C(t^6+t^2/d^2)\) for \(d\ge2\).
Using \(N\asymp d^3\) and \(\log(em/k)\ge1\), the preceding bound for
\(b_{(k)}(u)\) therefore yields
\[
  b_{(k)}(u)^3+\frac{b_{(k)}(u)}{d^2}
  \le C_{A,M}\left\{
    k^{-3}+(kd^2)^{-1}
    +N^{-1}\log^3\frac{em}{k}
  \right\}.
\]
For every \(I\subseteq[m]\) with \(|I|=r\), monotonicity of
\(b\mapsto b^3+b/d^2\) on \([0,\infty)\) gives
\[
\begin{aligned}
  \sum_{i\in I}Z_{3,d}(u,X_i)^2
  &\le C\sum_{i\in I}\left(
    |\langle u,X_i\rangle|^6+\frac{|\langle u,X_i\rangle|^2}{d^2}
  \right)\\
  &\le C\sum_{k=1}^r\left\{b_{(k)}(u)^3+
    \frac{b_{(k)}(u)}{d^2}\right\}\\
  &\le C_{A,M}\left\{1+\frac rN
    \left(\log\frac{em}{r}\right)^3\right\}.
\end{aligned}
\]
The last inequality uses \(\sum_{k=1}^r k^{-3}\le C\),
\(d^{-2}\sum_{k=1}^r k^{-1}\le d^{-2}(1+\log m)=O_A(1)\),
and the integral comparison
\[
  \sum_{k=1}^r\left(\log\frac{em}{k}\right)^3
  \le C r\left(\log\frac{em}{r}\right)^3.
\]
This proves \eqref{eq:cubic-zonal-hereditary} simultaneously for all
\(I\) and \(u\).

For \eqref{eq:cubic-feature-gram}, the feature identities in
Section~\ref{subsec:cubic-harmonic-kernel} give
\(\mathbb E[q_{3,d}(X)q_{3,d}(X)^{\mathsf T}]=N^{-1}I_N\).  Matrix Bernstein
\cite[Theorem~6.1]{Tropp2012} yields
\[
  \left\|\mathbf Q_m^{\mathsf T}\mathbf Q_m-\frac mN I_N\right\|_{\mathrm{op}}
  \le C_{A,M}\left(
    \sqrt{\frac{m\log N}N}+\log N
  \right)
\]
with probability \(1-o(d^{-M})\), uniformly for \(m\le d^A\), after increasing
\(C_{A,M}\).  The inequality
\(2\sqrt{xy}\le x+y\) now gives \eqref{eq:cubic-feature-gram}.

It remains to prove \eqref{eq:cubic-pair-separation}.  For \(i\ne j\),
\(X_j\) is independent of \(X_i\) and has law \(\sigma_d\).
Equation~\eqref{eq:cubic-spherical-coordinate-tail} therefore gives, almost surely,
\[
  \mathbb P\left\{d\langle X_i,X_j\rangle^2>t\mid X_i\right\}
  \le2e^{-ct},\qquad t\ge1.
\]
Integrating over \(X_i\) and taking a union bound over \(i<j\) yields
\[
\begin{aligned}
  &\mathbb P\left\{\max_{i\ne j}|\langle X_i,X_j\rangle|
    >C_{A,M}\sqrt{\frac{\log d}d}\right\}\\
  &\qquad\le m(m-1)d^{-cC_{A,M}^2}
    \le d^{2A-cC_{A,M}^2}=o(d^{-M}),
\end{aligned}
\]
where \(C_{A,M}\) is chosen so that \(cC_{A,M}^2>2A+M\).
Substituting this inner-product bound into
\eqref{eq:cubic-zonal-polynomial} proves
\eqref{eq:cubic-pair-separation}.

Define \(\mathcal G_{m,M}\) by requiring
\eqref{eq:cubic-linear-hereditary}--\eqref{eq:cubic-pair-separation}
to hold for all the indicated subsets and directions.  The preceding
probability estimates give \(\rho_m(\mathcal G_{m,M}^c)=o(d^{-M})\).
To check its invariance, note that
\[
  \|\mathbf Q_m\|_{\mathrm{op}}^2
  =\|\mathbf Q_m\mathbf Q_m^{\mathsf T}\|_{\mathrm{op}}
  =\bigl\|[Z_{3,d}(X_i,X_j)]_{i,j\le m}\bigr\|_{\mathrm{op}}.
\]
For an orthogonal matrix \(O\), a permutation \(\tau\) of \([m]\), and
signs \(\varepsilon_i\in\{-1,1\}\), set
\(\widetilde X_i=\varepsilon_i OX_{\tau(i)}\).
Let \(P\) be the permutation matrix with
\(P_{ij}=\1_{\{j=\tau(i)\}}\), and put
\(D=\operatorname{diag}(\varepsilon_1,\ldots,\varepsilon_m)\).
Rotational invariance and oddness of \(Z_{3,d}\) give
\[
\begin{aligned}
  \sum_{i\in I}\widetilde X_i\widetilde X_i^{\mathsf T}
  &=O\left(\sum_{j\in\tau(I)}X_jX_j^{\mathsf T}\right)O^{\mathsf T},\\
  \sum_{i\in I}Z_{3,d}(u,\widetilde X_i)^2
  &=\sum_{j\in\tau(I)}Z_{3,d}(O^{\mathsf T}u,X_j)^2,\\
  [Z_{3,d}(\widetilde X_i,\widetilde X_j)]_{i,j\le m}
  &=DP[Z_{3,d}(X_i,X_j)]_{i,j\le m}P^{\mathsf T}D.
\end{aligned}
\]
Since \(|\tau(I)|=|I|\) and \(O^{\mathsf T}u\) ranges over the unit sphere,
the first two identities preserve
\eqref{eq:cubic-linear-hereditary}--\eqref{eq:cubic-zonal-hereditary}.
The last identity preserves both the operator norm and the maximum absolute
off-diagonal entry of the Gram matrix, giving invariance of
\eqref{eq:cubic-feature-gram}--\eqref{eq:cubic-pair-separation}.
Thus membership in \(\mathcal G_{m,M}\) is unchanged.
\end{proof}

In the remainder of this appendix, we write
\(\mathcal G_m=\mathcal G_{m,M}\) for a fixed \(M>0\) and suppress the
dependence of constants on \(M\).

Lemma~\ref{lem:cubic-regular-channel-endpoints} gives edge and likelihood
bounds after conditioning on \(\mathcal G_m\) or on the enlarged regularity
event used for the Gaussian comparison in
Appendix~\ref{app:cubic-franz-parisi}.

\begin{lemma}[Edge and likelihood bounds after conditioning]
\label{lem:cubic-regular-channel-endpoints}
Let \(\mathcal E_m\subseteq(\Sph^{d-1})^m\) have prior probability
\(p_m=\rho_m(\mathcal E_m)=1-o(1)\), and suppose that on this event
\[
 \max_{i<j}|Z_{3,d}(X_i,X_j)|\le\varepsilon_d=o(1).
\]
Put \(\pi_m=\rho_m(\,\cdot\mid\mathcal E_m)\).

Then the following bounds hold uniformly over \(|\beta|\le1\).
Under the graph model with prior \(\pi_m\), the edge signs
\(Y_{ij}=2a_{ij}-1\) satisfy, for \(X\in\mathcal E_m\) and \(i<j\),
\[
\begin{aligned}
  |\mathbb E(Y_{ij}\mid X)|
  &=|\beta Z_{3,d}(X_i,X_j)|\le\varepsilon_d,\\
  \operatorname{Var}(Y_{ij}\mid X)
  &=1-\beta^2Z_{3,d}(X_i,X_j)^2\ge1-\varepsilon_d^2.
\end{aligned}
\]
In particular, these variances are at least \(3/4\) for all sufficiently
large \(d\).  For \(X,X'\in\mathcal E_m\), we also have
\(|\beta^2h_{ij}|\le\varepsilon_d^2\) for every \(i<j\).

The original and conditioned graph likelihoods satisfy
\[
 \|L_{\rho_m,\beta}-L_{\pi_m,\beta}\|_{L^1(\Qq_m)}
 \le2(1-p_m)=o(1).
\]
Both likelihoods are strictly positive on every graph for all sufficiently
large \(d\).
In particular, \(\mathcal E_m=\mathcal G_m\) satisfies these hypotheses
with \(\varepsilon_d=C_A\log^{3/2} d/\sqrt N\).
\end{lemma}

\begin{proof}
The conditional identities
\(\mathbb E(Y_{ij}\mid X)=\beta Z_{3,d}(X_i,X_j)\) and
\(\mathbb E(Y_{ij}^2\mid X)=1\) give the moment bounds.
Multiplying the kernel bounds for \(X\) and \(X'\) gives
\(|\beta^2h_{ij}|\le\varepsilon_d^2\).

Splitting the likelihood integral over \(\mathcal E_m\) and its complement gives
\[
  L_{\rho_m,\beta}-p_mL_{\pi_m,\beta}
  =\mathbb E_{X\sim\rho_m}\!\left[
    \1_{\{X\notin\mathcal E_m\}}
    \prod_{i<j}\{1+\beta Y_{ij}Z_{3,d}(X_i,X_j)\}\right]
  \ge0.
\]
Since both likelihood ratios have expectation one under \(\Qq_m\),
\[
\begin{aligned}
  \|L_{\rho_m,\beta}-L_{\pi_m,\beta}\|_{L^1(\Qq_m)}
  &\le\mathbb E_{\Qq_m}(L_{\rho_m,\beta}-p_mL_{\pi_m,\beta})
    +(1-p_m)\mathbb E_{\Qq_m}L_{\pi_m,\beta}\\
  &=2(1-p_m).
\end{aligned}
\]
On \(\mathcal E_m\), each factor satisfies
\(1+\beta Y_{ij}Z_{3,d}(X_i,X_j)\ge1-\varepsilon_d\).  Thus, for every
graph and all sufficiently large \(d\),
\[
  L_{\rho_m,\beta}\ge p_mL_{\pi_m,\beta}
  \ge p_m(1-\varepsilon_d)^{\binom m2}>0.
\]
The final assertion follows from \eqref{eq:cubic-pair-separation}.
\end{proof}

On \(\mathcal G_m\), inequality~\eqref{eq:cubic-pair-separation} gives,
for every \(I\subseteq[m]\) with \(|I|=r\),
\begin{equation*}
  \sum_{\substack{i<j\\i,j\in I}}|Z_{3,d}(X_i,X_j)|^6
  \le\binom r2\max_{i\ne j}|Z_{3,d}(X_i,X_j)|^6
  \le C_A\frac{r^2\log^9 d}{N^3}.
\end{equation*}
For \(r=o(Nd)\), we have
\(r^2\log^9 d/N^3
  =\{r/(Nd)\}^2\,d^2\log^9 d/N=o(1)\).

\subsection{Likelihood-weighted overlap tail bounds}
\label{subsec:cubic-overlap-deletion}

For matrices \(H,K\in\mathbb R^{N\times N}\), we use the
inner product \(\langle H,K\rangle=\tr(H^{\mathsf T}K)\) and
Hilbert--Schmidt norm \(\|H\|_{\HS}=\langle H,H\rangle^{1/2}\).
The feature representation in Section~\ref{subsec:cubic-harmonic-kernel}
gives \(h_{ij}=\langle w_i,w_j\rangle\).  Put
\[
  \kappa_m=\frac{d}{N^{-1}+(\log N)/m}.
\]
For a function \(f:(\Sph^{d-1})^m\to\mathbb R\), write
\[
  \operatorname{Lip}(f)
  :=\sup_{\substack{x',y'\in(\Sph^{d-1})^m\\x'\ne y'}}
    \frac{|f(x')-f(y')|}{\|x'-y'\|_2},
\]
where \(\|x'-y'\|_2\) is the Euclidean norm of the difference of the
sample vectors in \(\mathbb R^{md}\).

Proposition~\ref{prop:cubic-weighted-overlap} shows that, for a sufficiently
large fixed \(R\), the likelihood-weighted contribution from \(\|S_m\|_{\HS}^2>Rm\)
is at most \(e^{-c_R\kappa_m}=o(1)\) when \(m\beta^2=o(Nd)\).
This bounds the contribution from \(\|S_m\|_{\HS}^2>Rm\) to the \(o(1)\) remainder
in \eqref{eq:cubic-roadmap-product-reduction}.

\begin{proposition}[Weighted tail bound for the overlap]
\label{prop:cubic-weighted-overlap}
There are constants \(R<\infty\) and \(c_R>0\) such that, uniformly for
\(|\beta|\le1\),
\begin{equation}
  \mathbb E_{\rho_m\otimes\rho_m}\left[
    \1_{\mathcal G_m(X)}\1_{\mathcal G_m(X')}
    \1_{\{\|S_m\|_{\HS}^2>Rm\}}
    W_{m,\beta}
  \right]
  \le e^{-c_R\kappa_m},
  \label{eq:cubic-weighted-overlap}
\end{equation}
whenever \(m\beta^2=o(Nd)\) and \(d\) is sufficiently large.
\end{proposition}

\begin{proof}
Fix \(X\in\mathcal G_m\), and let
\(\mathbf Q,\mathbf R\in\mathbb R^{m\times N}\) have rows
\(q_{3,d}(X_i)^{\mathsf T}\) and \(q_{3,d}(X_i')^{\mathsf T}\), respectively,
where \(X'\sim\rho_m\).  Define
\[
  F(X'):=\frac{\|\mathbf Q^{\mathsf T}\mathbf R\|_{\HS}}{\sqrt m}
  =\frac{\|S_m\|_{\HS}}{\sqrt m}.
\]
Since \(\sum_{i<j}h_{ij}=(\|S_m\|_{\HS}^2-m)/2\) and
\(0\le1+\beta^2h_{ij}\le e^{\beta^2h_{ij}}\),
\[
  0\le W_{m,\beta}
  \le e^{\beta^2(\|S_m\|_{\HS}^2-m)/2}
  \le e^{\beta^2mF(X')^2/2}.
\]
Thus it suffices to prove, uniformly over \(X\in\mathcal G_m\),
\[
  \mathbb E_{X'}\!\left[
    e^{\beta^2mF(X')^2/2}\1_{\{F(X')>\sqrt R\}}\mid X
  \right]\le e^{-c_R\kappa_m}.
\]

We first control the tail of \(F\).  The oddness and unit norm of
\(q_{3,d}\), together with independence of the \(X_i'\),
give \(\mathbb E_{X'}(\mathbf R\mathbf R^{\mathsf T})=I_m\).  Hence
\[
  \mathbb E_{X'}F(X')^2
  =\frac1m\tr\left\{\mathbf Q^{\mathsf T}
    \mathbb E_{X'}(\mathbf R\mathbf R^{\mathsf T})\mathbf Q\right\}
  =\frac1m\tr(\mathbf Q^{\mathsf T}\mathbf Q)=1,
\]
and Cauchy--Schwarz gives \(\mathbb E_{X'}F(X')\le1\).

For the Lipschitz bound, use the identity
\(Z_{3,d}(y,y')=\langle q_{3,d}(y),q_{3,d}(y')\rangle\) from
Section~\ref{subsec:cubic-harmonic-kernel}.  The polynomial formula
\eqref{eq:cubic-zonal-polynomial} gives
\[
  L_{3,d}^2:=\sup_{-1\le t\le1}|Z_{3,d}'(t)|
  =Z_{3,d}'(1)=\frac{3(d+1)}{d-1}.
\]
Thus, for \(t=\langle y,y'\rangle\),
\begin{equation}
  \|q_{3,d}(y)-q_{3,d}(y')\|^2
  =2\{1-Z_{3,d}(t)\}
  \le2L_{3,d}^2(1-t)=L_{3,d}^2\|y-y'\|_2^2.
  \label{eq:cubic-feature-Lipschitz}
\end{equation}
Writing \(\mathbf R(x')\) and \(\mathbf R(y')\) for the corresponding
feature matrices, the triangle inequality and
\eqref{eq:cubic-feature-Lipschitz} give
\[
  |F(x')-F(y')|
  \le\frac{\|\mathbf Q\|_{\mathrm{op}}}{\sqrt m}
    \|\mathbf R(x')-\mathbf R(y')\|_{\HS}
  \le\frac{L_{3,d}\|\mathbf Q\|_{\mathrm{op}}}{\sqrt m}
    \|x'-y'\|_2.
\]
Consequently, \eqref{eq:cubic-feature-gram} yields
\begin{equation}
  \operatorname{Lip}(F)^2
  \le \frac{L_{3,d}^2\|\mathbf Q\|_{\mathrm{op}}^2}{m}
  \le C_A\left(\frac1N+\frac{\log N}m\right).
  \label{eq:cubic-overlap-Lipschitz}
\end{equation}
The spherical log-Sobolev inequality has constant \(O(d^{-1})\)
\cite[Corollary~2 and the spherical example, pp.~199--200]{BakryEmery1985}.
By tensorization \cite[Theorem~5.7]{Meckes2019}, the same bound holds for
\(\rho_m=\sigma_d^{\otimes m}\), uniformly in \(m\).
The Euclidean Lipschitz bound \eqref{eq:cubic-overlap-Lipschitz} also bounds
the squared norm of the tangential gradient on the product sphere.  The Herbst argument
\cite[proof of Theorem~5.5]{Meckes2019}, together with
\(\mathbb E_{X'}F\le1\), therefore gives
\begin{equation}
  \mathbb P_{X'}\{F(X')\ge1+t\mid X\}
  \le \exp\!\left(-\frac{cdt^2}{N^{-1}+(\log N)/m}\right)
  =\exp(-c\kappa_mt^2),
  \qquad t\ge0.
  \label{eq:cubic-conditional-overlap-tail}
\end{equation}

The assumption \(m\beta^2=o(Nd)\) implies
\(m\beta^2/\kappa_m=m\beta^2/(Nd)+\beta^2(\log N)/d=o(1)\).
The conditional tail-integral formula gives
\begin{align*}
  &\mathbb E_{X'}\!\left[
    e^{\beta^2mF(X')^2/2}\1_{\{F(X')>\sqrt R\}}\mid X
  \right]\\
  &\quad= e^{\beta^2mR/2}\mathbb P_{X'}\{F>\sqrt R\mid X\}
  +\frac{\beta^2m}{2}\int_R^\infty e^{\beta^2mu/2}
       \mathbb P_{X'}\{F^2>u\mid X\}\,du.
\end{align*}
For \(u\ge R\ge4\), one has \((\sqrt u-1)^2\ge u/4\).  Since
\(\beta^2m/\kappa_m=o(1)\), the tail bound
\eqref{eq:cubic-conditional-overlap-tail} makes the last display at most
\(e^{-c_R\kappa_m}\) after \(R\) is chosen large enough.
Inserting \(\1_{\mathcal G_m(X')}\le1\) and integrating over
\(X\in\mathcal G_m\) proves \eqref{eq:cubic-weighted-overlap}.
\end{proof}

\subsection{A one-row exponential-moment bound}
\label{subsec:cubic-row-laplace}

To estimate the exponential moment in
\eqref{eq:cubic-roadmap-compensated-target}, we first decompose
\(F_m^{(2)}(\beta)\) into contributions from individual rows.  With
\(S_{i-1}=\sum_{j<i}w_j\), the identity \(h_{ij}=\langle w_i,w_j\rangle\)
gives
\[
  F_m^{(2)}(\beta)
  =\sum_{i=1}^m\left(
    \beta^2\langle S_{i-1},w_i\rangle
    -\frac{\beta^4}{2}\sum_{j<i}\langle w_j,w_i\rangle^2
  \right).
\]
Recall \(P_0=\sigma_d\otimes\sigma_d\).  For \((X,X')\sim P_0\), write
\(w=q_{3,d}(X)q_{3,d}(X')^{\mathsf T}\).
For \(H\in\mathbb R^{N\times N}\), let
\[
  \ell(H)=\sup_{x,y\in\Sph^{d-1}}
  |\langle H,q_{3,d}(x)q_{3,d}(y)^{\mathsf T}\rangle|
\]
be the largest absolute value of \(\langle H,w\rangle\) over all possible
feature matrices \(w\).  Then we have
\(\bigl|\sum_{j<i}h_{ij}\bigr|\le\ell(S_{i-1})\).

Lemma~\ref{lem:cubic-compensated-row} estimates the exponential moment of
one row under \(P_0\), with \(w_1,\ldots,w_{i-1}\) fixed.  The summand
indexed by \(i\) corresponds to \(r=i-1\), \(v_j=w_j\),
\(u=\beta^2\), and \(c_0=1\).  Its leading coefficient involves
\(\|S_{i-1}\|_{\HS}^2-(i-1)=2\sum_{j<k<i}h_{jk}\), so the diagonal
terms cancel.

\begin{lemma}[Exponential-moment bound with a quadratic correction]
\label{lem:cubic-compensated-row}
There is a universal constant \(c_{\rm src}>0\) with the following
property.  For deterministic points
\(x_1,\ldots,x_r,x_1',\ldots,x_r'\in\Sph^{d-1}\), put
\[
  v_j=q_{3,d}(x_j)q_{3,d}(x_j')^{\mathsf T},\quad 1\le j\le r,
  \qquad S=\sum_{j=1}^r v_j.
\]
Fix \(\rho<\infty\), and let \(|u|\le\rho\) and \(0\le c_0\le1\).
Suppose that, as \(d\to\infty\),
\[
  2|u|\ell(S)\le c_{\rm src}d,
  \qquad \frac{u^2(r+\|S\|_{\HS}^2)}{N^2}=o(1).
\]
Then, for all sufficiently large \(d\),
\begin{equation}
  \log\mathbb E_{P_0}
  \exp\left\{u\langle S,w\rangle
  -\frac{c_0u^2}{2}\sum_{j=1}^r\langle v_j,w\rangle^2\right\}
  \le \frac{u^2(\|S\|_{\HS}^2-c_0r)}{2N^2}
  +C_\rho u^4\frac{\|S\|_{\HS}^4+r^2}{N^4}.
  \label{eq:cubic-compensated-row}
\end{equation}
The constant \(C_\rho\) depends only on \(\rho\), and the bound is uniform
over the choices above.  For \(c_0=0\), the numerator
\(\|S\|_{\HS}^4+r^2\) in \eqref{eq:cubic-compensated-row} can be replaced
by \(\|S\|_{\HS}^4\).
\end{lemma}

\begin{proof}
For scalar random variables, write
\(\|U\|_p=(\mathbb E_{P_0}|U|^p)^{1/p}\), \(1\le p<\infty\).
Write \(f(w)=\langle S,w\rangle\) and
\(g(w)=\sum_{j=1}^r\langle v_j,w\rangle^2\), and define
\[
  Z_{c_0}(t)=\mathbb E_{P_0}
  \exp\left\{tf-\frac{c_0t^2}{2}g\right\}.
\]
The left-hand side of \eqref{eq:cubic-compensated-row} is
\(\log Z_{c_0}(u)\).  We will justify the Taylor bound
\[
  \log Z_{c_0}(u)\le Z_{c_0}(u)-1
  \le\frac{u^2}{2}Z_{c_0}''(0)
    +\frac{u^4}{24}\sup_{|t|\le|u|}|Z_{c_0}^{(4)}(t)|.
\]
Step 1 verifies evenness and computes
\(Z_{c_0}''(0)=(\|S\|_{\HS}^2-c_0r)/N^2\), giving the main term in
\eqref{eq:cubic-compensated-row}.  To control the remainder, Step 2
bounds the exponential weights that occur in \(Z_{c_0}^{(4)}\), and
Step 3 uses this bound to estimate the fourth derivative and conclude.

\medskip
\noindent\textit{Step 1. Compute the quadratic coefficient.}
By definition, \(Z_{c_0}(0)=1\).  Since \(w\) and \(-w\) have the same
law, \(Z_{c_0}\) is even, so \(Z_{c_0}'(0)=Z_{c_0}'''(0)=0\).
The same symmetry gives \(\mathbb E_{P_0}f=0\), while isotropy yields
\(\mathbb E_{P_0}f^2=\|S\|_{\HS}^2/N^2\) and
\(\mathbb E_{P_0}g=r/N^2\).  Hence
\[
  Z_{c_0}''(0)=\mathbb E_{P_0}f^2-c_0\mathbb E_{P_0}g
  =\frac{\|S\|_{\HS}^2-c_0r}{N^2}.
\]
For \(c_0=1\), the diagonal contribution cancels, since
\(\|S\|_{\HS}^2-r=2\sum_{j<k}\langle v_j,v_k\rangle\).

\medskip
\noindent\textit{Step 2. Control the exponential weights.}
For deterministic \(H\), set
\[
  F_H(x,y)=\langle H,q_{3,d}(x)q_{3,d}(y)^{\mathsf T}\rangle
  =q_{3,d}(x)^{\mathsf T}Hq_{3,d}(y).
\]
In expectations under \(P_0\), \(F_H\) denotes \(F_H(X,X')\) for the
independent pair \((X,X')\sim P_0\).
Recall from Section~\ref{subsec:cubic-harmonic-kernel} that
\(q_{3,d}(-x)=-q_{3,d}(x)\) and
\(\mathbb E_{\sigma_d}[q_{3,d}(X)q_{3,d}(X)^{\mathsf T}]=I_N/N\).
The first identity and symmetry of \(\sigma_d\) give
\(\mathbb E_{P_0}F_H=0\).  Since \(X\) and \(X'\) are independent under
\(P_0\), applying the second identity first to \(X'\) and then to \(X\) gives
\[
  \mathbb E_{P_0}F_H^2
  =\frac1N\mathbb E_X
    [q_{3,d}(X)^{\mathsf T}HH^{\mathsf T}q_{3,d}(X)]
  =\frac{\tr(HH^{\mathsf T})}{N^2}
  =\frac{\|H\|_{\HS}^2}{N^2}.
\]
The definition of the feature map also gives
\[
  F_H(x,y)=\frac1N\sum_{a,b=1}^N
  H_{ab}\psi_{3,a}(x)\psi_{3,b}(y).
\]
Thus \(F_H(\cdot,y)\in\cH_{3,d}\) for every fixed \(y\), and likewise
\(F_H(x,\cdot)\in\cH_{3,d}\) for every fixed \(x\).
The spherical hypercontractive inequality
\cite[Theorem~1]{Beckner1992}, tensorized over the two sphere factors,
contributes a factor \((p-1)^{3/2}\) for each degree-three variable.
Hence, for every fixed \(p\ge2\),
\begin{equation}
  \|\langle H,w\rangle\|_p
  \le (p-1)^3\|F_H\|_{L^2(P_0)}
  =(p-1)^3\frac{\|H\|_{\HS}}{N}.
  \label{eq:cubic-product-hypercontractivity}
\end{equation}
We next show that, if \(|t|\ell(H)\le c_{\rm src}d\), then
\begin{equation}
  \log\mathbb E_{P_0}e^{t\langle H,w\rangle}
  \le C\frac{t^2\|H\|_{\HS}^2}{N^2}.
  \label{eq:cubic-bounded-source}
\end{equation}
We first bound the Riemannian Hessian of \(F_H\) on the product sphere.
For \(x_0,y_0\in\Sph^{d-1}\), choose unit tangent vectors
\(e_x\perp x_0\) and \(e_y\perp y_0\).  The curves
\[
  x(\theta_1)=x_0\cos\theta_1+e_x\sin\theta_1,
  \qquad
  y(\theta_2)=y_0\cos\theta_2+e_y\sin\theta_2
\]
are great circles with \(\|x'(\theta_1)\|_2=\|y'(\theta_2)\|_2=1\).
Each basis function \(\psi_{3,a}\) is the restriction of a homogeneous
harmonic cubic polynomial.  Along either great circle, it is therefore
a polynomial of degree at most three in the sine and cosine of the
corresponding angle.  Substituting into the feature expansion of \(F_H\)
gives a trigonometric polynomial of degree at most three in each angle,
which we denote by
\[
  p_H(\theta_1,\theta_2)
  :=F_H(x(\theta_1),y(\theta_2))
  =\sum_{k,l=-3}^3 c_{kl}e^{\mathrm i(k\theta_1+l\theta_2)}.
\]
The definition of \(\ell(H)\) gives \(\|p_H\|_\infty\le\ell(H)\).
Bernstein's inequality for trigonometric polynomials
\cite[Theorem~1.1]{QueffelecZarouf2019} states that
\(\|T'\|_\infty\le m\|T\|_\infty\) when \(T\) has degree at most \(m\).
Differentiation does not increase either trigonometric degree, so applying
this inequality twice, with the other variable fixed at each application, gives
\[
  \|\partial_{\theta_i}\partial_{\theta_j}p_H\|_\infty
  \le3\|\partial_{\theta_j}p_H\|_\infty
  \le9\|p_H\|_\infty\le9\ell(H),
  \qquad i,j\in\{1,2\}.
\]
For a tangent vector \(v=(v_x,v_y)\) at \((x_0,y_0)\), choose
\(e_x=v_x/\|v_x\|_2\) if \(v_x\ne0\), and
\(e_y=v_y/\|v_y\|_2\) if \(v_y\ne0\).
For a zero component, choose any unit tangent vector at the corresponding
base point.  Since
\(x''(\theta_1)=-x(\theta_1)\) and \(y''(\theta_2)=-y(\theta_2)\),
the accelerations are normal to the respective spheres, so both curves
are geodesics.  Thus
\(s\mapsto(x(s\|v_x\|_2),y(s\|v_y\|_2))\) is a product geodesic
with initial velocity \((\|v_x\|_2e_x,\|v_y\|_2e_y)=v\).
The Hessian along this geodesic therefore satisfies
\[
\begin{aligned}
  |\nabla^2F_H[v,v]|
  &=\left|\left.\frac{d^2}{ds^2}
    F_H\bigl(x(s\|v_x\|_2),y(s\|v_y\|_2)\bigr)\right|_{s=0}\right|\\
  &\le9\ell(H)(\|v_x\|_2+\|v_y\|_2)^2
   \le18\ell(H)(\|v_x\|_2^2+\|v_y\|_2^2).
\end{aligned}
\]
Taking the supremum over unit tangent vectors gives
\(\|\nabla^2F_H\|_{\mathrm{op}}\le18\ell(H)\).
Let \(\nu_t\) be the probability measure on
\(\Sph^{d-1}\times\Sph^{d-1}\) with density proportional to \(e^{tF_H}\)
relative to \(P_0\), and write \(\mathbb E_t\) for expectation under
\(\nu_t\).  On the product of the unit spheres, the Ricci tensor satisfies
\(\operatorname{Ric}(v,v)=(d-2)\|v\|_2^2\) for every tangent vector \(v\)
\cite[Sections~1.16 and~2.2.4]{BakryGentilLedoux2014}.
The weighted Ricci tensor for \(\nu_t\) is
\(\operatorname{Ric}-t\nabla^2F_H\), so its lower bound is
\(d-2-18|t|\ell(H)\ge c'd\) when
\(|t|\ell(H)\le c_{\rm src}d\), for sufficiently small universal
\(c_{\rm src}\) and all sufficiently large \(d\).
Its Poincar\'e inequality \cite{BakryEmery1985} is therefore
\[
  \operatorname{Var}_t(H_0)
  \le \frac C d\mathbb E_t\|\nabla H_0\|_2^2
  \qquad\text{for every smooth }H_0.
\]
Put \(G=\|\nabla F_H\|_2^2\).  Since
\(\nabla G=2\nabla^2F_H(\nabla F_H,\cdot)\),
\[
  \|\nabla G\|_2^2\le C\ell(H)^2G.
\]
Differentiation under the tilted law and two applications of the Poincar\'e
inequality now give
\[
  \left|\frac{d}{dt}\mathbb E_tG\right|
  =|\operatorname{Cov}_t(G,F_H)|
  \le\sqrt{\operatorname{Var}_t(G)\operatorname{Var}_t(F_H)}
  \le C\frac{\ell(H)}d\mathbb E_tG.
\]
Let \(\Delta=\Delta_x+\Delta_y\) be the Laplace--Beltrami operator on
the product sphere, with sign convention \(\Delta=\operatorname{div}\nabla\).
Since \(F_H\) has spherical harmonic degree three in each variable,
\(-\Delta_xF_H=-\Delta_yF_H=3(d+1)F_H\)
\cite[Section~2.2.3]{BakryGentilLedoux2014}, and hence
\(-\Delta F_H=6(d+1)F_H\).
Integration by parts under \(P_0=\nu_0\) gives
\[
  \mathbb E_0G
  =-\mathbb E_0[F_H\Delta F_H]
  =6(d+1)\mathbb E_0F_H^2
  =6(d+1)\frac{\|H\|_{\HS}^2}{N^2}.
\]
Gronwall's inequality yields
\(\mathbb E_tG\le Cd\|H\|_{\HS}^2/N^2\) throughout the specified range of \(t\).  A final
application of Poincar\'e gives
\[
  \operatorname{Var}_t(F_H)\le C\frac{\|H\|_{\HS}^2}{N^2}.
\]
If \(\Lambda_H(t)=\log\mathbb E_{P_0}e^{tF_H}\), then
\(\Lambda_H(0)=\Lambda_H'(0)=0\) and
\(\Lambda_H''(t)=\operatorname{Var}_t(F_H)\).  Integrating twice proves
\eqref{eq:cubic-bounded-source}.
Taking \(H=S\) and applying \eqref{eq:cubic-bounded-source} at
\(\pm2t\), which is permitted by \(2|u|\ell(S)\le c_{\rm src}d\), gives
\[
  \sup_{|t|\le|u|}\mathbb E_{P_0}e^{2|t||f|}
  \le 2\exp\left(\frac{Cu^2\|S\|_{\HS}^2}{N^2}\right)=O(1).
\]
Here we used \(e^{2|t||f|}\le e^{2tf}+e^{-2tf}\) and
\(u^2\|S\|_{\HS}^2/N^2=o(1)\).

\medskip
\noindent\textit{Step 3. Bound the fourth derivative and conclude.}
By \eqref{eq:cubic-product-hypercontractivity} and the triangle inequality
in \(L^p(P_0)\),
\begin{equation}
  \|f\|_8\le C\frac{\|S\|_{\HS}}{N},
  \qquad
  \|g\|_4+\|g\|_8\le C\frac r{N^2}.
  \label{eq:cubic-row-moment-bounds}
\end{equation}

With \(\phi_t=tf-c_0t^2g/2\), four differentiations give
\[
  Z_{c_0}^{(4)}(t)=\mathbb E_{P_0}e^{\phi_t}
  \left\{(f-c_0tg)^4-6c_0g(f-c_0tg)^2+3c_0^2g^2\right\}.
\]
Since \(g\ge0\) and \(0\le c_0\le1\),
\[
  |f-c_0tg|^4\le8(f^4+t^4g^4),
  \qquad g(f-c_0tg)^2\le2(gf^2+t^2g^3).
\]
Using \(e^{\phi_t}\le e^{|t||f|}\), Cauchy--Schwarz, the moment bounds
in \eqref{eq:cubic-row-moment-bounds}, and
\(\sup_{|t|\le|u|}\mathbb E_{P_0}e^{2|t||f|}=O(1)\) from Step 2, we obtain
\[
 \mathbb E_{P_0}e^{\phi_t}
 \left(f^4+gf^2+g^2+t^2g^3+t^4g^4\right)
 \le C_\rho\left(
 \frac{\|S\|_{\HS}^4+r^2}{N^4}
 +\frac{t^2r^3}{N^6}
 +\frac{t^4r^4}{N^8}\right).
\]
Here \(r\|S\|_{\HS}^2\le(\|S\|_{\HS}^4+r^2)/2\) bounds the
contribution of \(gf^2\).
The terms \(t^2r^3/N^6\) and \(t^4r^4/N^8\) are \(o(r^2/N^4)\),
uniformly for \(|t|\le|u|\), because
\(t^2r/N^2\le u^2r/N^2=o(1)\).  Hence
\[
  \sup_{|t|\le |u|}|Z_{c_0}^{(4)}(t)|
  \le C_\rho\frac{\|S\|_{\HS}^4+r^2}{N^4}.
\]
When \(c_0=0\), every term containing \(g\) disappears and the right side is
\(C_\rho\|S\|_{\HS}^4/N^4\).
Taylor's theorem at zero, with \(Z_{c_0}(0)=1\),
\(Z_{c_0}'(0)=Z_{c_0}'''(0)=0\), and
\(Z_{c_0}''(0)=(\|S\|_{\HS}^2-c_0r)/N^2\) as computed in Step 1, gives
\[
  Z_{c_0}(u)-1
  \le \frac{u^2(\|S\|_{\HS}^2-c_0r)}{2N^2}
    +C_\rho u^4\frac{\|S\|_{\HS}^4+r^2}{N^4},
\]
with \(r^2\) omitted when \(c_0=0\).
Finally, \(\log Z_{c_0}(u)\le Z_{c_0}(u)-1\) proves
\eqref{eq:cubic-compensated-row}.
\end{proof}

\subsection{Partial-sum bounds and variational estimates}
\label{subsec:cubic-prefix-moments}

This subsection supplies the partial-sum bounds and entropy formulas used
in \eqref{eq:cubic-main-entropy-reduction}.
Lemma~\ref{lem:cubic-subset-tensor-bounds} gives deterministic bounds
that verify the hypotheses of Lemma~\ref{lem:cubic-compensated-row}
in that argument.
Proposition~\ref{prop:cubic-linear-reduction} then records the resulting
bound for the exponential moment of \(F_m^{(2)}(\beta)\).
We begin with Lemma~\ref{lem:cubic-random-prefix}, which bounds the
second and fourth moments of \(S_r\) under exchangeable laws satisfying
\(\|S_m\|_{\HS}^2\le Rm\).

\begin{lemma}[Partial-sum moments under exchangeability]
\label{lem:cubic-random-prefix}
Fix \(R<\infty\).  Let \(Q\) be a law of
\(((X_i,X_i'))_{i=1}^m\) invariant under permutations of the \(m\) pairs,
and suppose \(Q(\|S_m\|_{\HS}^2\le Rm)=1\).  Then, for
\(S_r=\sum_{i=1}^rw_i\) and every \(0\le r\le m\),
\begin{equation}
  \mathbb E_Q\|S_r\|_{\HS}^2\le C_Rr,
  \qquad
  \mathbb E_Q\|S_r\|_{\HS}^4\le C_Rr^2.
  \label{eq:cubic-Q-prefix}
\end{equation}
The constant \(C_R\) depends only on \(R\).
\end{lemma}

\begin{proof}
The case \(r=0\) is immediate.  For \(1\le r\le m\), we first prove a
sampling bound for deterministic unit vectors \(v_1,\ldots,v_m\) in a
real Hilbert space with norm \(\|\cdot\|\), satisfying
\(\|\sum_i v_i\|^2\le Rm\).  Let \(I_r\) be a uniform \(r\)-subset of
\([m]\), and put \(\bar v=m^{-1}\sum_i v_i\) and \(x_i=v_i-\bar v\).
Then \(\sum_i x_i=0\) and \(\|x_i\|\le2\).
For independent uniform draws \(\xi_1,\ldots,\xi_r\) from the indexed population
\(x_1,\ldots,x_m\), Hoeffding's convex comparison
\cite[Theorem~4]{Hoeffding1963}, in its vector form
\cite[Section~D]{GrossNesme2010}, gives
\[
  \mathbb E_{I_r}\left\|\sum_{i\in I_r}x_i\right\|^p
  \le \mathbb E\left\|\sum_{k=1}^r\xi_k\right\|^p,
  \qquad p\in\{2,4\}.
\]
Since \(\mathbb E\xi_k=0\) and \(\|\xi_k\|\le2\), independence gives
\(\mathbb E\|\sum_{k=1}^r\xi_k\|^2
=\sum_{k=1}^r\mathbb E\|\xi_k\|^2\le4r\).
In the fourth-moment expansion, independence and centering eliminate
terms with an index occurring exactly once.  The remaining terms have
either four equal indices or two pairs, with \(r\) and \(6\binom r2\)
choices, respectively.  Since each term has absolute value at most \(16\),
\[
  \mathbb E\left\|\sum_{k=1}^r\xi_k\right\|^4
  =\sum_{i,j,k,l=1}^r
    \mathbb E[\langle\xi_i,\xi_j\rangle\langle\xi_k,\xi_l\rangle]
  \le16\left(r+6\binom r2\right)\le48r^2.
\]
The mean part satisfies
\[
  \|r\bar v\|^2\le Rr^2/m\le Rr,
  \qquad
  \|r\bar v\|^4\le R^2r^4/m^2\le R^2r^2.
\]
Combining the centered and mean parts yields
\begin{equation}
  \mathbb E_{I_r}\left\|\sum_{i\in I_r}v_i\right\|^2\le C_Rr,
  \qquad
  \mathbb E_{I_r}\left\|\sum_{i\in I_r}v_i\right\|^4\le C_Rr^2.
  \label{eq:cubic-random-prefix}
\end{equation}
Since \(\|w_i\|_{\HS}=1\) and
\(\|S_m\|_{\HS}^2\le Rm\) hold \(Q\)-almost surely,
\eqref{eq:cubic-random-prefix} applies to \(v_i=w_i\) for almost every
sample.  By exchangeability, \(\sum_{i\in I}w_i\) has the same law as
\(S_r\) for every \(r\)-subset \(I\).  Hence, for \(p\in\{2,4\}\),
\[
  \mathbb E_Q\|S_r\|_{\HS}^p
  =\binom mr^{-1}\sum_{\substack{I\subseteq[m]\\|I|=r}}
    \mathbb E_Q\left\|\sum_{i\in I}w_i\right\|_{\HS}^p
  =\mathbb E_Q\left[
    \binom mr^{-1}\sum_{\substack{I\subseteq[m]\\|I|=r}}
    \left\|\sum_{i\in I}w_i\right\|_{\HS}^p\right]
    \le C_Rr^{p/2},
\]
where the last inequality follows from \eqref{eq:cubic-random-prefix}
for each such sample.  This proves \eqref{eq:cubic-Q-prefix}.
\end{proof}

\paragraph{Variational and conditional entropy formulas.}
Under \(\rho_m\otimes\rho_m\), the full paired sample
\(((X_i,X_i'))_{i=1}^m\) has law \(P_0^{\otimes m}\).
Suppose \(P_0^{\otimes m}(\mathcal E_{m,R})>0\), and let \(H\) be a
bounded measurable function of the paired sample.  A law with finite
relative entropy with respect to
\(P_0^{\otimes m}(\cdot\mid\mathcal E_{m,R})\) is supported on
\(\mathcal E_{m,R}\).
Conditioning on \(\mathcal E_{m,R}\) and applying the Gibbs variational
formula \cite[Lemma~4.4]{Fischer2013RelativeEntropy} therefore give
\begin{equation}
\begin{aligned}
  &\log\mathbb E_{P_0^{\otimes m}}[\1_{\mathcal E_{m,R}}e^H]\\
  &\quad=\log P_0^{\otimes m}(\mathcal E_{m,R})
    +\log\mathbb E_{P_0^{\otimes m}}[e^H\mid\mathcal E_{m,R}]\\
  &\quad=\log P_0^{\otimes m}(\mathcal E_{m,R})
    +\sup_{Q:Q(\mathcal E_{m,R})=1}
    \left[\mathbb E_QH-D\{Q\Vert
      P_0^{\otimes m}(\cdot\mid\mathcal E_{m,R})\}\right]\\
  &\quad=\sup_{Q:Q(\mathcal E_{m,R})=1}
    \{\mathbb E_QH-D(Q\Vert P_0^{\otimes m})\}.
\end{aligned}
  \label{eq:cubic-restricted-variational}
\end{equation}
The last equality uses
\[
  D\{Q\Vert P_0^{\otimes m}(\cdot\mid\mathcal E_{m,R})\}
  =D(Q\Vert P_0^{\otimes m})+\log P_0^{\otimes m}(\mathcal E_{m,R}),
  \qquad Q(\mathcal E_{m,R})=1.
\]
The supremum in \eqref{eq:cubic-restricted-variational} is attained at
the Gibbs law \(Q_H\) with density
\begin{equation}
  \frac{dQ_H}{dP_0^{\otimes m}}
  =\frac{\1_{\mathcal E_{m,R}}e^H}
    {\mathbb E_{P_0^{\otimes m}}[\1_{\mathcal E_{m,R}}e^H]}.
  \label{eq:cubic-gibbs-law}
\end{equation}
Since \(H\) is bounded, \(Q_H\) has finite relative entropy.
If \(H\) is invariant under permutations of the \(m\) pairs, then
\(Q_H\) is exchangeable, because both \(P_0^{\otimes m}\) and
\(\mathcal E_{m,R}\) are invariant under these permutations.
In that case, the supremum in \eqref{eq:cubic-restricted-variational}
is unchanged if we restrict to exchangeable laws \(Q\) with
\(Q(\mathcal E_{m,R})=1\) and finite relative entropy.

For each \(i\), let \(Q_i\) be the
conditional law of \((X_i,X_i')\) given the pairs \((X_j,X_j')\),
\(1\le j<i\).  Write
\[
  \mathcal D_Q:=D(Q\Vert P_0^{\otimes m}),
  \qquad
  \mathcal D_{Q,i}:=D(Q_i\Vert P_0).
\]
The conditional relative entropy \(\mathcal D_{Q,i}\) is a random variable
depending on the preceding pairs.  The chain rule for relative entropy
\cite[Lemma~4.3]{Fischer2013RelativeEntropy} gives
\(\mathcal D_Q=\sum_{i=1}^m\mathbb E_Q\mathcal D_{Q,i}\).
With the preceding pairs fixed, applying the same Gibbs variational formula
to \(P_0\) gives, for every bounded measurable function \(H\) of the current pair,
\begin{equation}
  \mathbb E_{Q_i}H-\mathcal D_{Q,i}\le\log\mathbb E_{P_0}e^H.
  \label{eq:cubic-linear-entropy-duality}
\end{equation}

To apply Lemma~\ref{lem:cubic-compensated-row} after the conditional
entropy inequality \eqref{eq:cubic-linear-entropy-duality}, we need bounds
on \(\ell(S_r)\) and \(\|S_r\|_{\HS}^2\).
Lemma~\ref{lem:cubic-subset-tensor-bounds} obtains these from the
single-sample estimate \eqref{eq:cubic-zonal-hereditary}.

\begin{lemma}[Uniform bounds for sums of tensor features]
\label{lem:cubic-subset-tensor-bounds}
Fix \(A<\infty\) and \(m\le d^A\).  There is a constant \(C_A\) such
that, with \(K_0=0\) and
\[
 K_r=C_A\left\{1+\frac rN\left(\log\frac{em}{r}\right)^3\right\},
 \qquad 1\le r\le m,
\]
the following bounds hold on \(\mathcal G_m(X)\cap\mathcal G_m(X')\)
simultaneously for every \(I\subseteq[m]\), where \(r=|I|\) and
\(S_I=\sum_{i\in I}w_i\):
\begin{equation}
  \ell(S_I)\le K_r,
  \qquad \|S_I\|_{\HS}^2\le rK_r.
  \label{eq:cubic-prefix-source}
\end{equation}
If, in addition, \(m\beta^2=o(Nd)\), then, for another constant \(C_A'\),
\begin{equation}
\begin{aligned}
  \sup_{0\le r\le m}\frac{\beta^2K_r}{d}
  &\le C_A'\left(\frac{\beta^2}{d}+\frac{m\beta^2}{Nd}\right)=o(1),\\
  \sup_{0\le r\le m}\frac{\beta^4(r+rK_r)}{N^2}
  &\le C_A'\left(\frac{\beta^4m}{N^2}
    +\frac{\beta^4m^2}{N^3}\right)=o(1).
\end{aligned}
\label{eq:cubic-scaled-prefix-source}
\end{equation}
\end{lemma}

\begin{proof}
For \(I=\varnothing\), both bounds in \eqref{eq:cubic-prefix-source}
hold because \(S_\varnothing=0\) and \(K_0=0\).
For \(r\ge1\), the feature representation and Cauchy--Schwarz give,
for every \(x,y\in\Sph^{d-1}\),
\[
\begin{aligned}
  \left|\left\langle S_I,q_{3,d}(x)q_{3,d}(y)^{\mathsf T}\right\rangle\right|
  &=\left|\sum_{i\in I}Z_{3,d}(x,X_i)Z_{3,d}(y,X_i')\right|\\
  &\le\left[\sum_{i\in I}Z_{3,d}(x,X_i)^2\right]^{1/2}
    \left[\sum_{i\in I}Z_{3,d}(y,X_i')^2\right]^{1/2}
  \le K_r.
\end{aligned}
\]
The last inequality applies \eqref{eq:cubic-zonal-hereditary} from
Lemma~\ref{lem:cubic-hereditary-envelope} to each sample on
\(\mathcal G_m(X)\cap\mathcal G_m(X')\).
Taking the supremum over \(x,y\) gives \(\ell(S_I)\le K_r\).
Since each \(w_i\) is one of the feature matrices in the definition of
\(\ell\), we also have
\[
  \|S_I\|_{\HS}^2
  =\sum_{i\in I}\langle S_I,w_i\rangle
  \le\sum_{i\in I}|\langle S_I,w_i\rangle|
  \le r\ell(S_I)\le rK_r,
\]
which proves \eqref{eq:cubic-prefix-source}.

The inequalities in \eqref{eq:cubic-scaled-prefix-source} follow from
the definition of \(K_r\) and
\(\sup_{0<x\le1}x^j\log^3(e/x)<\infty\) for \(j=1,2\), with
\(x=r/m\).  Under \(m\beta^2=o(Nd)\), their right-hand sides tend
to zero because \(|\beta|\le1\), \(N\asymp d^3\), and
\[
 \frac{\beta^4m}{N^2}
 =\beta^2\frac{m\beta^2}{Nd}\frac dN=o(1),
 \qquad
 \frac{\beta^4m^2}{N^3}
 =\left(\frac{m\beta^2}{Nd}\right)^2\frac{d^2}{N}=o(1).
\]
Finally, for \(S_r=\sum_{i=1}^r w_i\) and \(|u|\le2\beta^2\),
\[
  \frac{2|u|\ell(S_r)}d\le\frac{4\beta^2K_r}d=o(1),
  \qquad
  \frac{u^2(r+\|S_r\|_{\HS}^2)}{N^2}
  \le\frac{4\beta^4(r+rK_r)}{N^2}=o(1).
\]
Thus Lemma~\ref{lem:cubic-compensated-row} applies with \(\rho=2\),
uniformly over these partial sums and \(0\le c_0\le1\) on
\(\mathcal G_m(X)\cap\mathcal G_m(X')\).
\end{proof}

For \(t\ge0\), write
\begin{equation}
  \Psi_m(t):=\log\mathbb E_{P_0^{\otimes m}}
    [\1_{\mathcal E_{m,R}}e^{t(\|S_m\|_{\HS}^2-m)/2}].
  \label{eq:cubic-linear-partition-definition}
\end{equation}
Proposition~\ref{prop:cubic-linear-reduction} follows from
\eqref{eq:cubic-main-entropy-reduction}, with \(n=m\), and the
variational formula \eqref{eq:cubic-restricted-variational}.

\begin{proposition}[Reduction to a linear partition function]
\label{prop:cubic-linear-reduction}
Fix \(R<\infty\).  Suppose \(m\ge3\), \(0<|\beta|\le1\), and
\(m\beta^2=o(Nd)\).  With
\(\eta=\beta^2(m-2)/(3N^2)\), one has, for all sufficiently large \(d\),
\begin{equation}
  \log\mathbb E_{P_0^{\otimes m}}
    [\1_{\mathcal E_{m,R}}e^{F_m^{(2)}(\beta)}]
  \le\eta\Psi_m(\beta^2)
    +C_R\left(\eta\frac{\beta^4m^2}{N^2}
      +\frac{\beta^8m^3}{N^4}\right).
  \label{eq:cubic-linear-reduction}
\end{equation}
\end{proposition}

\begin{proof}
If \(P_0^{\otimes m}(\mathcal E_{m,R})=0\), both logarithmic
expectations in \eqref{eq:cubic-linear-reduction} equal \(-\infty\),
and the bound is immediate.
Otherwise, \(\eta=o(1)\), so \(0<\eta<1/2\) for all sufficiently large
\(d\).  Lemma~\ref{lem:cubic-subset-tensor-bounds} verifies the
hypotheses of Lemma~\ref{lem:cubic-compensated-row} with
\(u=\beta^2/(1-\eta)\) and \(c_0=1-\eta\).
Together with the partial-sum moments \eqref{eq:cubic-Q-prefix},
these are the inputs to \eqref{eq:cubic-main-entropy-reduction},
which therefore applies with \(n=m\) to each exchangeable law \(Q\)
supported on \(\mathcal E_{m,R}\) with finite relative entropy.
By \eqref{eq:cubic-restricted-variational} with
\(H=\beta^2(\|S_m\|_{\HS}^2-m)/2\),
\[
  \frac{\beta^2}{2}\mathbb E_Q(\|S_m\|_{\HS}^2-m)-\mathcal D_Q
  \le\Psi_m(\beta^2).
\]
Substitute this inequality into \eqref{eq:cubic-main-entropy-reduction}
with \(n=m\), and take the supremum of its left-hand side over \(Q\).
Formula \eqref{eq:cubic-restricted-variational} with
\(H=F_m^{(2)}(\beta)\) identifies this supremum with the left-hand side
of \eqref{eq:cubic-linear-reduction}, proving the claim.
\end{proof}

\subsection{A bound for the linear log-partition function}
\label{subsec:cubic-linear-partition}

We obtain the bound on \(\Psi_m(\beta^2)\) needed in
\eqref{eq:cubic-linear-reduction} from
Lemma~\ref{lem:cubic-linear-free-energy} below.  Its proof combines
Lemma~\ref{lem:cubic-compensated-row} at \(c_0=0\) with the partial-sum
bounds \eqref{eq:cubic-Q-prefix}.

\begin{lemma}[Linear log-partition bound]
\label{lem:cubic-linear-free-energy}
For fixed \(\rho,R<\infty\) and \(0\le t\le\rho\), suppose that, with
\(K_r\) as in Lemma~\ref{lem:cubic-subset-tensor-bounds},
\[
 \sup_{1\le r\le m}\frac{2tK_r}{d}\le c_{\rm src},
 \qquad
 \sup_{1\le r\le m}\frac{t^2(r+rK_r)}{N^2}=o(1).
\]
Then
\begin{equation}
  \Psi_m(t)
  \le C_{\rho,R}\left(\frac{t^2m^2}{N^2}
  +\frac{t^4m^3}{N^4}\right).
  \label{eq:cubic-linear-free-energy}
\end{equation}
\end{lemma}

\begin{proof}
If \(P_0^{\otimes m}(\mathcal E_{m,R})=0\), then \(\Psi_m(t)=-\infty\)
and the bound is immediate.  Otherwise, take \(Q=Q_H\) from
\eqref{eq:cubic-gibbs-law}, with \(H=t(\|S_m\|_{\HS}^2-m)/2\).
This law is exchangeable, has finite relative entropy, and satisfies
\(Q(\mathcal E_{m,R})=1\).
Use the conditional laws \(Q_i\) and entropies
\(\mathcal D_Q,\mathcal D_{Q,i}\) from
Section~\ref{subsec:cubic-prefix-moments}.
Since \((\|S_m\|_{\HS}^2-m)/2=\sum_i\langle S_{i-1},w_i\rangle\), apply
inequality~\eqref{eq:cubic-linear-entropy-duality} to
\(t\langle S_{i-1},w_i\rangle\), and then use
Lemma~\ref{lem:cubic-compensated-row} with \(c_0=0\), to obtain
\[
  t\mathbb E_{Q_i}\langle S_{i-1},w_i\rangle-\mathcal D_{Q,i}
  \le \frac{t^2}{2N^2}\|S_{i-1}\|_{\HS}^2
  +C_\rho t^4\frac{\|S_{i-1}\|_{\HS}^4}{N^4}.
\]
The two conditions of Lemma~\ref{lem:cubic-compensated-row} with \(u=t\)
hold pointwise by the assumptions on \(t\) and the bounds
\eqref{eq:cubic-prefix-source} of Lemma~\ref{lem:cubic-subset-tensor-bounds},
applied to \(I=\{1,\ldots,i-1\}\).
Taking \(Q\)-expectations, summing over \(i\), and using the entropy chain
rule and \eqref{eq:cubic-Q-prefix} gives
\[
\begin{aligned}
  \frac t2\mathbb E_Q(\|S_m\|_{\HS}^2-m)-\mathcal D_Q
  &\le\frac{t^2}{2N^2}\sum_{r=0}^{m-1}\mathbb E_Q\|S_r\|_{\HS}^2
    +\frac{C_\rho t^4}{N^4}\sum_{r=0}^{m-1}\mathbb E_Q\|S_r\|_{\HS}^4\\
  &\le C_{\rho,R}\left(\frac{t^2m^2}{N^2}
    +\frac{t^4m^3}{N^4}\right).
\end{aligned}
\]
For this maximizing law, the left-hand side equals \(\Psi_m(t)\) by
\eqref{eq:cubic-restricted-variational}, proving
\eqref{eq:cubic-linear-free-energy}.
\end{proof}

\subsection{Partition-function bounds under further restrictions}
\label{subsec:cubic-partition-proof}

The proof of Proposition~\ref{prop:cubic-amplitude-tracked-reverse}
combines the estimates above to bound the expectation of \(W_{m,\beta}\)
on \(\mathcal G_m(X)\cap\mathcal G_m(X')\).
Theorem~\ref{thm:cubic-amplitude-tracked-partition} extends the upper
bound to arbitrary single-sample restrictions within \(\mathcal G_m\).
When both restrictions are the same event of probability \(1-o(1)\),
Cauchy--Schwarz gives the matching lower bound.

\begin{theorem}[Cubic partition-function bound]
\label{thm:cubic-amplitude-tracked-partition}
Fix \(A<\infty\).  Uniformly over \(m\le d^A\), \(|\beta|\le1\),
and measurable single-sample sets
\(\mathcal H_X,\mathcal H_{X'}\subseteq\mathcal G_m\), if \(m\beta^2=o(Nd)\), then
\begin{equation}
\begin{aligned}
  &\mathbb E_{\rho_m\otimes\rho_m}\left[
    \1_{\mathcal H_X}(X)\1_{\mathcal H_{X'}}(X')
    \prod_{i<j}(1+\beta^2h_{ij})
  \right]\\
  &\qquad\le \exp\left[C_A\left(
    \frac{\beta^6m^3}{N^4}
    +\frac{\beta^6m^2\log^9 d}{N^3}\right)+o(1)\right]\\
  &\qquad=1+o(1),
  \label{eq:cubic-amplitude-tracked-partition}
\end{aligned}
\end{equation}
uniformly over \(|\beta|\le1\) and the two sets.  If
\(\mathcal H_X=\mathcal H_{X'}=\mathcal H\) and
\(\rho_m(\mathcal H)=1-o(1)\), the left side equals \(1+o(1)\).
\end{theorem}

\begin{proof}
Each factor of \(W_{m,\beta}\) is nonnegative, since
\(|h_{ij}|\le1\) and \(|\beta|\le1\).  If \(\beta=0\) or \(m\le1\),
then \(W_{m,\beta}=1\).  For \(m=2\), independence of the two samples
and the zero mean of \(Z_{3,d}(X_1,X_2)\) give
\[
  \mathbb E_{\rho_2\otimes\rho_2}W_{2,\beta}
  =1+\beta^2\bigl[\mathbb E_{\rho_2}Z_{3,d}(X_1,X_2)\bigr]^2=1.
\]
The restricted expectation is therefore at most one in these cases.

For \(m\ge3\) and \(0<|\beta|\le1\), use
\eqref{eq:cubic-conditioned-partition-bound} with \(n=m\) and
multiply by \(\rho_m(\mathcal G_m)^2\).
Since \(\mathcal H_X,\mathcal H_{X'}\subseteq\mathcal G_m\),
nonnegativity gives
\[
\begin{aligned}
  &\mathbb E_{\rho_m\otimes\rho_m}
    [\1_{\mathcal H_X}(X)\1_{\mathcal H_{X'}}(X')W_{m,\beta}]\\
  &\quad\le\mathbb E_{\rho_m\otimes\rho_m}
    [\1_{\mathcal G_m(X)}\1_{\mathcal G_m(X')}W_{m,\beta}]\\
  &\quad\le\exp\left[C_A\left(\frac{\beta^6m^3}{N^4}
      +\frac{\beta^6m^2\log^9 d}{N^3}\right)+o(1)\right]
    =1+o(1).
\end{aligned}
\]
The last equality follows from \(m\beta^2=o(Nd)\) and
\(N\asymp d^3\), as in Step 3 of the proof of
Proposition~\ref{prop:cubic-amplitude-tracked-reverse}.
The bound does not depend on the two restriction sets, proving
\eqref{eq:cubic-amplitude-tracked-partition} uniformly over those sets.

For the lower bound, let \(\mathcal H_X=\mathcal H_{X'}=\mathcal H\).
Since \(\rho_m(\mathcal H)=1-o(1)>0\) for all sufficiently large \(d\),
we may apply \eqref{eq:cubic-exact-second-moment} with \(n=m\) and
\(\mu=\rho_m(\,\cdot\mid\mathcal H)\).
Using \(\mathbb E_{\Qq_m}L_{\mu,\beta}=1\), Cauchy--Schwarz gives
\[
  \mathbb E_{\rho_m\otimes\rho_m}
    [\1_{\mathcal H}(X)\1_{\mathcal H}(X')W_{m,\beta}]
  =\rho_m(\mathcal H)^2\mathbb E_{\Qq_m}L_{\mu,\beta}^2
  \ge\rho_m(\mathcal H)^2=1-o(1),
\]
which, together with the upper bound, proves the final assertion.
\end{proof}

\begin{remark}[Role of the restriction \(\|S_m\|_{\HS}^2\le Rm\)]
For exchangeable laws supported on \(\{\|S_m\|_{\HS}^2\le Rm\}\),
Lemma~\ref{lem:cubic-random-prefix} gives the partial-sum moment bounds
\eqref{eq:cubic-Q-prefix}.  These bounds control the accumulated errors in
Proposition~\ref{prop:cubic-linear-reduction} and
Lemma~\ref{lem:cubic-linear-free-energy}.
Proposition~\ref{prop:cubic-weighted-overlap} justifies imposing this
restriction by bounding the omitted weighted expectation in
\eqref{eq:cubic-weighted-overlap} by \(e^{-c_R\kappa_m}=o(1)\).
\end{remark}


\clearpage
\section{Gaussian limits for the signed triangle}
\label{app:cubic-triangle-clt}

We use the model and notation of Section~\ref{sec:model}.  Recall that
\(\sigma_d\) is uniform probability measure on \(\Sph^{d-1}\),
\(N=N_{3,d}=d(d-1)(d+4)/6\), and
\(Z_{3,d}(t)=\{(d+2)t^3-3t\}/(d-1)\).
The normalized harmonic feature map \(q_{3,d}:\Sph^{d-1}\to\mathbb R^N\)
satisfies
\(Z_{3,d}(x,y):=Z_{3,d}(\langle x,y\rangle)
=\langle q_{3,d}(x),q_{3,d}(y)\rangle\).
Under \(\Pp_n^{(3)}\), the positions \(X_i\) are iid with law
\(\sigma_d\), and the edges are conditionally independent with means
\(\{1+\beta Z_{3,d}(X_i,X_j)\}/2\), where \(|\beta|\le1\).  The null law is
\(\Qq_n=G(n,1/2)\), and \(L_n=d\Pp_n^{(3)}/d\Qq_n\).

Throughout Appendices B--D, graph limits are taken as \(n\to\infty\), with
\(d=d_n\to\infty\) and \(\beta=\beta_n\).  We write \(d\) and
\(\beta\) for these sequences.

Lemma~\ref{lem:cubic-critical-triangle-clt} proves the Gaussian limits
in Theorem~\ref{thm:pure-cubic-triangle-clt} and
shows that the alternative limit remains valid when the latent sample is
conditioned on an event of probability \(1-o(1)\).  Its proof provides the
conditional fourth-moment and remainder estimates used in
Section~\ref{subsec:cubic-triangle-critical}.
In this appendix, write \(N=N_{3,d}\) and \(M_n=\binom n3\).
Recall \(\widehat T_n=T_n/\sqrt{M_n}\),
\(\gamma_n=\sqrt{M_n}\beta^3/N^2\), and \(r_n=n\beta^2/(Nd)\)
from \eqref{eq:cubic-triangle-statistic}--\eqref{eq:cubic-model-parameters}.
Since \(d^3/(6N)\to1\),
\begin{equation}
  \gamma_n^2
  =\frac{M_n\beta^6}{N^4}
  =r_n^3\frac{d^3}{6N}\{1+o(1)\}.
  \label{eq:cubic-critical-parameter-match}
\end{equation}
Thus \(r_n\to r\in(0,\infty)\) is exactly the cubic critical scale, and
\(|\gamma_n|\to r^{3/2}\).

\begin{lemma}[Gaussian limits for the signed triangle]
\label{lem:cubic-critical-triangle-clt}
Suppose \(d=d_n\to\infty\), \(n\le d^A\), \(|\beta|\le1\), and
\(r_n\to r\in(0,\infty)\).  Then
\[
  \widehat T_n\Rightarrow\Normal(0,1)
  \quad\text{under }\Qq_n,
\]
whereas
\[
  \widehat T_n-\gamma_n\Rightarrow\Normal(0,1)
  \quad\text{under }\Pp_n^{(3)}.
\]
The alternative limit is unchanged after conditioning the latent prior on
any event of probability \(1-o(1)\).
\end{lemma}

\begin{proof}
We prove the null central limit theorem, establish the same limit for the
leading conditional edge chaos, and show that the mean-factor remainders
contribute only the stated deterministic shift.

\medskip
\noindent\textit{Step 1. The null limit.}
Under \(\Qq_n\), the normalized triangle is a degree-three homogeneous
multilinear form in independent Rademacher variables.  Its maximum influence
is \((n-2)/M_n=O(n^{-2})\).
The fourth-moment calculation in
\eqref{eq:cubic-null-triangle-fourth-moment} gives
\[
  \mathbb E_{\Qq_n}\widehat T_n^4
  =
  3-\frac2{M_n}+\frac{24\binom n4}{M_n^2}\longrightarrow3.
\]
The degree-three de Jong theorem \cite[Theorem~1]{deJong1990} therefore proves the first
assertion.

\medskip
\noindent\textit{Step 2. The leading conditional chaos.}
Write \(B_{ij}=Z_{3,d}(X_i,X_j)\),
\(m_{ij}=\mathbb E(Y_{ij}\mid X)=\beta B_{ij}\), and
\(\varepsilon_{ij}=Y_{ij}-m_{ij}\).  Conditional on the latent sample \(X\), the
\(\varepsilon_{ij}\) are independent, centered, and bounded by two.
Recall the decomposition \eqref{eq:cubic-triangle-mean-expansion}, in which
\[
  T_n^{(0)}:=\sum_{i<j<k}
  \varepsilon_{ij}\varepsilon_{jk}\varepsilon_{ki}
\]
and \(R_{s,n}\) collects the terms with exactly \(s\) conditional mean
factors \(m_{ij}\), for \(s=1,2,3\).  Thus
\(T_n=T_n^{(0)}+R_{1,n}+R_{2,n}+R_{3,n}\).

The required kernel moment estimates are
\begin{equation}
  \mathbb EB_{12}^2=\frac1N,
  \qquad
  \mathbb EB_{12}^4\le\frac C{N^2},
  \qquad
  \mathbb EB_{12}^6\le\frac C{N^3}.
  \label{eq:cubic-critical-zonal-moments}
\end{equation}
By \eqref{eq:cubic-addition-formula} and orthonormality,
\begin{equation}
 \int_{\Sph^{d-1}}Z_{3,d}(x,z)Z_{3,d}(z,y)\,d\sigma_d(z)
 =\frac{Z_{3,d}(x,y)}{N}.
 \label{eq:cubic-zonal-reproduction}
\end{equation}
Taking \(y=x\) and using \(Z_{3,d}(1)=1\) proves the second-moment
identity in \eqref{eq:cubic-critical-zonal-moments}.
For the higher moments, rotational invariance and
\cite[Proposition~2.5]{Meckes2019}, rescaled to the unit sphere, give
\begin{equation}
 \mathbb E\langle X_1,X_2\rangle^{2j}
 =\prod_{\ell=1}^j\frac{2\ell-1}{d+2\ell-2},\qquad j\ge1.
 \label{eq:cubic-spherical-even-moments}
\end{equation}
The polynomial formula \eqref{eq:cubic-zonal-polynomial} implies
\(|Z_{3,d}(t)|\le C(|t|^3+|t|/d)\).  Together with
\eqref{eq:cubic-spherical-even-moments} and \(N\asymp d^3\), this yields
\begin{equation}
 \|Z_{3,d}\|_{L^p(\sigma_d^{\otimes2})}
 \le C_pN^{-1/2},\qquad 2\le p<\infty\ \text{fixed}.
 \label{eq:cubic-zonal-Lp-bound}
\end{equation}
The calculation first gives \eqref{eq:cubic-zonal-Lp-bound} for even
integers \(p\).  Monotonicity of \(L^p\)-norms extends it to every fixed
real \(p\ge2\).  Taking \(p=4,6\) proves the remaining estimates in
\eqref{eq:cubic-critical-zonal-moments}.
Let \(v_{ij}:=\mathbb E(\varepsilon_{ij}^2\mid X)=1-m_{ij}^2\) be the
conditional edge variance.  On the event \(\mathcal G_n\) from
Lemma~\ref{lem:cubic-hereditary-envelope}, which has probability \(1-o(1)\),
equation~\eqref{eq:cubic-pair-separation} gives
\[
 \max_{i<j}|m_{ij}|
 \le C_A\frac{\log^{3/2} d}{\sqrt N}=o(1).
\]
Thus \(|m_{ij}|\le1/2\) for all sufficiently large \(d\), and hence
\(3/4\le v_{ij}=1-m_{ij}^2\le1\), uniformly for \(|\beta|\le1\).
Conditional orthogonality of distinct triangle products gives
\[
 V_n:=\operatorname{Var}(T_n^{(0)}\mid X)
 =\sum_{i<j<k}v_{ij}v_{jk}v_{ki},\qquad
 0\le M_n-V_n\le(n-2)\sum_{i<j}m_{ij}^2.
\]
Using \(\mathbb Em_{ij}^2=\beta^2/N\) from
\eqref{eq:cubic-critical-zonal-moments}, we obtain
\[
 \mathbb E\left|\frac{V_n}{M_n}-1\right|
 \le\frac{n-2}{M_n}\sum_{i<j}\mathbb Em_{ij}^2
 =\frac{3\beta^2}{N}=o(1).
\]
Therefore \(V_n/M_n\to1\) in \(L^1\) and in probability.
On \(\mathcal G_n\), normalize
\[
\widetilde\varepsilon_{ij}=\varepsilon_{ij}/\sqrt{v_{ij}},\qquad
c_{ijk}=\frac{\sqrt{v_{ij}v_{jk}v_{ki}}}{\sqrt{V_n}}.
\]
Here \(c_{ijk}\) is symmetric in the three distinct indices.
Conditional on \(X=(X_1,\ldots,X_n)\in\mathcal G_n\), the
\(\widetilde\varepsilon_{ij}\) are independent and satisfy
\[
 \mathbb E(\widetilde\varepsilon_{ij}\mid X)=0,\qquad
 \mathbb E(\widetilde\varepsilon_{ij}^2\mid X)=1,\qquad
 |\widetilde\varepsilon_{ij}|\le\frac4{\sqrt3}.
\]
For a triangle \(t=\{i,j,k\}\), let \(E(t)\) be its edge set and write
\(c_t=c_{ijk}\).  For an edge \(e=\{i,j\}\), write
\(\widetilde\varepsilon_e=\widetilde\varepsilon_{ij}\).
Then
\[
 \frac{T_n^{(0)}}{\sqrt{V_n}}
 =\sum_t c_t\prod_{e\in E(t)}\widetilde\varepsilon_e,
 \qquad
 \sum_t c_t^2=1,
\]
where \(t\) ranges over all triangles on \(\{1,\ldots,n\}\).
For each edge \(\{i,j\}\), its coefficient influence satisfies
\begin{equation}
 I_{ij}:=\sum_{k\ne i,j}c_{ijk}^2
 =\frac{v_{ij}}{V_n}\sum_{k\ne i,j}v_{ik}v_{jk}
 \le\frac{n-2}{V_n}
 \le\left(\frac43\right)^3\frac{n-2}{M_n}=O(n^{-2}),
 \label{eq:cubic-conditional-influence}
\end{equation}
where we used \(v_{ij}\le1\) and \(V_n\ge(3/4)^3M_n\).
Writing \(I_e=I_{ij}\) for \(e=\{i,j\}\), we have
\(\max_e I_e=O(n^{-2})\) uniformly over \(X\in\mathcal G_n\).

We next show that the conditional fourth moment of
\(T_n^{(0)}/\sqrt{V_n}\) tends to three, uniformly on \(\mathcal G_n\).
Given triangles \(t_1,\ldots,t_4\), let
\(\nu_e:=\sum_{a=1}^4\1_{\{e\in E(t_a)\}}\) count the occurrences of edge
\(e\).  Conditional independence gives
\[
 \mathbb E\left[\left.
  \prod_{a=1}^4\prod_{e\in E(t_a)}\widetilde\varepsilon_e\right|X\right]
 =\prod_{e\in\bigcup_a E(t_a)}
       \mathbb E(\widetilde\varepsilon_e^{\nu_e}\mid X).
\]
Since the edge variables are conditionally centered, any term containing an
edge that appears exactly once has zero conditional expectation.
The absence of such edges leaves the three types counted in
\eqref{eq:cubic-null-triangle-fourth-moment}, namely two distinct triangles
each appearing twice, four identical triangles, and the four faces of a tetrahedron.
For distinct \(t,s\), conditional independence and unit conditional variances give
\begin{equation}
 \mathbb E\!\left[
  \prod_{a\in E(t)}\widetilde\varepsilon_a^2
  \prod_{b\in E(s)}\widetilde\varepsilon_b^2\,\middle|\,X\right]
 =\prod_{e\in E(t)\cap E(s)}\mathbb E(\widetilde\varepsilon_e^4\mid X).
 \label{eq:cubic-conditional-paired-moment}
\end{equation}
Grouping the three types gives the decomposition
\begin{equation}
\begin{aligned}
 \mathbb E\!\left[(T_n^{(0)}/\sqrt{V_n})^4\mid X\right]
 &=3\sum_{t\ne s}c_t^2c_s^2
   \prod_{e\in E(t)\cap E(s)}\mathbb E(\widetilde\varepsilon_e^4\mid X)\\
 &\quad+\sum_t c_t^4\prod_{e\in E(t)}\mathbb E(\widetilde\varepsilon_e^4\mid X)
   +24\sum_{i<j<k<\ell}c_{ijk}c_{ij\ell}c_{ik\ell}c_{jk\ell}.
\end{aligned}
 \label{eq:cubic-conditional-fourth-decomposition}
\end{equation}

The ordered sum over \(t\ne s\) in
\eqref{eq:cubic-conditional-fourth-decomposition} gives the coefficient
\(3=\{4!/(2!\,2!)\}/2\).
Four identical triangles have one ordering and contribute the product
of their edge fourth moments.  The tetrahedron has \(4!=24\) orderings,
and each edge occurs twice, giving
\mbox{\(\mathbb E[\prod_{1\le i<j\le4}\widetilde\varepsilon_{ij}^{\,2}\mid X]=1\)}.
We estimate these three terms in the same order.
For the first term, the normalization \(\sum_t c_t^2=1\) gives
\(3\sum_{t\ne s}c_t^2c_s^2=3-3\sum_t c_t^4\).
To bound the difference between the first term and
\(3\sum_{t\ne s}c_t^2c_s^2\),
we use the fact that distinct triangles share at most one edge.  For each
edge \(e\),
\[
 \sum_{\substack{t\ne s\\E(t)\cap E(s)=\{e\}}}c_t^2c_s^2
 \le\left(\sum_{t:\,e\in E(t)}c_t^2\right)^2=I_e^2.
\]
Moreover, \(\mathbb E(\widetilde\varepsilon_e^2\mid X)=1\) and
\(|\widetilde\varepsilon_e|\le4/\sqrt3\) imply
\(1\le\mathbb E(\widetilde\varepsilon_e^4\mid X)\le16/3\).
It follows that
\[
\begin{aligned}
 0&\le3\sum_{t\ne s}c_t^2c_s^2
   \left[\prod_{e\in E(t)\cap E(s)}
     \mathbb E(\widetilde\varepsilon_e^4\mid X)-1\right]\\
 &=3\sum_e\left[\mathbb E(\widetilde\varepsilon_e^4\mid X)-1\right]
   \sum_{\substack{t\ne s\\E(t)\cap E(s)=\{e\}}}c_t^2c_s^2
   \le13\sum_e I_e^2.
\end{aligned}
\]
Consequently, the first term in
\eqref{eq:cubic-conditional-fourth-decomposition} satisfies
\begin{equation}
 3\sum_{t\ne s}c_t^2c_s^2
  \prod_{e\in E(t)\cap E(s)}\mathbb E(\widetilde\varepsilon_e^4\mid X)
 =3-3\sum_t c_t^4+O\!\left(\sum_e I_e^2\right),
 \label{eq:cubic-conditional-pairing-estimate}
\end{equation}
uniformly on \(\mathcal G_n\).
Since \(c_t\le V_n^{-1/2}\) and \(\sum_e I_e=3\sum_t c_t^2=3\),
\[
 \sum_t c_t^4\le\frac1{V_n}\sum_t c_t^2=\frac1{V_n},\qquad
 \sum_e I_e^2\le(\max_e I_e)\sum_e I_e=3\max_e I_e.
\]
Together with \eqref{eq:cubic-conditional-influence} and
\(V_n\ge(3/4)^3M_n\asymp n^3\), these bounds show that the first term
in \eqref{eq:cubic-conditional-fourth-decomposition} is \(3+O(n^{-2})\).

For the second term, \(\mathbb E(\widetilde\varepsilon_e^4\mid X)\le16/3\)
and \(\sum_t c_t^4\le1/V_n\) give
\[
 \sum_t c_t^4\prod_{e\in E(t)}\mathbb E(\widetilde\varepsilon_e^4\mid X)
 \le\left(\frac{16}{3}\right)^3\sum_t c_t^4
 \le\frac{4096}{27V_n}=O(n^{-3}).
\]

For the third term, \(c_t\le V_n^{-1/2}\) gives
\[
 24\sum_{i<j<k<\ell}c_{ijk}c_{ij\ell}c_{ik\ell}c_{jk\ell}
 \le\frac{24\binom n4}{V_n^2}=O(n^{-2}).
\]
Substituting the bounds for all three terms into
\eqref{eq:cubic-conditional-fourth-decomposition}, we conclude that
\begin{equation}
 \left|
 \mathbb E\!\left[(T_n^{(0)}/\sqrt{V_n})^4\mid X\right]-3\right|
 \le C\left(\frac1{V_n}+\max_e I_e+\frac{n^4}{V_n^2}\right)
 =O(n^{-2}),
 \label{eq:cubic-conditional-fourth-error}
\end{equation}
uniformly on \(\mathcal G_n\).

We can now apply de Jong's theorem conditionally on the latent sample.
Fix any deterministic sequence \(x_n\in\mathcal G_n\).
For each triangle \(t\) and edge \(e\in E(t)\), conditioning on \(X\)
and all other edge variables leaves \(\widetilde\varepsilon_e\) centered
by conditional independence.  Hence
\[
 \mathbb E\!\left[
 c_t\prod_{a\in E(t)}\widetilde\varepsilon_a
 \,\middle|\,X,(\widetilde\varepsilon_b)_{b\ne e}\right]
 =c_t\prod_{a\in E(t)\setminus\{e\}}\widetilde\varepsilon_a\,
   \mathbb E(\widetilde\varepsilon_e\mid X)=0.
\]
Thus integrating any one of the three edge variables makes the triangle
summand vanish, which is the degeneracy condition in de Jong's theorem.
Under \(X=x_n\), conditional orthogonality, the normalization
\(\sum_t c_t^2=1\), \eqref{eq:cubic-conditional-influence}, and
\eqref{eq:cubic-conditional-fourth-error} give
\[
\begin{aligned}
 &\mathbb E\!\left[T_n^{(0)}/\sqrt{V_n}\mid X=x_n\right]=0,
 \qquad
 \operatorname{Var}\!\left(T_n^{(0)}/\sqrt{V_n}\mid X=x_n\right)=1,\\
 &\max_e I_e(x_n)\longrightarrow0,
 \qquad
 \mathbb E\!\left[(T_n^{(0)}/\sqrt{V_n})^4\mid X=x_n\right]
   \longrightarrow3.
\end{aligned}
\]
Together with \(|\widetilde\varepsilon_e|\le4/\sqrt3\), these facts verify
the hypotheses of \cite[Theorem~1]{deJong1990}, which gives
\[
 \mathcal L\!\left(T_n^{(0)}/\sqrt{V_n}\mid X=x_n\right)
 \Rightarrow\Normal(0,1).
\]
Finally, we remove the conditioning on \(X\).  Since \(x_n\in\mathcal G_n\)
was arbitrary, the conditional characteristic functions converge uniformly
over \(\mathcal G_n\) for each fixed \(u\in\mathbb R\).
Otherwise, one could choose a deterministic sequence contradicting the
conditional limit just proved.  After extending \(T_n^{(0)}/\sqrt{V_n}\)
by zero on \(\mathcal G_n^c\), we average over \(X\) to obtain
\[
 \left|\mathbb E e^{iuT_n^{(0)}/\sqrt{V_n}}-e^{-u^2/2}\right|
 \le\sup_{x\in\mathcal G_n}
 \left|\mathbb E\!\left[e^{iuT_n^{(0)}/\sqrt{V_n}}\mid X=x\right]
       -e^{-u^2/2}\right|
 +2\mathbb P(X\notin\mathcal G_n)\longrightarrow0.
\]
L\'evy's continuity theorem \cite[Theorem~3.3.17(ii)]{Durrett2019}
therefore yields the first limit below.
The second follows from Slutsky's theorem, since \(V_n/M_n\to1\) in
probability and \(\mathbb P(X\notin\mathcal G_n)=o(1)\).
Thus
\[
  \frac{T_n^{(0)}}{\sqrt{V_n}}\Rightarrow\Normal(0,1),
  \qquad
  \frac{T_n^{(0)}}{\sqrt{M_n}}\Rightarrow\Normal(0,1).
\]

\medskip
\noindent\textit{Step 3. The mean-factor remainders.}
Regrouping the terms in \eqref{eq:cubic-triangle-mean-expansion} according
to the edge carrying the mean factor gives
\(R_{1,n}=\sum_{i<j}\sum_{k\ne i,j}m_{ij}\varepsilon_{ik}\varepsilon_{jk}\).
Expanding the conditional second moment gives
\[
 \mathbb E(R_{1,n}^2\mid X)
 =\sum_{i<j}\sum_{k\ne i,j}m_{ij}^2v_{ik}v_{jk}
   \le(n-2)\sum_{i<j}m_{ij}^2.
\]
Indeed, the unordered pair of centered edges \(\{\{i,k\},\{j,k\}\}\)
determines \((i,j,k)\) with \(i<j\).
Products of distinct summands therefore contain an edge variable
appearing exactly once and have zero conditional expectation.
The inequality uses \(v_{ik},v_{jk}\le1\).
Taking expectations and using
\(\mathbb Em_{ij}^2=\beta^2/N\) from
\eqref{eq:cubic-critical-zonal-moments}, we obtain
\[
 \frac{\mathbb ER_{1,n}^2}{M_n}
 \le\frac{\binom n2(n-2)}{M_n}\frac{\beta^2}{N}
 =\frac{3\beta^2}{N}\le\frac3N\longrightarrow0.
\]

Recall that \(R_{2,n}=\sum_{i<j}\varepsilon_{ij}
\sum_{k\ne i,j}m_{ik}m_{jk}\).
Conditional independence and centering of the \(\varepsilon_{ij}\) give
\[
 \mathbb E(R_{2,n}^2\mid X)
 =\sum_{i<j}v_{ij}\left(\sum_{k\ne i,j}m_{ik}m_{jk}\right)^2
 \le \sum_{i<j}\left(\sum_{k\ne i,j}m_{ik}m_{jk}\right)^2.
\]
For fixed \(i<j\), separate the diagonal and off-diagonal terms in
the square.  Conditional independence and
\eqref{eq:cubic-zonal-reproduction} give, for distinct
\(k,\ell\notin\{i,j\}\),
\[
\begin{aligned}
 \mathbb E(m_{ik}^2m_{jk}^2)
 &=\beta^4\mathbb E\!\left[
    \mathbb E(B_{ik}^2\mid X_k)\mathbb E(B_{jk}^2\mid X_k)\right]
   =\frac{\beta^4}{N^2},\\
 \mathbb E(m_{ik}m_{jk}m_{i\ell}m_{j\ell})
 &=\frac{\beta^4}{N^2}\mathbb EB_{ij}^2=\frac{\beta^4}{N^3}.
\end{aligned}
\]
The first line conditions on \(X_k\), and the second integrates
\(X_k,X_\ell\) conditional on \(X_i,X_j\), then uses
\eqref{eq:cubic-critical-zonal-moments}.
There are \(n-2\) diagonal terms and \((n-2)(n-3)\) off-diagonal
terms per edge.  Since \(\binom n2(n-2)=3M_n\), we obtain
\[
 \frac{\mathbb ER_{2,n}^2}{M_n}
 \le\frac{3\beta^4}{N^2}+\frac{3(n-3)\beta^4}{N^3}
 \le\frac{3\beta^4}{N^2}+\frac{3r_n\beta^2d}{N^2}
 \longrightarrow0,
\]
where we used \(n\beta^2=r_nNd\), \(r_n\to r\in(0,\infty)\),
\(|\beta|\le1\), and \(N\sim d^3/6\).

Finally, recall that
\(R_{3,n}=\beta^3\sum_{i<j<k}B_{ij}B_{jk}B_{ki}\).
Equation~\eqref{eq:cubic-zonal-reproduction} gives
\(\mathbb E(B_{23}B_{31}\mid X_1,X_2)=B_{12}/N\), so
\[
 \mathbb ER_{3,n}
 =\frac{M_n\beta^3}{N}\mathbb EB_{12}^2
 =\frac{M_n\beta^3}{N^2}.
\]
Products \(B_{ij}B_{jk}B_{ki}\) associated with disjoint triangles are
independent.  If two triangles share only vertex \(1\), their products
are conditionally independent given \(X_1\).  By
\eqref{eq:cubic-zonal-reproduction},
\(\mathbb E(B_{12}B_{23}B_{31}\mid X_1)=1/N^2\), so the conditional
means are constant and the covariance is zero.
For identical triangles, H\"older's inequality and
\eqref{eq:cubic-critical-zonal-moments} give
\[
 \operatorname{Var}(B_{12}B_{23}B_{31})
 \le\mathbb E(B_{12}^2B_{23}^2B_{31}^2)
 \le(\mathbb EB_{12}^6\mathbb EB_{23}^6\mathbb EB_{31}^6)^{1/3}
 \le\frac C{N^3}.
\]
For triangles sharing \(\{1,2\}\), integrating \(X_3,X_4\) conditional
on \(X_1,X_2\) by \eqref{eq:cubic-zonal-reproduction} gives
\[
 \operatorname{Cov}(B_{12}B_{23}B_{31},B_{12}B_{24}B_{41})
 =\frac{\mathbb EB_{12}^4}{N^2}-\frac1{N^4}
 \le\frac C{N^4}.
\]
There are \(M_n\) identical pairs and
\(\binom n2(n-2)(n-3)\) ordered pairs sharing an edge.  Thus
\[
 \frac{\operatorname{Var}(R_{3,n})}{M_n}
 \le C\beta^6\left(\frac1{N^3}+\frac n{N^4}\right)
 =C\left(\frac{\beta^6}{N^3}
       +\frac{r_n\beta^4d}{N^3}\right)\longrightarrow0,
\]
again using \(n\beta^2=r_nNd\), bounded \(r_n\), \(|\beta|\le1\),
and \(N\sim d^3/6\).

The three estimates and Cauchy--Schwarz imply
\[
 \mathbb E\left[
  \left(\frac{R_{1,n}+R_{2,n}+R_{3,n}-\mathbb ER_{3,n}}{\sqrt{M_n}}\right)^2
 \right]
 \le\frac3{M_n}
  \left[\mathbb ER_{1,n}^2+\mathbb ER_{2,n}^2
        +\operatorname{Var}(R_{3,n})\right]=o(1).
\]
Since \(\mathbb ER_{3,n}/\sqrt{M_n}=\gamma_n\),
the decomposition \eqref{eq:cubic-triangle-mean-expansion}, Step 2,
and Slutsky's theorem yield
\(\widehat T_n-\gamma_n\Rightarrow\Normal(0,1)\).
For any latent-sample event \(E_n\) with \(\rho_n(E_n)=1-o(1)\),
the conditioning comparison \eqref{eq:cubic-conditioning-comparison},
with \(\mathcal G_n\) replaced by \(E_n\), gives
\[
 \TV\!\left(\mathcal L(\Adj\mid X\in E_n),\Pp_n^{(3)}\right)
 \le\TV\{\rho_n(\,\cdot\mid E_n),\rho_n\}
 =\rho_n(E_n^c)=o(1).
\]
This proves that the alternative limit is unchanged by the conditioning.
\end{proof}

\section{Variational analysis and critical likelihood asymptotics}
\label{app:cubic-franz-parisi}

This appendix proves the likelihood and total-variation assertions of
Theorem~\ref{thm:pure-cubic-critical}.  Write \(N=N_{3,d}\) and
\(L_n=d\Pp_n^{(3)}/d\Qq_n\).
For \(r_n\to r\in(0,r_{\mathrm{tens}})\), we construct a normalized
truncated likelihood \(\widetilde L_n\) such that
\[
 \|L_n-\widetilde L_n\|_{L^1(\Qq_n)}\longrightarrow0,
 \qquad
 \limsup_n\mathbb E_{\Qq_n}\widetilde L_n^2\le e^{r^3}.
\]
Together with the triangle limits in Appendix~\ref{app:cubic-triangle-clt},
these estimates identify the Gaussian log-likelihood limit through the
argument in Section~\ref{subsec:cubic-likelihood-mechanism}.

Let \(\rho_n=\sigma_d^{\otimes n}\) and
\(\pi_n^\dagger=\rho_n(\,\cdot\mid\mathcal G_n^\dagger)\), where
\(\mathcal G_n^\dagger\) is the event in
Proposition~\ref{prop:cubic-raw-local-partition}.
Recall \(L_x(\Adj)=\prod_{i<j}\{1+\beta Y_{ij}Z_{3,d}(x_i,x_j)\}\),
with \(Y_{ij}=2a_{ij}-1\), and
\(\mathbb E_{\Qq_n}[L_xL_{x'}]=W_{n,\beta}(x,x')\).
The conditioned likelihood \(L_n^\dagger=\int L_x\,d\pi_n^\dagger(x)\)
satisfies \(\|L_n-L_n^\dagger\|_{L^1(\Qq_n)}=o(1)\) by
\eqref{eq:cubic-regular-prior-likelihood-transfer}.
For two latent samples, write
\[
 R(x,x')=\frac1n\sum_{i=1}^n
 q_{3,d}(x_i)q_{3,d}(x_i')^{\mathsf T},
 \qquad s_n(x,x')=N\|R(x,x')\|_{\HS}^2.
\]
For a small fixed \(\delta>0\), define
\(Z_{x,\delta}^\dagger(\Adj)
=\int_{\{s_n(x,x')\ge\delta\}}L_{x'}(\Adj)\,d\pi_n^\dagger(x')\).
With \(g>0\) chosen below, set
\[
 \mathcal A_{x,\delta}^\dagger
 =\{\Adj:Z_{x,\delta}^\dagger(\Adj)\le e^{-gnd}\},
 \qquad
 \overline L_n(\Adj)
 =\int L_x(\Adj)\1_{\mathcal A_{x,\delta}^\dagger}(\Adj)
       \,d\pi_n^\dagger(x).
\]
Splitting the second moment at \(s_n=\delta\) and applying
Proposition~\ref{prop:cubic-raw-local-partition} gives
\[
 \mathbb E_{\Qq_n}\overline L_n^2
 \le \iint_{\{s_n<\delta\}}W_{n,\beta}\,
        d(\pi_n^\dagger\otimes\pi_n^\dagger)+e^{-gnd}
 \le e^{r^3}+o(1).
\]
It remains to show that the truncation loses \(o(1)\) likelihood mass.
Writing \(c_n=\mathbb E_{\Qq_n}\overline L_n\), we need
\[
 1-c_n
 =\mathbb P_{X\sim\pi_n^\dagger,\ \Adj\mid X}
   \{Z_{X,\delta}^\dagger(\Adj)>e^{-gnd}\}=o(1).
\]

To prove this probability bound,
Section~\ref{subsec:cubic-bernoulli-Gaussian-reduction} replaces the
Bernoulli exponent by the Gaussian exponent in
Lemma~\ref{lem:cubic-constrained-edge-replacement}, keeping the same
prior \(\pi_n^\dagger\) and restriction \(s_n\ge\delta\).
The interpolation in Proposition~\ref{prop:cubic-Gaussian-variational-bound}
then bounds the Gaussian conditional mean by
\(\mathcal V_{d,r}(Q)
=d^{-1}D(P_{rNdQ}\Vert\gamma_N)-rN\|Q\|_{\HS}^2/4\).
For \(B\succeq0\), \(P_B=\operatorname{Law}(B^{1/2}q_{3,d}(\widetilde X)+G)\),
where \(\widetilde X\sim\sigma_d\) and
\(G\sim\gamma_N=\Normal(0,I_N)\) are independent.
Sections~\ref{subsec:cubic-scalar-channel}--\ref{subsec:cubic-matrix-potential}
prove \(\mathcal V_{d,r_n}(Q)\le-\kappa N\|Q\|_{\HS}^2\)
when \(r_n\) stays in a compact subinterval of \((0,r_{\mathrm{tens}})\).
For the covering error \(\varepsilon_{\rm cov}>0\) in
Proposition~\ref{prop:cubic-Gaussian-variational-bound}, these estimates
give, for \(x\in\mathcal G_n^\dagger\setminus\mathcal E_n\) with
\(\pi_n^\dagger(\mathcal E_n)=o(1)\),
\[
 \frac1{nd}\mathbb E_{\Adj\mid x}\log Z_{x,\delta}^\dagger
 \le
 \sup_{\substack{Q\succeq0\\N\|Q\|_{\HS}^2\ge\delta/2}}
       \mathcal V_{d,r_n}(Q)+C\varepsilon_{\rm cov}+o(1)
 \le-\frac{\kappa\delta}{2}+C\varepsilon_{\rm cov}+o(1).
\]
The logarithmic bound is needed only when
\(\pi_n^\dagger\{x':s_n(x,x')\ge\delta\}>0\), since otherwise
\(Z_{x,\delta}^\dagger=0\).
Choosing \(\varepsilon_{\rm cov}\) small makes the last bound at most
\(-2g\) for some fixed \(g>0\).
Conditional concentration in Lemma~\ref{lem:cubic-legal-truncation} yields
\[
 \sup_{x\in\mathcal G_n^\dagger\setminus\mathcal E_n}
 \mathbb P_{\Adj\mid x}\{Z_{x,\delta}^\dagger>e^{-gnd}\}=o(1),
 \qquad 1-c_n\le\pi_n^\dagger(\mathcal E_n)+o(1)=o(1).
\]
Section~\ref{subsec:cubic-supp-likelihood} completes the construction
with \(\widetilde L_n=\overline L_n/c_n\).

For \(r>r_{\mathrm{tens}}\),
Section~\ref{subsec:cubic-supp-supercritical} instead uses a positive
scalar variational value to prove, for some \(\eta>0\),
\[
 \Pp_n^{(3)}\{\log L_n\ge\eta nd\}\longrightarrow1,
 \qquad
 \Qq_n\{\log L_n\ge\eta nd\}\le e^{-\eta nd}\longrightarrow0.
\]
Proposition~\ref{prop:cubic-endpoint} then constructs sequences approaching
\(r_{\mathrm{tens}}\) with different total-variation limits.
We begin with the regularity and local second-moment estimates.

\begin{proposition}[Regularity and local partition-function bound]
\label{prop:cubic-raw-local-partition}
Fix \(A<\infty\), a compact interval \(K\subset(0,\infty)\), and
\(M_\star>4A+10\).  Suppose \(d=d_n\to\infty\), \(n\le d^A\),
\(0<|\beta|\le1\), and \(r_n=n\beta^2/(Nd)\in K\).
Let \(\mathcal G_{n,M_\star}\) be the event from
Lemma~\ref{lem:cubic-hereditary-envelope}.
There are events \(\mathcal G_n^\dagger\subseteq\mathcal G_{n,M_\star}\),
invariant under common orthogonal rotations, coordinate permutations,
and individual sign changes, such that
\(\rho_n(\mathcal G_n^\dagger)=1-o(d^{-M_\star})\).
The events depend on \(A,K,M_\star,n,d\), but not on \(\beta\).
For every \(X\in\mathcal G_n^\dagger\),
\[
 \left\|\sum_{i=1}^nq_{3,d}(X_i)q_{3,d}(X_i)^{\mathsf T}\right\|_{\rm op}
 \le C\frac nN,
 \qquad
 \max_{i\ne j}|Z_{3,d}(X_i,X_j)|
 \le C\frac{\log^3 d}{\sqrt N},
\]
where \(C\) depends only on \(A,K,M_\star\).
There is \(\delta_0=\delta_0(A,K,M_\star)>0\) such that, for every fixed
\(0<\delta\le\delta_0\),
\begin{equation}
 \limsup_{n\to\infty}
 \mathbb E_{\rho_n\otimes\rho_n}\!\left[
 \1_{\mathcal G_n^\dagger}(X)\1_{\mathcal G_n^\dagger}(X')
 W_{n,\beta}\1_{\{s_n<\delta\}}\right]
 \le e^{r^3}
 \quad\text{whenever }r_n\to r\in K.
 \label{eq:cubic-raw-local-partition}
\end{equation}
For each such \(\delta\), the expectation is uniformly bounded when
\(r_n\in K\).
\end{proposition}

The proof of Proposition~\ref{prop:cubic-raw-local-partition} is given in
Appendix~\ref{app:cubic-local-partition}.

\subsection{Conditioning and Gaussian comparison}
\label{subsec:cubic-bernoulli-Gaussian-reduction}

We first bound the change from \(L_n\) to \(L_n^\dagger\), then compare
the conditional means of the Bernoulli and Gaussian log integrals.
For the likelihood analysis below \(r_{\mathrm{tens}}\), fix
\(A<\infty\), a compact interval \(K\subset(0,r_{\mathrm{tens}})\), and
\(M_\star>4A+10\).  Suppose \(d=d_n\to\infty\), \(n\le d^A\),
\(0<|\beta|\le1\), and \(r_n=n\beta^2/(Nd)\in K\).
Take \(\mathcal G_n^\dagger\) from
Proposition~\ref{prop:cubic-raw-local-partition} for these parameters,
and use the same conditioned prior for the Gaussian comparison,
the variational estimates, and the truncated second moment:
\begin{equation}
 \pi_n^\dagger:=\rho_n(\,\cdot\mid\mathcal G_n^\dagger),
 \qquad g_n^\dagger:=\rho_n(\mathcal G_n^\dagger)
       =1-o(d^{-M_\star}).
 \label{eq:cubic-common-conditioned-prior}
\end{equation}

Write \(L_n^\dagger(\Adj)=\int L_x(\Adj)\,d\pi_n^\dagger(x)\), and let
\(\Pp_{\pi_n^\dagger,\beta}^{(3)}\) be the graph law with latent
positions distributed according to \(\pi_n^\dagger\).
For any graph event \(\mathcal B\), split
\(\Pp_n^{(3)}(\mathcal B)
=\int\mathbb P(\Adj\in\mathcal B\mid X=x)\,d\rho_n(x)\)
over \(\mathcal G_n^\dagger\) and its complement to obtain
\[
\begin{aligned}
 \big|\Pp_n^{(3)}(\mathcal B)
       -\Pp_{\pi_n^\dagger,\beta}^{(3)}(\mathcal B)\big|
 &=\left|\int_{(\mathcal G_n^\dagger)^c}
     \big[\mathbb P(\Adj\in\mathcal B\mid X=x)
        -\Pp_{\pi_n^\dagger,\beta}^{(3)}(\mathcal B)\big]\,d\rho_n(x)\right|\\
 &\le\rho_n\big((\mathcal G_n^\dagger)^c\big)=1-g_n^\dagger.
\end{aligned}
\]
Taking the supremum over \(\mathcal B\) therefore gives
\begin{equation}
 \|L_n-L_n^\dagger\|_{L^1(\Qq_n)}
 =2\,\TV(\Pp_n^{(3)},\Pp_{\pi_n^\dagger,\beta}^{(3)})
 \le2(1-g_n^\dagger)=o(d^{-M_\star}).
 \label{eq:cubic-regular-prior-likelihood-transfer}
\end{equation}

It remains to estimate \(\mathbb E_{\Adj\mid x}\log Z_{x,\delta}^\dagger\).
Lemma~\ref{lem:cubic-constrained-edge-replacement} first compares the
Bernoulli and Gaussian integrals for a common prior and integration set.
We then apply it with prior \(\pi_n^\dagger\) and use the kernel bound
in Proposition~\ref{prop:cubic-raw-local-partition} to verify its hypotheses.

\begin{lemma}[Bernoulli--Gaussian comparison]
\label{lem:cubic-constrained-edge-replacement}
Let \({\cal E}\) be a finite set, \(\nu\) a probability measure on a
configuration space \({\cal X}\), and \({\cal S}\subseteq{\cal X}\) a
measurable set of positive \(\nu\)-measure.  Suppose there exists
\(0<\varepsilon_{\rm ch}\le1/2\) such that, for every \(e\in{\cal E}\),
\(|b_e^0|\le\varepsilon_{\rm ch}\) and
\(\sup_{x\in{\cal S}}|b_e(x)|\le\varepsilon_{\rm ch}\).
Let the independent signs \(Y_e\) have means \(b_e^0\), and let
\((g_e)_{e\in{\cal E}}\) be independent standard Gaussians.  Define
\begin{align*}
  Z_{\rm B}
  &:=
  \int_{\cal S}\prod_{e\in{\cal E}}
       \{1+Y_eb_e(x)\}\,d\nu(x),\\
  Z_{\rm G}
  &:=
  \int_{\cal S}\exp\left[
       \sum_{e\in{\cal E}}
       \left\{g_eb_e(x)+b_e^0b_e(x)-\frac12b_e(x)^2\right\}
       \right]\,d\nu(x).
\end{align*}
Then, for an absolute constant \(C\) independent of \(\nu\) and
\({\cal S}\),
\begin{equation}
  \left|\mathbb E\log Z_{\rm B}-\mathbb E\log Z_{\rm G}\right|
  \le C|{\cal E}|(\varepsilon_{\rm ch}^3+\varepsilon_{\rm ch}^4).
  \label{eq:cubic-constrained-edge-replacement}
\end{equation}
\end{lemma}

\begin{proof}
We apply the Lindeberg comparison theorem
\cite[Theorem~1.1]{Chatterjee2006Lindeberg} to
\[
 \Phi(z):=\log\int_{\cal S}\exp\left[
   \sum_{e\in{\cal E}}\left\{z_eb_e(x)+b_e^0b_e(x)
                               -\frac12b_e(x)^2\right\}\right]d\nu(x),
 \qquad z\in\mathbb R^{\cal E}.
\]
Here \(Y\), \(b^0\), and \(g\) denote the vectors indexed by
\(e\in{\cal E}\), so that \(\Phi(g)=\log Z_{\rm G}\).
Since \(Y_e^2=1\), Taylor's formula gives
\[
 \sup_{x\in{\cal S}}\left|
 \sum_{e\in{\cal E}}
 \left[\log\{1+Y_eb_e(x)\}-Y_eb_e(x)+\frac12b_e(x)^2\right]
 \right|
 \le C|{\cal E}|\varepsilon_{\rm ch}^3.
\]
The exponent in \(\Phi(Y-b^0)\) is
\(\sum_e\{Y_eb_e(x)-b_e(x)^2/2\}\).
Exponentiating the preceding bound and integrating over \({\cal S}\)
therefore gives
\[
 e^{-C|{\cal E}|\varepsilon_{\rm ch}^3}e^{\Phi(Y-b^0)}
 \le Z_{\rm B}\le
 e^{C|{\cal E}|\varepsilon_{\rm ch}^3}e^{\Phi(Y-b^0)},
 \qquad
 |\log Z_{\rm B}-\Phi(Y-b^0)|
 \le C|{\cal E}|\varepsilon_{\rm ch}^3.
\]

It remains to compare \(\mathbb E\Phi(Y-b^0)\) and
\(\mathbb E\Phi(g)\).  Let \(\nu_z\) be the probability measure on
\({\cal S}\) given by
\[
 \nu_z(dx):=\exp\left[
   \sum_e\left\{z_eb_e(x)+b_e^0b_e(x)-\frac12b_e(x)^2\right\}
   -\Phi(z)\right]\nu(dx),\qquad x\in{\cal S}.
\]
Boundedness of the functions \(b_e\) permits differentiation under the
integral and gives, uniformly in \(z\),
\[
\begin{aligned}
 |\partial_{z_e}\Phi(z)|
   &=|\mathbb E_{\nu_z}b_e|\le\varepsilon_{\rm ch},\\
 0\le\partial_{z_e}^2\Phi(z)
   &=\operatorname{Var}_{\nu_z}(b_e)\le\varepsilon_{\rm ch}^2,\\
 |\partial_{z_e}^3\Phi(z)|
   &=\left|\mathbb E_{\nu_z}
           [(b_e-\mathbb E_{\nu_z}b_e)^3]\right|
     \le8\varepsilon_{\rm ch}^3.
\end{aligned}
\]
The independent coordinates \(Y_e-b_e^0\) and the independent standard
Gaussians satisfy
\[
 \mathbb E(Y_e-b_e^0)=\mathbb E g_e=0,\qquad
 \left|\mathbb E(Y_e-b_e^0)^2-\mathbb E g_e^2\right|=(b_e^0)^2,
 \qquad
 \mathbb E|Y_e-b_e^0|^3+\mathbb E|g_e|^3\le C,
\]
where the last bound uses \(|Y_e-b_e^0|\le3/2\).
Thus the first-moment error in
\cite[Theorem~1.1]{Chatterjee2006Lindeberg} is zero, and its
second-moment and third-derivative terms yield
\[
 \left|\mathbb E\Phi(Y-b^0)-\mathbb E\Phi(g)\right|
 \le\frac{\varepsilon_{\rm ch}^2}{2}\sum_e(b_e^0)^2
       +C|{\cal E}|\varepsilon_{\rm ch}^3
 \le C|{\cal E}|(\varepsilon_{\rm ch}^4+\varepsilon_{\rm ch}^3).
\]
Combining this estimate with
\(|\log Z_{\rm B}-\Phi(Y-b^0)|
\le C|{\cal E}|\varepsilon_{\rm ch}^3\) proves
\eqref{eq:cubic-constrained-edge-replacement}.
\end{proof}

We now apply Lemma~\ref{lem:cubic-constrained-edge-replacement} to the
prior \(\pi_n^\dagger\) fixed above.  For \(x\in\mathcal G_n^\dagger\)
and \(\delta>0\), recall
\[
 Z_{x,\delta}^\dagger(\Adj)
 :=\int_{\{s_n(x,x')\ge\delta\}}L_{x'}(\Adj)\,d\pi_n^\dagger(x').
\]
Let \(Z_{x,\delta}^{\mathrm G,\dagger}\) be the Gaussian integral
\(Z_{\rm G}\) in Lemma~\ref{lem:cubic-constrained-edge-replacement}
with \(\nu=\pi_n^\dagger\), \({\cal S}=\{x':s_n(x,x')\ge\delta\}\), and
\[
 b_{ij}^0=\beta Z_{3,d}(x_i,x_j),\qquad
 b_{ij}(x')=\beta Z_{3,d}(x_i',x_j'),\qquad i<j.
\]
Thus \(Z_{\rm B}=Z_{x,\delta}^\dagger\), and both integrals use
the same candidate prior and integration region.
If this region has zero \(\pi_n^\dagger\)-measure, both integrals
vanish.  The logarithmic comparison below concerns positive
\(\pi_n^\dagger\)-measure.

The kernel bound in Proposition~\ref{prop:cubic-raw-local-partition}
gives, uniformly for \(x\in\mathcal G_n^\dagger\),
\[
 \max_{i<j}\max\left\{|b_{ij}^0|,
   \sup_{x'\in\mathcal G_n^\dagger}|b_{ij}(x')|\right\}
 \le\tau_{n,d}:=C_{A,K,M_\star}|\beta|\frac{\log^3d}{\sqrt N}=o(1).
\]
Using \(n\beta^2=r_nNd\), we obtain
\begin{align}
 \frac{n^2\tau_{n,d}^3}{nd}
 &\le C_{A,K,M_\star}\frac{|\beta|\log^9d}{\sqrt N}=o(1),
 \label{eq:cubic-third-replacement-ledger}\\
 \frac{n^2\tau_{n,d}^4}{nd}
 &\le C_{A,K,M_\star}\frac{\beta^2\log^{12}d}{N}=o(1).
 \label{eq:cubic-fourth-replacement-ledger}
\end{align}
For sufficiently large \(n\), \(\tau_{n,d}\le1/2\).
Lemma~\ref{lem:cubic-constrained-edge-replacement}, applied
conditionally on \(X=x\) with \(|{\cal E}|=\binom n2\), therefore gives
\begin{equation}
 \frac1{nd}\left|
 \mathbb E_{\Adj\mid x}\log Z_{x,\delta}^\dagger
 -\mathbb E_g\log Z_{x,\delta}^{\mathrm G,\dagger}\right|
 \le\frac{Cn^2}{nd}(\tau_{n,d}^3+\tau_{n,d}^4)=o(1).
 \label{eq:cubic-Bernoulli-Gaussian-reduction}
\end{equation}
Here \(\mathbb E_g\) averages over the independent Gaussian variables
in \(Z_{x,\delta}^{\mathrm G,\dagger}\).  The bound is uniform over
\(x\in\mathcal G_n^\dagger\) with positive constrained prior mass.
It remains to bound \(\mathbb E_g\log Z_{x,\delta}^{\mathrm G,\dagger}\)
by a negative multiple of \(nd\) outside a set of small
\(\pi_n^\dagger\)-probability.

\subsection{Gaussian interpolation and reduction to a matrix potential}
\label{subsec:cubic-franz-parisi-gap}

Equation~\eqref{eq:cubic-Bernoulli-Gaussian-reduction} leaves the
Gaussian conditional mean
\((nd)^{-1}\mathbb E_g\log Z_{x,\delta}^{\mathrm G,\dagger}\) to be bounded.
We express this integral in harmonic feature coordinates and use
Gaussian interpolation to reduce its conditional mean to a
deterministic matrix optimization.
Proposition~\ref{prop:cubic-Gaussian-variational-bound} makes this
reduction precise.

Let \(\mu_d\) be the law of \(q_{3,d}(X)\) for \(X\sim\sigma_d\).
For \(u\sim\mu_d\), the feature identities give
\(\|u\|=1\) and \(\mathbb E uu^{\mathsf T}=N^{-1}I_N\).
Retain \(A,K,M_\star,\pi_n^\dagger\) from
\eqref{eq:cubic-common-conditioned-prior}.
Let \(\mu_{d,n}^\dagger\) be the law of
\((q_{3,d}(X_i))_{i=1}^n\) for \(X\sim\pi_n^\dagger\), and denote
the image of the retained sample set by
\[
 \mathcal U_n^\dagger
 :=\{(q_{3,d}(x_i))_{i=1}^n:x\in\mathcal G_n^\dagger\}.
\]
Thus \(\mu_{d,n}^\dagger\) is the feature-coordinate representation
of \(\pi_n^\dagger\), with the same normalization \(g_n^\dagger\).
The sampling relation gives
\[
 \frac{N\log N}{n}
 =\frac{\beta^2\log N}{r_nd}
 \le\frac{\log N}{d\min K}=o(1).
\]
This will make the logarithm of the covering cardinality \(o(nd)\).
All constants in the following sample estimates may depend on the
fixed \(A,K,M_\star\).

For \(r\in K\), put \(c=rNd\) and fix a latent feature sample
\(u^0=(u_i^0)_{i=1}^n\) in \(\mathcal U_n^\dagger\).
Conditional on \(u^0\),
consider the auxiliary Gaussian observations
\[
 Y_{ij}^{\rm G}=\sqrt{\frac cn}\langle u_i^0,u_j^0\rangle+g_{ij},
 \qquad 1\le i<j\le n,
\]
where the \(g_{ij}\) are independent standard Gaussian variables.
For a candidate feature sample \(u=(u_i)_{i=1}^n\) in
\(\mathcal U_n^\dagger\), the Gaussian density formula gives its
log-likelihood ratio relative to independent standard Gaussian
observations.  We call this log-likelihood ratio the Gaussian Hamiltonian
and use \(e^{\mathsf{H}^{\rm pair}}\) for its Gibbs weight.  Explicitly,
\[
 \mathsf{H}^{\rm pair}(u;u^0)
 :=\sqrt{\frac cn}\sum_{i<j}Y_{ij}^{\rm G}\langle u_i,u_j\rangle
   -\frac{c}{2n}\sum_{i<j}\langle u_i,u_j\rangle^2.
\]

Substituting the expression for \(Y_{ij}^{\rm G}\) separates the
Gaussian noise from the deterministic terms.  Writing \(\mathbb E_g\)
for expectation over \((g_{ij})_{i<j}\), we denote the centered Gaussian
term and its covariance for candidate samples \(u,v\) by
\begin{align*}
  H_n(u)
  &:=
  \sqrt{\frac cn}\sum_{i<j}g_{ij}\langle u_i,u_j\rangle, \\
  C_n(u,v)
  &:=
  \mathbb E_g\{H_n(u)H_n(v)\}
  =
  \frac cn\sum_{i<j}
    \langle u_i,u_j\rangle\langle v_i,v_j\rangle .
\end{align*}
The Hamiltonian therefore takes the form
\[
 \mathsf{H}^{\rm pair}(u;u^0)
 =H_n(u)+C_n(u,u^0)-\frac12C_n(u,u).
\]

To express \(C_n(u,v)\) in matrix form, define the overlap matrix
\[
 R(u,v):=\frac1n\sum_{i=1}^n u_i v_i^{\mathsf T}.
\]
Its \((k,\ell)\) entry is
\(n^{-1}\sum_i(u_i)_k(v_i)_\ell\), the empirical cross-moment of the
corresponding feature coordinates.
Since \(\|u_i\|=\|v_i\|=1\), the diagonal terms in
\(n^2\|R(u,v)\|_{\HS}^2
=\sum_{i,j}\langle u_i,u_j\rangle\langle v_i,v_j\rangle\)
sum to \(n\).  Hence
\begin{equation}
  C_n(u,v)
  =
  \frac{cn}{2}\|R(u,v)\|_{\HS}^2
  -\frac c2.
  \label{eq:cubic-pair-covariance-overlap}
\end{equation}
For the sample \(u^0\) generating the observations and candidate
samples \(u^1,u^2\), all
restricted to \(\mathcal U_n^\dagger\), write
\(R_{ab}:=R(u^a,u^b)\).
We integrate \(e^{\mathsf{H}^{\rm pair}(u;u^0)}\) over candidates
\(u=u^1\) satisfying \(N\|R_{01}\|_{\HS}^2\ge\delta\), defining
\[
 Z^{\rm G}_{u^0,\delta}
 :=\int\1_{\{N\|R_{01}\|_{\HS}^2\ge\delta\}}
        e^{\mathsf{H}^{\rm pair}(u;u^0)}\,d\mu_{d,n}^\dagger(u).
\]
We set \(\log Z^{\rm G}_{u^0,\delta}=-\infty\) when the integration set
has zero \(\mu_{d,n}^\dagger\)-measure.  Our main target is the
conditional mean \(\mathbb E_g\log Z^{\rm G}_{u^0,\delta}\) for fixed
\(u^0\).  Averaging also over
\(u^0\sim\mu_{d,n}^\dagger\), independently of \((g_{ij})_{i<j}\),
defines the expected constrained free energy
\begin{equation}
  {\cal F}_{n,d}(r,\delta)
  :=\frac1{nd}\mathbb E_{u^0\sim\mu_{d,n}^\dagger}
                   \mathbb E_g\log Z^{\rm G}_{u^0,\delta}.
  \label{eq:cubic-constrained-Gaussian-pressure}
\end{equation}
To relate \(Z^{\rm G}_{u^0,\delta}\) to
\(Z_{x,\delta}^{\mathrm G,\dagger}\), take
\(u_i^0=q_{3,d}(x_i)\) and \(u_i=q_{3,d}(x_i')\).
Then \(N\|R(u^0,u)\|_{\HS}^2=s_n(x,x')\), so the two integrals
use the corresponding indicators
\[
 \1_{\{N\|R(u^0,u)\|_{\HS}^2\ge\delta\}}
 =\1_{\{s_n(x,x')\ge\delta\}}.
\]
When \(r=r_n\), \(c=r_nNd=n\beta^2\) and
\(\sqrt{c/n}=|\beta|\).  Symmetry of the independent standard Gaussian
variables gives
\((\beta g_{ij})_{i<j}\overset{\mathrm d}=(|\beta|g_{ij})_{i<j}\),
while the deterministic terms in \(\mathsf{H}^{\rm pair}\) depend only on
\(\beta^2\).  Consequently, at \(r=r_n\),
\(Z^{\rm G}_{u^0,\delta}\) has the same distribution as
\(Z_{x,\delta}^{\mathrm G,\dagger}\) for
\(u_i^0=q_{3,d}(x_i)\).  Both integrals use the prior
\(\pi_n^\dagger\), expressed in the corresponding coordinates.

The interpolation compares the pair observations with independent
Gaussian observations of individual feature vectors.
For these single-vertex integrals, take \(u\sim\mu_d\).
For \(B\succeq0\), put
\begin{equation*}
  Z_B(g)
  :=
  \int
  \exp\left\{
    \langle g,B^{1/2}u\rangle-\frac12u^{\mathsf T}Bu
  \right\}\,d\mu_d(u),
\end{equation*}
where \(g\sim\Normal(0,I_N)\) is independent of \(u\).
Define \(P_B:=\operatorname{Law}(B^{1/2}u+g)\) and write
\(\gamma_N=\operatorname{Law}(g)\).  The Gaussian density formula gives
\(dP_B/d\gamma_N=Z_B\).

For \(r>0\) and \(Q\succeq0\), define
\begin{equation*}
  {\cal V}_{d,r}(Q)
  :=
  \frac1dD(P_{rNdQ}\Vert\gamma_N)
  -\frac{rN}{4}\|Q\|_{\HS}^2.
\end{equation*}
Convexity of relative entropy and
\(\mathbb E uu^{\mathsf T}=N^{-1}I_N\) give
\[
 D(P_B\Vert\gamma_N)\le\frac{\tr B}{2N},
 \qquad
 {\cal V}_{d,r}(Q)\le\frac r2\tr Q-\frac{rN}{4}\|Q\|_{\HS}^2
 \le\frac r4.
\]
Thus the suprema of \({\cal V}_{d,r}\) appearing below are finite.

For \(\mathsf A\in\mathbb R^{N\times N}\), the singular value decomposition
gives an orthogonal matrix \(U\) such that
\begin{equation}
  \mathsf A=UQ_R,\qquad
  Q_R=(\mathsf A^{\mathsf T}\mathsf A)^{1/2},\qquad
  Q_L=(\mathsf A\mathsf A^{\mathsf T})^{1/2}=UQ_RU^{\mathsf T}.
  \label{eq:cubic-polar-overlap}
\end{equation}
Thus \(Q_L,Q_R\succeq0\), and they have the same operator and
Hilbert--Schmidt norms as \(\mathsf A\).
Lemma~\ref{lem:cubic-polar-cross-entropy} converts the expected
single-vertex log integral associated with \(\mathsf A\) into
\({\cal V}_{d,r}(Q_L)\).

\begin{lemma}[Polar cross-entropy bound]
\label{lem:cubic-polar-cross-entropy}
Put \(c=rNd\) and \(T_{\mathsf A}=\sqrt c\,Q_R^{1/2}U^{\mathsf T}\).
Let \(u^0\sim\mu_d\) and \(g\sim\Normal(0,I_N)\) be independent, and
set \(\mathsf Y_{\mathsf A}=T_{\mathsf A}u^0+g\).  Then
\begin{align}
  \frac1d\mathbb E\log Z_{cQ_R}(\mathsf Y_{\mathsf A})
  -\frac{rN}{4}\|\mathsf A\|_{\HS}^2
  &\le {\cal V}_{d,r}(Q_L).
  \label{eq:cubic-polar-cross-entropy}
\end{align}
\end{lemma}

\begin{proof}
Let \(\nu_{\mathsf A}=\operatorname{Law}(\mathsf Y_{\mathsf A})\).  Recall that
\(P_B=\operatorname{Law}(B^{1/2}u^0+g)\) for \(B\succeq0\), and
\(Z_B=dP_B/d\gamma_N\), where \(\gamma_N=\Normal(0,I_N)\).
Applying this density identity with \(B=cQ_R\), we obtain
\[
  \mathbb E\log Z_{cQ_R}(\mathsf Y_{\mathsf A})
  =D(\nu_{\mathsf A}\Vert\gamma_N)
    -D(\nu_{\mathsf A}\Vert P_{cQ_R})
  \le D(\nu_{\mathsf A}\Vert\gamma_N).
\]

To express this upper bound in terms of \({\cal V}_{d,r}(Q_L)\),
we identify the law of \(U\mathsf Y_{\mathsf A}\).  The identity
\(Q_L=UQ_RU^{\mathsf T}\) gives
\[
 UT_{\mathsf A}=\sqrt c\,UQ_R^{1/2}U^{\mathsf T}=(cQ_L)^{1/2},
 \qquad U\mathsf Y_{\mathsf A}=(cQ_L)^{1/2}u^0+Ug.
\]
Because \(U\) is orthogonal, \(Ug\sim\gamma_N\) and \(Ug\) remains
independent of \(u^0\).  Hence
\(\operatorname{Law}(U\mathsf Y_{\mathsf A})=P_{cQ_L}\).
Since the invertible transformation \(y\mapsto Uy\) preserves
\(\gamma_N\), invariance of relative entropy gives
\[
 D(\nu_{\mathsf A}\Vert\gamma_N)
 =D(P_{cQ_L}\Vert\gamma_N).
\]

Using \(c=rNd\) and \(\|Q_L\|_{\HS}=\|\mathsf A\|_{\HS}\), we conclude that
\begin{align*}
 \frac1d\mathbb E\log Z_{cQ_R}(\mathsf Y_{\mathsf A})
 -\frac{rN}{4}\|\mathsf A\|_{\HS}^2
 &\le \frac1dD(P_{cQ_L}\Vert\gamma_N)
       -\frac{rN}{4}\|Q_L\|_{\HS}^2 \\
 &= {\cal V}_{d,r}(Q_L),
\end{align*}
which proves \eqref{eq:cubic-polar-cross-entropy}.
\end{proof}

We can now bound the Gaussian conditional mean by the supremum of
\({\cal V}_{d,r}\) over matrices with \(N\|Q\|_{\HS}^2\ge\delta/2\).
The constraint is inherited from \(s_n\ge\delta\) through the covering
construction in Proposition~\ref{prop:cubic-Gaussian-variational-bound}.

\begin{proposition}[Variational upper bound for constrained Gaussian integrals]
\label{prop:cubic-Gaussian-variational-bound}
Fix \(A<\infty\), a compact interval \(K\subset(0,r_{\mathrm{tens}})\),
and \(M_\star>4A+10\).
Suppose \(d=d_n\to\infty\), \(n\le d^A\), \(0<|\beta|\le1\),
and \(r_n=n\beta^2/(Nd)\in K\).
Use \(\pi_n^\dagger\) from
\eqref{eq:cubic-common-conditioned-prior} and its feature law
\(\mu_{d,n}^\dagger\).
For every fixed \(\delta,\varepsilon,M>0\), uniformly for \(r\in K\),
\begin{equation}
 \mu_{d,n}^\dagger\left\{
 u^0:\frac1{nd}\mathbb E_g\log Z^{\rm G}_{u^0,\delta}
 >
 \sup_{\substack{Q\succeq0\\N\|Q\|_{\HS}^2\ge\delta/2}}
       {\cal V}_{d,r}(Q)+\varepsilon\right\}
 =o(d^{-M}).
 \label{eq:cubic-Gaussian-variational-probability}
\end{equation}
Moreover,
\begin{equation}
 \limsup_{n\to\infty}\sup_{r\in K}
 \left\{{\cal F}_{n,d}(r,\delta)
       -\sup_{\substack{Q\succeq0\\N\|Q\|_{\HS}^2\ge\delta/2}}
              {\cal V}_{d,r}(Q)\right\}\le0.
 \label{eq:cubic-Gaussian-variational-average}
\end{equation}
\end{proposition}

\begin{proof}
Step~1 covers \(\{N\|R_{01}\|_{\HS}^2\ge\delta\}\) by \(e^{o(nd)}\)
balls.  Step~2 uses Guerra's Gaussian interpolation \cite{Guerra2001}
and Lemma~\ref{lem:cubic-polar-cross-entropy} to prove
\eqref{eq:cubic-Gaussian-cell-upper-bound}, with
\({\cal V}_{d,r}(Q_L)\) as the upper bound for each restricted integral.
We adapt the overlap-constrained argument in
\cite{ElAlaouiKrzakala2018} to matrix overlaps.
Step~3 controls the maximum over these integrals and the sample-average
errors, giving \eqref{eq:cubic-Gaussian-variational-probability}
and \eqref{eq:cubic-Gaussian-variational-average}.

\emph{Step 1: construct a finite cover.}
We cover the matrix set
\[
  \mathcal O_\delta=
  \{R(u^0,u^1):u^0,u^1\in\mathcal U_n^\dagger,\
                 N\|R(u^0,u^1)\|_{\HS}^2\ge\delta\}.
\]
If this set is empty, both conclusions follow from
\(Z^{\rm G}_{u^0,\delta}=0\) and \(\log0=-\infty\).
Otherwise, fix \(0<\varepsilon_{\rm cov}<1/4\), put
\(\rho=\varepsilon_{\rm cov}\sqrt{\delta/N}\) and
\(k=\lceil16/\rho^2\rceil\), and approximate each overlap by a
\(k\)-term average.  For fixed \(u_i^a=q_{3,d}(x_i^a)\), \(a\in\{0,1\}\),
take independent uniform indices \(I_1,\ldots,I_k\) in \([n]\) and set
\(\widehat R_k=k^{-1}\sum_{\ell=1}^k u_{I_\ell}^0(u_{I_\ell}^1)^{\mathsf T}\).
Independence and the unit feature norms give
\[
 \mathbb E_I\widehat R_k=R_{01},\qquad
 \mathbb E_I\|\widehat R_k-R_{01}\|_{\HS}^2
 =\frac1k\mathbb E_I
       \|u_{I_1}^0(u_{I_1}^1)^{\mathsf T}-R_{01}\|_{\HS}^2
 =\frac{1-\|R_{01}\|_{\HS}^2}{k}\le\frac{\rho^2}{16},
\]
where \(\mathbb E_I\) averages over the indices.
Thus some choice satisfies \(\|\widehat R_k-R_{01}\|_{\HS}\le\rho/4\).

By \eqref{eq:cubic-feature-Lipschitz}, the feature map is
\(L_{3,d}\)-Lipschitz, with \(L_{3,d}^2=3(d+1)/(d-1)\le9\).
Choose a Euclidean \(\eta\)-net \(\mathcal N_\eta\) of \(\Sph^{d-1}\),
where \(\eta=\rho/(8L_{3,d})\), with
\(|\mathcal N_\eta|\le(1+2/\eta)^d\le(C/\rho)^d\).
Replace each \(x_{I_\ell}^a\) by a net point \(\bar x_\ell^a\) within
distance \(\eta\), and let
\(\widetilde R_k=k^{-1}\sum_{\ell=1}^k
q_{3,d}(\bar x_\ell^0)q_{3,d}(\bar x_\ell^1)^{\mathsf T}\).
The inequality
\(\|ab^{\mathsf T}-\bar a\bar b^{\mathsf T}\|_{\HS}
 \le\|a-\bar a\|+\|b-\bar b\|\) for unit vectors gives
\[
 \|\widehat R_k-\widetilde R_k\|_{\HS}
 \le\frac1k\sum_{\ell=1}^k\sum_{a=0}^1
       \|u_{I_\ell}^a-q_{3,d}(\bar x_\ell^a)\|
 \le2L_{3,d}\eta=\frac\rho4.
\]
Consequently, the deterministic set
\[
 \mathcal C_0=
 \left\{\frac1k\sum_{\ell=1}^k
 q_{3,d}(a_\ell)q_{3,d}(b_\ell)^{\mathsf T}:
 a_\ell,b_\ell\in\mathcal N_\eta\right\}
\]
covers \(\mathcal O_\delta\) with radius \(\rho/2\) and satisfies
\(\log|\mathcal C_0|\le2dk\log(C/\rho)\).

The centers must also satisfy the operator-norm bound for the overlaps.
If \(\mathbf U_a\) has rows \((u_i^a)^{\mathsf T}\), then
\(R_{ab}=n^{-1}\mathbf U_a^{\mathsf T}\mathbf U_b\), and
Proposition~\ref{prop:cubic-raw-local-partition} gives
\(\|\mathbf U_a\|_{\rm op}^2\le Cn/N\).  Hence
\begin{equation}
 N\|R_{ab}\|_{\rm op}
 \le\frac Nn\|\mathbf U_a\|_{\rm op}\|\mathbf U_b\|_{\rm op}
 \le C=:L_0
 \label{eq:cubic-overlap-operator-cap}
\end{equation}
for all \(u^a,u^b\in\mathcal U_n^\dagger\) and sufficiently large \(n\).
Project \(\mathcal C_0\), in Hilbert--Schmidt norm, onto the closed
convex set \(\mathcal B=\{\mathsf A:N\|\mathsf A\|_{\rm op}\le L_0\}\).
For \(\mathsf A=\Pi_{\mathcal B}(\mathsf A')\), minimality gives
\(\langle\mathsf A'-\mathsf A,R-\mathsf A\rangle_{\HS}\le0\)
for \(R\in\mathcal B\), and therefore
\[
 \|\mathsf A-R\|_{\HS}^2
 \le\|\mathsf A'-R\|_{\HS}^2-\|\mathsf A'-\mathsf A\|_{\HS}^2
 \le\|\mathsf A'-R\|_{\HS}^2.
\]
Retain only projected centers within distance \(\rho\) of
\(\mathcal O_\delta\), writing
\(\mathcal C=\{\mathsf A\in\Pi_{\mathcal B}(\mathcal C_0):
 \exists R\in\mathcal O_\delta,\ \|\mathsf A-R\|_{\HS}\le\rho\}\)
and \(J=|\mathcal C|\).
This preserves the covering radius \(\rho/2\), while
\(J\le|\mathcal C_0|\).  Every retained center satisfies
\[
 N\|\mathsf A\|_{\rm op}\le L_0,\qquad
 \sqrt N\|\mathsf A\|_{\HS}
 \ge\sqrt\delta-\sqrt N\rho
 =(1-\varepsilon_{\rm cov})\sqrt\delta,
\]
so \(N\|\mathsf A\|_{\HS}^2\ge\delta/2\).
Since \(k\le16N/(\varepsilon_{\rm cov}^2\delta)+1\) and
\(N\log N=o(n)\),
\begin{equation}
 \log J\le\log|\mathcal C_0|
 =O_{\varepsilon_{\rm cov},\delta}\{Nd\log(N/\delta)\}
 =o(nd).
 \label{eq:cubic-relative-cover}
\end{equation}

For fixed \(u^0\), define
\[
 Z_{\mathsf A}^{\rm G}(g;u^0)
 :=\int_{\{\|R(u^0,u)-\mathsf A\|_{\HS}\le\rho\}}
       e^{\mathsf{H}^{\rm pair}(u;u^0)}\,d\mu_{d,n}^\dagger(u).
\]
The expected constrained free energy
\((nd)^{-1}\mathbb E_g\log Z_{\mathsf A}^{\rm G}(g;u^0)\)
is a matrix-overlap analogue of the Franz--Parisi potential
\cite{FranzParisi1995}.  The covering property gives
\begin{equation}
 Z^{\rm G}_{u^0,\delta}\le
 \sum_{\mathsf A\in\mathcal C}Z_{\mathsf A}^{\rm G}(g;u^0).
 \label{eq:cubic-Gaussian-cell-cover}
\end{equation}
For this fixed \(u^0\), omit balls of zero \(\mu_{d,n}^\dagger\)-measure,
independently of the Gaussian variables.
If none remain, \(Z^{\rm G}_{u^0,\delta}=0\) and the bound is immediate.
Otherwise, Step~2 bounds \(\mathbb E_g\log Z_{\mathsf A}^{\rm G}\)
for each remaining center, and Step~3 combines them using
\[
 \mathbb E_g\log Z^{\rm G}_{u^0,\delta}
 \le\log J+\mathbb E_g\max_{\mathsf A}\log Z_{\mathsf A}^{\rm G}(g;u^0).
\]

\emph{Step 2: bound the expected logarithm of each restricted integral.}
Fix one of the remaining centers \(\mathsf A\), keeping \(u^0\) fixed.

\emph{Step 2(a): construct the interpolation.}
To bound \(\mathbb E_g\log Z_{\mathsf A}^{\rm G}(g;u^0)\), we first examine
the Hamiltonian
\[
 \mathsf{H}^{\rm pair}(u;u^0)
 =H_n(u)+C_n(u,u^0)-\frac12C_n(u,u).
\]
We begin with \(C_n(u,u^0)\).
For candidates satisfying \(\|R(u^0,u)-\mathsf A\|_{\HS}\le\rho\), expanding
the square in \eqref{eq:cubic-pair-covariance-overlap} around \(\mathsf A\) gives
\[
 \begin{aligned}
 C_n(u,u^0)
 &=M_{\mathsf A}(u;u^0)-\frac{cn}{2}\|\mathsf A\|_{\HS}^2-\frac c2
   +\frac{cn}{2}\|R(u^0,u)-\mathsf A\|_{\HS}^2\\
 &\le M_{\mathsf A}(u;u^0)-\frac{cn}{2}\|\mathsf A\|_{\HS}^2-\frac c2
       +\frac{cn\rho^2}{2},
 \end{aligned}
\]
where \(M_{\mathsf A}(u;u^0):=cn\langle \mathsf A,R(u^0,u)\rangle_{\HS}
=c\sum_i(u_i^0)^{\mathsf T}\mathsf A u_i\) is a sum of terms depending on
individual candidate feature vectors \(u_i\).

The remaining expression is
\[
 H_n(u)-\frac12C_n(u,u)
 =\sum_{i<j}\left(
   \sqrt{\frac cn}\,g_{ij}\langle u_i,u_j\rangle
   -\frac{c}{2n}\langle u_i,u_j\rangle^2
 \right),
\]
whose summands involve both \(u_i\) and \(u_j\).
We compare \(\mathsf{H}^{\rm pair}\) with a Hamiltonian whose summands
each depend on only one \(u_i\), with deterministic term \(M_{\mathsf A}\).
Recall \(\mathsf A=UQ_R\) and \(Q_L=UQ_RU^{\mathsf T}\) from
\eqref{eq:cubic-polar-overlap}.
The matrix \(Q_R=(\mathsf A^{\mathsf T}\mathsf A)^{1/2}\) is positive semidefinite
and satisfies \(\|Q_R\|_{\HS}=\|\mathsf A\|_{\HS}\).
Let \(g_1,\ldots,g_n\) be independent \(\Normal(0,I_N)\) vectors,
independent of \((g_{ij})_{i<j}\), and define the Gaussian sum and
its covariance by
\begin{align*}
  G_{\mathsf A}(u)
  &:=\sqrt c\sum_i\langle g_i,Q_R^{1/2}u_i\rangle,\\
  C_{\mathsf A}(u,v)
  &:=\mathbb E_{(g_i)}\{G_{\mathsf A}(u)G_{\mathsf A}(v)\}
   =c\sum_i u_i^{\mathsf T}Q_Rv_i.
\end{align*}
The resulting Hamiltonian is
\[
 G_{\mathsf A}(u)+M_{\mathsf A}(u;u^0)-\frac12C_{\mathsf A}(u,u)
 =\sum_{i=1}^n\left\{
   \sqrt c\,\langle g_i,Q_R^{1/2}u_i\rangle
   +c(u_i^0)^{\mathsf T}\mathsf A u_i-\frac c2u_i^{\mathsf T}Q_Ru_i
 \right\}.
\]
For fixed \(u^0\) and \((g_i)_{i=1}^n\), the \(i\)th summand depends
only on \(u_i\), so integrating the exponential of this sum under
\(\mu_d^{\otimes n}\) gives a product of one-variable integrals.
To compare the corresponding expected free energies, define the
interpolating Hamiltonian, for \(0\le t\le1\), by
\[
\begin{split}
  \mathsf{H}_{\mathsf A,t}(u;u^0)
  :={}&
  \sqrt t\,H_n(u)+\sqrt{1-t}\,G_{\mathsf A}(u)
  +tC_n(u,u^0)+(1-t)M_{\mathsf A}(u;u^0)\\
  &-\frac12\{tC_n(u,u)+(1-t)C_{\mathsf A}(u,u)\}.
\end{split}
\]
Write \(\mathbb E_g\) for expectation over both Gaussian arrays
\((g_{ij})_{i<j}\) and \((g_i)_{i=1}^n\).  The expected constrained
free energy along the interpolation is
\[
 \phi_{\mathsf A}(t;u^0):=\frac1{nd}\mathbb E_g\log
 \int_{\{\|R(u^0,u)-\mathsf A\|_{\HS}\le\rho\}}
 e^{\mathsf{H}_{\mathsf A,t}(u;u^0)}\,d\mu_{d,n}^\dagger(u).
\]
At \(t=1\), we have
\(\mathsf{H}_{\mathsf A,1}(u;u^0)=\mathsf{H}^{\rm pair}(u;u^0)\), so
\(\phi_{\mathsf A}(1;u^0)=(nd)^{-1}\mathbb E_g\log Z_{\mathsf A}^{\rm G}(g;u^0)\).
At \(t=0\), we have
\(\mathsf{H}_{\mathsf A,0}(u;u^0)=G_{\mathsf A}(u)+M_{\mathsf A}(u;u^0)-C_{\mathsf A}(u,u)/2\).
We first bound \(\phi_{\mathsf A}(0;u^0)\), then integrate
\(\partial_t\phi_{\mathsf A}(t;u^0)\) from \(0\) to \(1\).

\emph{Step 2(b): bound \(\phi_{\mathsf A}(0;u^0)\).}
Recall \(g_n^\dagger=\rho_n(\mathcal G_n^\dagger)=1-o(1)\)
from \eqref{eq:cubic-common-conditioned-prior}.
For every nonnegative measurable function \(G\), the definition of
\(\mu_{d,n}^\dagger\) gives
\begin{equation}
 \int G(u)\,d\mu_{d,n}^\dagger(u)
 =\frac1{g_n^\dagger}\int_{\mathcal G_n^\dagger}
      G((q_{3,d}(x_i))_{i=1}^n)\,d\rho_n(x)
 \le (g_n^\dagger)^{-1}\int G(u)\,d\mu_d^{\otimes n}(u).
 \label{eq:cubic-regularized-product-upper-bound}
\end{equation}
Recall \(T_{\mathsf A}=\sqrt c\,Q_R^{1/2}U^{\mathsf T}\) from
Lemma~\ref{lem:cubic-polar-cross-entropy}.
Since \(T_{\mathsf A}^{\mathsf T}(cQ_R)^{1/2}=c\mathsf A\), the definition of \(Z_{cQ_R}\)
gives
\[
 \int\exp\left\{\langle g_i,(cQ_R)^{1/2}u_i\rangle
       +c(u_i^0)^{\mathsf T}\mathsf A u_i-\frac12\|(cQ_R)^{1/2}u_i\|^2\right\}
       \,d\mu_d(u_i)
 =Z_{cQ_R}(T_\mathsf A u_i^0+g_i).
\]
Since \(e^{\mathsf{H}_{\mathsf A,0}(u;u^0)}>0\), integrating over all \(u\)
instead of restricting to \(\|R(u^0,u)-\mathsf A\|_{\HS}\le\rho\) increases
the integral.
Then \eqref{eq:cubic-regularized-product-upper-bound} with
\(G(u)=e^{\mathsf{H}_{\mathsf A,0}(u;u^0)}\) and the product structure of
\(\mu_d^{\otimes n}\) give
\[
\begin{aligned}
 \int_{\{\|R(u^0,u)-\mathsf A\|_{\HS}\le\rho\}}
 e^{\mathsf{H}_{\mathsf A,0}(u;u^0)}\,d\mu_{d,n}^\dagger(u)
 &\le\int e^{\mathsf{H}_{\mathsf A,0}(u;u^0)}\,d\mu_{d,n}^\dagger(u)\\
 &\le (g_n^\dagger)^{-1}\int e^{\mathsf{H}_{\mathsf A,0}(u;u^0)}
       \,d\mu_d^{\otimes n}(u)\\
 &=(g_n^\dagger)^{-1}\prod_{i=1}^n Z_{cQ_R}(T_\mathsf A u_i^0+g_i).
\end{aligned}
\]
Let \(f_{\mathsf A}(v):=\mathbb E_g\log Z_{cQ_R}(T_{\mathsf A}v+g)\) be the deterministic
function obtained by averaging over \(g\sim\Normal(0,I_N)\) with
\(v\) fixed.
Taking logarithms and Gaussian expectations in the preceding bound gives
\begin{equation}
 \phi_{\mathsf A}(0;u^0)\le\frac1{nd}\sum_i f_{\mathsf A}(u_i^0)-\frac{\log g_n^\dagger}{nd}.
 \label{eq:cubic-Gaussian-vector-endpoint}
\end{equation}
Here \(-\log g_n^\dagger/(nd)=o(1)\), since \(g_n^\dagger=1-o(1)\).

\emph{Step 2(c): compute and integrate \(\partial_t\phi_{\mathsf A}(t;u^0)\).}
To extend the bound for \(\phi_{\mathsf A}(0;u^0)\) in
\eqref{eq:cubic-Gaussian-vector-endpoint} to \(\phi_{\mathsf A}(1;u^0)\),
we compute the derivative by Gaussian integration by parts and use
\(\|R_{01}-\mathsf A\|_{\HS}\le\rho\) to bound its integral.

For the constrained integral defining \(\phi_{\mathsf A}(t;u^0)\), introduce
the Gibbs measure
\[
 \nu_{\mathsf A,t}(du):=
 \frac{\1_{\{\|R(u^0,u)-\mathsf A\|_{\HS}\le\rho\}}
       e^{\mathsf{H}_{\mathsf A,t}(u;u^0)}\,d\mu_{d,n}^\dagger(u)}
      {\displaystyle\int_{\{\|R(u^0,v)-\mathsf A\|_{\HS}\le\rho\}}
       e^{\mathsf{H}_{\mathsf A,t}(v;u^0)}\,d\mu_{d,n}^\dagger(v)}.
\]
The denominator is positive by the choice of \(\mathsf A\) in Step~1.
For \(k=1,2\) and a bounded measurable function \(F\), define
\[
 \langle F(u^1,\ldots,u^k)\rangle_t
 :=\int F(u^1,\ldots,u^k)\prod_{a=1}^k\nu_{\mathsf A,t}(du^a),
\]
where each \(u^a=(u_i^a)_{i=1}^n\) is a full candidate sample.
We call \(u^1,u^2\) replicas, meaning that, conditional on \(u^0\)
and both Gaussian arrays, they are independent draws from the same
Gibbs measure \(\nu_{\mathsf A,t}\).  The expectation \(\mathbb E_g\) averages
over the Gaussian arrays.  We retain this convention in Step~3.

For fixed \(n,d,r,\mathsf A\) and \(0<\varepsilon<1/2\), the expression for
\(\mathsf{H}_{\mathsf A,t}\) in Step~2(a) and \(\|u_i\|=1\) give
\[
 \mathbb E_g\sup_{\varepsilon\le t\le1-\varepsilon}\sup_u
 |\partial_t\mathsf{H}_{\mathsf A,t}(u;u^0)|
 \le C_{\varepsilon}\left(
  1+\sum_{i<j}\mathbb E_g|g_{ij}|+\sum_i\mathbb E_g\|g_i\|
 \right)<\infty.
\]
The supremum is over \(\|u_i\|=1\), and \(C_\varepsilon\) may depend on
the fixed parameters \(n,d,r,\mathsf A\).
Since
\(|\langle\partial_t\mathsf{H}_{\mathsf A,t}\rangle_t|
\le\sup_u|\partial_t\mathsf{H}_{\mathsf A,t}|\),
this integrable bound justifies differentiation under the candidate
integral and Gaussian expectation, giving
\[
\begin{aligned}
 &\partial_t\phi_{\mathsf A}(t;u^0)
 =\frac1{nd}\mathbb E_g\langle\partial_t\mathsf{H}_{\mathsf A,t}(u^1;u^0)\rangle_t\\
 &\quad=\frac1{nd}\mathbb E_g\Big\langle
   \frac{H_n(u^1)}{2\sqrt t}-\frac{G_{\mathsf A}(u^1)}{2\sqrt{1-t}}
   +C_n(u^1,u^0)-M_{\mathsf A}(u^1;u^0)
   -\frac12\{C_n(u^1,u^1)-C_{\mathsf A}(u^1,u^1)\}\Big\rangle_t.
\end{aligned}
\]
To evaluate the terms involving \(H_n\) and \(G_{\mathsf A}\), we differentiate
averages under \(\nu_{\mathsf A,t}\) with respect to their Gaussian
coordinates.  The constraint \(\|R(u^0,u)-\mathsf A\|_{\HS}\le\rho\) is
independent of these coordinates.  Writing
\(g_i=(g_{i,1},\ldots,g_{i,N})\), we have
\(\partial_{g_{ij}}\mathsf{H}_{\mathsf A,t}(u;u^0)
=\sqrt{tc/n}\,u_i^{\mathsf T}u_j\) and
\(\partial_{g_{i,k}}\mathsf{H}_{\mathsf A,t}(u;u^0)
=\sqrt{(1-t)c}\,(Q_R^{1/2}u_i)_k\).
Differentiating the numerator and denominator in the definition of
\(\nu_{\mathsf A,t}\) therefore gives
\[
\begin{aligned}
 \partial_{g_{ij}}\langle (u_i^1)^{\mathsf T}u_j^1\rangle_t
 &=\sqrt{\frac{tc}{n}}\,
   \operatorname{Var}_{\nu_{\mathsf A,t}}(u_i^{\mathsf T}u_j),\\
 \partial_{g_{i,k}}\langle(Q_R^{1/2}u_i^1)_k\rangle_t
 &=\sqrt{(1-t)c}\,
   \operatorname{Var}_{\nu_{\mathsf A,t}}\big((Q_R^{1/2}u_i)_k\big).
\end{aligned}
\]
The bounds \(|u_i^{\mathsf T}u_j|\le1\) and
\(|(Q_R^{1/2}u_i)_k|\le\|Q_R\|_{\rm op}^{1/2}\) bound both averages
and their derivatives, justifying Gaussian integration by parts.
Summing over the coordinates gives
\[
\begin{aligned}
 \mathbb E_g\langle H_n(u^1)\rangle_t
 &=\frac{c\sqrt t}{n}\sum_{i<j}\mathbb E_g
   \operatorname{Var}_{\nu_{\mathsf A,t}}(u_i^{\mathsf T}u_j)\\
 &=\sqrt t\,\mathbb E_g\langle
   C_n(u^1,u^1)-C_n(u^1,u^2)\rangle_t,\\
 \mathbb E_g\langle G_{\mathsf A}(u^1)\rangle_t
 &=c\sqrt{1-t}\sum_{i,k}\mathbb E_g
    \operatorname{Var}_{\nu_{\mathsf A,t}}\big((Q_R^{1/2}u_i)_k\big)\\
 &=c\sqrt{1-t}\sum_i\mathbb E_g\langle
    (u_i^1)^{\mathsf T}Q_Ru_i^1-(u_i^1)^{\mathsf T}Q_Ru_i^2\rangle_t\\
 &=\sqrt{1-t}\,\mathbb E_g\langle
   C_{\mathsf A}(u^1,u^1)-C_{\mathsf A}(u^1,u^2)\rangle_t,
\end{aligned}
\]
Here we use independence of the two replicas and
\(C_n(u,v)=\frac cn\sum_{i<j}(u_i^{\mathsf T}u_j)(v_i^{\mathsf T}v_j)\)
and \(C_{\mathsf A}(u,v)=c\sum_i u_i^{\mathsf T}Q_Rv_i\).

Substituting these expressions for
\(\mathbb E_g\langle H_n(u^1)\rangle_t\) and
\(\mathbb E_g\langle G_{\mathsf A}(u^1)\rangle_t\) into the derivative of
\(\phi_{\mathsf A}\) cancels the terms involving
\(C_n(u^1,u^1)\) and \(C_{\mathsf A}(u^1,u^1)\), giving
\[
\begin{aligned}
 \partial_t\phi_{\mathsf A}(t;u^0)
 &=\frac1{nd}\mathbb E_g\Big\langle
   C_n(u^1,u^0)-M_{\mathsf A}(u^1;u^0)
   +\frac12\{C_n(u^1,u^1)-C_n(u^1,u^2)\}\\
 &\qquad-\frac12\{C_{\mathsf A}(u^1,u^1)-C_{\mathsf A}(u^1,u^2)\}
   -\frac12\{C_n(u^1,u^1)-C_{\mathsf A}(u^1,u^1)\}\Big\rangle_t\\
 &=\frac1{nd}\mathbb E_g\Big\langle
   C_n(u^1,u^0)-M_{\mathsf A}(u^1;u^0)
   -\frac12C_n(u^1,u^2)+\frac12C_{\mathsf A}(u^1,u^2)\Big\rangle_t.
\end{aligned}
\]
It remains to write this derivative in terms of \(R_{01}\) and
\(R_{12}\).  Recall from Step~2(a) that
\(M_{\mathsf A}(u^1;u^0)=cn\langle \mathsf A,R_{01}\rangle_{\HS}\).
Using \(R_{12}=n^{-1}\sum_i u_i^1(u_i^2)^{\mathsf T}\) and the
Hilbert--Schmidt inner product, we also have
\[
 C_{\mathsf A}(u^1,u^2)
 =c\sum_i(u_i^1)^{\mathsf T}Q_Ru_i^2
 =c\sum_i\tr\{Q_R^{\mathsf T}u_i^1(u_i^2)^{\mathsf T}\}
 =cn\langle Q_R,R_{12}\rangle_{\HS}.
\]
After substituting \eqref{eq:cubic-pair-covariance-overlap}, we use
the identities
\[
\begin{aligned}
 \frac12\|R_{01}\|_{\HS}^2-\langle \mathsf A,R_{01}\rangle_{\HS}
 &=\frac12\|R_{01}-\mathsf A\|_{\HS}^2-\frac12\|\mathsf A\|_{\HS}^2,\\
 -\frac14\|R_{12}\|_{\HS}^2+\frac12\langle Q_R,R_{12}\rangle_{\HS}
 &=-\frac14\|R_{12}-Q_R\|_{\HS}^2+\frac14\|Q_R\|_{\HS}^2.
\end{aligned}
\]
Together with \(c=rNd\), these identities give
\begin{align}
  \partial_t\phi_{\mathsf A}(t;u^0)
  ={}&
  \frac{rN}{2}\,
    \mathbb E_g\left\langle\|R_{01}-\mathsf A\|_{\HS}^2\right\rangle_t
  -\frac{rN}{4}\,
    \mathbb E_g\left\langle\|R_{12}-Q_R\|_{\HS}^2\right\rangle_t
  \notag\\
  &-\frac{rN}{2}\|\mathsf A\|_{\HS}^2
   +\frac{rN}{4}\|Q_R\|_{\HS}^2-\frac{rN}{4n}.
  \label{eq:cubic-fixed-overlap-derivative}
\end{align}
The term \(-rN/(4n)=(-c/2+c/4)/(nd)\) comes from the constant
\(-c/2\) in \eqref{eq:cubic-pair-covariance-overlap}.
On the support of \(\nu_{\mathsf A,t}\), we have
\(\|R_{01}-\mathsf A\|_{\HS}\le\rho\).  Substituting
\(\|Q_R\|_{\HS}=\|\mathsf A\|_{\HS}\) and the choice
\(\rho=\varepsilon_{\rm cov}\sqrt{\delta/N}\) from Step~1 into
\eqref{eq:cubic-fixed-overlap-derivative} gives
\[
\begin{aligned}
 \partial_t\phi_{\mathsf A}(t;u^0)
 &=\frac{rN}{4}\left(
   2\mathbb E_g\langle\|R_{01}-\mathsf A\|_{\HS}^2\rangle_t
   -\mathbb E_g\langle\|R_{12}-Q_R\|_{\HS}^2\rangle_t
   -\|\mathsf A\|_{\HS}^2-\frac1n\right)\\
 &\le\frac{rN\rho^2}{2}-\frac{rN}{4}\|\mathsf A\|_{\HS}^2
 =\frac{r\delta\varepsilon_{\rm cov}^2}{2}
   -\frac{rN}{4}\|\mathsf A\|_{\HS}^2.
\end{aligned}
\]

For integration up to \(t=0,1\), the fixed measure and constraint give,
for \(s,t\in[0,1]\),
\[
 \begin{aligned}
 |\phi_{\mathsf A}(t;u^0)-\phi_{\mathsf A}(s;u^0)|
 &\le\frac1{nd}\mathbb E_g\sup_u
   |\mathsf{H}_{\mathsf A,t}(u;u^0)-\mathsf{H}_{\mathsf A,s}(u;u^0)|\\
 &\le C\left(|\sqrt t-\sqrt s|
       +|\sqrt{1-t}-\sqrt{1-s}|+|t-s|\right).
 \end{aligned}
\]
Here \(C<\infty\) may depend on \(n,d,r,\mathsf A\), and the second
inequality uses \(\|u_i\|=1\) and finite Gaussian first moments.
Thus \(\phi_{\mathsf A}\) is continuous on \([0,1]\).
Moreover, \(\|R_{ab}\|_{\HS}\le1\) and
\(\|Q_R\|_{\HS}=\|\mathsf A\|_{\HS}\) imply, by
\eqref{eq:cubic-fixed-overlap-derivative},
\[
 \sup_{0<t<1}|\partial_t\phi_{\mathsf A}(t;u^0)|
 \le\frac{rN}{4}\left(3(1+\|\mathsf A\|_{\HS})^2
                     +\|\mathsf A\|_{\HS}^2+\frac1n\right)<\infty.
\]
Integrating on \([\varepsilon,1-\varepsilon]\) and letting
\(\varepsilon\downarrow0\) therefore yields
\begin{equation}
 \phi_{\mathsf A}(1;u^0)-\phi_{\mathsf A}(0;u^0)
 =\int_0^1\partial_t\phi_{\mathsf A}(t;u^0)\,dt
 \le-\frac{rN}{4}\|\mathsf A\|_{\HS}^2
       +\frac{r\delta\varepsilon_{\rm cov}^2}{2}.
 \label{eq:cubic-integrated-interpolation-bound}
\end{equation}

\emph{Step 2(d): express the bound through the matrix potential.}
Combining \eqref{eq:cubic-Gaussian-vector-endpoint} and
\eqref{eq:cubic-integrated-interpolation-bound} gives
\[
 \phi_{\mathsf A}(1;u^0)
 \le\frac1{nd}\sum_i f_{\mathsf A}(u_i^0)-\frac{rN}{4}\|\mathsf A\|_{\HS}^2
     +\frac{r\delta\varepsilon_{\rm cov}^2}{2}
     -\frac{\log g_n^\dagger}{nd}.
\]

The matrix potential controls the \(\mu_d\)-mean of \(f_{\mathsf A}\).
We therefore separate this mean from the dependence on the fixed
sample \(u^0\) by writing
\[
 \frac1{nd}\sum_i f_{\mathsf A}(u_i^0)
 =\frac1d\mathbb E_{\mu_d}f_{\mathsf A}+\Xi_{\mathsf A},
\]
where \(\Xi_{\mathsf A}:=(nd)^{-1}\sum_i\{f_{\mathsf A}(u_i^0)-\mathbb E_{\mu_d}f_{\mathsf A}\}\).
For independent \(v\sim\mu_d\) and \(g\sim\Normal(0,I_N)\), the
definition of \(f_{\mathsf A}\) gives
\(\mathbb E_{\mu_d}f_{\mathsf A}=\mathbb E_{v,g}\log Z_{cQ_R}(T_{\mathsf A}v+g)\).
Lemma~\ref{lem:cubic-polar-cross-entropy} gives
\[
 \frac1d\mathbb E_{\mu_d}f_{\mathsf A}-\frac{rN}{4}\|\mathsf A\|_{\HS}^2
 \le {\cal V}_{d,r}(Q_L),
 \qquad N\|Q_L\|_{\HS}^2=N\|\mathsf A\|_{\HS}^2\ge\delta/2,
\]
where the last inequality follows from Step~1.
Substituting the decomposition of \((nd)^{-1}\sum_i f_{\mathsf A}(u_i^0)\)
and this bound into the upper bound for \(\phi_{\mathsf A}(1;u^0)\), with
\(\varepsilon_{\rm cov}^2\le\varepsilon_{\rm cov}\) and
\(-\log g_n^\dagger/(nd)=o(1)\), yields
\begin{equation}
 \frac1{nd}\mathbb E_g\log Z_{\mathsf A}^{\rm G}(g;u^0)
 =\phi_{\mathsf A}(1;u^0)
 \le{\cal V}_{d,r}(Q_L)+\Xi_{\mathsf A}
       +C_{K,\delta}\varepsilon_{\rm cov}+o(1),
 \label{eq:cubic-Gaussian-cell-upper-bound}
\end{equation}
uniformly in the center \(\mathsf A\) and retained sample \(u^0\).
The right-hand side of \eqref{eq:cubic-Gaussian-cell-upper-bound}
depends on \(u^0\) only through \(\Xi_{\mathsf A}\).

\emph{Step 3: combine the integrals and control the sample averages.}
For fixed \(u^0\), first take maxima only over the balls of positive
\(\mu_{d,n}^\dagger\)-measure retained in Step~1.
For candidates \(u\in\mathcal U_n^\dagger\), the matrix \(\mathbf U\)
with rows \(u_i^{\mathsf T}\) satisfies
\(\|\mathbf U\|_{\rm op}^2\le Cn/N\) by
Proposition~\ref{prop:cubic-raw-local-partition}, and hence
\[
 \sum_{i<j}\langle u_i,u_j\rangle^2
 \le\frac12\tr\{(\mathbf U^{\mathsf T}\mathbf U)^2\}
 \le\frac12\|\mathbf U\|_{\rm op}^2
            \tr(\mathbf U^{\mathsf T}\mathbf U)
 \le\frac{Cn^2}{N}.
\]
Using the candidate expectation \(\langle\cdot\rangle_1\) from
Step~2(c), differentiation and Jensen's inequality give
\[
 \|\nabla_g\log Z_{\mathsf A}^{\rm G}\|^2
 =\frac cn\sum_{i<j}\langle u_i^{\mathsf T}u_j\rangle_1^2
 \le\frac cn\left\langle\sum_{i<j}(u_i^{\mathsf T}u_j)^2\right\rangle_1
 \le\frac{Ccn}{N}=Crnd\le C_Knd.
\]
Thus Gaussian concentration
\cite[Theorem~5.5]{BoucheronLugosiMassart2013} yields, for \(t>0\),
\(\mathbb E_g e^{t(\log Z_{\mathsf A}^{\rm G}-\mathbb E_g\log Z_{\mathsf A}^{\rm G})}
\le e^{C_Kndt^2/2}\).
Jensen's inequality and \(\max_{\mathsf A}e^{x_{\mathsf A}}
\le\sum_{\mathsf A}e^{x_{\mathsf A}}\) imply
\[
 \mathbb E_g\max_{\mathsf A}
 \{\log Z_{\mathsf A}^{\rm G}-\mathbb E_g\log Z_{\mathsf A}^{\rm G}\}
 \le\inf_{t>0}\left(\frac{\log J}{t}+\frac{C_Kndt}{2}\right)
 \le C_K\sqrt{nd\log(2J)}.
\]
Combining this with \eqref{eq:cubic-Gaussian-cell-cover} gives
\begin{equation}
 \frac1{nd}\mathbb E_g\log Z^{\rm G}_{u^0,\delta}
 \le \max_{\mathsf A}\frac1{nd}\mathbb E_g\log Z_{\mathsf A}^{\rm G}(g;u^0)
       +C_K\sqrt{\frac{\log(2J)}{nd}}+\frac{\log J}{nd}.
 \label{eq:cubic-cover-Gaussian-maximum}
\end{equation}
Substituting \eqref{eq:cubic-Gaussian-cell-upper-bound}, and enlarging
the maximum of \(|\Xi_{\mathsf A}|\) to the fixed cover \(\mathcal C\), yields
\begin{equation}
 \frac1{nd}\mathbb E_g\log Z^{\rm G}_{u^0,\delta}
 \le\sup_{\substack{Q\succeq0\\N\|Q\|_{\HS}^2\ge\delta/2}}
       {\cal V}_{d,r}(Q)+\max_{\mathsf A\in\mathcal C}|\Xi_{\mathsf A}|
       +C_{K,\delta}\varepsilon_{\rm cov}
       +C_K\sqrt{\frac{\log(2J)}{nd}}+\frac{\log J}{nd}+o(1).
 \label{eq:cubic-Gaussian-conditional-upper-bound}
\end{equation}
For fixed \(\varepsilon_{\rm cov}>0\), the last three terms tend to
zero uniformly in \(u^0\), by \eqref{eq:cubic-relative-cover}.

To control the centered average \(\Xi_{\mathsf A}\) from Step~2(d),
we bound the gradient of
\(f_{\mathsf A}(v)=\mathbb E_g\log Z_{cQ_R}(T_{\mathsf A}v+g)\).
Differentiating the integral defining \(Z_{cQ_R}\) gives
\[
 \nabla_y\log Z_{cQ_R}(y)
 =(cQ_R)^{1/2}
   \frac{\int u\,e^{\langle y,(cQ_R)^{1/2}u\rangle-u^{\mathsf T}cQ_Ru/2}
                    \,d\mu_d(u)}{Z_{cQ_R}(y)}.
\]
The fraction has norm at most one.
Since \(T_{\mathsf A}^{\mathsf T}(cQ_R)^{1/2}=c\mathsf A\),
we have \(\|\nabla f_{\mathsf A}\|\le c\|\mathsf A\|_{\rm op}\le C_Kd\).
By \eqref{eq:cubic-feature-Lipschitz}, \(\Xi_{\mathsf A}\) therefore has
Lipschitz constant at most \(C_K/\sqrt n\) as a function of
\((X_1,\ldots,X_n)\), with \(u_i^0=q_{3,d}(X_i)\).
The product-sphere concentration estimate used in
Proposition~\ref{prop:cubic-weighted-overlap}
\cite[Theorem~5.5]{Meckes2019} gives
\begin{equation}
 \mathbb P_{\mu_d^{\otimes n}}\{|\Xi_{\mathsf A}|>t\}
 \le 2e^{-c_Kndt^2},\qquad t>0.
 \label{eq:cubic-endpoint-empirical-concentration}
\end{equation}
A union bound and \eqref{eq:cubic-regularized-product-upper-bound} yield
\begin{equation}
 \mathbb P_{\mu_{d,n}^\dagger}\{\max_{\mathsf A\in\mathcal C}|\Xi_{\mathsf A}|>t\}
 \le (g_n^\dagger)^{-1}\mathbb P_{\mu_d^{\otimes n}}
                   \{\max_{\mathsf A\in\mathcal C}|\Xi_{\mathsf A}|>t\}
 \le\frac{2J}{g_n^\dagger}e^{-c_Kndt^2}.
 \label{eq:cubic-endpoint-uniform-concentration}
\end{equation}
Integrating the union bound under \(\mu_d^{\otimes n}\), then applying
\eqref{eq:cubic-regularized-product-upper-bound}, also gives
\begin{equation}
 \mathbb E_{\mu_{d,n}^\dagger}\max_{\mathsf A\in\mathcal C}|\Xi_{\mathsf A}|
 \le (g_n^\dagger)^{-1}\mathbb E_{\mu_d^{\otimes n}}\max_{\mathsf A\in\mathcal C}|\Xi_{\mathsf A}|
 \le\frac{C_K}{g_n^\dagger}\sqrt{\frac{\log(2J)}{nd}}=o(1).
 \label{eq:cubic-endpoint-empirical-mean-bound}
\end{equation}

Choose \(\varepsilon_{\rm cov}\) so that
\(C_{K,\delta}\varepsilon_{\rm cov}\le\varepsilon/3\).
For large \(n\), the remaining deterministic errors in
\eqref{eq:cubic-Gaussian-conditional-upper-bound} sum to at most
\(\varepsilon/3\), so the event in
\eqref{eq:cubic-Gaussian-variational-probability} is contained in
\(\{\max_{\mathsf A\in\mathcal C}|\Xi_{\mathsf A}|>\varepsilon/3\}\).
By \eqref{eq:cubic-endpoint-uniform-concentration},
\[
 \mu_{d,n}^\dagger\{\max_{\mathsf A\in\mathcal C}|\Xi_{\mathsf A}|>\varepsilon/3\}
 \le\frac{2J}{g_n^\dagger}
       \exp\left(-\frac{c_Knd\varepsilon^2}{9}\right)
 =o(d^{-M}),
\]
where \(g_n^\dagger=1-o(1)\), \(\log J=o(nd)\), and
\(\log d=o(nd)\).  This proves
\eqref{eq:cubic-Gaussian-variational-probability}, uniformly for \(r\in K\).
Taking expectations in \eqref{eq:cubic-Gaussian-conditional-upper-bound}
and using \eqref{eq:cubic-endpoint-empirical-mean-bound} gives
\[
 \limsup_{n\to\infty}\sup_{r\in K}
 \left\{{\cal F}_{n,d}(r,\delta)
       -\sup_{\substack{Q\succeq0\\N\|Q\|_{\HS}^2\ge\delta/2}}
              {\cal V}_{d,r}(Q)\right\}
 \le C_{K,\delta}\varepsilon_{\rm cov}.
\]
Letting \(\varepsilon_{\rm cov}\downarrow0\) proves
\eqref{eq:cubic-Gaussian-variational-average}.
\end{proof}

\subsection{The Gaussian potential and the scalar threshold}
\label{subsec:cubic-scalar-channel}

Proposition~\ref{prop:cubic-Gaussian-variational-bound} reduces the
Gaussian estimate to controlling \({\cal V}_{d,r}(Q)\) when
\(N\|Q\|_{\HS}^2\ge\delta/2\).
We first study its restriction to scalar matrices,
\(v_{d,r}(z):={\cal V}_{d,r}(zI_N/N)
=d^{-1}D(P_{rzdI_N}\Vert\gamma_N)-rz^2/4\), for \(z\ge0\).
Lemma~\ref{lem:cubic-gaussian-entropy-comparison} and the spherical
tensor formula below give the locally uniform limit
\begin{equation}
  v_{d,r}(z)\longrightarrow
  v_r(z):={\cal D}_3(rz)-\frac r4z^2,
  \label{eq:cubic-scalar-channel-limit}
\end{equation}
where \({\cal D}_3(s):=\tfrac12\max_{0\le q<1}
\{sq^3+\log(1-q)+q\}\).
Optimizing \(v_r\) gives the threshold characterized in
Lemma~\ref{lem:cubic-scalar-threshold}.
For compact \(K\subset(0,r_{\mathrm{tens}})\), the limiting potential
is strictly negative for \(z>0\).
Lemma~\ref{lem:cubic-scalar-branch-coercivity} combines this fact with
the small-\(z\) estimate from Lemma~\ref{lem:cubic-directional-mgf}
and the bound \(v_{d,r}(z)\le rz/2-rz^2/4\) for large \(z\), obtaining
\[
 v_{d,r}(z)\le-4\kappa_Kz^2,
 \qquad r\in K,\quad z\ge0,\quad d\ge d_K.
\]
Section~\ref{subsec:cubic-matrix-potential} extends this bound to
general positive semidefinite matrices.

For the scalar limit, let \(X_0\sim\sigma_d\) and let \(G_0\) be an
independent standard Gaussian tensor in \((\mathbb R^d)^{\otimes3}\).
The replica-symmetric (RS) formula in
\cite[Appendix~A.3]{JagannathLopattoMiolane2020}, with dimension \(d\),
order \(k=3\), and parameters \(\lambda=\sqrt s\) and
\(\gamma=\lambda^2=s\), gives
\begin{equation}
 \frac1dD\!\left(
   \operatorname{Law}\{\sqrt{sd}\,X_0^{\otimes3}+G_0\}\,\middle\|\,
   \operatorname{Law}(G_0)\right)
 \longrightarrow {\cal D}_3(s).
 \label{eq:cubic-standard-tensor-channel-limit}
\end{equation}
The Gaussian density formula identifies the relative entropy on the
left with the expected free energy in that reference,
\[
 \frac1d\mathbb E_{X_0,G_0}\log\int
 \exp\left(
   \sqrt{sd}\,\langle G_0,x^{\otimes3}\rangle
   +sd\langle X_0,x\rangle^3-\frac{sd}{2}
 \right)\,d\sigma_d(x).
\]
Since \(P_{sdI_N}=\operatorname{Law}(\sqrt{sd}\,q_{3,d}(X)+g)\),
Lemma~\ref{lem:cubic-gaussian-entropy-comparison} transfers
\eqref{eq:cubic-standard-tensor-channel-limit} to the harmonic feature
with an error at most \(Cs/d\).

\begin{lemma}[Gaussian relative-entropy comparison]
\label{lem:cubic-gaussian-entropy-comparison}
For \(d\ge2\), let \(X_0\sim\sigma_d\) and let \(G_0\) be an
independent standard Gaussian tensor in \((\mathbb R^d)^{\otimes3}\).
There is a universal constant \(C\) such that, for every \(s\ge0\),
\begin{equation}
 \frac1d\left|
 D(P_{sdI_N}\Vert\gamma_N)
 -D\!\left(
   \operatorname{Law}\{\sqrt{sd}\,X_0^{\otimes3}+G_0\}\,\middle\|\,
   \operatorname{Law}(G_0)\right)\right|
 \le\frac{Cs}{d}.
 \label{eq:cubic-harmonic-tensor-free-energy-comparison}
\end{equation}
\end{lemma}

\begin{proof}
At \(s=0\), both relative entropies are zero.  Assume \(s>0\).
We use Gaussian interpolation to bound the left-hand side of
\eqref{eq:cubic-harmonic-tensor-free-energy-comparison} by
\((s/2)\sup_{|a|\le1}|Z_{3,d}(a)-a^3|\).
The feature inner products are
\(\langle x^{\otimes3},y^{\otimes3}\rangle=\langle x,y\rangle^3\)
and
\(\langle q_{3,d}(x),q_{3,d}(y)\rangle=Z_{3,d}(\langle x,y\rangle)\).
By \eqref{eq:cubic-zonal-polynomial},
\begin{equation}
  Z_{3,d}(t)=t^3+\frac3{d-1}(t^3-t),\qquad -1\le t\le1,
  \label{eq:cubic-harmonic-tensor-covariance-comparison}
\end{equation}
so \(\sup_{|t|\le1}|Z_{3,d}(t)-t^3|\le C/d\).

Let \(G_1\sim\Normal(0,I_N)\) be independent of \((X_0,G_0)\), and
interpolate between the two observations through
\[
 \mathsf Y_t^{\mathrm{int}}=
 \left(\sqrt{sd(1-t)}\,X_0^{\otimes3}+G_0,
       \sqrt{sdt}\,q_{3,d}(X_0)+G_1\right),\qquad 0\le t\le1.
\]
The Gaussian density formula and the unit norms of the features give
the conditional log-likelihood ratio at this observation, relative to
\(\operatorname{Law}(G_0,G_1)\), as
\[
 \mathsf{H}_t(x)=h_t(x)+sd\,k_t(x,X_0)-\frac{sd}{2},
\]
where
\[
\begin{aligned}
 h_t(x)&=\sqrt{sd(1-t)}\,\langle G_0,x^{\otimes3}\rangle
           +\sqrt{sdt}\,\langle G_1,q_{3,d}(x)\rangle,\\
 k_t(x,y)&=(1-t)\langle x,y\rangle^3+tZ_{3,d}(x,y).
\end{aligned}
\]
In particular, \(\mathbb E[h_t(x)h_t(y)]=sd\,k_t(x,y)\) and
\(k_t(x,x)=1\).  Integrating \(e^{\mathsf{H}_t(x)}\) against
\(\sigma_d\) gives the mixture likelihood ratio, so
\[
 F_{d,s}(t):=\mathbb E\log\int e^{\mathsf{H}_t(x)}\,d\sigma_d(x)
 =D\!\left(
   \operatorname{Law}(\mathsf Y_t^{\mathrm{int}})\,\middle\|\,
   \operatorname{Law}(G_0,G_1)\right),
\]
where \(\mathbb E\) averages over \(X_0,G_0,G_1\).
At each endpoint, one component is independent pure noise.  Hence
\begin{equation}
\begin{aligned}
 F_{d,s}(0)&=D\!\left(
   \operatorname{Law}\{\sqrt{sd}\,X_0^{\otimes3}+G_0\}\,\middle\|\,
   \operatorname{Law}(G_0)\right),\\
 F_{d,s}(1)&=D(P_{sdI_N}\Vert\gamma_N).
\end{aligned}
\label{eq:cubic-channel-endpoint-entropies}
\end{equation}

To compare these endpoint values, let
\(\nu_t(dx):=e^{\mathsf{H}_t(x)}d\sigma_d(x)/
\int e^{\mathsf{H}_t(z)}d\sigma_d(z)\) denote the posterior law of \(X_0\)
given \(\mathsf Y_t^{\mathrm{int}}\).
For \(k=1,2\), write
\[
 \big\langle f(x^{(1)},\ldots,x^{(k)})\big\rangle_t
 :=\int f(x^{(1)},\ldots,x^{(k)})
        \prod_{j=1}^k\nu_t(dx^{(j)}),
\]
holding \(X_0,G_0,G_1\) fixed, with
\(\langle f(x^{(1)},X_0)\rangle_t=\int f(x,X_0)\,\nu_t(dx)\).
Conditional on \(\mathsf Y_t^{\mathrm{int}}\), the variables
\(X_0,x^{(1)},x^{(2)}\) are iid with law \(\nu_t\).
Bayes' formula therefore gives, for every bounded measurable \(f\),
\begin{equation}
 \mathbb E\langle f(x^{(1)},X_0)\rangle_t
 =\mathbb E\langle f(x^{(1)},x^{(2)})\rangle_t.
 \label{eq:cubic-tensor-channel-replica-identity}
\end{equation}
For \(0<t<1\), differentiating the expected log integral gives
\begin{equation}
 F'_{d,s}(t)
 =\mathbb E\langle\partial_t h_t(x^{(1)})\rangle_t
  +sd\,\mathbb E\langle\partial_t k_t(x^{(1)},X_0)\rangle_t.
 \label{eq:cubic-channel-log-partition-derivative}
\end{equation}
To evaluate its first term, the definitions of \(h_t\) and \(k_t\) give
\[
 \mathbb E_{G_0,G_1}[\partial_t h_t(x)\,h_t(y)]
 =\frac{sd}{2}\partial_t k_t(x,y),\qquad
 \partial_t k_t(x,y)=Z_{3,d}(x,y)-\langle x,y\rangle^3.
\]
Conditional on \(X_0\), Gaussian integration by parts in the Gibbs
expectation, as in \cite[Appendix~A.4]{JagannathLopattoMiolane2020},
therefore yields
\begin{equation}
\begin{aligned}
 \mathbb E\langle\partial_t h_t(x^{(1)})\rangle_t
 &=\frac{sd}{2}\mathbb E\big\langle
     \partial_t k_t(x^{(1)},x^{(1)})
       -\partial_t k_t(x^{(1)},x^{(2)})\big\rangle_t\\
 &=-\frac{sd}{2}\mathbb E\langle
      \partial_t k_t(x^{(1)},x^{(2)})\rangle_t,
\end{aligned}
\label{eq:cubic-channel-noise-derivative}
\end{equation}
where the diagonal term vanishes because \(k_t(x,x)=1\).
Here the factor \(\exp\{sd\,k_t(x,X_0)-sd/2\}\) in \(\nu_t\) is
deterministic conditional on \(X_0\) and does not change the
Gaussian integration-by-parts identity.
Substituting \eqref{eq:cubic-channel-noise-derivative} into
\eqref{eq:cubic-channel-log-partition-derivative} and using
\eqref{eq:cubic-tensor-channel-replica-identity} gives
\begin{equation}
 F'_{d,s}(t)=\frac{sd}{2}\mathbb E\langle
          \partial_t k_t(x^{(1)},x^{(2)})\rangle_t.
 \label{eq:cubic-harmonic-tensor-interpolation-derivative}
\end{equation}

For fixed \(d,s\) and \(0<\delta<1/2\), the bounds
\[
 \sup_{\substack{0\le t\le1\\x\in\Sph^{d-1}}}|\mathsf{H}_t(x)|
 \le\sqrt{sd}(\|G_0\|+\|G_1\|)+\frac{3sd}{2},
 \qquad
 \sup_{\substack{\delta\le t\le1-\delta\\x\in\Sph^{d-1}}}
 |\partial_t\mathsf{H}_t(x)|
 \le\frac{\sqrt{sd}}{2\sqrt\delta}(\|G_0\|+\|G_1\|)+2sd
\]
have right-hand sides with finite expectations under the joint law of
\(X_0,G_0,G_1\), because Gaussian norms have finite first moments.
These bounds justify differentiation under the integrals on
\([\delta,1-\delta]\) and continuity of \(F_{d,s}\) on \([0,1]\).
Integrating \eqref{eq:cubic-harmonic-tensor-interpolation-derivative}
on \([\delta,1-\delta]\), then letting \(\delta\downarrow0\) and
using \eqref{eq:cubic-harmonic-tensor-covariance-comparison}, gives
\[
 \frac{|F_{d,s}(1)-F_{d,s}(0)|}{d}
 \le\frac s2\sup_{|a|\le1}|Z_{3,d}(a)-a^3|
 \le\frac{Cs}{d}.
\]
Together with \eqref{eq:cubic-channel-endpoint-entropies}, this proves
\eqref{eq:cubic-harmonic-tensor-free-energy-comparison}.
\end{proof}

Put \(f_d(s):=d^{-1}D(P_{sdI_N}\Vert\gamma_N)\).
Combining Lemma~\ref{lem:cubic-gaussian-entropy-comparison} with
\eqref{eq:cubic-standard-tensor-channel-limit} gives
\(f_d(s)\to{\cal D}_3(s)\) for every fixed \(s\ge0\).
To obtain the locally uniform convergence asserted in
\eqref{eq:cubic-scalar-channel-limit}, we bound \(f_d'\) uniformly in \(d\).
For a scalar signal-to-noise parameter \(\alpha>0\), the relative-entropy
form of the Gaussian I--MMSE identity
\cite[Theorem~5]{GuoShamaiVerdu2005} states that
\[
 \frac{d}{d\alpha}D(P_{\alpha I_N}\Vert\gamma_N)
 =\frac12\mathbb E\left\|
   \mathbb E(u\mid\sqrt\alpha\,u+g)\right\|^2.
\]
Its finite-second-moment assumption holds because \(\|u\|=1\).
Setting \(\alpha=sd\) and applying the chain rule therefore gives
\[
 0\le f_d'(s)
 =\frac12\mathbb E\left\|
   \mathbb E(u\mid\sqrt{sd}\,u+g)\right\|^2\le\frac12
\]
for \(s>0\), with the right derivative at \(s=0\).
The last inequality follows from conditional Jensen's inequality.
Thus \(f_d\) and its pointwise limit \({\cal D}_3\) are
\(1/2\)-Lipschitz.  Pointwise convergence on a finite mesh, followed by
this common Lipschitz bound between mesh points, yields
\[
 \sup_{0\le s\le S}|f_d(s)-{\cal D}_3(s)|\longrightarrow0
 \qquad\text{for every }0<S<\infty.
\]
Since \(v_{d,r}(z)=f_d(rz)-rz^2/4\), for every nonempty compact
\(K\subset(0,\infty)\) and \(0<Z<\infty\), we obtain
\begin{equation}
 \sup_{\substack{r\in K\\0\le z\le Z}}
 |v_{d,r}(z)-v_r(z)|
 \le\sup_{0\le s\le(\max K)Z}|f_d(s)-{\cal D}_3(s)|
 \longrightarrow0.
 \label{eq:cubic-scalar-potential-uniform-limit}
\end{equation}
This proves \eqref{eq:cubic-scalar-channel-limit}, uniformly also for
\(r\) in compact subsets of \((0,\infty)\).
We next determine for which \(r\) the limiting function \(v_r\)
has a positive maximum.  Optimizing over \(z\ge0\) gives
\[
\begin{split}
  \sup_{z\ge0}v_r(z)
  &=
  \sup_{\substack{z\ge0\\0\le q<1}}
  \left[
    \frac r2zq^3-\frac r4z^2
    +\frac12\{\log(1-q)+q\}
  \right]\\
  &=
  \sup_{0\le q<1}
  \left[
    \frac r4q^6+\frac12\{\log(1-q)+q\}
  \right].
  \end{split}
\]
For each fixed \(q\), completing the square gives the optimizer
\(z=q^3\).  Thus \(\sup_zv_r(z)>0\) exactly when
\(r>r_{\mathrm{tens}}\), where the threshold from
Theorem~\ref{thm:pure-cubic-critical} is
\begin{equation}
  r_{\mathrm{tens}}
  =
  \min_{0<q<1}
  \frac{-2\{\log(1-q)+q\}}{q^6}.
  \label{eq:cubic-tensor-threshold}
\end{equation}
Lemma~\ref{lem:cubic-scalar-threshold} identifies the unique minimizer
and gives its stationarity equation.

\begin{lemma}[Scalar threshold and its optimizer]
\label{lem:cubic-scalar-threshold}
The minimum in \eqref{eq:cubic-tensor-threshold} is attained at a unique
point \(q_\star\in(4/5,1)\).  It is characterized by
\begin{equation}
 \frac{q_\star^2}{1-q_\star}
 =6\{-\log(1-q_\star)-q_\star\},
 \qquad
 r_{\mathrm{tens}}
 =\frac1{3q_\star^4(1-q_\star)}.
 \label{eq:cubic-threshold-stationarity}
\end{equation}
\end{lemma}

\begin{proof}
Let \(f(q)=-2\{\log(1-q)+q\}/q^6\).  Since
\(-\log(1-q)-q\sim q^2/2\) as \(q\downarrow0\) and
\(-\log(1-q)-q\to\infty\) as \(q\uparrow1\),
\(f(q)\to\infty\) at both endpoints.  Direct differentiation gives
\[
 f'(q)=\frac2{q^7}g(q),
 \qquad
 g(q):=\frac{q^2}{1-q}+6\{\log(1-q)+q\},
 \qquad
 g'(q)=\frac{q(5q-4)}{(1-q)^2}.
\]
The expansion \(g(q)=-2q^2+O(q^3)\) shows that \(g<0\) near zero.
The derivative formula makes \(g\) decreasing on \((0,4/5)\) and strictly
increasing on \((4/5,1)\), while \(g(q)\to\infty\) as \(q\uparrow1\).
Thus \(g\) has a unique zero \(q_\star\in(4/5,1)\), which is the unique
minimizer of \(f\).  At this zero, substituting
\(6\{-\log(1-q_\star)-q_\star\}=q_\star^2/(1-q_\star)\)
into \(f(q_\star)\) gives the second
identity in \eqref{eq:cubic-threshold-stationarity}.
\end{proof}

Numerically, \(q_\star\approx0.9063\).
It follows from Lemma~\ref{lem:cubic-scalar-threshold} and
\eqref{eq:cubic-tensor-threshold} that, for
\(0<r<r_{\mathrm{tens}}\), we have \(v_r(0)=0\) and
\(v_r(z)<0\) for every \(z>0\), so \(z=0\) is the unique global
maximizer.  At \(r=r_{\mathrm{tens}}\),
\(0\) and \(q_\star^3\) are both global maximizers.

For \(r\) in a compact subset of \((0,r_{\mathrm{tens}})\), the strict
negativity of \(v_r\) and
\eqref{eq:cubic-scalar-potential-uniform-limit} give a negative quadratic
bound for \(v_{d,r}(z)\) on bounded intervals separated from zero.
Near \(z=0\), we need
an estimate on the \(z^2\) scale, since the absolute \(o(1)\) error in
\eqref{eq:cubic-scalar-potential-uniform-limit} need not be
\(o(z^2)\) when \(z\) decreases with \(d\).
Lemma~\ref{lem:cubic-directional-mgf} provides the exponential-moment
estimate used in \eqref{eq:cubic-small-channel-chi-square} to prove
\(d^{-1}D(P_{sdI_N}\Vert\gamma_N)\le Cs^2/d^2\)
for sufficiently small \(s\).

\begin{lemma}[Directional exponential-moment bound]
\label{lem:cubic-directional-mgf}
There are universal constants \(c_0,C>0\) such that, for every
\(v\in\mathbb S^{N-1}\),
\begin{equation}
 \log\mathbb E_{\sigma_d}
  \exp\{\theta\langle v,q_{3,d}(X)\rangle\}
 \le\frac{C\theta^2}{N},
 \qquad |\theta|\le c_0d.
 \label{eq:cubic-directional-mgf}
\end{equation}
\end{lemma}

\begin{proof}
Put \(F_v(x)=\langle v,q_{3,d}(x)\rangle\).  Oddness of \(q_{3,d}\) and
\(\mathbb E_{\sigma_d}[q_{3,d}(X)q_{3,d}(X)^{\mathsf T}]=I_N/N\),
from Section~\ref{subsec:cubic-harmonic-kernel}, give
\(\mathbb E_{\sigma_d}F_v=0\),
and \(\mathbb E_{\sigma_d}F_v^2=\|v\|^2/N=N^{-1}\).
Also, \(|F_v(x)|\le\|v\|\|q_{3,d}(x)\|=1\).
For \(d=2\), we have \(N=2\), and oddness gives
\(\mathbb E_{\sigma_d}e^{\theta F_v}
=\mathbb E_{\sigma_d}\cosh(\theta F_v)
\le e^{\theta^2/2}=e^{\theta^2/N}\), which proves the claim.
We may therefore assume \(d\ge3\).
On each unit-speed great circle, \(F_v\) is a trigonometric polynomial
of degree at most three.  Applying the single-sphere part of the
Bernstein argument in Step 2 of the proof of
Lemma~\ref{lem:cubic-compensated-row} gives
\(\|\nabla^2F_v\|_{\rm op}\le9\|F_v\|_\infty\le9\).
Let \(\nu_t(dx):=e^{tF_v(x)}\,d\sigma_d(x)/
\mathbb E_{\sigma_d}e^{tF_v}\), and write \(\mathbb E_t\) for
expectation under \(\nu_t\).
For a density proportional to \(e^{-V}\) relative to Riemannian volume,
the Bakry--{\'E}mery Ricci tensor is
\(\operatorname{Ric}+\nabla^2V\)
\cite[p.~377]{WeiWylie2009}.  Taking \(V=-tF_v\) gives
\(\operatorname{Ric}_{\nu_t}=\operatorname{Ric}-t\nabla^2F_v\).
On the unit sphere \(\mathbb S^{d-1}\),
\(\operatorname{Ric}(\xi,\xi)=(d-2)\|\xi\|^2\).
Hence, taking \(c_0\le1/54\), for \(|t|\le c_0d\) and every tangent
vector \(\xi\), we obtain
\[
 \operatorname{Ric}_{\nu_t}(\xi,\xi)
 \ge(d-2-9|t|)\|\xi\|^2\ge\frac d6\|\xi\|^2.
\]
The curvature criterion and its spectral-gap consequence
\cite[Proposition~3 and p.~203]{BakryEmery1985} therefore give
\begin{equation}
 \operatorname{Var}_t(H)\le\frac6d\mathbb E_t\|\nabla H\|^2
 \qquad\text{for every smooth }H.
 \label{eq:cubic-directional-poincare}
\end{equation}

Put \(G=\|\nabla F_v\|^2\).  Using
\(\|\nabla^2F_v\|_{\rm op}\le9\), we obtain
\[
 \nabla G=2\nabla^2F_v(\nabla F_v,\cdot),\qquad
 \|\nabla G\|^2\le4\|\nabla^2F_v\|_{\rm op}^2G\le324G.
\]
Differentiating the normalized expectation under \(\nu_t\), applying
Cauchy--Schwarz, and using \eqref{eq:cubic-directional-poincare} for
\(H=G\) and \(H=F_v\) gives
\begin{equation}
 \left|\frac d{dt}\mathbb E_tG\right|
 =|\operatorname{Cov}_t(G,F_v)|
 \le\sqrt{\operatorname{Var}_t(G)\operatorname{Var}_t(F_v)}
 \le\frac C d\mathbb E_tG.
 \label{eq:cubic-directional-energy-derivative}
\end{equation}
Since \(F_v\in\cH_{3,d}\), we have
\(-\Delta F_v=3(d+1)F_v\), where
\(\Delta=\operatorname{div}\nabla\) is the spherical Laplace--Beltrami
operator.  Integration by parts under \(\nu_0=\sigma_d\) gives
\[
 \mathbb E_0G=-\mathbb E_{\sigma_d}(F_v\Delta F_v)
 =3(d+1)\mathbb E_{\sigma_d}F_v^2=\frac{3(d+1)}N.
\]
Gronwall's inequality applied to
\eqref{eq:cubic-directional-energy-derivative}, in both directions from
\(t=0\), yields
\[
 \mathbb E_tG\le e^{C|t|/d}\frac{3(d+1)}N\le\frac{Cd}N,
 \qquad |t|\le c_0d.
\]
By \eqref{eq:cubic-directional-poincare}, this implies
\(\operatorname{Var}_t(F_v)\le(6/d)\mathbb E_tG\le C/N\) throughout
the same interval.
Finally, let \(\Lambda_v(\theta)=\log\mathbb E_{\sigma_d}e^{\theta F_v}\).
Then \(\Lambda_v(0)=\Lambda_v'(0)=0\) and
\(\Lambda_v''(t)=\operatorname{Var}_t(F_v)\), so Taylor's formula with
integral remainder gives
\[
 \Lambda_v(\theta)
 =\theta^2\int_0^1(1-u)\operatorname{Var}_{u\theta}(F_v)\,du
 \le\frac{C\theta^2}N,\qquad |\theta|\le c_0d.
\]
This proves \eqref{eq:cubic-directional-mgf}.
\end{proof}

We now prove the uniform quadratic bound in
Lemma~\ref{lem:cubic-scalar-branch-coercivity}, beginning with the
small-\(z\) estimate obtained from \eqref{eq:cubic-directional-mgf}.

\begin{lemma}[Uniform quadratic bound for the scalar potential]
\label{lem:cubic-scalar-branch-coercivity}
For every compact
\(K\subset(0,r_{\mathrm{tens}})\), there are
\(\kappa_K>0\) and \(d_K<\infty\) such that, uniformly for
\(r\in K\), \(d\ge d_K\), and \(z\ge0\),
\begin{equation*}
 v_{d,r}(z)\le-4\kappa_Kz^2.
\end{equation*}
\end{lemma}

\begin{proof}
We first bound \(D(P_{sdI_N}\Vert\gamma_N)\) for small \(s\).
Recall that the likelihood-ratio function is \(Z_B=dP_B/d\gamma_N\).
Jensen's inequality under \(P_B\) gives
\[
 D(P_B\Vert\gamma_N)
 =\mathbb E_{P_B}\log Z_B
 \le\log\mathbb E_{P_B}Z_B
 =\log\mathbb E_{\gamma_N}Z_B^2.
\]
To estimate \(\mathbb E_{\gamma_N}Z_B^2\), let \(u,u'\) be independent with
law \(\mu_d\), and let \(g\sim\gamma_N\) be independent of both.
Here \(\mathbb E_{\gamma_N}Z_B^2=\mathbb E_g[Z_B(g)^2]\).
Expanding \(Z_B(g)^2\) and integrating first over \(g\), we obtain
\[
\begin{aligned}
 \mathbb E_g[Z_B(g)^2]
 &=\mathbb E_{u,u'}\left[
   e^{-\frac12u^{\mathsf T}Bu-\frac12u'^{\mathsf T}Bu'}
   \,\mathbb E_g e^{\langle g,B^{1/2}(u+u')\rangle}\right]\\
 &=\mathbb E_{u,u'}\exp\left(
   \frac12\|B^{1/2}(u+u')\|^2
   -\frac12u^{\mathsf T}Bu-\frac12u'^{\mathsf T}Bu'\right)\\
 &=\mathbb E_{u,u'}e^{u^{\mathsf T}Bu'}.
\end{aligned}
\]
The second equality uses
\(\mathbb E_g e^{\langle g,h\rangle}=e^{\|h\|^2/2}\) with
\(h=B^{1/2}(u+u')\), keeping \(u,u'\) fixed.
Now take \(B=sdI_N\), \(u=q_{3,d}(X)\), and \(u'=q_{3,d}(X')\), where
\(X,X'\) are independent with law \(\sigma_d\).  Substituting
\(\mathbb E_{\gamma_N}Z_B^2=\mathbb E_{u,u'}e^{u^{\mathsf T}Bu'}\)
into \(D(P_B\Vert\gamma_N)\le\log\mathbb E_{\gamma_N}Z_B^2\)
and using
\(\langle q_{3,d}(X),q_{3,d}(X')\rangle=Z_{3,d}(X,X')\), we obtain
\begin{equation}
 D(P_{sdI_N}\Vert\gamma_N)
 \le
 \log\mathbb E_{X,X'}
 \exp\{sd\,Z_{3,d}(X,X')\}
 \le C\frac{s^2d^2}{N}
 \le C\frac{s^2}{d},
\label{eq:cubic-small-channel-chi-square}
\end{equation}
for \(0\le s\le s_0\), where \(s_0>0\) is absolute.
For the second inequality, condition on \(X\)
and apply Lemma~\ref{lem:cubic-directional-mgf} with
\(v=q_{3,d}(X)\) and \(\theta=sd\).  Since \(\|v\|=1\), the lemma gives
the bound \(Cs^2d^2/N\) whenever \(s\le s_0\le c_0\).

Write \(K=[r_-,r_+]\) after enlarging it harmlessly to an interval.  With
\(s=rz\), \eqref{eq:cubic-small-channel-chi-square} yields, for all large
\(d\),
\[
 v_{d,r}(z)
 \le C\frac{r^2z^2}{d^2}-\frac r4z^2
 \le-\frac{r_-}{8}z^2,
 \qquad 0\le rz\le s_0.
\]
To cover a bounded interval away from zero, choose
\(Z>\max\{4,s_0/r_+\}\).  Strict negativity of \(v_r(z)\) for
\(r\in K\) and \(z>0\), together with continuity, gives
\[
 c_K:=\frac12\min_{\substack{r\in K\\s_0/r_+\le z\le Z}}
       \frac{-v_r(z)}{z^2}>0.
\]
By \eqref{eq:cubic-scalar-potential-uniform-limit}, for all large \(d\),
\(\sup_{r\in K,\,0\le z\le Z}|v_{d,r}(z)-v_r(z)|
\le c_K(s_0/r_+)^2\).  Hence, uniformly for \(r\in K\) and
\(s_0/r_+\le z\le Z\),
\[
 v_{d,r}(z)\le v_r(z)+c_K\left(\frac{s_0}{r_+}\right)^2
 \le-2c_Kz^2+c_K\left(\frac{s_0}{r_+}\right)^2\le-c_Kz^2.
\]
It remains to treat \(z\ge Z\).  We use
\[
 D(P_B\Vert\gamma_N)
 =\frac12\frac{\tr B}{N}-I(u;B^{1/2}u+g)
 \le\frac12\frac{\tr B}{N}
\]
where the inequality follows from nonnegativity of mutual information.
Substituting \(B=rzdI_N\), we obtain
\[
 v_{d,r}(z)\le\frac r2z-\frac r4z^2
 \le-\frac{r_-}{8}z^2,\qquad z\ge Z,
\]
because \(Z>4\).  Choose \(\kappa_K>0\) with
\(4\kappa_K\le\min\{r_-/8,c_K\}\).
Combining the three ranges gives \(v_{d,r}(z)\le-4\kappa_Kz^2\) for
all \(z\ge0\), uniformly for \(r\in K\) and all sufficiently large
\(d\), proving Lemma~\ref{lem:cubic-scalar-branch-coercivity}.
\end{proof}

\subsection{Scalar maximizers and a trace perturbation}
\label{subsec:cubic-matrix-potential}

Lemma~\ref{lem:cubic-scalar-branch-coercivity} gives
\({\cal V}_{d,r}(zI_N/N)\le-4\kappa_Kz^2\) for \(z\ge0\), uniformly
for \(r\) in a compact subset \(K\subset(0,r_{\mathrm{tens}})\) and all
sufficiently large \(d\).  The objective of this subsection is the matrix
bound in Proposition~\ref{prop:cubic-matrix-branch-gap},
\[
 {\cal V}_{d,r}(Q)\le-\kappa_KN\|Q\|_{\HS}^2,
 \qquad Q\succeq0.
\]
To reduce the matrix optimization to scalar matrices, we first combine
the first- and second-order conditions at a positive-definite local
maximum with rotational invariance to show that every such maximum is
a scalar multiple of \(I_N\).  We then apply the same argument to
\({\cal V}_{d,r}(Q)+\kappa_KN\|Q\|_{\HS}^2+\varepsilon\tr Q\).
The positive trace term forces its maximizers to be positive definite,
and letting \(\varepsilon\downarrow0\) yields the required bound.

The derivative formulas involve the posterior moments of \(u\).
For \(B\succeq0\), let \(\mathsf Y_B=B^{1/2}u+g\), so that
\(\mathsf Y_B\) has law \(P_B\).  Bayes' formula defines the posterior
law of \(u\) given \(\mathsf Y_B=y\), for every \(y\in\mathbb R^N\), by
\[
 \mu_{B,y}(dv)
 :=\frac{\exp\left(\langle y,B^{1/2}v\rangle
              -\frac12v^{\mathsf T}Bv\right)}{Z_B(y)}\,\mu_d(dv).
\]
Its mean and covariance are the continuous functions
\[
 m_B(y):=\int v\,\mu_{B,y}(dv),\qquad
 \Sigma_B(y):=\int (v-m_B(y))(v-m_B(y))^{\mathsf T}\,\mu_{B,y}(dv).
\]
Averaging over \(\mathsf Y_B\sim P_B\) defines the deterministic
second-moment matrix of the posterior mean,
\[
  M(B):=\mathbb E\{m_B(\mathsf Y_B)m_B(\mathsf Y_B)^{\mathsf T}\}.
\]
Lemma~\ref{lem:cubic-gaussian-entropy-derivatives} expresses the first
and second derivatives of \(D(P_B\Vert\gamma_N)\) in terms of
these posterior moments.  Related derivatives with respect to the Gaussian
channel matrix are given in \cite[Theorem~2]{PayaroPalomar2009}.
The proof below works directly with the signal-to-noise matrix \(B\).

\begin{lemma}[Derivatives of the Gaussian relative entropy]
\label{lem:cubic-gaussian-entropy-derivatives}
Let \(B\succ0\) and let \(H\) be symmetric.  With the gradient taken
in the Hilbert--Schmidt inner product,
\begin{align}
  \nabla_BD(P_B\Vert\gamma_N)
  &=\frac12M(B),
  \label{eq:cubic-matrix-immse}\\
  \nabla_B^2 D(P_B\Vert\gamma_N)[H,H]
  &=\frac12\mathbb E\tr\{\Sigma_B(\mathsf Y_B)H\Sigma_B(\mathsf Y_B)H\}.
  \label{eq:cubic-matrix-hessian}
\end{align}
The derivative of \(M\) is
\begin{equation}
 D M(B)[H]
 =\mathbb E\{\Sigma_B(\mathsf Y_B)H\Sigma_B(\mathsf Y_B)\}.
 \label{eq:cubic-posterior-mean-derivative}
\end{equation}
For \(B\succeq0\) and \(H\succeq0\),
\eqref{eq:cubic-matrix-immse} extends to the one-sided derivative
\[
 \left.\frac d{dt}D(P_{B+tH}\Vert\gamma_N)\right|_{t=0+}
 =\frac12\tr\{M(B)H\}.
\]
\end{lemma}

\begin{proof}
We represent \(B+tH\), initially for \(H\succeq0\), by an additional
Gaussian observation.  The I--MMSE identity gives
\eqref{eq:cubic-matrix-immse}, and an \(L^2\) expansion of the posterior
mean gives \eqref{eq:cubic-posterior-mean-derivative}.
We then extend to arbitrary symmetric directions and to the boundary.

\noindent\textit{Step 1. Represent the matrix perturbation.}
Fix \(B\succ0\), \(H\succeq0\), and let
\(\mathsf Y_t^H=\sqrt t\,H^{1/2}u+g'\), with \(g'\sim\gamma_N\)
independent of \((u,g)\).
Writing \(\phi_N\) for the density of \(\gamma_N\), we have
\[
 \log\frac{\phi_N(y-B^{1/2}u)\phi_N(z-\sqrt t\,H^{1/2}u)}
               {\phi_N(y)\phi_N(z)}
 =\langle u,B^{1/2}y+\sqrt t\,H^{1/2}z\rangle
  -\frac12u^{\mathsf T}(B+tH)u.
\]
Thus \(S_t=B^{1/2}\mathsf Y_B+\sqrt t\,H^{1/2}\mathsf Y_t^H\)
is sufficient for \(u\), and
\(\operatorname{Law}(S_t\mid u)=\Normal((B+tH)u,B+tH)\).
Whitening \(S_t\) gives the conditional law of \(\mathsf Y_{B+tH}\).
Sufficiency and the chain rule therefore yield
\[
 I(u;\mathsf Y_{B+tH})
 =I(u;\mathsf Y_B,\mathsf Y_t^H)
 =I(u;\mathsf Y_B)+I(u;\mathsf Y_t^H\mid\mathsf Y_B).
\]

\noindent\textit{Step 2. Differentiate relative entropy.}
Let \(m_{B,t}=\mathbb E(u\mid\mathsf Y_B,\mathsf Y_t^H)\).
The vector I--MMSE identity \cite[Theorem~2]{GuoShamaiVerdu2005} gives
\[
 0\le\frac d{dt}I(u;\mathsf Y_t^H\mid\mathsf Y_B=y)
 =\frac12\mathbb E\bigl(\|H^{1/2}(u-m_{B,t})\|^2
                         \mid\mathsf Y_B=y\bigr)
 \le\frac12\|H\|_{\rm op}.
\]
Since the conditional mutual information is zero at \(t=0\), its
difference quotient is bounded by \(\|H\|_{\rm op}/2\), which is
integrable under the fixed probability law \(P_B(dy)\).
Dominated convergence and \(m_{B,0}=m_B(\mathsf Y_B)\) imply
\[
 \left.\frac d{dt}I(u;\mathsf Y_{B+tH})\right|_{t=0+}
 =\frac12\mathbb E\tr\{H\Sigma_B(\mathsf Y_B)\}.
\]
Using \(\mathbb E\Sigma_B(\mathsf Y_B)=N^{-1}I_N-M(B)\) and
\(D(P_B\Vert\gamma_N)=\tr B/(2N)-I(u;\mathsf Y_B)\), we obtain
\[
 \left.\frac d{dt}D(P_{B+tH}\Vert\gamma_N)\right|_{t=0+}
 =\frac{\tr H}{2N}
  -\frac12\tr\{H(N^{-1}I_N-M(B))\}
 =\frac12\tr\{HM(B)\}.
\]

\noindent\textit{Step 3. Differentiate the posterior second moment.}
Step~1 and the tower property give
\(\mathbb E(m_{B,t}m_{B,t}^{\mathsf T})=M(B+tH)\) and
\(\mathbb E(m_{B,t}\mid\mathsf Y_B)=m_B(\mathsf Y_B)\).
Consequently,
\[
 \frac{M(B+tH)-M(B)}t
 =\mathbb E\left\{
 \frac{m_{B,t}-m_B(\mathsf Y_B)}{\sqrt t}
 \left(\frac{m_{B,t}-m_B(\mathsf Y_B)}{\sqrt t}\right)^{\mathsf T}
 \right\}.
\]
We show that the normalized difference converges in \(L^2\) to
\(\Sigma_B(\mathsf Y_B)H^{1/2}g'\).
For fixed \(\mathsf Y_B=y,u,g'\), Bayes' formula gives
\[
 \nu_s(dv)=\frac{e^{k_s(v)}\mu_{B,y}(dv)}
                   {\int e^{k_s(w)}\mu_{B,y}(dw)},\qquad
 m_s=\int v\,\nu_s(dv),\qquad 0\le s\le1,
\]
where \(k_s(v)=s(g')^{\mathsf T}H^{1/2}v
+s^2(u^{\mathsf T}Hv-v^{\mathsf T}Hv/2)\).
In particular, \(m_0=m_B(y)\) and \(m_{\sqrt t}=m_{B,t}\).
Differentiating these normalized integrals gives
\[
 \begin{aligned}
 m_s'&=\int(v-m_s)\partial_s k_s(v)\,\nu_s(dv),\\
 m_s''&=\int(v-m_s)\left\{\partial_s^2k_s(v)
   +\left(\partial_s k_s(v)-\int\partial_s k_s(w)\,\nu_s(dw)\right)^2
                         \right\}\nu_s(dv).
 \end{aligned}
\]
At zero, \(m_0'=\Sigma_B(y)H^{1/2}g'\).
For \(0\le s\le1\), the unit norms of \(u,v\) give
\[
 \|\partial_s k_s\|_\infty
 \le\|H\|_{\rm op}^{1/2}\|g'\|+3\|H\|_{\rm op},
 \qquad \|\partial_s^2k_s\|_\infty\le3\|H\|_{\rm op}.
\]
Since \(\|v-m_s\|\le2\), it follows that
\[
 \|m_s''\|
 \le2\|\partial_s^2k_s\|_\infty+8\|\partial_s k_s\|_\infty^2
 \le C_H(1+\|g'\|^2).
\]
Taylor's formula yields
\(m_{B,t}-m_B(y)=\sqrt t\,\Sigma_B(y)H^{1/2}g'+\mathcal R_t\),
where \(\mathcal R_t=\int_0^{\sqrt t}(\sqrt t-s)m_s''\,ds\).
The bound
\[
 \mathbb E(\|\mathcal R_t\|^2\mid\mathsf Y_B=y)
 \le\frac{C_H^2t^2}{4}\mathbb E(1+\|g'\|^2)^2
 =\frac{C_H^2}{4}(N^2+4N+1)t^2=:C_{N,H}t^2
\]
is uniform in \(y\), since \(g'\) is independent of \(\mathsf Y_B\).
It proves the required \(L^2\) convergence under the joint law of
\((u,g,g')\), with \(B,H,N\) fixed.
Moreover,
\(\mathbb E\|\Sigma_B(\mathsf Y_B)H^{1/2}g'\|^2
=\mathbb E\tr\{H\Sigma_B(\mathsf Y_B)^2\}\le\|H\|_{\rm op}\).
Cauchy--Schwarz therefore gives
\[
 \left\|\frac{M(B+tH)-M(B)}t
       -\mathbb E\{\Sigma_B(\mathsf Y_B)H\Sigma_B(\mathsf Y_B)\}
 \right\|_{\HS}
 \le2\sqrt{C_{N,H}\|H\|_{\rm op}t}+C_{N,H}t\longrightarrow0,
\]
proving \eqref{eq:cubic-posterior-mean-derivative} for \(H\succeq0\).

\noindent\textit{Step 4. Extend the derivatives and obtain the Hessian.}
Using \(dP_B=Z_B\,d\gamma_N\), write
\[
 D(P_B\Vert\gamma_N)=\int Z_B\log Z_B\,d\gamma_N,\qquad
 M(B)=\int Z_Bm_Bm_B^{\mathsf T}\,d\gamma_N.
\]
On a compact neighborhood \(\mathcal U\) within the cone of
positive-definite matrices, \(B\mapsto B^{1/2}\) has bounded continuous
derivative.  Differentiating the defining integrals and their quotient
gives
\[
 |D_BZ_B(y)[H]|\le C_{\mathcal U}(1+\|y\|)Z_B(y)\|H\|_{\HS},
 \qquad
 \|D_Bm_B(y)[H]\|\le2C_{\mathcal U}(1+\|y\|)\|H\|_{\HS}.
\]
Together with \(Z_B(y)\le e^{C_{\mathcal U}\|y\|}\),
\(|\log Z_B(y)|\le C_{\mathcal U}(1+\|y\|)\), and \(\|m_B(y)\|\le1\),
these bounds imply
\[
 \begin{aligned}
 &|D_B[Z_B(y)\log Z_B(y)][H]|
 +\|D_B[Z_B(y)m_B(y)m_B(y)^{\mathsf T}][H]\|_{\HS}\\
 &\qquad\le C_{\mathcal U}(1+\|y\|)^2
                    e^{C_{\mathcal U}\|y\|}\|H\|_{\HS}.
 \end{aligned}
\]
The right-hand side is \(\gamma_N\)-integrable for fixed \(N\), so both
integrals are continuously differentiable.
Writing \(H=H_+-H_-\), with \(H_\pm\succeq0\), extends
\eqref{eq:cubic-matrix-immse} and
\eqref{eq:cubic-posterior-mean-derivative} to every symmetric direction.
Differentiating the former now gives
\[
 \nabla_B^2D(P_B\Vert\gamma_N)[H,H]
 =\frac12\tr\{HDM(B)[H]\}
 =\frac12\mathbb E\tr\{\Sigma_B(\mathsf Y_B)H\Sigma_B(\mathsf Y_B)H\}.
\]

For the boundary assertion, \(B_j\to B\) with
\(0\preceq B_j\preceq RI_N\) implies pointwise convergence of
\(Z_{B_j}(y)\) and \(m_{B_j}(y)\).
The bounds
\(e^{-\sqrt R\|y\|-R/2}\le Z_{B_j}(y)\le e^{\sqrt R\|y\|}\) and
\(\|m_{B_j}(y)\|\le1\) give the common integrable envelope
\((1+\sqrt R\|y\|+R/2)e^{\sqrt R\|y\|}\)
for the two integrands above.
Thus relative entropy and \(M\) are continuous on \(B\succeq0\).
For \(B,H\succeq0\), integrate \eqref{eq:cubic-matrix-immse} along
\(B+\varepsilon I_N+sH\) and let \(\varepsilon\downarrow0\).
Since \(0\le\tr\{M(B+\varepsilon I_N+sH)H\}\le\|H\|_{\rm op}\),
dominated convergence gives
\[
 \frac{D(P_{B+tH}\Vert\gamma_N)-D(P_B\Vert\gamma_N)}t
 =\frac1{2t}\int_0^t\tr\{M(B+sH)H\}\,ds
 \longrightarrow\frac12\tr\{M(B)H\}.
\]
\end{proof}

We now apply Lemma~\ref{lem:cubic-gaussian-entropy-derivatives} to the
matrix potential.  For \(Q\succ0\), the chain rule with \(B=rNdQ\) gives
\begin{align}
  \nabla_Q{\cal V}_{d,r}(Q)
  &=\frac{rN}{2}\{M(B)-Q\},
  \label{eq:cubic-branch-gradient}\\
  D^2{\cal V}_{d,r}(Q)[H,H]
  &=\frac{rN}{2}
  \left[
    rNd\,\mathbb E\tr\{\Sigma_B(\mathsf Y_B)H\Sigma_B(\mathsf Y_B)H\}
    -\|H\|_{\HS}^2
  \right].
  \label{eq:cubic-branch-hessian}
\end{align}

For \(B\succ0\), define the self-adjoint map on symmetric matrices by
\[
 {\cal T}_B(H):=\mathbb E\{\Sigma_B(\mathsf Y_B)H\Sigma_B(\mathsf Y_B)\}.
\]
Since the expectation in \eqref{eq:cubic-branch-hessian} equals
\(\langle H,{\cal T}_B(H)\rangle_{\HS}\), the gradient formula
\eqref{eq:cubic-branch-gradient} and the nonpositivity of the Hessian
imply that every positive-definite local maximum satisfies
\begin{equation}
 Q=M(rNdQ),\qquad
 \lambda_{\max}({\cal T}_{rNdQ})\le\frac1{rNd}.
 \label{eq:cubic-local-maximizer-conditions}
\end{equation}
To combine \eqref{eq:cubic-local-maximizer-conditions} with rotational
invariance, we need to identify the eigenspace corresponding to the
largest eigenvalue of \({\cal T}_B\).
Lemma~\ref{lem:cubic-posterior-covariance-irreducible} establishes the
irreducibility needed to show that this eigenspace is one-dimensional
and spanned by a positive-definite matrix.

\begin{lemma}[Irreducibility of the posterior covariance family]
\label{lem:cubic-posterior-covariance-irreducible}
Assume \(d\ge3\) and \(B\succ0\).
The family \(\{\Sigma_B(y):y\in\mathbb R^N\}\) has no common
nontrivial invariant subspace, meaning that the only subspaces
\(S\subseteq\mathbb R^N\) satisfying \(\Sigma_B(y)S\subseteq S\) for every
\(y\in\mathbb R^N\) are \(\{0\}\) and \(\mathbb R^N\).
Hence \({\cal T}_B\) is irreducible
on the cone of positive-semidefinite matrices.
\end{lemma}

\begin{proof}
We prove the first assertion by contradiction.  A common nontrivial
invariant subspace for all \(\Sigma_B(y)\) would force the log-Laplace
transform introduced below to split additively, and hence the feature
image \(q_{3,d}(\Sph^{d-1})\) to split as an orthogonal Cartesian product.
The strictly positive sectional curvature of this image excludes two
positive-dimensional factors, while oddness of \(q_{3,d}\) and the fact
that its image spans \(\mathbb R^N\) exclude a singleton factor.
Finally, we prove the second assertion by showing that \({\cal T}_B\)
preserves no nonzero proper face of the positive-semidefinite cone.
Here a nonzero proper face is a face other than \(\{0\}\) or the entire
cone.  Any such invariant face would yield a common nontrivial invariant
subspace for all \(\Sigma_B(y)\), contradicting the first assertion.

Let \(\nu_B\) be the probability measure on the image
\(M_d:=q_{3,d}(\Sph^{d-1})\) whose density with respect to \(\mu_d\) is
proportional to \(\exp(-u^{\mathsf T}Bu/2)\), and put
\[
  \psi_B(\theta)
  :=\log\int_{M_d}e^{\langle\theta,u\rangle}\,d\nu_B(u).
\]
The posterior law \(\mu_{B,y}\) has density proportional to
\(e^{\langle B^{1/2}y,u\rangle}\) relative to \(\nu_B\).  Differentiating
\(\psi_B\) twice therefore gives
\begin{equation}
  \Sigma_B(y)=\nabla^2\psi_B(B^{1/2}y).
  \label{eq:cubic-posterior-Hessian}
\end{equation}
Because \(B\succ0\), the parameter \(B^{1/2}y\) ranges over all of
\(\mathbb R^N\).

Suppose that a subspace \(S\subseteq\mathbb R^N\) with \(0<\dim S<N\)
is invariant under every matrix in \eqref{eq:cubic-posterior-Hessian},
and let \(P\) be its orthogonal
projection, with \(P^\perp:=I_N-P\).  The matrices are symmetric, so \(S^\perp\) is invariant as
well and
\[
  P^\perp\nabla^2\psi_B(\theta)P=0,
  \qquad \theta\in\mathbb R^N.
\]
For \(\theta_S\in S\) and \(\theta_\perp\in S^\perp\), the mixed
second derivative in these two directions therefore satisfies
\[
 \frac{\partial^2}{\partial s\,\partial t}
   \psi_B(s\theta_S+t\theta_\perp)
 =\theta_\perp^{\mathsf T}\nabla^2\psi_B(s\theta_S+t\theta_\perp)\theta_S
 =0,\qquad s,t\in[0,1].
\]
Integrating over \([0,1]^2\) gives
\begin{equation}
  \psi_B(\theta_S+\theta_\perp)-\psi_B(0)
  =\psi_B(\theta_S)-\psi_B(0)
   +\psi_B(\theta_\perp)-\psi_B(0).
  \label{eq:cubic-mgf-factorization}
\end{equation}
Let \(U_B\sim\nu_B\).
By \eqref{eq:cubic-mgf-factorization} and uniqueness of Laplace transforms
for compactly supported laws, \(PU_B\) and \(P^\perp U_B\) are independent.
Since \(\nu_B\) has a
strictly positive density on \(M_d\), equality of supports in this
independence statement gives the orthogonal Cartesian decomposition
\begin{equation}
  M_d=PM_d\times P^\perp M_d.
  \label{eq:cubic-orbit-product}
\end{equation}

We recall why \eqref{eq:cubic-orbit-product} is impossible.
Equation~\eqref{eq:cubic-zonal-polynomial} gives
\[
 1-Z_{3,d}(t)
 =\frac{1-t}{d-1}
   \left\{(d+2)(t+1/2)^2+\frac{3(d-2)}4\right\}.
\]
For \(d\ge3\), this vanishes on \([-1,1]\) only at \(t=1\).
If \(q_{3,d}(x)=q_{3,d}(y)\), their unit norms and
\eqref{eq:cubic-addition-formula} give \(Z_{3,d}(x,y)=1\), so
\(\langle x,y\rangle=1\).  Thus \(q_{3,d}\) is injective.
The pullback of the Euclidean metric under \(q_{3,d}\) is given, for
tangent vectors \(v,w\in T_x\Sph^{d-1}\), by
\begin{equation}
  \bigl\langle Dq_{3,d}(x)v,Dq_{3,d}(x)w\bigr\rangle
  =\frac{3(d+1)}{d-1}\langle v,w\rangle.
  \label{eq:cubic-harmonic-pullback-metric}
\end{equation}
The addition formula \eqref{eq:cubic-addition-formula} implies that the
pullback is rotation invariant.
Rotations fixing \(x\) act transitively on the unit tangent vectors at
\(x\), so the pullback is a constant multiple of the standard spherical
metric.  Let \(\Delta_{\Sph}\) be the Laplace--Beltrami operator on
\(\Sph^{d-1}\), with sign convention \(\Delta_{\Sph}=\operatorname{div}\nabla\).
Each coordinate of \(q_{3,d}\) is a degree-three spherical harmonic, so
\(\Delta_{\Sph}q_{3,d}=-3(d+1)q_{3,d}\), where the Laplacian acts
coordinatewise \cite[Section~2.2.3]{BakryGentilLedoux2014}.
Since \(\|q_{3,d}(x)\|^2=1\), the product rule gives
\[
 0=\frac12\Delta_{\Sph}\|q_{3,d}(x)\|^2
  =\tr\{Dq_{3,d}(x)^{\mathsf T}Dq_{3,d}(x)\}-3(d+1).
\]
The trace is taken on the \((d-1)\)-dimensional tangent space
\(T_x\Sph^{d-1}\).  Dividing it by \(d-1\) determines the proportionality
constant and proves \eqref{eq:cubic-harmonic-pullback-metric}.
Taking \(w=v\) in \eqref{eq:cubic-harmonic-pullback-metric} gives
\[
 \|Dq_{3,d}(x)v\|^2=\frac{3(d+1)}{d-1}\|v\|^2>0
 \qquad\text{for }v\in T_x\Sph^{d-1}\setminus\{0\}.
\]
Thus \(Dq_{3,d}(x)\) is injective on \(T_x\Sph^{d-1}\).
The map is consequently a smooth embedding, since its domain is compact.
Thus \(M_d\), with its induced metric, is a
constant rescaling of the round sphere and has strictly positive sectional
curvature \((d-1)/\{3(d+1)\}\).

We next show that \eqref{eq:cubic-orbit-product} makes \(M_d\)
locally isometric to a Riemannian product.  Fix \(z\in M_d\) and
\(v\in T_zM_d\).  Choose a smooth curve \(c\) in \(M_d\) with
\(c(0)=z\) and \(c'(0)=v\).  The Cartesian decomposition implies
\[
 Pc(t)+P^\perp z\in M_d,\qquad
 Pz+P^\perp c(t)\in M_d.
\]
Differentiating at zero shows that \(Pv,P^\perp v\in T_zM_d\), so
\begin{equation*}
 T_zM_d=(T_zM_d\cap S)\oplus(T_zM_d\cap S^\perp).
\end{equation*}
Since \(P(T_zM_d)\subseteq T_zM_d\), the restriction \(P|_{T_zM_d}\)
inherits the identities \(P^2=P\) and \(P^{\mathsf T}=P\).
It is therefore the orthogonal projection onto \(T_zM_d\cap S\).
Thus its eigenvalues belong to \(\{0,1\}\), so its rank equals its trace.
In a local smooth orthonormal frame \(e_1(z),\ldots,e_{d-1}(z)\) for
the tangent spaces, this gives
\[
 \operatorname{rank}(P|_{T_zM_d})
 =\tr(P|_{T_zM_d})
 =\sum_{\ell=1}^{d-1}\langle e_\ell(z),Pe_\ell(z)\rangle.
\]
The right-hand side is continuous in \(z\), so the rank is a continuous
integer-valued function on \(M_d\).
Since \(M_d\) is the continuous image of the connected sphere
\(\Sph^{d-1}\), it is connected.  Hence this rank is constant.  Denote its value by
\(p\).  The complementary projection satisfies
\(P^\perp|_{T_zM_d}=I_{T_zM_d}-P|_{T_zM_d}\).
Its rank is therefore \(d-1-p\).
The maps \(P|_{M_d}\) and \(P^\perp|_{M_d}\) consequently have constant
ranks \(p\) and \(d-1-p\).  The constant-rank theorem
\cite[Theorem~4.12]{Lee2013Smooth} gives local smooth image patches
\(\mathcal U\subset PM_d\) and \(\mathcal V\subset P^\perp M_d\), of these respective
dimensions, through \(Pz\) and \(P^\perp z\).  By
\eqref{eq:cubic-orbit-product}, the map \((a,b)\mapsto a+b\) sends
\(\mathcal U\times\mathcal V\) into \(M_d\).  Its differential is injective because
\(S\perp S^\perp\), and the dimensions of its domain and target both
equal \(d-1\).  After shrinking \(\mathcal U\) and \(\mathcal V\), the inverse
function theorem therefore makes \((a,b)\mapsto a+b\) a diffeomorphism
from \(\mathcal U\times\mathcal V\) onto a neighborhood of \(z\) in \(M_d\).

For \(v_1,v_2\in T_a\mathcal U\) and \(w_1,w_2\in T_b\mathcal V\),
orthogonality of \(S\) and \(S^\perp\) gives
\[
 \langle v_1+w_1,v_2+w_2\rangle
 =\langle v_1,v_2\rangle+\langle w_1,w_2\rangle.
\]
Thus this diffeomorphism is an isometry when each factor carries its
induced Euclidean metric and \(\mathcal U\times\mathcal V\) carries
the product metric.  Hence \(M_d\) is locally isometric to a
Riemannian product.  If
\(0<p<d-1\), choose nonzero tangent vectors \(v\in T_{Pz}\mathcal U\) and
\(w\in T_{P^\perp z}\mathcal V\).  The curvature tensor of a product metric
is the sum of the curvature tensors of its factors
\cite[Section~4.7]{Petersen2016Riemannian}, so a plane with one direction
in each factor has
\[
 \operatorname{Sec}_{M_d}
 \bigl(\operatorname{span}\{(v,0),(0,w)\}\bigr)=0.
\]
This contradicts the positive curvature given by
\eqref{eq:cubic-harmonic-pullback-metric}.  Hence \(p=0\) or \(p=d-1\).
The corresponding projection has zero differential everywhere on the
connected manifold \(M_d\), and is therefore constant.  Oddness of
\(q_{3,d}\) makes that constant zero.  Since \(M_d\) spans
\(\mathbb R^N\), either \(S=\{0\}\) or \(S=\mathbb R^N\), contrary to
the choice of \(S\).  This proves the first assertion.

To deduce irreducibility of \({\cal T}_B\) from the first assertion,
recall from \cite[Section~3.7.1]{Pataki2000Geometry} that every nonzero
proper face of the positive-semidefinite cone has the form
\[
 {\cal F}_S=\{H\succeq0:\operatorname{range}(H)\subseteq S\}
\]
for a subspace \(S\subseteq\mathbb R^N\) with \(0<\dim S<N\).
If \({\cal T}_B\) preserved this
face, then for every \(x\in S\),
\[
 0=\tr\{P_{S^\perp}{\cal T}_B(xx^{\mathsf T})\}
  =\mathbb E\|P_{S^\perp}\Sigma_B(\mathsf Y_B)x\|^2.
\]
The law \(P_B\) has Lebesgue density
\((2\pi)^{-N/2}e^{-\|y\|^2/2}Z_B(y)>0\), and \(y\mapsto \Sigma_B(y)\)
is continuous.  Hence \(\Sigma_B(y)S\subseteq S\) for every \(y\).  This contradicts
the first assertion.  Hence \({\cal T}_B\) preserves no nontrivial face
and is irreducible on the positive-semidefinite cone.
\end{proof}

The irreducibility established in
Lemma~\ref{lem:cubic-posterior-covariance-irreducible} allows us to determine
the eigenspace of \({\cal T}_B\) corresponding to its largest eigenvalue.
Lemma~\ref{lem:cubic-perron-rigidity} combines this spectral characterization
with the local-maximizer conditions \eqref{eq:cubic-local-maximizer-conditions}
and rotational invariance to show that every positive-definite local
maximizer of \({\cal V}_{d,r}\) is a scalar matrix.

\begin{lemma}[Scalar form of positive-definite local maximizers]
\label{lem:cubic-perron-rigidity}
For \(d\ge3\) and \(r>0\), every positive-definite local maximum of
\({\cal V}_{d,r}\) is of the form \(zI_N/N\).
\end{lemma}

\begin{proof}
For each \(B\succ0\), we first prove that there exists \(P\succ0\) such that
\[
 \{H=H^{\mathsf T}:{\cal T}_B(H)=\lambda_{\max}({\cal T}_B)H\}
 =\{cP:c\in\mathbb R\}.
\]
This is the Perron--Frobenius property \cite{EvansHoeghKrohn1978} needed
for the self-adjoint operator \({\cal T}_B\) on the finite-dimensional
space of real symmetric matrices.  We prove it directly below.
We then combine it with \(D M(B)[H]={\cal T}_B(H)\) from
Lemma~\ref{lem:cubic-gaussian-entropy-derivatives} and rotational invariance
to show that a positive-definite local maximum must be a scalar multiple
of \(I_N\).

For a real symmetric matrix \(H\), choose an orthonormal eigenbasis
\(e_1,\ldots,e_N\) with \(He_i=h_i e_i\), and let \(|H|:=(H^2)^{1/2}\).
Then
\[
 H=\sum_{i=1}^N h_i e_i e_i^{\mathsf T},\qquad
 |H|=\sum_{i=1}^N |h_i|e_i e_i^{\mathsf T}.
\]
Consequently, for every \(y\in\mathbb R^N\),
\[
  \begin{aligned}
   \tr\{H\Sigma_B(y)H\Sigma_B(y)\}
   &=\sum_{i,j}h_ih_j\bigl(e_i^{\mathsf T}\Sigma_B(y)e_j\bigr)^2\\
   &\le\sum_{i,j}|h_i h_j|\bigl(e_i^{\mathsf T}\Sigma_B(y)e_j\bigr)^2
    =\tr\{|H|\Sigma_B(y)|H|\Sigma_B(y)\}.
  \end{aligned}
\]
Integrating against \(P_B(dy)\) gives
\(\langle H,{\cal T}_B(H)\rangle_{\HS}
\le\langle |H|,{\cal T}_B(|H|)\rangle_{\HS}\).
Self-adjointness gives the Rayleigh variational formula
\[
 \lambda_{\max}({\cal T}_B)
 =\max_{\substack{H=H^{\mathsf T}\\\|H\|_{\HS}=1}}
   \langle H,{\cal T}_B(H)\rangle_{\HS}.
\]
Since \(\||H|\|_{\HS}=\|H\|_{\HS}\), the preceding inequality shows that
the absolute value of a maximizing matrix is also a maximizer.
There is therefore a matrix \(P\succeq0\) satisfying
\[
 \|P\|_{\HS}=1,\qquad
 {\cal T}_B(P)=\lambda_{\max}({\cal T}_B)P.
\]
For \(x\in\ker P\), we consequently have
\[
 0=\lambda_{\max}({\cal T}_B)x^{\mathsf T}Px
  =x^{\mathsf T}{\cal T}_B(P)x
  =\mathbb E\|P^{1/2}\Sigma_B(\mathsf Y_B)x\|^2.
\]
Continuity of \(\Sigma_B(y)\) and the full support of \(P_B\) imply that \(\ker P\)
is invariant under every \(\Sigma_B(y)\).
Lemma~\ref{lem:cubic-posterior-covariance-irreducible} therefore forces
\(P\succ0\).  The same argument applies to every nonzero positive-semidefinite
matrix in the eigenspace corresponding to \(\lambda_{\max}({\cal T}_B)\).
If a symmetric matrix \(H\) satisfying
\({\cal T}_B(H)=\lambda_{\max}({\cal T}_B)H\) were not proportional to \(P\), then
\[
 \lambda_{\max}(P^{-1/2}HP^{-1/2})P-H
\]
would be nonzero, singular, and positive semidefinite, and would belong
to the same eigenspace.  This contradicts the positive definiteness just
proved.  The largest eigenvalue is therefore simple.

Let \(Q\succ0\) be a local maximum, so that
\eqref{eq:cubic-local-maximizer-conditions} holds.
To prove that \(Q\) is a scalar multiple of \(I_N\), we show that it is
unchanged by rotations of the latent sphere.
For \(R\in\operatorname{SO}(d)\), composition with \(R\) preserves the
degree and harmonicity of a homogeneous harmonic polynomial \(p\), since
\(\Delta(p\circ R)=(\Delta p)\circ R=0\).
Thus the rotated functions \(\psi_{3,i}\circ R\) belong to
\(\cH_{3,d}\).  Expanding them in the orthonormal basis from
Section~\ref{subsec:cubic-harmonic-kernel} defines a matrix \(U_R\), and
rotation invariance of \(\sigma_d\) gives
\[
 \begin{aligned}
 \psi_{3,i}(Rx)&=\sum_{j=1}^N(U_R)_{ij}\psi_{3,j}(x),\\
 (U_RU_R^{\mathsf T})_{ij}
 &=\int\psi_{3,i}(Rx)\psi_{3,j}(Rx)\,d\sigma_d(x)=\delta_{ij}.
 \end{aligned}
\]
Consequently, \(U_R\) is orthogonal and
\(q_{3,d}(Rx)=U_Rq_{3,d}(x)\).
The coefficient formula
\[
 (U_R)_{ij}=\int\psi_{3,i}(Rx)\psi_{3,j}(x)\,d\sigma_d(x)
\]
shows that \(U_R\) depends smoothly on \(R\), and composition gives
\(U_{R_1R_2}=U_{R_1}U_{R_2}\).
In particular, if \(u=q_{3,d}(X)\) with \(X\sim\sigma_d\), then
\(U_Ru=q_{3,d}(RX)\) also has law \(\mu_d\).

Recall the integral formula for the posterior mean,
\[
 m_B(y)=\frac1{Z_B(y)}\int v
 \exp\left(\langle y,B^{1/2}v\rangle-\frac12v^{\mathsf T}Bv\right)
 \,\mu_d(dv).
\]
Since \((U_RBU_R^{\mathsf T})^{1/2}=U_RB^{1/2}U_R^{\mathsf T}\),
the change of variables \(v=U_Rw\) in the normalizing integral gives
\(Z_{U_RBU_R^{\mathsf T}}(U_Ry)=Z_B(y)\).  Applying the same substitution
to the numerator yields
\[
 m_{U_RBU_R^{\mathsf T}}(U_Ry)
 =\frac1{Z_B(y)}\int U_Rw
 e^{\langle y,B^{1/2}w\rangle-\frac12w^{\mathsf T}Bw}\,\mu_d(dw)
 =U_Rm_B(y).
\]
Moreover, \(U_R\mathsf Y_B=(U_RBU_R^{\mathsf T})^{1/2}U_Ru+U_Rg\)
has law \(P_{U_RBU_R^{\mathsf T}}\), since the independent variables
\(U_Ru\) and \(U_Rg\) have laws \(\mu_d\) and \(\gamma_N\), respectively.
Using \(M(B)=\mathbb E\{m_B(\mathsf Y_B)m_B(\mathsf Y_B)^{\mathsf T}\}\),
we obtain
\begin{equation}
 M(U_RBU_R^{\mathsf T})
 =U_RM(B)U_R^{\mathsf T}.
 \label{eq:cubic-posterior-equivariance}
\end{equation}
To differentiate the stationary equation along rotations of \(Q\), let
\(R(t)\) be a smooth
curve in \(\operatorname{SO}(d)\) with \(R(0)=I_d\), and set
\(Q(t):=U_{R(t)}QU_{R(t)}^{\mathsf T}\).
Since \(M(rNdQ)=Q\) by \eqref{eq:cubic-local-maximizer-conditions},
\eqref{eq:cubic-posterior-equivariance} gives
\[
 M(rNdQ(t))=U_{R(t)}M(rNdQ)U_{R(t)}^{\mathsf T}=Q(t).
\]
Write \(L:=\left.\frac d{dt}U_{R(t)}\right|_{t=0}\), so that
\(L^{\mathsf T}=-L\) and \(H:=Q'(0)=[L,Q]:=LQ-QL\) is symmetric.
Differentiating the displayed equation and applying
\(DM(B)[H]={\cal T}_B(H)\) from
Lemma~\ref{lem:cubic-gaussian-entropy-derivatives} gives
\[
 rNd\,{\cal T}_{rNdQ}(H)=H,\qquad \tr H=\tr(LQ-QL)=0.
\]
If \(H\ne0\), then \(1/(rNd)\) is an eigenvalue, and
\eqref{eq:cubic-local-maximizer-conditions} makes it the largest one.
The spectral conclusion proved above then gives \(H=cP\) for some
\(P\succ0\).  Taking traces yields \(0=c\,\tr P\), hence \(c=0\),
a contradiction.  Therefore \([L,Q]=0\) for every such \(L\).

To pass from these derivatives to an arbitrary rotation, fix
\(R\in\operatorname{SO}(d)\).  This group is connected
\cite[Chapter~2, Exercise~11]{Hall2000}.
Its Lie algebra consists of real skew-symmetric matrices, so
\cite[Chapter~3, Section~5 and Corollary~3.26]{Hall2000} gives
\(R=e^{A_1}\cdots e^{A_k}\) with \(A_j^{\mathsf T}=-A_j\).
Thus \(R(t):=e^{tA_1}\cdots e^{tA_k}\), \(0\le t\le1\), is a smooth
path in \(\operatorname{SO}(d)\) with \(R(0)=I_d\) and \(R(1)=R\).
Write \(U(t)=U_{R(t)}\) and \(L(t)=U(t)^{\mathsf T}U'(t)\).
For each \(t\in(0,1)\), the curve \(s\mapsto R(t)^{-1}R(t+s)\)
starts at \(I_d\), and its representing matrix is
\(U(t)^{\mathsf T}U(t+s)\).  The preceding argument therefore gives
\([L(t),Q]=0\), so
\[
 \frac d{dt}\{U(t)QU(t)^{\mathsf T}\}
 =U(t)[L(t),Q]U(t)^{\mathsf T}=0.
\]
Evaluating at \(t=0\) and \(t=1\) yields
\(U_RQU_R^{\mathsf T}=Q\), hence \(QU_R=U_RQ\).

Under the identification \(a\mapsto\langle a,q_{3,d}(\cdot)\rangle\)
of \(\mathbb R^N\) with \(\cH_{3,d}\), the matrix \(U_R\) represents
\(f\mapsto f\circ R^{-1}\).  This representation is irreducible for
\(d\ge3\) \cite[Section~1.7]{DaiXu2013}, so its only invariant subspaces
are \(\{0\}\) and \(\mathbb R^N\).
If \(Qv=\lambda v\), the commutation relation gives
\(Q(U_Rv)=U_RQv=\lambda U_Rv\).  Thus each eigenspace of \(Q\) is
invariant under every \(U_R\).  Since \(Q\) is symmetric, irreducibility
forces it to have only one eigenvalue.  Therefore \(Q=zI_N/N\), where
\(z=\tr Q>0\).
\end{proof}

Proposition~\ref{prop:cubic-matrix-branch-gap} extends the scalar bound of
Lemma~\ref{lem:cubic-scalar-branch-coercivity} to all positive-semidefinite
matrices.  Its proof uses the spectral and rotation arguments established
in the proof of Lemma~\ref{lem:cubic-perron-rigidity}.

\begin{proposition}[Uniform quadratic bound for the matrix potential]
\label{prop:cubic-matrix-branch-gap}
For every compact \(K\subset(0,r_{\mathrm{tens}})\), there are
\(\kappa_K>0\) and \(d_K<\infty\) such that, uniformly for
\(r\in K\), \(d\ge d_K\), and \(Q\succeq0\),
\begin{equation}
 {\cal V}_{d,r}(Q)
 \le-\kappa_KN\|Q\|_{\HS}^2.
 \label{eq:cubic-uniform-quadratic-branch-coercivity}
\end{equation}
Consequently, for every fixed \(\delta>0\),
\begin{equation}
  \limsup_{d\to\infty}
  \sup_{\substack{Q\succeq0\\N\|Q\|_{\HS}^2\ge\delta}}
  {\cal V}_{d,r}(Q)
  \le-\kappa_K\delta
  \label{eq:cubic-uniform-branch-gap}
\end{equation}
uniformly for \(r\in K\).
\end{proposition}

\begin{proof}
We first show that the perturbed potential below attains its maximum and
that every maximizer is positive definite.  The spectral and rotation
arguments from the proof of Lemma~\ref{lem:cubic-perron-rigidity} then
reduce the optimization to scalar matrices.  We bound the resulting
scalar maximum and let the trace perturbation tend to zero to obtain
\eqref{eq:cubic-uniform-quadratic-branch-coercivity}.
Choose \(\kappa_K,d_K\) as in
Lemma~\ref{lem:cubic-scalar-branch-coercivity}, increasing \(d_K\) to at
least three and decreasing \(\kappa_K\) so that
\(\kappa_K<(\min K)/8\).
Fix \(d\ge d_K\) and \(r\in K\).
For \(\varepsilon>0\), set
\[
  {\cal W}_{r,\varepsilon}(Q)
  :={\cal V}_{d,r}(Q)
    +\kappa_K N\|Q\|_{\HS}^2+\varepsilon\tr Q.
\]
To prove that \({\cal W}_{r,\varepsilon}\) attains its maximum, we first
bound its relative-entropy term.  Applying the relative-entropy identity
from Step~2 of the proof of
Lemma~\ref{lem:cubic-gaussian-entropy-derivatives} with \(B=rNdQ\) gives
\[
 \begin{aligned}
  {\cal W}_{r,\varepsilon}(Q)
  &=\frac1dD(P_{rNdQ}\Vert\gamma_N)
    -\left(\frac r4-\kappa_K\right)N\|Q\|_{\HS}^2+\varepsilon\tr Q\\
  &=\left(\frac r2+\varepsilon\right)\tr Q
    -\left(\frac r4-\kappa_K\right)N\|Q\|_{\HS}^2
    -\frac1d I\bigl(u;(rNdQ)^{1/2}u+g\bigr)\\
  &\le\left(\frac r2+\varepsilon\right)\tr Q
    -\left(\frac r4-\kappa_K\right)N\|Q\|_{\HS}^2.
 \end{aligned}
\]
Here \(\tr(rNdQ)/(2Nd)=r\tr Q/2\), and the inequality uses
nonnegativity of mutual information.
Since \((\tr Q)^2\le N\|Q\|_{\HS}^2\) and \(r/4-\kappa_K>0\),
\({\cal W}_{r,\varepsilon}(Q)\to-\infty\) as \(\|Q\|_{\HS}\to\infty\).
Hence the set
\(\{Q\succeq0:{\cal W}_{r,\varepsilon}(Q)\ge0\}\) is bounded and is
nonempty because \({\cal W}_{r,\varepsilon}(0)=0\).
The continuity proved in Step~4 of the proof of
Lemma~\ref{lem:cubic-gaussian-entropy-derivatives} makes this set closed,
and thus compact.  The continuous function \({\cal W}_{r,\varepsilon}\)
attains its maximum on this set.  Its values outside the set are negative,
so this is also a global maximum over \(Q\succeq0\).

Let \(Q\) be any global maximizer.  If \(Q\) were singular, choose a
unit vector \(v\) with \(Qv=0\).  The matrices
\(Q+tvv^{\mathsf T}\) remain positive semidefinite for \(t\ge0\).
The one-sided derivative in
Lemma~\ref{lem:cubic-gaussian-entropy-derivatives}, together with
\(v^{\mathsf T}Qv=0\) and \(M(rNdQ)\succeq0\), gives
\[
 \left.\frac d{dt}{\cal W}_{r,\varepsilon}(Q+tvv^{\mathsf T})\right|_{t=0+}
 =\frac{rN}{2}v^{\mathsf T}M(rNdQ)v+\varepsilon
 \ge\varepsilon>0.
\]
Thus \({\cal W}_{r,\varepsilon}(Q+tvv^{\mathsf T})>
{\cal W}_{r,\varepsilon}(Q)\) for all sufficiently small \(t>0\),
contradicting maximality.  Therefore \(Q\succ0\), so the maximum is
interior to the positive-semidefinite cone and its gradient vanishes.
Using \eqref{eq:cubic-branch-gradient}, we obtain
\[
 0=\nabla_Q{\cal W}_{r,\varepsilon}(Q)
  =\frac{rN}{2}\{M(rNdQ)-Q\}+2\kappa_K NQ+\varepsilon I_N.
\]
Rearranging gives
\begin{equation}
  M(rNdQ)
  =\left(1-\frac{4\kappa_K}{r}\right)Q
    -\frac{2\varepsilon}{rN}I_N.
  \label{eq:cubic-perturbed-stationarity}
\end{equation}
Let \(R(t)\) be a smooth curve in \(\operatorname{SO}(d)\) with
\(R(0)=I_d\), and use the matrices \(U_R\) constructed in the proof of
Lemma~\ref{lem:cubic-perron-rigidity}.
Write \(U(t)=U_{R(t)}\), \(Q(t)=U(t)QU(t)^{\mathsf T}\), and
\(L=U'(0)\).  Since \(U(0)=I_N\) and \(L^{\mathsf T}=-L\),
\(H:=Q'(0)=LQ-QL\) is symmetric.
Equation~\eqref{eq:cubic-posterior-equivariance} gives
\[
 M(rNdQ(t))=U(t)M(rNdQ)U(t)^{\mathsf T}.
\]
Differentiating at \(t=0\), using
\(DM(B)[H]={\cal T}_B(H)\) from
Lemma~\ref{lem:cubic-gaussian-entropy-derivatives}, and substituting
\eqref{eq:cubic-perturbed-stationarity}, we obtain
\[
 \begin{aligned}
 rNd\,{\cal T}_{rNdQ}(H)
 &=LM(rNdQ)-M(rNdQ)L\\
 &=\left(1-\frac{4\kappa_K}{r}\right)[L,Q]
   -\frac{2\varepsilon}{rN}[L,I_N]
  =\left(1-\frac{4\kappa_K}{r}\right)H.
 \end{aligned}
\]
If \(H\ne0\), this calculation identifies
\((1-4\kappa_K/r)/(rNd)\) as an eigenvalue of \({\cal T}_{rNdQ}\).
It follows that
\(\lambda_{\max}({\cal T}_{rNdQ})\ge(1-4\kappa_K/r)/(rNd)\).
On the other hand, the Hessian formula
\eqref{eq:cubic-branch-hessian} and the nonpositivity of the Hessian of
\({\cal W}_{r,\varepsilon}\) at its maximizer imply
\(\lambda_{\max}({\cal T}_{rNdQ})\le(1-4\kappa_K/r)/(rNd)\).
Combining these two bounds gives
\[
 \lambda_{\max}({\cal T}_{rNdQ})
 =\frac{1-4\kappa_K/r}{rNd}.
\]
Hence \(H\) lies in the eigenspace for the largest eigenvalue.
The first part of the proof of Lemma~\ref{lem:cubic-perron-rigidity}
shows that this space is \(\{cP:c\in\mathbb R\}\) for some \(P\succ0\).
Writing \(H=cP\) and taking traces gives
\[
 0=\tr(LQ-QL)=\tr H=c\,\tr P.
\]
Since \(\tr P>0\), we have \(c=0\), contradicting \(H\ne0\).
Therefore \([L,Q]=0\) for every
rotation derivative \(L\).  The path argument and the irreducibility of
the harmonic representation used in the proof of
Lemma~\ref{lem:cubic-perron-rigidity} then show that every maximizer is a
scalar multiple of \(I_N\).  Hence
\begin{equation*}
  \sup_{Q\succeq0}
  \{{\cal V}_{d,r}(Q)+\kappa_K N\|Q\|_{\HS}^2
       +\varepsilon\tr Q\}
  =
  \sup_{z\ge0}\{v_{d,r}(z)+\kappa_K z^2+\varepsilon z\}.
\end{equation*}
Lemma~\ref{lem:cubic-scalar-branch-coercivity} bounds the right side by
\[
 \sup_{z\ge0}\{-3\kappa_K z^2+\varepsilon z\}
 \le\frac{\varepsilon^2}{12\kappa_K}.
\]
For an arbitrary \(Q\succeq0\), drop the nonnegative term
\(\varepsilon\tr Q\) and send \(\varepsilon\downarrow0\).  This proves
\eqref{eq:cubic-uniform-quadratic-branch-coercivity}.
Restricting to \(N\|Q\|_{\HS}^2\ge\delta\) gives
\eqref{eq:cubic-uniform-branch-gap}.
\end{proof}

The matrix bound now completes the variational estimate required by
the appendix roadmap.
Proposition~\ref{prop:cubic-strict-Franz-Parisi} combines it with
Proposition~\ref{prop:cubic-Gaussian-variational-bound} to give a
strictly negative conditional mean on \(s_n\ge\delta\).

\begin{proposition}[Negative bound for the expected constrained free energy]
\label{prop:cubic-strict-Franz-Parisi}
Fix \(A<\infty\), a compact interval \(K\subset(0,r_{\mathrm{tens}})\),
and \(M_\star>4A+10\).
Suppose \(d=d_n\to\infty\), \(n\le d^A\), \(0<|\beta|\le1\),
and \(r_n=n\beta^2/(Nd)\in K\).
Use the prior \(\pi_n^\dagger\) from
\eqref{eq:cubic-common-conditioned-prior} and its feature law
\(\mu_{d,n}^\dagger\) from Section~\ref{subsec:cubic-franz-parisi-gap}.
For every \(\delta>0\), there is \(g=g(K,\delta)>0\) such
that, for every fixed \(M>0\), uniformly for \(r\in K\),
\begin{equation}
 \mu_{d,n}^\dagger\left\{
 u^0:\frac1{nd}\mathbb E_g\log Z^{\rm G}_{u^0,\delta}
 >-2g(K,\delta)\right\}=o(d^{-M}).
 \label{eq:cubic-quenched-Gaussian-gap}
\end{equation}
Moreover,
\begin{equation}
  \limsup_{n\to\infty}\sup_{r\in K}{\cal F}_{n,d}(r,\delta)
  \le-2g(K,\delta).
  \label{eq:cubic-strict-Franz-Parisi}
\end{equation}
In particular, these bounds apply at the graph parameter \(r=r_n\).
\end{proposition}

\begin{proof}
Proposition~\ref{prop:cubic-matrix-branch-gap} gives, for all sufficiently
large \(d\), uniformly for \(r\in K\),
\[
 \sup_{\substack{Q\succeq0\\N\|Q\|_{\HS}^2\ge\delta/2}}
       {\cal V}_{d,r}(Q)
 \le-\frac{\kappa_K\delta}{2}.
\]
Apply \eqref{eq:cubic-Gaussian-variational-probability} in
Proposition~\ref{prop:cubic-Gaussian-variational-bound} with
\(\varepsilon=\kappa_K\delta/4\).  The event inclusion
\[
 \left\{u^0:\frac1{nd}\mathbb E_g\log Z^{\rm G}_{u^0,\delta}
              >-\frac{\kappa_K\delta}{4}\right\}
 \subseteq
 \left\{u^0:\frac1{nd}\mathbb E_g\log Z^{\rm G}_{u^0,\delta}
    >\sup_{\substack{Q\succeq0\\N\|Q\|_{\HS}^2\ge\delta/2}}
          {\cal V}_{d,r}(Q)+\frac{\kappa_K\delta}{4}\right\}
\]
therefore proves \eqref{eq:cubic-quenched-Gaussian-gap} with
\(g(K,\delta)=\kappa_K\delta/8\).
For the averaged conclusion,
\eqref{eq:cubic-Gaussian-variational-average} gives
\[
 \limsup_{n\to\infty}\sup_{r\in K}{\cal F}_{n,d}(r,\delta)
 \le-\frac{\kappa_K\delta}{2}\le-2g(K,\delta),
\]
which proves \eqref{eq:cubic-strict-Franz-Parisi}.
\end{proof}

\subsection{Observation-dependent truncation}
\label{subsec:cubic-bernoulli-truncation}

Proposition~\ref{prop:cubic-strict-Franz-Parisi} gives a negative
Gaussian conditional mean, and
\eqref{eq:cubic-Bernoulli-Gaussian-reduction} compares it with the
Bernoulli conditional mean.  We now control the region \(s_n\ge\delta\)
in the second moment by retaining \(L_x(\Adj)\) only when
\(Z_{x,\delta}^\dagger(\Adj)\le e^{-gnd}\).
Lemma~\ref{lem:cubic-legal-truncation} bounds the probability of the
discarded event under the joint law of \(X\sim\pi_n^\dagger\) and \(\Adj\mid X\).
Equations~\eqref{eq:cubic-truncation-L1} and
\eqref{eq:cubic-truncation-pair} then give an \(o(1)\) loss of
likelihood mass and an exponentially small second-moment integral
over \(s_n\ge\delta\).

\begin{lemma}[Observation-dependent macroscopic truncation]
\label{lem:cubic-legal-truncation}
Under the assumptions of Proposition~\ref{prop:cubic-strict-Franz-Parisi},
use the same prior \(\pi_n^\dagger\) and fix \(\delta>0\).
Then there is \(g=g(K,\delta)>0\) such that
\begin{equation}
  \mathbb P_{X\sim\pi_n^\dagger,\ \Adj\mid X}
  \{\log Z_{X,\delta}^\dagger(\Adj)>-gnd\}=o(d^{-M})
  \label{eq:cubic-component-Franz-Parisi}
\end{equation}
for every fixed \(M>0\).
\end{lemma}

\begin{proof}
Step~1 combines the Gaussian estimate in
Proposition~\ref{prop:cubic-strict-Franz-Parisi} with the comparison
\eqref{eq:cubic-Bernoulli-Gaussian-reduction} to prove
\(\mathbb E_{\Adj\mid x}\log Z_{x,\delta}^\dagger(\Adj)\le-g_0nd\) for some
\(g_0>0\), outside a set of \(\pi_n^\dagger\)-probability \(o(d^{-M})\).
Step~2 bounds the change in \(\log Z_{x,\delta}^\dagger\) when one edge is
changed, as in \eqref{eq:cubic-edge-log-likelihood-oscillation}.
The bounded differences inequality and integration over \(x\sim\pi_n^\dagger\)
then give \eqref{eq:cubic-component-Franz-Parisi} with \(g=g_0/2\).

We use the convention \(\log0=-\infty\).
If \(\pi_n^\dagger\{x':s_n(x,x')\ge\delta\}=0\), then
\(Z_{x,\delta}^\dagger\equiv0\), so this \(x\) contributes zero to the
probability in \eqref{eq:cubic-component-Franz-Parisi}.
It therefore suffices to prove the conditional estimates below when
this set has positive \(\pi_n^\dagger\)-measure.

\medskip
\noindent\textit{Step 1. Bound the conditional mean of \(\log Z_{x,\delta}^\dagger\).}
Let \(g_0=g(K,\delta)>0\) be the constant in
Proposition~\ref{prop:cubic-strict-Franz-Parisi}.
By \eqref{eq:cubic-quenched-Gaussian-gap}, applied at \(r=r_n\),
the set
\[
 \mathcal T_n
 :=\left\{x\in\mathcal G_n^\dagger:
       \mathbb E_g\log Z_{x,\delta}^{\mathrm G,\dagger}\le-2g_0nd\right\}
 \quad\text{satisfies}\quad
 \pi_n^\dagger(\mathcal T_n^c)=o(d^{-M})
\]
for every fixed \(M>0\).
The Gaussian integral here is the one identified with
\(Z^{\rm G}_{u^0,\delta}\), for \(u_i^0=q_{3,d}(x_i)\), before
Proposition~\ref{prop:cubic-Gaussian-variational-bound}.
For \(x\in\mathcal T_n\) with positive constrained prior mass, the comparison
\eqref{eq:cubic-Bernoulli-Gaussian-reduction} now gives
\begin{equation}
\begin{split}
  \mathbb E_{\Adj\mid x}\log Z_{x,\delta}^\dagger(\Adj)
  &\le\mathbb E_g\log Z_{x,\delta}^{\mathrm G,\dagger}
       +Cn^2(\tau_{n,d}^3+\tau_{n,d}^4)\\
  &\le-2g_0nd+o(nd)\le-g_0nd
\end{split}
  \label{eq:cubic-quenched-Bernoulli-gap}
\end{equation}
for all sufficiently large \(n\), uniformly in \(x\).

\medskip
\noindent\textit{Step 2. Control deviations from the conditional mean.}
Fix \(x\in\mathcal T_n\) with
\(\pi_n^\dagger\{x':s_n(x,x')\ge\delta\}>0\), and let \(\Adj^{(e)}\) differ
from \(\Adj\) only in edge \(e=\{i,j\}\).
For every \(x'\in\mathcal G_n^\dagger\),
Proposition~\ref{prop:cubic-raw-local-partition} gives
\(|b_e(x')|=|\beta Z_{3,d}(x_i',x_j')|\le\tau_{n,d}=o(1)\).
Only the factor \(1+Y_e b_e(x')\) changes in \(L_{x'}\), so
\[
 \frac{1-\tau_{n,d}}{1+\tau_{n,d}}
 \le\frac{L_{x'}(\Adj^{(e)})}{L_{x'}(\Adj)}
 \le\frac{1+\tau_{n,d}}{1-\tau_{n,d}}.
\]
For all sufficiently large \(n\), \(\tau_{n,d}<1\), and for every
graph \(\Adj\),
\[
 Z_{x,\delta}^\dagger(\Adj)
 \ge(1-\tau_{n,d})^{\binom n2}
       \pi_n^\dagger\{x':s_n(x,x')\ge\delta\}>0.
\]
Multiplying the bounds for \(L_{x'}(\Adj^{(e)})/L_{x'}(\Adj)\)
by \(L_{x'}(\Adj)\) and integrating with respect to \(\pi_n^\dagger\)
over \(\{x':s_n(x,x')\ge\delta\}\) gives the same bounds for
\(Z_{x,\delta}^\dagger(\Adj^{(e)})/Z_{x,\delta}^\dagger(\Adj)\).
Taking logarithms of the ratio bounds therefore yields
\begin{equation}
 \left|\log Z_{x,\delta}^\dagger(\Adj^{(e)})-\log Z_{x,\delta}^\dagger(\Adj)\right|
 \le\log\frac{1+\tau_{n,d}}{1-\tau_{n,d}}
 \le C\tau_{n,d}.
 \label{eq:cubic-edge-log-likelihood-oscillation}
\end{equation}
The \(\binom n2\) edges are independent under \(\mathbb P_{\Adj\mid x}\).
By \eqref{eq:cubic-quenched-Bernoulli-gap}, on the event
\(\{\log Z_{x,\delta}^\dagger(\Adj)>-g_0nd/2\}\),
\[
 \log Z_{x,\delta}^\dagger(\Adj)
 -\mathbb E_{\Adj\mid x}\log Z_{x,\delta}^\dagger(\Adj)
 \ge\log Z_{x,\delta}^\dagger(\Adj)+g_0nd
 >\frac{g_0nd}{2}.
\]
Applying the bounded differences inequality
\cite[Theorem~6.2]{BoucheronLugosiMassart2013} with the bounds in
\eqref{eq:cubic-edge-log-likelihood-oscillation} therefore gives
\[
  \mathbb P_{\Adj\mid x}
  \{\log Z_{x,\delta}^\dagger(\Adj)>-g_0nd/2\}
  \le\exp\left(-c\frac{g_0^2d^2}{\tau_{n,d}^2}\right).
\]
Recall \(\tau_{n,d}=C_{A,K,M_\star}|\beta|\log^3d/\sqrt N\).
Since \(|\beta|\le1\) and \(N\asymp d^3\),
\[
 \frac{d^2}{\tau_{n,d}^2}
 \ge c\frac{d^2N}{\log^6d}\gg\log d.
\]
Thus \(\exp(-cg_0^2d^2/\tau_{n,d}^2)=o(d^{-M})\) for every fixed \(M\).
The same probability bound holds trivially when the constrained prior
mass is zero.
Taking \(g=g_0/2\) and integrating with respect to \(x\sim\pi_n^\dagger\), we obtain
\begin{align*}
 \mathbb P_{X\sim\pi_n^\dagger,\ \Adj\mid X}
   \{\log Z_{X,\delta}^\dagger(\Adj)>-gnd\}
 &=\int\mathbb P_{\Adj\mid x}
       \{\log Z_{x,\delta}^\dagger(\Adj)>-gnd\}\,d\pi_n^\dagger(x)\\
 &\le\pi_n^\dagger(\mathcal T_n^c)
       +\exp\left(-c\frac{g_0^2d^2}{\tau_{n,d}^2}\right)
   =o(d^{-M}).
\end{align*}
This proves \eqref{eq:cubic-component-Franz-Parisi}.
\end{proof}

For fixed \(\delta\), define
\[
  {\cal A}_{x,\delta}^\dagger
  :=\{\Adj:Z_{x,\delta}^\dagger(\Adj)\le e^{-gnd}\},\qquad
  \overline L_n(\Adj)
  :=\int L_x(\Adj)\1_{{\cal A}_{x,\delta}^\dagger}(\Adj)\,d\pi_n^\dagger(x).
\]
Since \(L_x=d\mathbb P(\,\cdot\mid x)/d\Qq_n\), Fubini's theorem and
Lemma~\ref{lem:cubic-legal-truncation} give
\begin{equation}
  \mathbb E_{\Qq_n}(L_n^\dagger-\overline L_n)
  =
  \int\mathbb P((\mathcal A_{x,\delta}^\dagger)^c\mid x)\,d\pi_n^\dagger(x)
  =o(1).
  \label{eq:cubic-truncation-L1}
\end{equation}
For the second moment, drop \(\1_{\mathcal A_{x',\delta}^\dagger}\) and use
\(\int_{\{s_n(x,x')\ge\delta\}}L_{x'}\,d\pi_n^\dagger(x')=Z_{x,\delta}^\dagger\).
Fubini's theorem then yields
\begin{align}
  \iint_{\{s_n(x,x')\ge\delta\}}
  \mathbb E_{\Qq_n}\!
  \left[
    L_xL_{x'}\1_{\mathcal A_{x,\delta}^\dagger}
              \1_{\mathcal A_{x',\delta}^\dagger}
  \right]d\pi_n^\dagger(x)d\pi_n^\dagger(x')
  &\le\mathbb E_{\Qq_n}\int
       L_x\1_{\mathcal A_{x,\delta}^\dagger}Z_{x,\delta}^\dagger\,d\pi_n^\dagger(x)
       \notag\\
  &\le e^{-gnd}\int\mathbb E_{\Qq_n}L_x\,d\pi_n^\dagger(x)
   =e^{-gnd}.
  \label{eq:cubic-truncation-pair}
\end{align}
By \eqref{eq:cubic-truncation-L1}, normalizing the truncated likelihood
changes the graph law with latent prior \(\pi_n^\dagger\) by \(o(1)\) in
total variation.
By \eqref{eq:cubic-truncation-pair}, the integral over
\(s_n(x,x')\ge\delta\) in the truncated second moment is at most
\(e^{-gnd}\), and hence tends to zero.

\subsection{Gaussian likelihood asymptotics below the threshold}
\label{subsec:cubic-supp-likelihood}

The estimates above justify the truncation constructed in
Section~\ref{subsec:cubic-variational-mechanism}.
Proposition~\ref{prop:cubic-full-critical-LAN} records the resulting
log-likelihood expansion.  Its proof also establishes the
\(L^2\) approximation by conditioning on \(\widehat T_n\)
in \eqref{eq:cubic-likelihood-projection}.

\begin{proposition}[Gaussian log-likelihood expansion below the tensor threshold]
\label{prop:cubic-full-critical-LAN}
Fix \(A<\infty\).  Suppose
\[
  d=d_n\to\infty,\qquad n\le d^A,\qquad
  |\beta|\le1,\qquad
  r_n=\frac{n\beta^2}{N_{3,d}d}\longrightarrow
  r\in(0,r_{\mathrm{tens}}).
\]
Write \(L_n=d\Pp_n^{(3)}/d\Qq_n\) and \(Y_{ij}=2a_{ij}-1\).  Define
\[
  \widehat T_n:=\binom n3^{-1/2}
  \sum_{i<j<k}Y_{ij}Y_{jk}Y_{ki},
  \qquad
  \gamma_n:=
  \frac{\sqrt{\binom n3}\,\beta^3}{N_{3,d}^2}.
\]
Then, under \(\Qq_n\),
\begin{equation}
  \log L_n
  =\gamma_n\widehat T_n-\frac12r^3+o_{\Qq_n}(1),
  \qquad
  \gamma_n^2\longrightarrow r^3.
  \label{eq:cubic-full-critical-LAN}
\end{equation}
\end{proposition}

\begin{proof}
We use the truncation constructed in
Section~\ref{subsec:cubic-variational-mechanism}.
Step~1 recalls its \(L^1\) and second-moment bounds.
Step~2 proves the additional \(L^2\) approximation
\eqref{eq:cubic-likelihood-projection}, and Step~3 applies the
likelihood-identification argument of
Section~\ref{subsec:cubic-likelihood-mechanism}.

\noindent\textit{Step 1. Recall the normalized truncation.}
Take \(\widetilde L_n=\overline L_n/c_n\) from
Section~\ref{subsec:cubic-variational-mechanism}, where
\(c_n=\mathbb E_{\Qq_n}\overline L_n\to1\).
This likelihood has mean one and, by
\eqref{eq:cubic-main-likelihood-interface}, satisfies
\begin{align}
 \|L_n-\widetilde L_n\|_{L^1(\Qq_n)}
 &\longrightarrow0,
 \label{eq:cubic-normalized-truncated-L1}\\
 \limsup_n\mathbb E_{\Qq_n}\widetilde L_n^2
 &\le e^{r^3}.
 \label{eq:cubic-truncated-L2-upper}
\end{align}

\noindent\textit{Step 2. Prove the conditional \(L^2\) approximation.}
Let \(G_n=\mathbb E_{\Qq_n}(\widetilde L_n\mid\widehat T_n)\).
The orthogonal-projection identity is
\begin{equation}
 \|\widetilde L_n-G_n\|_{L^2(\Qq_n)}^2
 =\mathbb E_{\Qq_n}\widetilde L_n^2-\mathbb E_{\Qq_n}G_n^2.
 \label{eq:cubic-likelihood-projection-identity}
\end{equation}
Let \(d\widetilde\Pp_n=\widetilde L_n\,d\Qq_n\).
By \eqref{eq:cubic-normalized-truncated-L1},
\(\TV(\widetilde\Pp_n,\Pp_n^{(3)})=o(1)\), so the alternative
triangle limit in Lemma~\ref{lem:cubic-critical-triangle-clt}
also holds under \(\widetilde\Pp_n\).
Moreover,
\(\gamma_n^2=r_n^3\{(n-1)(n-2)/n^2\}\{d^3/(6N)\}\to r^3\).
Along any subsequence, pass to a further one with
\(\gamma_n\to\Delta\in\{-r^{3/2},r^{3/2}\}\).
For bounded continuous \(f\), the inequality
\(\mathbb E_{\Qq_n}\{G_n-f(\widehat T_n)\}^2\ge0\) gives
\(\mathbb E_{\Qq_n}G_n^2\ge
2\mathbb E_{\widetilde\Pp_n}f(\widehat T_n)
-\mathbb E_{\Qq_n}f(\widehat T_n)^2\).
The triangle limits therefore imply
\[
 \liminf_n\mathbb E_{\Qq_n}G_n^2
 \ge 2\mathbb Ef(Z+\Delta)-\mathbb Ef(Z)^2,
 \qquad Z\sim\Normal(0,1),
\]
and taking \(f=f_M\), where
\(f_M(z)=\min\{M,e^{\Delta z-\Delta^2/2}\}\), yields
\[
 \liminf_n\mathbb E_{\Qq_n}G_n^2
 \ge e^{\Delta^2}
      -\mathbb E\big(e^{\Delta Z-\Delta^2/2}-f_M(Z)\big)^2
 \xrightarrow[M\to\infty]{}e^{\Delta^2}=e^{r^3}.
\]
The last limit uses dominated convergence and
\(\mathbb E e^{2\Delta Z-\Delta^2}=e^{\Delta^2}<\infty\).
Conditional Jensen's inequality and
\eqref{eq:cubic-truncated-L2-upper} give the matching upper bound,
\[
 \limsup_n\mathbb E_{\Qq_n}G_n^2
 \le\limsup_n\mathbb E_{\Qq_n}\widetilde L_n^2\le e^{r^3}.
\]
Both second moments therefore converge to \(e^{r^3}\) along every such
further subsequence.  Equation~\eqref{eq:cubic-likelihood-projection-identity}
then gives, along the full sequence,
\begin{equation}
 \big\|\widetilde L_n-
       \mathbb E_{\Qq_n}(\widetilde L_n\mid\widehat T_n)
 \big\|_{L^2(\Qq_n)}^2
 =\mathbb E_{\Qq_n}\widetilde L_n^2-\mathbb E_{\Qq_n}G_n^2
 \longrightarrow0.
 \label{eq:cubic-likelihood-projection}
\end{equation}

\noindent\textit{Step 3. Identify the original likelihood.}
Applying the argument in
Section~\ref{subsec:cubic-likelihood-mechanism} to
\eqref{eq:cubic-normalized-truncated-L1}--\eqref{eq:cubic-truncated-L2-upper}
and the triangle limits gives
\(L_n/e^{\gamma_n\widehat T_n-\gamma_n^2/2}=1+o_{\Qq_n}(1)\).
Taking logarithms and using \(\gamma_n^2\to r^3\) proves
\eqref{eq:cubic-full-critical-LAN}.
\end{proof}

\subsection{Strong detection and limits at the threshold}
\label{subsec:cubic-supp-supercritical}

Proposition~\ref{prop:cubic-full-critical-LAN} establishes the
below-threshold likelihood expansion in
Theorem~\ref{thm:pure-cubic-critical}.
We now prove the theorem's strong-detection conclusion for
\(r>r_{\mathrm{tens}}\).
Specifically, Proposition~\ref{prop:cubic-supercritical-free-energy} shows that
\(\{\log L_n\ge\eta(r)nd\}\) has probability tending to one under
\(\Pp_n^{(3)}\) and to zero under \(\Qq_n\).
This proves \(\TV(\Pp_n^{(3)},\Qq_n)\to1\), which
complements the below-threshold limit proved in
Section~\ref{subsec:cubic-likelihood-mechanism}.
Finally, Proposition~\ref{prop:cubic-endpoint} uses these fixed-\(r\)
limits to construct two sequences with \(r_n\to r_{\mathrm{tens}}\)
but different total-variation limits, completing the theorem's
assertion at the threshold.

For this argument, Lemma~\ref{lem:cubic-superpoly-regularity} constructs
an event \(\mathcal R_{n,d}\) with
\(\rho_n(\mathcal R_{n,d}^c)=o(d^{-M})\) for every fixed \(M\).
On this event, the conditional edge means satisfy
\(\max_{i\ne j}|\beta Z_{3,d}(X_i,X_j)|
\le C_A|\beta|\log^3d/\sqrt N=o(1)\).
These bounds allow us to apply the comparison in
Lemma~\ref{lem:cubic-constrained-edge-replacement} and control the
polynomial factors multiplying the conditioning error.

\begin{lemma}[Sample bounds with superpolynomially small failure probability]
\label{lem:cubic-superpoly-regularity}
Fix \(A<\infty\).  There are a constant \(C_A\) and events
\({\cal R}_{n,d}\subseteq(\Sph^{d-1})^n\) such that, on
\({\cal R}_{n,d}\), the bounds
\eqref{eq:cubic-linear-hereditary}--\eqref{eq:cubic-zonal-hereditary}
hold with \(m=n\) and \(C_{A,M}\) replaced by \(C_A\), and
\begin{align}
 \left\|\sum_{i=1}^n
 q_{3,d}(X_i)q_{3,d}(X_i)^{\mathsf T}\right\|_{\rm op}
 &\le C_A\left(\frac nN+\log^2 N\right),
 \label{eq:cubic-superpoly-feature-gram}\\
 \max_{i\ne j}|Z_{3,d}(X_i,X_j)|
 &\le C_A\frac{\log^3 d}{\sqrt N}.
 \label{eq:cubic-superpoly-regularity}
\end{align}
The events and \(C_A\) are independent of \(\beta\) and \(M\).
For every fixed \(M<\infty\), these same events satisfy, uniformly
over \(n\le d^A\),
\begin{equation}
 \varpi_n:=\rho_n({\cal R}_{n,d})=1-o(d^{-M}).
 \label{eq:cubic-regular-prior-mass}
\end{equation}
Membership in \({\cal R}_{n,d}\) is unchanged by common orthogonal
rotations, coordinate permutations, and individual sign changes.
\end{lemma}

\begin{proof}
We construct \({\cal R}_{n,d}\) as the intersection of the events
on which the four asserted sample bounds hold.  Compared with
\eqref{eq:cubic-feature-gram} and \eqref{eq:cubic-pair-separation},
the larger logarithmic factors allow us to choose the constants
independently of \(M\).
For the bounds
\eqref{eq:cubic-linear-hereditary}--\eqref{eq:cubic-zonal-hereditary}
from Lemma~\ref{lem:cubic-hereditary-envelope}, the \(1/4\)-net
used in the proof of \eqref{eq:cubic-linear-hereditary}
has at most \(9^d\) points, and the tail probability for a fixed net
point and a fixed \(k\)-subset is at most \(2e^{-cs}\).
Taking \(s=C_0\{d+k\log(en/k)\}\), the total failure probability is
bounded, for a sufficiently large absolute \(C_0\), by
\[
 \sum_{k=1}^n 2\,9^d\binom nk
       e^{-cC_0\{d+k\log(en/k)\}}
 \le 2e^{-(cC_0-\log9)d}
       \sum_{k=1}^n e^{-(cC_0-1)k\log(en/k)}
 \le Ce^{-c'd}.
\]
Here we used \(\binom nk\le(en/k)^k\) and \(\log(en/k)\ge1\).
The bound \eqref{eq:cubic-zonal-hereditary} follows
deterministically from \eqref{eq:cubic-linear-hereditary} by the
argument in that lemma.

We next prove \eqref{eq:cubic-superpoly-feature-gram}.
Write \(q_i=q_{3,d}(X_i)\), so that \(\|q_i\|=1\) and
\(\E q_iq_i^{\mathsf T}=N^{-1}I_N\).  Matrix
Bernstein \cite[Theorem~6.1]{Tropp2012}, at tail parameter
\(C\log^2 N\), gives
\[
 \left\|\sum_iq_iq_i^{\mathsf T}-\frac nN I_N\right\|_{\rm op}
 \le C\left(\sqrt{\frac nN}\log N+\log^2 N\right)
\]
outside an event of probability at most \(2N e^{-c\log^2 N}\).
Using \(2\sqrt{xy}\le x+y\) and adding the mean matrix
\((n/N)I_N\) gives \eqref{eq:cubic-superpoly-feature-gram}.

Finally, to prove \eqref{eq:cubic-superpoly-regularity}, we combine the
spherical concentration bound \cite[Theorem~3.4.6]{Vershynin2018}
with a union bound over at most \(n^2\) pairs to obtain
\[
 \max_{i\ne j}|\langle X_i,X_j\rangle|
 \le C\frac{\log d}{\sqrt d}
\]
outside an event of probability at most
\(2d^{2A}e^{-c\log^2 d}\).
The polynomial formula \eqref{eq:cubic-zonal-polynomial} then gives
\eqref{eq:cubic-superpoly-regularity}, since \(N\asymp d^3\).

Combining the three failure probabilities gives
\[
 \rho_n({\cal R}_{n,d}^{\,c})
 \le Ce^{-c'd}+2N e^{-c\log^2 N}
             +2d^{2A}e^{-c\log^2d}
 =o(d^{-M})
\]
for every fixed \(M\), since \(N\asymp d^3\).
The constants were chosen independently of \(M\), so this proves
\eqref{eq:cubic-regular-prior-mass} for the same event
\({\cal R}_{n,d}\) throughout.
Each defining condition is preserved by common orthogonal rotations,
coordinate permutations, and individual sign changes, and none involves
\(\beta\).  Their intersection \({\cal R}_{n,d}\) therefore has the
remaining asserted properties.
\end{proof}

We now use \(\mathcal R_{n,d}\) in
Proposition~\ref{prop:cubic-supercritical-free-energy}, first to compare
the expected log-likelihoods and then to control their variation.

\begin{proposition}[Likelihood-ratio separation above the tensor threshold]
\label{prop:cubic-supercritical-free-energy}
Fix \(A<\infty\).  Suppose
\[
 d=d_n\to\infty,\qquad n\le d^A,\qquad |\beta|\le1,
 \qquad r_n=\frac{n\beta^2}{Nd}\longrightarrow
 r>r_{\mathrm{tens}}.
\]
There is \(\eta(r)>0\) such that, for the original likelihood ratio
\(L_n=d\Pp_n^{(3)}/d\Qq_n\),
\begin{align}
  \mathbb P_{\Pp_n^{(3)}}
  \{\log L_n\ge\eta(r)nd\}&\longrightarrow1,
  \label{eq:cubic-supercritical-planted}\\
  \mathbb P_{\Qq_n}
  \{\log L_n\ge\eta(r)nd\}&\longrightarrow0,
  \label{eq:cubic-supercritical-null}
\end{align}
\end{proposition}

\begin{proof}
For \eqref{eq:cubic-supercritical-planted}, we analyze the mean
and variance of a conditioned log-likelihood, then return to \(\log L_n\).
Step~1 uses Lemma~\ref{lem:cubic-constrained-edge-replacement} and bounds
the conditioning error to reduce the mean estimate to Gaussian relative
entropy in \eqref{eq:cubic-supercritical-mean-reduction}.
Step~2 applies interpolation to obtain the positive mean bound
\eqref{eq:cubic-supercritical-regular-mean}.
Step~3 proves the variance bound
\eqref{eq:cubic-supercritical-regular-concentration}.
Step~4 combines these two bounds by Chebyshev's inequality and removes
the conditioning to prove \eqref{eq:cubic-supercritical-planted}.

\medskip
\noindent\textit{Step 1. Reduce the log-likelihood mean to Gaussian relative entropy.}
Let \(\rho_n=\sigma_d^{\otimes n}\) and use the event
\(\mathcal R_{n,d}\) from Lemma~\ref{lem:cubic-superpoly-regularity}.
Recall \(\varpi_n=\rho_n(\mathcal R_{n,d})\to1\) from
\eqref{eq:cubic-regular-prior-mass}, and set
\(\pi_n^{\rm reg}=\rho_n(\,\cdot\mid\mathcal R_{n,d})\).
Write \(\Pp_n^{\rm reg}\) for the graph law with this
distribution of positions and
\(L_n^{\rm reg}=d\Pp_n^{\rm reg}/d\Qq_n\).
We start from
\(\mathbb E_{\Pp_n^{\rm reg}}\log L_n^{\rm reg}
=D(\Pp_n^{\rm reg}\Vert\Qq_n)\).
To estimate this relative entropy, we first apply the Bernoulli--Gaussian
comparison in Lemma~\ref{lem:cubic-constrained-edge-replacement}.
Let \(u_i^0=q_{3,d}(X_i)\), where \(X\sim\rho_n\), and let
\(\Pp_{\rho_n}^{\rm G}\) be the law of
\((|\beta|\langle u_i^0,u_j^0\rangle+g_{ij})_{i<j}\), where the standard
Gaussian variables \(g_{ij}\) are independent of one another and of \(X\).
Write \(\Qq^{\rm G}\) for the law of \((g_{ij})_{i<j}\).
Write \(\Pp_{\pi_n^{\rm reg}}^{\rm G}\) for the same Gaussian observation
law with \(X\sim\pi_n^{\rm reg}\), and \(\Pp_x^{\rm G}\) for the law
conditional on \(X=x\).

For \(x,x'\in\mathcal R_{n,d}\),
\eqref{eq:cubic-superpoly-regularity} bounds both
\(b_{ij}^0=\beta Z_{3,d}(x_i,x_j)\) and
\(b_{ij}(x')=\beta Z_{3,d}(x_i',x_j')\) in absolute value by
\(\tau_{n,d}:=C_A|\beta|\log^3d/\sqrt N=o(1)\).
Apply Lemma~\ref{lem:cubic-constrained-edge-replacement} with
\(\nu=\pi_n^{\rm reg}\), \(\mathcal S=\mathcal R_{n,d}\), and
\(\mathcal E=\{\{i,j\}:i<j\}\), first with \(X=x\) fixed and then
averaging over \(x\sim\pi_n^{\rm reg}\).
The two expected log-likelihoods are the Bernoulli and Gaussian
relative entropies.  For the Gaussian term, this uses
\((\beta g_{ij})_{i<j}\stackrel{\rm d}=(|\beta|g_{ij})_{i<j}\)
to match the Gaussian law just defined.
Consequently,
\begin{equation}
 \left|
 D(\Pp_n^{\rm reg}\Vert\Qq_n)
 -D(\Pp_{\pi_n^{\rm reg}}^{\rm G}\Vert\Qq^{\rm G})
 \right|
 \le Cn^2(\tau_{n,d}^3+\tau_{n,d}^4)=o(nd).
 \label{eq:cubic-supercritical-channel-comparison}
\end{equation}
Indeed, \(n\beta^2=r_nNd\) gives
\[
 \frac{n\tau_{n,d}^3}{d}
 =O_r\!\left(\frac{|\beta|\log^9d}{\sqrt N}\right)=o(1),
 \qquad
 \frac{n\tau_{n,d}^4}{d}
 =O_r\!\left(\frac{\beta^2\log^{12}d}{N}\right)=o(1).
\]

We next compare the Gaussian laws with position distributions
\(\pi_n^{\rm reg}\) and \(\rho_n\).
By the definition of \(\pi_n^{\rm reg}\),
\[
 \Pp_{\rho_n}^{\rm G}
 =\int \Pp_x^{\rm G}\,d\rho_n(x)
 =\varpi_n\Pp_{\pi_n^{\rm reg}}^{\rm G}
   +\int_{\mathcal R_{n,d}^c}\Pp_x^{\rm G}\,d\rho_n(x).
\]
The last integral has mass \(1-\varpi_n\), so convexity of
\(P\mapsto D(P\Vert\Qq^{\rm G})\) gives
\begin{equation}
 \varpi_n D(\Pp_{\pi_n^{\rm reg}}^{\rm G}\Vert\Qq^{\rm G})
 \ge D(\Pp_{\rho_n}^{\rm G}\Vert\Qq^{\rm G})
     -\int_{\mathcal R_{n,d}^c}
           D(\Pp_x^{\rm G}\Vert\Qq^{\rm G})\,d\rho_n(x).
 \label{eq:cubic-supercritical-Gaussian-conditioning}
\end{equation}
For fixed \(x\),
\(D(\Pp_x^{\rm G}\Vert\Qq^{\rm G})
=\frac{\beta^2}{2}\sum_{i<j}Z_{3,d}(x_i,x_j)^2\).
Since \(|Z_{3,d}|\le1\) and \(n\beta^2=r_nNd\),
\begin{equation}
 \int_{\mathcal R_{n,d}^c}
       D(\Pp_x^{\rm G}\Vert\Qq^{\rm G})\,d\rho_n(x)
 \le\frac{\beta^2}{2}\binom n2(1-\varpi_n)
 \le C_r nNd(1-\varpi_n)=o(nd),
 \label{eq:cubic-supercritical-conditioning-loss}
\end{equation}
where \(N(1-\varpi_n)\to0\) follows from \(N\asymp d^3\) and
\eqref{eq:cubic-regular-prior-mass}.
Combining \eqref{eq:cubic-supercritical-channel-comparison},
\eqref{eq:cubic-supercritical-Gaussian-conditioning}, and
\eqref{eq:cubic-supercritical-conditioning-loss},
and using \(\varpi_n\to1\), gives
\begin{equation}
 \mathbb E_{\Pp_n^{\rm reg}}\log L_n^{\rm reg}
 \ge D(\Pp_{\pi_n^{\rm reg}}^{\rm G}\Vert\Qq^{\rm G})-o(nd)
 \ge \varpi_n^{-1}
       D(\Pp_{\rho_n}^{\rm G}\Vert\Qq^{\rm G})-o(nd).
 \label{eq:cubic-supercritical-mean-reduction}
\end{equation}
Thus it remains to prove an order-\(nd\) positive lower bound on
\(D(\Pp_{\rho_n}^{\rm G}\Vert\Qq^{\rm G})\).

\medskip
\noindent\textit{Step 2. Complete the mean lower bound by Gaussian interpolation.}
For each \(0\le z\le1\), we will prove
\((nd)^{-1}D(\Pp_{\rho_n}^{\rm G}\Vert\Qq^{\rm G})
\ge v_{d,r_n}(z)-\beta^2/(4d)\) by a Guerra-type interpolation
\cite{Guerra2001} between the pairwise Gaussian observations and
independent Gaussian observations of the individual \(u_i^0\).
Specifically, set \(Q=zI_N/N\), and let \(g_1,\ldots,g_n\) be independent
\(\Normal(0,I_N)\) vectors, independent of \(u^0\) and \((g_{ij})_{i<j}\).
For \(0\le t\le1\), define
\begin{equation}
\begin{aligned}
 \mathsf{H}_{Q,t}(u;u^0)
 :={}&\sqrt t\,|\beta|\sum_{i<j}g_{ij}\langle u_i,u_j\rangle
 +t\beta^2\sum_{i<j}\left(
   \langle u_i^0,u_j^0\rangle\langle u_i,u_j\rangle
   -\frac12\langle u_i,u_j\rangle^2\right)\\
 &+\sqrt{(1-t)r_nzd}\sum_i\langle g_i,u_i\rangle
 +(1-t)r_nzd\sum_i\left(
   \langle u_i^0,u_i\rangle-\frac12\|u_i\|^2\right).
\end{aligned}
\label{eq:cubic-supercritical-Hamiltonian}
\end{equation}
This is the Hamiltonian from Step~2(a) of the proof of
Proposition~\ref{prop:cubic-Gaussian-variational-bound}, specialized to
\(\mathsf A=Q_R=Q\) and \(c=r_nNd=n\beta^2\).
Here we use the full product measure \(\mu_d^{\otimes n}\) for both
\(u^0\) and the integral over \(u\), with no constraint on
\(R(u^0,u)=n^{-1}\sum_i u_i^0u_i^{\mathsf T}\).
Define the normalized expected free energy by
\[
 F_{n,d}(t):=\frac1{nd}\mathbb E_{u^0,g}
       \log\int e^{\mathsf{H}_{Q,t}(u;u^0)}\,d\mu_d^{\otimes n}(u),
 \qquad 0\le t\le1,
\]
where the expectation includes \(u^0\sim\mu_d^{\otimes n}\) and both
Gaussian arrays in \(\mathsf{H}_{Q,t}\).
At \(t=1\), the integral is the likelihood ratio of
\(\Pp_{\rho_n}^{\rm G}\) relative to \(\Qq^{\rm G}\), evaluated at the
Gaussian observations.
At \(t=0\), it factors into the \(n\) single-vertex likelihood ratios
\(Z_{r_nzdI_N}(\sqrt{r_nzd}\,u_i^0+g_i)\).
Thus
\begin{equation}
 F_{n,d}(1)=\frac1{nd}D(\Pp_{\rho_n}^{\rm G}\Vert\Qq^{\rm G}),
 \qquad
 F_{n,d}(0)=\frac1dD(P_{r_nzdI_N}\Vert\gamma_N).
 \label{eq:cubic-supercritical-interpolation-endpoints}
\end{equation}

The expression \eqref{eq:cubic-supercritical-Hamiltonian} is the
conditional log-likelihood ratio, relative to independent standard
Gaussian noise, for the two observation arrays
\((\sqrt t\,|\beta|\langle u_i^0,u_j^0\rangle+g_{ij})_{i<j}\) and
\((\sqrt{(1-t)r_nzd}\,u_i^0+g_i)_{i=1}^n\).
Bayes' formula gives the conditional law of \(u^0\) given these arrays as
\[
 \nu_{Q,t}(du):=
 \frac{e^{\mathsf{H}_{Q,t}(u;u^0)}\,d\mu_d^{\otimes n}(u)}
      {\displaystyle\int e^{\mathsf{H}_{Q,t}(v;u^0)}\,
                       d\mu_d^{\otimes n}(v)}.
\]
Draw \(u^1,u^2\) independently from this posterior, independently of
\(u^0\) conditional on the observations.
Then \(u^0,u^1,u^2\), conditional on the two arrays, are iid with law
\(\nu_{Q,t}\).
Write \(\langle\cdot\rangle_t\) for integration over \(u^1,u^2\)
with the observations and \(u^0\) fixed, and let \(\mathbb E\) average
over \(u^0\) and the Gaussian variables.
Applying the conditional-expectation argument of
\eqref{eq:cubic-tensor-channel-replica-identity} to
\(\|R(u,v)-Q\|_{\HS}^2\), where
\(R(u,v)=n^{-1}\sum_i u_iv_i^{\mathsf T}\), gives the Nishimori
identity \cite{LelargeMiolane2017}
\begin{equation}
\begin{aligned}
 \mathbb E\langle\|R_{01}-Q\|_{\HS}^2\rangle_t
 &=\mathbb E\iint\|R(u,v)-Q\|_{\HS}^2\,
                      \nu_{Q,t}(du)\nu_{Q,t}(dv)\\
 &=\mathbb E\langle\|R_{12}-Q\|_{\HS}^2\rangle_t,
\end{aligned}
 \label{eq:cubic-supercritical-Nishimori}
\end{equation}
where \(R_{ab}=R(u^a,u^b)\).
Differentiating \(F_{n,d}(t)\) and integrating by parts in
\((g_{ij})_{i<j}\) and \((g_i)_{i=1}^n\), as in the derivation of
\eqref{eq:cubic-fixed-overlap-derivative}, give
\[
 F_{n,d}'(t)
 =\frac{r_nN}{2}\mathbb E\langle\|R_{01}-Q\|_{\HS}^2\rangle_t
  -\frac{r_nN}{4}\mathbb E\langle\|R_{12}-Q\|_{\HS}^2\rangle_t
  -\frac{r_nz^2}{4}-\frac{r_nN}{4n}.
\]
Here we used \(N\|Q\|_{\HS}^2=z^2\).
Applying \eqref{eq:cubic-supercritical-Nishimori} to the first term gives
\begin{equation}
 F_{n,d}'(t)
 =\frac{r_n}{4}\,
  \mathbb E\left\langle
   N\left\|R_{12}-\frac zN I_N\right\|_{\HS}^2-z^2
  \right\rangle_t-\frac{r_nN}{4n}.
 \label{eq:cubic-supercritical-trial-derivative}
\end{equation}
Dropping the nonnegative squared-norm term in
\eqref{eq:cubic-supercritical-trial-derivative}, integrating from
\(t=0\) to \(t=1\), and using
\eqref{eq:cubic-supercritical-interpolation-endpoints} give
\[
 \frac1{nd}D(\Pp_{\rho_n}^{\rm G}\Vert\Qq^{\rm G})
 \ge \frac1dD(P_{r_nzdI_N}\Vert\gamma_N)
       -\frac{r_nz^2}{4}-\frac{r_nN}{4n}
 =v_{d,r_n}(z)-\frac{\beta^2}{4d}.
\]
The error \(\beta^2/(4d)\) tends to zero uniformly in \(z\).
Optimizing over \(0\le z\le1\) and applying the locally uniform
convergence \eqref{eq:cubic-scalar-potential-uniform-limit}, we obtain
\begin{align}
 \liminf_{n\to\infty}\frac1{nd}
 D(\Pp_{\rho_n}^{\rm G}\Vert\Qq^{\rm G})
 &\ge\sup_{0\le z\le1}
       \left\{{\cal D}_3(rz)-\frac r4z^2\right\}\notag\\
 &=\sup_{0\le q<1}
       \left[\frac r4q^6+\frac12\{\log(1-q)+q\}\right]
 =:\psi(r)>0.
 \label{eq:cubic-supercritical-Gaussian-positive}
\end{align}
For the equality, substitute
\({\cal D}_3(rz)=\tfrac12\sup_{0\le q<1}
\{rzq^3+\log(1-q)+q\}\) and use
\(\sup_{0\le z\le1}(rzq^3/2-rz^2/4)=rq^6/4\), attained at \(z=q^3\).
The definition \eqref{eq:cubic-tensor-threshold} of
\(r_{\mathrm{tens}}\) then gives \(\psi(r)>0\) for
\(r>r_{\mathrm{tens}}\).

Substituting \eqref{eq:cubic-supercritical-Gaussian-positive} into
the second inequality of \eqref{eq:cubic-supercritical-mean-reduction}
and using \(\varpi_n\to1\) gives
\begin{equation}
 D(\Pp_{\pi_n^{\rm reg}}^{\rm G}\Vert\Qq^{\rm G})
 \ge\{\psi(r)+o(1)\}nd.
 \label{eq:cubic-supercritical-conditioned-Gaussian-lower}
\end{equation}
Together with \eqref{eq:cubic-supercritical-channel-comparison}, this
proves, for all sufficiently large \(n\),
\begin{equation}
 \mathbb E_{\Pp_n^{\rm reg}}\log L_n^{\rm reg}
 =D(\Pp_n^{\rm reg}\Vert\Qq_n)
 \ge\frac12\psi(r)nd.
 \label{eq:cubic-supercritical-regular-mean}
\end{equation}

\medskip
\noindent\textit{Step 3. Bound the variance of \(\log L_n^{\rm reg}\).}
Define \(f(x):=\mathbb E_{\Adj\mid X=x}\log L_n^{\rm reg}(\Adj)\)
for every \(x\in(\Sph^{d-1})^n\).
The law of total variance gives
\[
 \operatorname{Var}_{\Pp_n^{\rm reg}}(\log L_n^{\rm reg})
 =\mathbb E_{\pi_n^{\rm reg}}
       \operatorname{Var}(\log L_n^{\rm reg}\mid X)
       +\operatorname{Var}_{\pi_n^{\rm reg}}f.
\]
We first bound \(\mathbb E_{\pi_n^{\rm reg}}
\operatorname{Var}(\log L_n^{\rm reg}\mid X)\).

For every candidate \(x'\in\mathcal R_{n,d}\),
\(|\beta Z_{3,d}(x_i',x_j')|\le\tau_{n,d}\).
As in the proof of \eqref{eq:cubic-edge-log-likelihood-oscillation},
if \(\Adj^{(e)}\) differs from \(\Adj\) only at edge \(e\), then
\[
 \left|\log L_n^{\rm reg}(\Adj^{(e)})
       -\log L_n^{\rm reg}(\Adj)\right|
 \le\log\frac{1+\tau_{n,d}}{1-\tau_{n,d}}
 \le C\tau_{n,d}.
\]
The edges are independent conditional on \(X\).
The Efron--Stein inequality
\cite[Theorem~3.1]{BoucheronLugosiMassart2013}, applied conditionally
and then averaged over \(X\sim\pi_n^{\rm reg}\), therefore gives
\begin{equation}
 \mathbb E_{\pi_n^{\rm reg}}
       \operatorname{Var}(\log L_n^{\rm reg}\mid X)
 \le C\binom n2\tau_{n,d}^2\le Cn^2\tau_{n,d}^2.
 \label{eq:cubic-supercritical-edge-variance}
\end{equation}

To bound \(\operatorname{Var}_{\pi_n^{\rm reg}}f\), let \(x^{(i)}\)
differ from \(x\) only at its \(i\)th coordinate.
Generate the graphs conditional on \(x\) and \(x^{(i)}\) using the same
independent uniform variable for each edge, and denote their signs by
\(Y\) and \(Y'\).
Only edges incident to \(i\) can differ, so
\[
 |f(x)-f(x^{(i)})|
 \le C\tau_{n,d}\sum_{j\ne i}\mathbb P(Y_{ij}\ne Y'_{ij})
 \le C(n-1)\tau_{n,d}.
\]
Here the probability is over the coupling with \(x,x^{(i)}\) fixed.
Let \(X\sim\rho_n=\sigma_d^{\otimes n}\), and let \(X^{(i)}\) replace \(X_i\) by an
independent draw from \(\sigma_d\).  Efron--Stein gives
\[
 \operatorname{Var}_{\rho_n}f
 \le\frac12\sum_{i=1}^n
        \mathbb E\{f(X)-f(X^{(i)})\}^2
 \le Cn^3\tau_{n,d}^2.
\]
The definition of \(\pi_n^{\rm reg}\) then implies
\begin{equation}
 \operatorname{Var}_{\pi_n^{\rm reg}}f
 \le\mathbb E_{\pi_n^{\rm reg}}(f-\mathbb E_{\rho_n}f)^2
 \le\varpi_n^{-1}\operatorname{Var}_{\rho_n}f
 \le Cn^3\tau_{n,d}^2.
 \label{eq:cubic-supercritical-position-variance}
\end{equation}
Substituting \eqref{eq:cubic-supercritical-edge-variance} and
\eqref{eq:cubic-supercritical-position-variance} into the variance
decomposition at the start of Step~3 gives
\begin{equation}
 \operatorname{Var}_{\Pp_n^{\rm reg}}(\log L_n^{\rm reg})
 \le C(n^2+n^3)\tau_{n,d}^2=o((nd)^2).
 \label{eq:cubic-supercritical-regular-concentration}
\end{equation}
For the last equality, \(n\beta^2=r_nNd\) gives
\(n\tau_{n,d}^2/d^2=O_r(\log^6d/d)=o(1)\).

\medskip
\noindent\textit{Step 4. Prove separation for the original likelihood.}
Choose \(0<\eta(r)<\psi(r)/4\).
Chebyshev's inequality, \eqref{eq:cubic-supercritical-regular-mean},
and \eqref{eq:cubic-supercritical-regular-concentration} give
\[
 \Pp_n^{\rm reg}
       \{\log L_n^{\rm reg}<2\eta(r)nd\}
 \le
 \frac{\operatorname{Var}_{\Pp_n^{\rm reg}}
                 (\log L_n^{\rm reg})}
      {(\psi(r)/2-2\eta(r))^2(nd)^2}
 =o(1).
\]
The integral defining \(L_n\) decomposes as
\[
 L_n=\varpi_n L_n^{\rm reg}
       +\int_{\mathcal R_{n,d}^c}L_x\,d\rho_n(x)
 \ge\varpi_n L_n^{\rm reg}.
\]
Since \(\varpi_n\to1\), we have
\(2\eta(r)nd+\log\varpi_n\ge\eta(r)nd\) for all sufficiently large \(n\).
Thus
\(\{\log L_n<\eta(r)nd\}
 \subseteq\{\log L_n^{\rm reg}<2\eta(r)nd\}\).
The same mixture decomposition also gives
\(\TV(\Pp_n^{(3)},\Pp_n^{\rm reg})\le1-\varpi_n\), so
\[
 \Pp_n^{(3)}\{\log L_n<\eta(r)nd\}
 \le \Pp_n^{\rm reg}
           \{\log L_n^{\rm reg}<2\eta(r)nd\}
       +1-\varpi_n=o(1).
\]
This proves \eqref{eq:cubic-supercritical-planted}.
Under \(\Qq_n\), the identity \(\mathbb E_{\Qq_n}L_n=1\) and Markov's
inequality give
\[
 \Qq_n\{\log L_n\ge\eta(r)nd\}
 =\Qq_n\{L_n\ge e^{\eta(r)nd}\}
 \le e^{-\eta(r)nd}\longrightarrow0,
\]
which proves \eqref{eq:cubic-supercritical-null}.
\end{proof}

For fixed \(A<\infty\), suppose \(d=d_n\to\infty\), \(n\le d^A\),
and \(|\beta|\le1\) as \(n\to\infty\).
Section~\ref{subsec:cubic-likelihood-mechanism} and
Proposition~\ref{prop:cubic-supercritical-free-energy} give
\begin{equation}
 \TV(\Pp_n^{(3)},\Qq_n)\longrightarrow
 \begin{cases}
  2\Phi(r^{3/2}/2)-1,& r_n\to r\in(0,r_{\mathrm{tens}}),\\
  1,& r_n\to r>r_{\mathrm{tens}}.
 \end{cases}
 \label{eq:cubic-fixed-r-TV}
\end{equation}

The case \(r_n\to r_{\mathrm{tens}}\) requires a separate argument.
Proposition~\ref{prop:cubic-matrix-branch-gap} assumes a compact set
\(K\subset(0,r_{\mathrm{tens}})\), so that
\(\max K<r_{\mathrm{tens}}\).
To see why \eqref{eq:cubic-uniform-branch-gap} requires this restriction,
set \(Q_d=q_\star^3I_N/N\).  At \(r=r_{\mathrm{tens}}\),
Lemma~\ref{lem:cubic-scalar-threshold} and
\eqref{eq:cubic-scalar-channel-limit} give
\[
 N\|Q_d\|_{\HS}^2=q_\star^6,
 \qquad
 {\cal V}_{d,r_{\mathrm{tens}}}(Q_d)
 \longrightarrow v_{r_{\mathrm{tens}}}(q_\star^3)=0.
\]
Thus, for every \(0<\delta\le q_\star^6\),
\[
 \liminf_{d\to\infty}
 \sup_{\substack{Q\succeq0\\N\|Q\|_{\HS}^2\ge\delta}}
 {\cal V}_{d,r_{\mathrm{tens}}}(Q)
 \ge\lim_{d\to\infty}{\cal V}_{d,r_{\mathrm{tens}}}(Q_d)=0.
\]
Hence the strictly negative upper bound in
\eqref{eq:cubic-uniform-branch-gap} does not hold at
\(r=r_{\mathrm{tens}}\) for these values of \(\delta\).
For \(A>4\), Proposition~\ref{prop:cubic-endpoint} uses a diagonal
construction based on \eqref{eq:cubic-fixed-r-TV} to obtain two sequences
with \(r_n\to r_{\mathrm{tens}}\) but different total-variation limits.

\begin{proposition}[Sequence-dependent total-variation limits at the tensor threshold]
\label{prop:cubic-endpoint}
Fix \(4<A<\infty\) and any \(\beta_0\in[-1,1]\setminus\{0\}\).
Put
\[
 \tau_\star=2\Phi(r_{\mathrm{tens}}^{3/2}/2)-1<1.
\]
There are two parameter sequences with
\(d=d_n\to\infty\), \(n\le d^A\), \(\beta\equiv\beta_0\), and
\(r_n=n\beta^2/(N_{3,d}d)\to r_{\mathrm{tens}}\):
along one the total variation converges to
$\tau_\star$, whereas along the other it converges to one.
\end{proposition}

\begin{proof}
For integers \(k\) large enough, set
\[
  r_k^-:=r_{\mathrm{tens}}-\frac1k,
  \qquad
  r_k^+:=r_{\mathrm{tens}}+\frac1k.
\]
For each sign \(\pm\) and each fixed \(k\), take \(\beta_n=\beta_0\) and
\[
  d_k^\pm(n):=
  \left\lfloor\left(\frac{6\beta_0^2n}{r_k^\pm}\right)^{1/4}\right\rfloor.
\]
Since \(N_{3,d}d=(d^4/6)\{1+O(d^{-1})\}\), these choices satisfy
\[
  d_k^\pm(n)\asymp n^{1/4},\qquad
  \frac{n\beta_0^2}{N_{3,d_k^\pm(n)}d_k^\pm(n)}
  =r_k^\pm+O(n^{-1/4}).
\]
In particular, \(n\le\{d_k^\pm(n)\}^A\) for all sufficiently large
\(n\), since \(A>4\).  Equation~\eqref{eq:cubic-fixed-r-TV}
therefore gives, as \(n\to\infty\),
\[
 \TV(\Pp_n^{(3)},\Qq_n)\longrightarrow
 \begin{cases}
  2\Phi\{(r_k^-)^{3/2}/2\}-1,& d=d_k^-(n),\\
  1,& d=d_k^+(n).
 \end{cases}
\]
Write \(\Pp_{n,d,\beta}^{(3)}\) for the cubic graph law with its dimension
and amplitude displayed.  The preceding convergences allow us to choose
integer thresholds \(n_k\) recursively, with \(n_{k+1}>n_k\), so that,
for every \(n\ge n_k\),
\[
 d_k^\pm(n)\ge k,\qquad n\le\{d_k^\pm(n)\}^A,\qquad
 \left|\frac{n\beta_0^2}{N_{3,d_k^\pm(n)}d_k^\pm(n)}-r_k^\pm\right|
 \le\frac1k,
\]
and
\[
\begin{aligned}
 \left|\TV(\Pp_{n,d_k^-(n),\beta_0}^{(3)},\Qq_n)
       -\left\{2\Phi\left(\frac{(r_k^-)^{3/2}}2\right)-1\right\}\right|
 &\le\frac1k,\\
 1-\TV(\Pp_{n,d_k^+(n),\beta_0}^{(3)},\Qq_n)&\le\frac1k.
\end{aligned}
\]
For sufficiently large \(n\), set
\[
 k(n)=\max\{k:n_k\le n\},\qquad
 d_n^\pm=d_{k(n)}^\pm(n),\qquad \beta_n=\beta_0.
\]
Then \(k(n)\to\infty\), \(d_n^\pm\ge k(n)\to\infty\), and
\(n\le(d_n^\pm)^A\).  Along either sequence, the parameter bound gives
\[
 \left|\frac{n\beta_0^2}{N_{3,d_n^\pm}d_n^\pm}-r_{\mathrm{tens}}\right|
 \le\frac1{k(n)}+|r_{k(n)}^\pm-r_{\mathrm{tens}}|
 =\frac2{k(n)}\longrightarrow0.
\]
The two total-variation bounds with \(k=k(n)\), together with
\(r_{k(n)}^-\to r_{\mathrm{tens}}\) and continuity of \(\Phi\),
give the claimed limits.
Since \(\Phi(x)<1\) for finite \(x\), \(\tau_\star<1\).
\end{proof}

\section{A local partition-function bound}
\label{app:cubic-local-partition}

This appendix proves Proposition~\ref{prop:cubic-raw-local-partition}.
The shared estimates are proved first in
Sections~\ref{subsec:cubic-local-regularity}--\ref{subsec:cubic-local-summed-estimates}.
Sections~\ref{subsec:cubic-local-deletion}--\ref{subsec:cubic-direct-orbit-closure}
then prove the first inequality below, and
Section~\ref{subsec:cubic-local-centered} proves the second.

For the partition-function estimates, fix \(A<\infty\) and a compact
set \(K\subset(0,\infty)\), and assume
\begin{equation}
 n\le d^A,\qquad 0<|\beta|\le1,\qquad
 r_n=\frac{n\beta^2}{N_{3,d}d}\in K.
 \label{eq:cubic-local-standing-assumptions}
\end{equation}

For two independent samples \(X,X'\sim\rho_n=\sigma_d^{\otimes n}\),
the pairs \((X_i,X_i')\) are iid with law
\(P_0=\sigma_d\otimes\sigma_d\).  With \(N=N_{3,d}\), recall
\[
\begin{gathered}
  w_i=q_{3,d}(X_i)q_{3,d}(X_i')^{\mathsf T},\qquad
  h_{ij}=\langle w_i,w_j\rangle_{\HS},\\
  S_m=\sum_{i=1}^mw_i,\qquad
  W_{n,\beta}=\prod_{i<j}(1+\beta^2h_{ij}).
\end{gathered}
\]
Here \(0\le m\le n\).  We also write
\(E_m=\sum_{i<j\le m}h_{ij}=(\|S_m\|_{\HS}^2-m)/2\) and recall
\(s_n=s_n(X,X')=N\|S_n\|_{\HS}^2/n^2\).

Let \(\mathcal E_n^{\rm reg}
=\{X\in\mathcal G_n^\dagger,\ X'\in\mathcal G_n^\dagger\}\), with
\(\mathcal G_n^\dagger\) constructed in
Lemma~\ref{lem:cubic-hereditary-square-profile}.  For \(K_d\uparrow\infty\) sufficiently slowly
and sufficiently small fixed \(\delta>0\), the proof reduces to
\begin{align}
 &\mathbb E_{P_0^{\otimes n}}\!\left[
   \1_{\mathcal G_n^\dagger(X)}\1_{\mathcal G_n^\dagger(X')}
   W_{n,\beta}\1_{\{s_n<\delta\}}\right]\notag\\
 &\quad\le\mathbb E_{P_0^{\otimes n}}\!\left[
   \1_{\mathcal G_n^\dagger(X)}\1_{\mathcal G_n^\dagger(X')}
   W_{n,\beta}\1_{\{\bigl|\|S_n\|_{\HS}^2-n\bigr|\le K_dd^2/\beta^2\}}\right]+o(1)
   \qquad\text{by Corollary~\ref{cor:cubic-local-central-reduction}}\notag\\
 &\quad\le\exp\left\{\binom n3\frac{\beta^6}{N^4}+o(1)\right\}+o(1)
   \qquad\text{by Lemma~\ref{lem:cubic-centered-simple-edge}}.
 \label{eq:cubic-local-roadmap}
\end{align}

For the first inequality,
Section~\ref{subsec:cubic-local-deletion} combines the probability estimates
with \eqref{eq:cubic-local-linear-reward} to obtain preliminary weighted
tail bounds.
Section~\ref{subsec:cubic-direct-orbit-closure} then restricts successively to
\(\|S_n\|_{\HS}^2-n\le n/(\beta^2\log^4d)\) and
\(\beta^2\ell(S_n)\le\epsilon_\star d\), at a weighted cost of \(o(1)\),
before controlling the remaining positive tail of \(\|S_n\|_{\HS}^2-n\).
Together these estimates give
\[
 \mathbb E_{P_0^{\otimes n}}\!\left[
   \1_{\mathcal G_n^\dagger(X)}\1_{\mathcal G_n^\dagger(X')}
   W_{n,\beta}\1_{\{s_n<\delta,\ \|S_n\|_{\HS}^2-n>K_dd^2/\beta^2\}}\right]=o(1).
\]
To control the region \(\|S_n\|_{\HS}^2-n<-K_dd^2/\beta^2\), we use
\(1+x\le e^x\) to obtain the pointwise bound
\begin{equation}
 0\le W_{n,\beta}
 \le e^{\beta^2\sum_{i<j}h_{ij}}
 =e^{\beta^2(\|S_n\|_{\HS}^2-n)/2}
 \le e^{r_n nd\,s_n/2}.
 \label{eq:cubic-local-linear-reward}
\end{equation}
Consequently,
\begin{equation}
 \mathbb E_{P_0^{\otimes n}}\!\left[
   W_{n,\beta}\1_{\{\|S_n\|_{\HS}^2-n<-K_dd^2/\beta^2\}}\right]
 \le e^{-K_dd^2/2}=o(1).
 \label{eq:cubic-local-negative-tail}
\end{equation}
Together they give the first inequality in \eqref{eq:cubic-local-roadmap}.

For the second inequality, let \(\mathcal E^{\rm term}_{n,d}\) be the event
defined in Step 1 of the proof of Lemma~\ref{lem:cubic-centered-simple-edge}.
Corollary~\ref{cor:cubic-central-weighted-deletions} applies the shared
probability bounds on the central region and gives
\[
 \mathbb E_{P_0^{\otimes n}}\!\left[
   \1_{\mathcal G_n^\dagger(X)}\1_{\mathcal G_n^\dagger(X')}
   W_{n,\beta}\1_{\{\bigl|\|S_n\|_{\HS}^2-n\bigr|\le K_dd^2/\beta^2\}}\right]
 =\mathbb E_{P_0^{\otimes n}}
   [\1_{\mathcal E^{\rm term}_{n,d}}W_{n,\beta}]+o(1).
\]
Set \(\eta=\beta^2(n-2)/(3N^2)=o(1)\), and let \(Q\) range over
exchangeable laws with \(Q(\mathcal E^{\rm term}_{n,d})=1\) and
\(D(Q\Vert P_0^{\otimes n})<\infty\).
Then, with the convention \(\log0=-\infty\),
\begin{align*}
 &\log\mathbb E_{P_0^{\otimes n}}
       [\1_{\mathcal E^{\rm term}_{n,d}}W_{n,\beta}]
 =\sup_Q\{\mathbb E_Q\log W_{n,\beta}-D(Q\Vert P_0^{\otimes n})\}
       \qquad\text{by \eqref{eq:cubic-restricted-variational}}\\
 &\;\le\sup_Q\!\biggl\{
   \frac{\beta^4}{2N^2}\sum_{m<n}\!
      \left(\frac{\mathbb E_Q\|S_m\|_{\HS}^2}{1-\eta}-m\right)
   -\eta D(Q\Vert P_0^{\otimes n})
   \biggr\}+o(1)
       \,\text{by \eqref{eq:cubic-terminal-fourth-error} and
                   Lemmas~\ref{lem:cubic-regular-prefix-estimates},
                   \ref{lem:cubic-regular-row-expansion}}\\
 &\;\le\frac{\eta\beta^4n(n-1)}{4N^2}
                 \left(1+\frac1{1-\eta}\right)+o(1)
       \quad\text{by \eqref{eq:cubic-prefix-linear-identity}--%
                    \eqref{eq:cubic-prefix-diagonal-identity}
                    and \eqref{eq:cubic-linear-entropy-absorption}}\\
 &\;=\binom n3\frac{\beta^6}{N^4}+o(1).
\end{align*}
Adding the discarded \(o(1)\) contribution gives the second inequality
in the roadmap.

Finally, since \(N\sim d^3/6\), the exponent satisfies
\(\binom n3\beta^6/N^4=(1+o(1))r_n^3d^3/(6N)\to r^3\) when
\(r_n\to r\in K\).
The final proof in Section~\ref{subsec:cubic-local-centered}
combines the two inequalities and
verifies uniform boundedness for \(r_n\in K\).

\subsection{Regularity bounds}
\label{subsec:cubic-local-regularity}

Lemma~\ref{lem:cubic-hereditary-square-profile} strengthens the event in
Lemma~\ref{lem:cubic-hereditary-envelope} by adding the fourth-moment bounds
\eqref{eq:cubic-hereditary-fourth-profile} and
\eqref{eq:cubic-hereditary-directional-fourth-profile}.
These bounds will control the coefficients in the conditional
exponential-moment estimates below.

\begin{lemma}[Uniform covariance and fourth-moment bounds over sample subsets]
\label{lem:cubic-hereditary-square-profile}
Fix \(A<\infty\), \(\underline c>0\), and \(M<\infty\).  Uniformly for
\[
  \underline c\,N_{3,d}d\le n\le d^A,
\]
there is an orthogonally invariant, sign-invariant, exchangeable event
\({\cal G}_{n,M}^\dagger\subseteq{\cal G}_{n,M}\) such that
\(\rho_n({\cal G}_{n,M}^\dagger)=1-o(d^{-M})\).  On this event, for every
\(J\subseteq[n]\) with \(|J|=m\ge n/3\) and every
\(0\preceq B\preceq I_N\),
\begin{align}
  \frac1m\sum_{j\in J}
    \{q_{3,d}(X_j)^{\mathsf T}Bq_{3,d}(X_j)\}^2
  &\le C_{A,\underline c,M}\left\{
    \left(\frac{\tr B}{N}\right)^2
    +\frac{\tr B}{Nd}
    +\frac{\log d\tr B}{m}
  \right\},
  \label{eq:cubic-hereditary-fourth-profile}\\
  \frac1m\sum_{j\in J}q_{3,d}(X_j)q_{3,d}(X_j)^{\mathsf T}
  &\preceq \frac{C_{A,\underline c,M}}{N}I_N.
  \label{eq:cubic-hereditary-covariance-profile}
\end{align}
For \(B=vv^{\mathsf T}\), the bound
\eqref{eq:cubic-hereditary-fourth-profile} improves to the following
estimate, uniformly for \(v\in\mathbb S^{N-1}\):
\begin{equation}
 \frac1m\sum_{j\in J}
   \langle v,q_{3,d}(X_j)\rangle^4
 \le C_{A,\underline c,M}\left(
       \frac1{N^2}+\frac1{Nd}+\frac1m\right).
 \label{eq:cubic-hereditary-directional-fourth-profile}
\end{equation}
Moreover,
\begin{equation}
 \max_{i\ne j}|\langle X_i,X_j\rangle|
 \le C_{A,M}\frac{\log d}{\sqrt d}.
 \label{eq:cubic-raw-pair-separation}
\end{equation}
The same assertions hold for the \(X'\)-sample.
\end{lemma}

\begin{proof}
Step 1 restricts the full-sample Gram bound
\eqref{eq:cubic-feature-gram} to \(J\), giving the covariance bound
\eqref{eq:cubic-hereditary-covariance-profile}.
Step 2 decomposes \(q_{3,d}(x)^{\mathsf T}Bq_{3,d}(x)\) into its
constant term \(\tr B/N\) and harmonics of degrees \(2,4,6\).
Coefficient estimates and matrix Bernstein bounds for the empirical
second moments of these harmonic features give
\eqref{eq:cubic-hereditary-fourth-profile}.
For the sharper directional estimate, Step 3 proves
\eqref{eq:cubic-raw-pair-separation} by spherical concentration.
Together with \eqref{eq:cubic-zonal-polynomial}, this bounds the pairwise
feature inner products and hence the number of large projections
\(|\langle v,q_{3,d}(X_j)\rangle|\).
The covariance bound \eqref{eq:cubic-hereditary-covariance-profile}
controls the smaller projections, giving, uniformly in \(v\) and \(J\),
\[
 \frac1m\sum_{j\in J}\langle v,q_{3,d}(X_j)\rangle^4
 \le\frac C m+\frac C{Nd},
\]
which implies \eqref{eq:cubic-hereditary-directional-fourth-profile}.
Step 4 intersects the events from Steps 2--3 with \(\mathcal G_{n,M}\)
and verifies the required invariances and failure probability.

\emph{Step 1: covariance bounds.}
Recall from Lemma~\ref{lem:cubic-hereditary-envelope} that
\(\mathbf Q_n\in\mathbb R^{n\times N}\) has \(i\)th row
\(q_{3,d}(X_i)^{\mathsf T}\).
By \eqref{eq:cubic-feature-gram}, \(m=|J|\ge n/3\), and
\(n\ge\underline c\,Nd\),
\[
 \frac1{|J|}\sum_{j\in J}q_{3,d}(X_j)q_{3,d}(X_j)^{\mathsf T}
 \preceq
 \frac3n\mathbf Q_n^{\mathsf T}\mathbf Q_n
 \preceq C_{A,M}\left(\frac1N+\frac{\log N}{n}\right)I_N
 \preceq \frac{C_{A,\underline c,M}}N I_N,
\]
which proves \eqref{eq:cubic-hereditary-covariance-profile}.

\emph{Step 2: fourth moments of quadratic forms.}
To prove \eqref{eq:cubic-hereditary-fourth-profile}, we expand
\(q_{3,d}(x)^{\mathsf T}Bq_{3,d}(x)\) in harmonic degrees
\(\ell\in\{0,2,4,6\}\).  We extend the degree-three notation from
Section~\ref{subsec:cubic-harmonic-kernel} to these even degrees.  Let
\(\cH_{\ell,d}\) be the space of degree-\(\ell\) spherical harmonics and
\(N_{\ell,d}=\dim\cH_{\ell,d}\).  Choose a basis
\((\psi_{\ell,a})_{a=1}^{N_{\ell,d}}\) of \(\cH_{\ell,d}\), orthonormal
in \(L^2(\sigma_d)\), and set
\[
 q_{\ell,d}(x):=N_{\ell,d}^{-1/2}
       (\psi_{\ell,a}(x))_{a=1}^{N_{\ell,d}},\qquad
 Z_{\ell,d}(\langle x,y\rangle):=
       \langle q_{\ell,d}(x),q_{\ell,d}(y)\rangle.
\]
We likewise abbreviate \(Z_{\ell,d}(\langle x,y\rangle)\) by
\(Z_{\ell,d}(x,y)\).  The addition formula for degree \(\ell\)
\cite[Section~1.2]{DaiXu2013} is
\begin{equation}
 Z_{\ell,d}(x,y)
 =\frac1{N_{\ell,d}}\sum_{a=1}^{N_{\ell,d}}
       \psi_{\ell,a}(x)\psi_{\ell,a}(y),\qquad Z_{\ell,d}(x,x)=1.
 \label{eq:cubic-general-degree-addition}
\end{equation}
Equation~\eqref{eq:cubic-general-degree-addition} and orthonormality give
\(\|q_{\ell,d}(x)\|=1\) and
\(\mathbb E_{X\sim\sigma_d}q_{\ell,d}(X)q_{\ell,d}(X)^{\mathsf T}
  =I_{N_{\ell,d}}/N_{\ell,d}\).
For \(\ell=0\), choose \(\psi_{0,1}=1\), so
\(N_{0,d}=q_{0,d}(x)=Z_{0,d}(t)=1\).
For each \(\ell\in\{0,2,4,6\}\), define the coefficient vector
\[
 c_{\ell,B}:=N_{\ell,d}\int_{\Sph^{d-1}}
   \bigl(q_{3,d}(y)^{\mathsf T}Bq_{3,d}(y)\bigr)
   q_{\ell,d}(y)\,d\sigma_d(y)\in\mathbb R^{N_{\ell,d}}.
\]
The harmonic decomposition of degree-six polynomials
\cite[Section~1.1]{DaiXu2013} then gives
\begin{equation*}
  q_{3,d}(x)^{\mathsf T}Bq_{3,d}(x)
  =\sum_{\ell\in\{0,2,4,6\}}
    \langle c_{\ell,B},q_{\ell,d}(x)\rangle .
\end{equation*}
In particular, \(c_{0,B}=\tr(B)/N\).  To bound the remaining coefficients,
write
\begin{equation}
 Z_{3,d}(t)^2=\sum_{\ell\in\{0,2,4,6\}}\alpha_{\ell,d}Z_{\ell,d}(t),
 \qquad
 \alpha_{\ell,d}
 =N_{\ell,d}\int Z_{3,d}(x_0,x)^2Z_{\ell,d}(x_0,x)\,d\sigma_d(x).
 \label{eq:cubic-zonal-square-coefficients}
\end{equation}
Here \(x_0\in\Sph^{d-1}\) is arbitrary and fixed.
For a real symmetric \(N\times N\) matrix \(B\), define
\[
 ({\cal M}_dB)(x)=q_{3,d}(x)^{\mathsf T}Bq_{3,d}(x).
\]
With the Hilbert--Schmidt and \(L^2(\sigma_d)\) inner products,
\[
 {\cal M}_d^*f=\int f(y)q_{3,d}(y)q_{3,d}(y)^{\mathsf T}\,d\sigma_d(y),
 \qquad
 ({\cal M}_d{\cal M}_d^*f)(x)
 =\int Z_{3,d}(x,y)^2f(y)\,d\sigma_d(y).
\]
By \eqref{eq:cubic-general-degree-addition} and
\eqref{eq:cubic-zonal-square-coefficients},
\[
 {\cal M}_d{\cal M}_d^*f=\frac{\alpha_{\ell,d}}{N_{\ell,d}}f,
 \qquad f\in\cH_{\ell,d},\quad \ell\in\{0,2,4,6\}.
\]
For distinct \(\ell,k\in\{0,2,4,6\}\), take
\(f\in\cH_{\ell,d}\) and \(g\in\cH_{k,d}\).  Then
\[
 \langle{\cal M}_d^*f,{\cal M}_d^*g\rangle_{\HS}
 =\langle f,{\cal M}_d{\cal M}_d^*g\rangle_{L^2(\sigma_d)}
 =\frac{\alpha_{k,d}}{N_{k,d}}\langle f,g\rangle_{L^2(\sigma_d)}=0.
\]
Thus the matrix subspaces \({\cal M}_d^*(\cH_{\ell,d})\),
\(\ell\in\{0,2,4,6\}\), are mutually orthogonal in the Hilbert--Schmidt
inner product.  Let \(P_\ell\) be the orthogonal projection onto
\({\cal M}_d^*(\cH_{\ell,d})\).  The definition of \(c_{\ell,B}\) and
the identity for \({\cal M}_d{\cal M}_d^*\) give
\begin{equation*}
 \|c_{\ell,B}\|^2
 =N_{\ell,d}\|{\cal M}_dP_\ell B\|_{L^2(\sigma_d)}^2
 =\alpha_{\ell,d}\|P_\ell B\|_{\HS}^2.
\end{equation*}
To compute \(\alpha_{\ell,d}\), use \(Z_{0,d}=1\) and the explicit
even-degree polynomials
\begin{align*}
 Z_{2,d}(t)&=\frac{dt^2-1}{d-1},\\
 Z_{4,d}(t)&=\frac{(d+2)(d+4)t^4-6(d+2)t^2+3}{(d-1)(d+1)},\\
 Z_{6,d}(t)&=\frac{(d+4)(d+6)(d+8)t^6-15(d+4)(d+6)t^4
                     +45(d+4)t^2-15}{(d-1)(d+1)(d+3)}.
\end{align*}
Substituting these expressions and \eqref{eq:cubic-zonal-polynomial}
into \eqref{eq:cubic-zonal-square-coefficients}, and comparing the
coefficients of \(t^6,t^4,t^2,1\) in that order, gives
\begin{align*}
 \alpha_{0,d}&=\frac1N,
 &\alpha_{2,d}&=\frac{18(d-2)(d+1)}{d(d-1)(d+4)(d+6)},\\
 \alpha_{4,d}&=\frac{9(d-2)(d+1)}{(d-1)(d+4)(d+8)},
 &\alpha_{6,d}&=\frac{(d+3)(d+2)^2(d+1)}
 {(d-1)(d+4)(d+6)(d+8)}.
\end{align*}
Since \(N_{\ell,d}\asymp d^\ell\) for fixed \(\ell\) and \(N\asymp d^3\),
\[
 \max_{\ell\in\{2,4,6\}}\alpha_{\ell,d}\le C,
 \qquad
 \max_{\ell\in\{2,4,6\}}\frac{\alpha_{\ell,d}}{N_{\ell,d}}
 \le\frac{C}{Nd}.
\]
Orthogonality of the \(P_\ell\)'s now gives
\begin{equation*}
 \sum_{\ell\in\{2,4,6\}}
   \frac{\|c_{\ell,B}\|^2}{N_{\ell,d}}
 \le \frac{C}{Nd}\|B\|_{\HS}^2,
 \qquad
 \sum_{\ell\in\{2,4,6\}}\|c_{\ell,B}\|^2
 \le C\|B\|_{\HS}^2.
\end{equation*}
For \(\ell\in\{2,4,6\}\), the matrix Bernstein argument used for
\eqref{eq:cubic-feature-gram} gives, simultaneously with probability
\(1-o(d^{-M})\),
\begin{equation}
 \frac1n\sum_{j=1}^nq_{\ell,d}(X_j)q_{\ell,d}(X_j)^{\mathsf T}
 \preceq C_{A,M}\left(\frac1{N_{\ell,d}}+\frac{\log d}{n}\right)
 I_{N_{\ell,d}}.
 \label{eq:cubic-even-degree-frame}
\end{equation}
Here \(\log N_{\ell,d}=O(\log d)\), and the constants are chosen for
the prescribed \(M\).  Since \(m\ge n/3\),
\[
 \frac1m\sum_{j\in J}
   \langle c_{\ell,B},q_{\ell,d}(X_j)\rangle^2
 \le\frac3n\sum_{j=1}^n
   \langle c_{\ell,B},q_{\ell,d}(X_j)\rangle^2
 \le C_{A,M}\left(\frac1{N_{\ell,d}}+\frac{\log d}{m}\right)
 \|c_{\ell,B}\|^2.
\]
Applying Cauchy--Schwarz to the four harmonic components yields
\begin{align*}
 \frac1m\sum_{j\in J}\{q_{3,d}(X_j)^{\mathsf T}Bq_{3,d}(X_j)\}^2
 &\le4\left(\frac{\tr B}{N}\right)^2
    +\frac4m\sum_{\ell\in\{2,4,6\}}\sum_{j\in J}
       \langle c_{\ell,B},q_{\ell,d}(X_j)\rangle^2\\
 &\le C_{A,M}\left\{
  \left(\frac{\tr B}{N}\right)^2+\frac{\|B\|_{\HS}^2}{Nd}
  +\frac{\log d}{m}\|B\|_{\HS}^2\right\}.
\end{align*}
Since \(0\preceq B\preceq I_N\) implies
\(\|B\|_{\HS}^2\le\tr B\), this proves
\eqref{eq:cubic-hereditary-fourth-profile}.

\emph{Step 3: directional fourth moments.}
For \(i\ne j\), condition on \(X_i\) and apply spherical concentration
\cite[Theorem~3.4.6]{Vershynin2018} to the independent point \(X_j\).
A union bound gives
\[
 \mathbb P\!\left\{\max_{i\ne j}|\langle X_i,X_j\rangle|
       >C_{A,M}\frac{\log d}{\sqrt d}\right\}
 \le 2n^2\exp\{-cC_{A,M}^2\log^2 d\}=o(d^{-M}),
\]
since \(n\le d^A\).  On the event
\eqref{eq:cubic-raw-pair-separation}, the explicit polynomial
\eqref{eq:cubic-zonal-polynomial} and \(N\asymp d^3\) give, for
sufficiently large \(d\),
\begin{equation}
 \max_{i\ne j}|\langle q_{3,d}(X_i),q_{3,d}(X_j)\rangle|
 \le C_{A,M}\frac{\log^3 d}{\sqrt N}\le\frac1{2d}.
 \label{eq:cubic-degree-three-feature-coherence}
\end{equation}

Fix \(v\in\mathbb S^{N-1}\), and write \(q_j=q_{3,d}(X_j)\).
Equation~\eqref{eq:cubic-hereditary-covariance-profile} gives
\begin{equation}
 \frac1m\sum_{j\in J}\langle v,q_j\rangle^2\le \frac C N.
 \label{eq:cubic-directional-second-profile}
\end{equation}
For any \(0<t\le1\), define
\(J_t=\{j\in J:|\langle v,q_j\rangle|\ge t\}\).
Cauchy--Schwarz, \(\|q_j\|=1\), and
\eqref{eq:cubic-degree-three-feature-coherence} give
\[
 t^2|J_t|^2
 \le\left(\sum_{j\in J_t}|\langle v,q_j\rangle|\right)^2
 \le\left\|\sum_{j\in J_t}
       \operatorname{sgn}(\langle v,q_j\rangle)q_j\right\|^2
 \le |J_t|+\sum_{\substack{i,j\in J_t\\i\ne j}}
          |\langle q_i,q_j\rangle|
 \le |J_t|+\frac{|J_t|^2}{2d}.
\]
Consequently, \(|J_t|\le(t^2-1/(2d))^{-1}\) whenever
\(t^2>1/(2d)\).  In particular, \(|J_t|\le2t^{-2}\) for
\(d^{-1/2}\le t\le1\).

For any \(d^{-1/2}\le t\le1\), split the fourth-moment sum over
\(J_t\) and \(J\setminus J_t\).  Since \(|\langle v,q_j\rangle|\le1\),
writing the sum over \(J_t\) as an integral and applying
\(|J_s|\le2s^{-2}\) for \(t\le s\le1\) gives
\[
 \sum_{j\in J_t}\langle v,q_j\rangle^4
 =t^4|J_t|+4\int_t^1s^3|J_s|\,ds
 \le2t^2+8\int_t^1s\,ds\le4.
\]
For the complementary sum, \eqref{eq:cubic-directional-second-profile}
gives
\[
 \frac1m\sum_{j\in J\setminus J_t}\langle v,q_j\rangle^4
 \le\frac{t^2}{m}\sum_{j\in J}\langle v,q_j\rangle^2
 \le\frac{Ct^2}{N}.
\]
Thus \(m^{-1}\sum_{j\in J}\langle v,q_j\rangle^4\le Ct^2/N+4/m\)
for every \(d^{-1/2}\le t\le1\).  Taking \(t=d^{-1/2}\) proves
\eqref{eq:cubic-hereditary-directional-fourth-profile}.

\emph{Step 4: construction of the event.}
Define \({\cal G}_{n,M}^\dagger\) by intersecting \({\cal G}_{n,M}\) with
the events in \eqref{eq:cubic-even-degree-frame} and
\eqref{eq:cubic-raw-pair-separation}.  Every
constituent is invariant under common orthogonal rotations, coordinate
permutations, and individual sign changes.  A union bound, using the
probability estimates from Steps 2--3 and
Lemma~\ref{lem:cubic-hereditary-envelope}, gives
\(\rho_n({\cal G}_{n,M}^\dagger)=1-o(d^{-M})\).
The same construction applies to the independent sample \(X'\).
\end{proof}

With \(M_\star>4A+10\) from
Proposition~\ref{prop:cubic-raw-local-partition}, write
\({\cal G}_n^\dagger={\cal G}_{n,M_\star}^\dagger\).

\subsection{Operator-norm and fourth-order estimates}
\label{subsec:cubic-local-fourth-order}

On \(\mathcal E_n^{\rm reg}\),
Lemma~\ref{lem:cubic-hereditary-square-profile} bounds the empirical
covariances and fourth moments of \(X\) and \(X'\) separately.
We now control \(\|S_n\|_{\rm op}\) and the fourth-order statistic
defined below, both of which involve the two samples.
Combined with the bound
\(W_{n,\beta}\le e^{\beta^2(\|S_n\|_{\HS}^2-n)/2}\) from
\eqref{eq:cubic-local-linear-reward}, their tail probabilities yield
the coarse weighted bound in
Lemma~\ref{lem:cubic-coarse-terminal-deletions} on
\(\bigl|\|S_n\|_{\HS}^2-n\bigr|\le b_\star d^2\log d/\beta^2\).
Corollary~\ref{cor:cubic-central-weighted-deletions} uses sharper cutoffs on
\(\bigl|\|S_n\|_{\HS}^2-n\bigr|\le K_dd^2/\beta^2\)
for the second inequality.

For the remainder of this appendix, fix any positive auxiliary sequence
\((\epsilon_d)\) satisfying
\begin{equation}
 \epsilon_d\longrightarrow0,\qquad
 \frac{\epsilon_d^2\sqrt d}{\log d}\longrightarrow\infty.
 \label{eq:cubic-local-epsilon-conditions}
\end{equation}
Lemmas~\ref{lem:cubic-operator-norm-tail}
and~\ref{lem:cubic-fourth-U-moments} give the probability bounds needed
to impose the operator-norm and fourth-order restrictions at this
threshold.

\begin{lemma}[Operator-norm tail bound]
\label{lem:cubic-operator-norm-tail}
Under \eqref{eq:cubic-local-standing-assumptions} and
\eqref{eq:cubic-local-epsilon-conditions}, for sufficiently large \(d\),
\begin{equation}
 P_0^{\otimes n}\left\{
 \mathcal E_n^{\rm reg},\
 \|S_n\|_{\rm op}>\frac{\epsilon_dd}{\beta^2}\right\}
 \le2\exp\left(-\frac{c\epsilon_d^2d^{5/2}}{\beta^2}\right).
 \label{eq:cubic-operator-norm-tail}
\end{equation}
\end{lemma}

\begin{proof}
Probabilities and expectations over \((X,X')\) are under \(P_0^{\otimes n}\).
Put \(T=\epsilon_dd/\beta^2\) and
\(\mathcal E=\mathcal E_n^{\rm reg}\cap\{\|S_n\|_{\rm op}>T\}\).
Fix \(I=\{1,\ldots,m\}\), \(m=\lfloor n/2\rfloor\), and
\(J=[n]\setminus I\), and write
\(S_I=\sum_{i\in I}w_i\), \(S_J=\sum_{j\in J}w_j\).
Step 1 uses exchanges of the indices \(i\) and \(i+m\) to control
\(\|S_I-(m/n)S_n\|_{\rm op}\), then selects singular vectors
\(u_I,v_I\) of \(S_I\).
Step 2 conditions on \((X_i,X_i')_{i\in I}\) and \((X_j')_{j\in J}\)
and applies Lemma~\ref{lem:cubic-directional-mgf} after truncation
to bound the probability of
\(\{\mathcal E_n^{\rm reg},\ \langle u_I,S_Jv_I\rangle>T/4\}\).
Together these two steps prove \eqref{eq:cubic-operator-norm-tail}.

\emph{Step 1: choosing directions from a fixed group of indices.}
If \(\mathbb P(\mathcal E)=0\), \eqref{eq:cubic-operator-norm-tail}
holds trivially.  We may therefore assume \(\mathbb P(\mathcal E)>0\).
Put \(E_I=S_I-(m/n)S_n\) and \(D_i=(w_i-w_{i+m})/2\) for
\(1\le i\le m\). For both even and odd \(n\),
\[
 E_I=\sum_{i=1}^mD_i+
 \1_{\{n=2m+1\}}\left(\frac{S_n}{2n}-\frac{w_n}{2}\right).
\]
Thus \(\|E_I-\sum_iD_i\|_{\rm op}\le1\), since
\(\|S_n\|_{\rm op}\le n\) and \(\|w_n\|_{\rm op}=1\).
Let \(\varepsilon_1,\ldots,\varepsilon_m\) be independent Rademacher random
variables, independent of \((X,X')\), with
\(\mathbb P(\varepsilon_i=1)=\mathbb P(\varepsilon_i=-1)=1/2\).
Exchanging \((X_i,X_i')\) with \((X_{i+m},X_{i+m}')\) changes
\(D_i\) to \(-D_i\) and preserves both the product law and
\(\mathcal E\). Independent exchanges therefore show that
\(\sum_iD_i\) and \(\sum_i\varepsilon_iD_i\) have the same
conditional law given \(\mathcal E\).

Since \(w_i=q_{3,d}(X_i)q_{3,d}(X_i')^{\mathsf T}\) and
\(\|q_{3,d}(x)\|=1\), we have
\[
 w_iw_i^{\mathsf T}=q_{3,d}(X_i)q_{3,d}(X_i)^{\mathsf T},
 \qquad
 w_i^{\mathsf T}w_i=q_{3,d}(X_i')q_{3,d}(X_i')^{\mathsf T}.
\]
On \(\mathcal E_n^{\rm reg}\), the inequality
\((A-B)(A-B)^{\mathsf T}\preceq2(AA^{\mathsf T}+BB^{\mathsf T})\)
and \eqref{eq:cubic-hereditary-covariance-profile}, applied separately
to \(X\) and \(X'\), give
\[
\begin{aligned}
 \sum_{i=1}^mD_iD_i^{\mathsf T}
 &\preceq\frac12\sum_{i=1}^{2m}q_{3,d}(X_i)q_{3,d}(X_i)^{\mathsf T}
 \preceq\frac{Cn}{N}I_N,\\
 \sum_{i=1}^mD_i^{\mathsf T}D_i
 &\preceq\frac12\sum_{i=1}^{2m}q_{3,d}(X_i')q_{3,d}(X_i')^{\mathsf T}
 \preceq\frac{Cn}{N}I_N.
\end{aligned}
\]
For a fixed sample, the rectangular matrix Rademacher bound
\cite[Corollary~4.2]{Tropp2012} applies to \(\sum_i\varepsilon_iD_i\).
Averaging over the sample conditional on \(\mathcal E\), and using
\(T\ge32\) for sufficiently large \(d\), gives
\begin{equation}
\begin{aligned}
 \mathbb P(\|E_I\|_{\rm op}>T/16\mid\mathcal E)
 &\le\mathbb E\!\left[
   \mathbb P_{\varepsilon}\!\left(
    \left\|\sum_{i=1}^m\varepsilon_iD_i\right\|_{\rm op}>T/32
       \,\middle|\,X,X'\right)
      \,\middle|\,\mathcal E\right]\\
 &\le2N\exp\left(-\frac{cT^2N}{n}\right)=o(1).
\end{aligned}
 \label{eq:cubic-terminal-split-concentration}
\end{equation}
Here \(T^2N/n\asymp\epsilon_d^2d/\beta^2\gg\log N\) by
\eqref{eq:cubic-local-epsilon-conditions}, since \(|\beta|\le1\)
and \(\log N\asymp\log d\).

Choose unit singular vectors \(u_I,v_I\), measurably from the coordinates
in \(I\), such that
\(\langle u_I,S_Iv_I\rangle=\|S_I\|_{\rm op}\).
The identity
\(S_J=(|J|/m)S_I-(n/m)E_I\) gives, on
\(\mathcal E\cap\{\|E_I\|_{\rm op}\le T/16\}\),
\begin{equation}
 \langle u_I,S_Jv_I\rangle
 \ge\frac{|J|}{n}\|S_n\|_{\rm op}
        -\frac{n+|J|}{m}\|E_I\|_{\rm op}
 \ge\frac12\|S_n\|_{\rm op}-4\|E_I\|_{\rm op}>\frac T4.
 \label{eq:cubic-terminal-crossfit-witness}
\end{equation}
We used \(|J|/n\ge1/2\) and \((n+|J|)/m\le4\) for \(n\ge5\).
Equations \eqref{eq:cubic-terminal-split-concentration} and
\eqref{eq:cubic-terminal-crossfit-witness} imply
\[
 (1-o(1))\mathbb P(\mathcal E)
 \le\mathbb P\{\mathcal E_n^{\rm reg},
                 \langle u_I,S_Jv_I\rangle>T/4\}.
\]

\emph{Step 2: truncation and conditional exponential moments.}
For \(j\in J\), let
\(\alpha_j=\langle u_I,q_{3,d}(X_j)\rangle\) and
\(\alpha_j'=\langle v_I,q_{3,d}(X_j')\rangle\).
On \(\mathcal E_n^{\rm reg}\), the covariance and directional
fourth-moment bounds
\eqref{eq:cubic-hereditary-covariance-profile} and
\eqref{eq:cubic-hereditary-directional-fourth-profile} in
Lemma~\ref{lem:cubic-hereditary-square-profile} give
\begin{equation}
 \sum_{j\in J}\{\alpha_j^2+(\alpha_j')^2\}\le\frac{Cd}{\beta^2},
 \qquad
 \sum_{j\in J}\{\alpha_j^4+(\alpha_j')^4\}
 \le C|J|\left(\frac1{N^2}+\frac1{Nd}+\frac1{|J|}\right)
 \le\frac C{\beta^2}.
 \label{eq:cubic-terminal-directional-second}
\end{equation}
Choose \(b_d=C_0/(\epsilon_d\sqrt d)\), with \(C_0\) sufficiently
large and fixed, and put
\(Z=\sum_{j\in J}\alpha_j'\1_{\{|\alpha_j'|\le b_d\}}\alpha_j\).
By Cauchy--Schwarz and \eqref{eq:cubic-terminal-directional-second},
on \(\mathcal E_n^{\rm reg}\),
\begin{equation}
\begin{aligned}
 |\langle u_I,S_Jv_I\rangle-Z|
 &=\left|\sum_{j\in J}\alpha_j\alpha_j'
             \1_{\{|\alpha_j'|>b_d\}}\right|\\
 &\le\left(\sum_{j\in J}\alpha_j^2\right)^{1/2}
     \left(\sum_{j\in J}(\alpha_j')^2
                 \1_{\{|\alpha_j'|>b_d\}}\right)^{1/2}\\
 &\le\left(\sum_{j\in J}\alpha_j^2\right)^{1/2}
     \left(\frac1{b_d^2}\sum_{j\in J}(\alpha_j')^4\right)^{1/2}
 \le\frac{C\sqrt d}{\beta^2b_d}\le\frac T8.
\end{aligned}
 \label{eq:cubic-terminal-truncation-error}
\end{equation}

Now condition only on
\(\mathcal F_I=\sigma((X_i,X_i')_{i\in I},(X_j')_{j\in J})\).
Under this conditioning, \(u_I,v_I,\alpha_j'\) are fixed and the
\(X_j\), \(j\in J\), remain independent with law \(\sigma_d\).
The event
\(\mathcal B_I=\{\sum_{j\in J}(\alpha_j')^2\le Cd/\beta^2\}\)
is \(\mathcal F_I\)-measurable and contains \(\mathcal E_n^{\rm reg}\).
Thus \eqref{eq:cubic-terminal-truncation-error} gives
\[
 \mathcal E_n^{\rm reg}\cap\{\langle u_I,S_Jv_I\rangle>T/4\}
 \subseteq\mathcal E_n^{\rm reg}\cap\{Z>T/8\}
 \subseteq\mathcal B_I\cap\{Z>T/8\}.
\]
Lemma~\ref{lem:cubic-directional-mgf}, with
\(\theta=\lambda\alpha_j'\1_{\{|\alpha_j'|\le b_d\}}\), yields on
\(\mathcal B_I\)
\[
 \log\mathbb E(e^{\lambda Z}\mid\mathcal F_I)
 \le\frac{C\lambda^2}{N}
        \sum_{j\in J}(\alpha_j')^2\1_{\{|\alpha_j'|\le b_d\}}
 \le\frac{C\lambda^2d}{\beta^2N},
 \qquad 0\le\lambda\le\frac{c_0d}{b_d}.
\]
Taking \(\lambda=c_0d/b_d\asymp\epsilon_dd^{3/2}\) and applying the
conditional Chernoff inequality to \(Z\) gives
\begin{equation}
\begin{aligned}
 \mathbb P\{\mathcal E_n^{\rm reg},
                  \langle u_I,S_Jv_I\rangle>T/4\}
 &\le\mathbb E[\1_{\mathcal B_I}
                    \mathbb P(Z>T/8\mid\mathcal F_I)]\\
 &\le\exp\left(-\frac{\lambda T}{8}
                   +\frac{C\lambda^2d}{\beta^2N}\right)
 \le\exp\left(-\frac{c\epsilon_d^2d^{5/2}}{\beta^2}\right).
\end{aligned}
 \label{eq:cubic-terminal-crossfit-tail}
\end{equation}
Indeed, the ratio of the positive term in the exponent to
\(\lambda T\) is \(O(d^{3/2}/N)=o(1)\).
The concentration estimate is under the product law conditional on
\(\mathcal F_I\), not conditional on \(\mathcal E_n^{\rm reg}\).

Combining \eqref{eq:cubic-terminal-split-concentration},
\eqref{eq:cubic-terminal-crossfit-witness}, and
\eqref{eq:cubic-terminal-crossfit-tail} gives
\(\mathbb P(\mathcal E)\le
2\exp(-c\epsilon_d^2d^{5/2}/\beta^2)\) for large \(d\).
This proves \eqref{eq:cubic-operator-norm-tail}.
\end{proof}

We next control the fourth-order contribution needed for the sharp
bound in Lemma~\ref{lem:cubic-centered-simple-edge}.
Grouping the factors of \(W_{n,\beta}\) by their larger vertex index gives
\[
 W_{n,\beta}
 =\prod_{m=1}^{n-1}\prod_{j=1}^{m}
   \left(1+\beta^2\langle w_j,w_{m+1}\rangle_{\HS}\right).
\]
For fixed \(w_1,\ldots,w_m\), let
\((\widetilde X,\widetilde X')\sim P_0\) be independent of the full sample
and recall \(w=q_{3,d}(\widetilde X)q_{3,d}(\widetilde X')^{\mathsf T}\).
Write \(\mathbb E_w\) for expectation over this independent pair.
Expanding the product gives
\[
 \prod_{j\le m}(1+\beta^2\langle w_j,w\rangle_{\HS})
 =\sum_{A\subseteq[m]}\beta^{2|A|}
       \prod_{j\in A}\langle w_j,w\rangle_{\HS}.
\]
Its fourth-order term in \(\beta^2\) is therefore
\[
 \beta^8e_{4,m}(w)
 :=\beta^8\sum_{\substack{A\subseteq[m]\\|A|=4}}
       \prod_{j\in A}\langle w_j,w\rangle_{\HS}.
\]
The quantity \(24\mathbb E_we_{4,m}(w)\) is the contribution of four
distinct indices to the fourth cumulant
\eqref{eq:cubic-complete-fourth-cumulant}.
On the central region \(\|S_n\|_{\HS}^2\le2n\), estimating that cumulant
by the moment bounds in
Lemmas~\ref{lem:cubic-regular-prefix-estimates}
and~\ref{lem:cubic-regular-row-expansion} gives only the error scale
\[
 \frac{\beta^8}{N^4}\sum_{m<n}m^2
 \asymp\frac{\beta^8n^3}{N^4}
 =r_n^3\beta^2\frac{d^3}{N}\asymp\beta^2.
\]
Since this need not tend to zero, we sum
\(\mathbb E_we_{4,m}(w)\) before taking absolute values.

To express this sum as a statistic of the full sample, define, for
\(x_1,\ldots,x_4\in\mathbb S^{d-1}\),
\[
 \gamma_{4,d}(x_1,x_2,x_3,x_4)
 :=\int_{\mathbb S^{d-1}}\prod_{\alpha=1}^4
 Z_{3,d}(x_\alpha,z)\,d\sigma_d(z).
\]
For \(\xi_i=(X_i,X_i')\), independence of
\(\widetilde X\) and \(\widetilde X'\) gives the kernel
\[
 K_{4,d}(\xi_1,\xi_2,\xi_3,\xi_4)
 :=\mathbb E_w\prod_{j=1}^4\langle w_j,w\rangle_{\HS}
 =\gamma_{4,d}(X_1,X_2,X_3,X_4)
   \gamma_{4,d}(X_1',X_2',X_3',X_4').
\]
Set \(T_{4,n}=\sum_{i<j<k<\ell}
K_{4,d}(\xi_i,\xi_j,\xi_k,\xi_\ell)\).
Lemma~\ref{lem:cubic-fourth-U-moments} bounds
\(P_0^{\otimes n}\{\mathcal E_n^{\rm reg},\
\beta^8n|T_{4,n}|>\epsilon_d\}\) using moment estimates for a
bounded-kernel statistic that agrees with \(T_{4,n}\) on
\(\mathcal E_n^{\rm reg}\).

\begin{lemma}[Moments and tail bound for the fourth-order statistic]
\label{lem:cubic-fourth-U-moments}
Under \eqref{eq:cubic-local-standing-assumptions}, there is a symmetric
kernel \(\widetilde K_{4,d}\) with
\(\|\widetilde K_{4,d}\|_\infty\le C_A\log^{12}d/N^6\).
It is completely degenerate under \(P_0\), meaning that
\[
 \int\widetilde K_{4,d}(\xi_1,\ldots,\xi_4)\,dP_0(\xi_i)=0,
 \qquad i=1,\ldots,4,
\]
for each fixed choice of the other three arguments.
For \(Z_{4,n}=\sum_{i<j<k<\ell}
\widetilde K_{4,d}(\xi_i,\xi_j,\xi_k,\xi_\ell)\),
\begin{align}
 T_{4,n}&=Z_{4,n}\quad\text{on }\mathcal E_n^{\rm reg},
 \label{eq:cubic-fourth-truncated-identity}\\
 \|Z_{4,n}\|_{L^p(P_0^{\otimes n})}
 &\le C_Ap^2n^2\frac{\log^{12}d}{N^6},\qquad p\ge2.
 \label{eq:cubic-fourth-U-moment}
\end{align}
If \(\epsilon_d\) satisfies \eqref{eq:cubic-local-epsilon-conditions},
then, for sufficiently large \(d\),
\begin{equation}
 P_0^{\otimes n}\{\mathcal E_n^{\rm reg},\
                  \beta^8n|T_{4,n}|>\epsilon_d\}
 \le e^{-d^2\log d}.
 \label{eq:cubic-fourth-U-tail}
\end{equation}
\end{lemma}

\begin{proof}
Step 1 bounds \(\gamma_{4,d}\) when all six pairwise inner products
are small.  Step 2 uses this bound to construct a uniformly bounded,
symmetric, completely degenerate kernel \(\widetilde K_{4,d}\).
It also proves \(T_{4,n}=Z_{4,n}\) on \(\mathcal E_n^{\rm reg}\)
and oddness of \(\widetilde K_{4,d}\) in each \(X_i\)-coordinate.
Step 3 uses the uniform bound and oddness to estimate
\(\|Z_{4,n}\|_p\), then takes \(p=d^2\)
to prove \eqref{eq:cubic-fourth-U-tail}.

\emph{Step 1: a bound for \(\gamma_{4,d}\) when the inner products are small.}
Fix \(x_1,\ldots,x_4\in\mathbb S^{d-1}\) satisfying
\[
 \max_{\alpha<\beta}|\langle x_\alpha,x_\beta\rangle|
 \le C_A\frac{\log d}{\sqrt d}.
\]
Using \eqref{eq:cubic-zonal-polynomial} in the definition of
\(\gamma_{4,d}\), with \(X\sim\sigma_d\) and
\(a_{3,d}=(d+2)/(d-1)\), gives
\[
\begin{aligned}
 \gamma_{4,d}(x_1,x_2,x_3,x_4)
 &=a_{3,d}^4\mathbb E_{\sigma_d}\prod_{\alpha=1}^4
   \left(\langle X,x_\alpha\rangle^3
              -\frac3{d+2}\langle X,x_\alpha\rangle\right)\\
 &=a_{3,d}^4\sum_{L\subseteq[4]}
 \left(-\frac3{d+2}\right)^{|L|}
 \mathbb E_{\sigma_d}\!\left[
   \prod_{\alpha\in L}\langle X,x_\alpha\rangle
   \prod_{\alpha\notin L}\langle X,x_\alpha\rangle^3\right].
\end{aligned}
\]
Here \(L\) records the indices at which the linear term is selected.
Each expectation contains \(|L|+3(4-|L|)=2(6-|L|)\) linear factors.
We evaluate these expectations using the spherical moment identity
\begin{equation}
 \mathbb E_{\sigma_d}\prod_{a=1}^{2k}\langle X,v_a\rangle
 =\frac1{\prod_{i=0}^{k-1}(d+2i)}\sum_{\pi\in\mathfrak P_2(2k)}
       \prod_{\{a,b\}\in\pi}\langle v_a,v_b\rangle,
 \label{eq:cubic-spherical-pairing}
\end{equation}
where \(v_1,\ldots,v_{2k}\in\mathbb R^d\) and
\(\mathfrak P_2(2k)\) is the set of partitions of
\(\{1,\ldots,2k\}\) into \(k\) disjoint unordered pairs.
To derive \eqref{eq:cubic-spherical-pairing}, let
\(g\sim\mathcal N(0,I_d)\).  Since \(g/\|g\|\sim\sigma_d\)
is independent of \(\|g\|\), Isserlis' formula \cite{Isserlis1918} gives
\[
 \mathbb E\|g\|^{2k}\,
     \mathbb E_{\sigma_d}\prod_{a=1}^{2k}\langle X,v_a\rangle
 =\mathbb E\prod_{a=1}^{2k}\langle g,v_a\rangle
 =\sum_{\pi\in\mathfrak P_2(2k)}
       \prod_{\{a,b\}\in\pi}\langle v_a,v_b\rangle,
\]
where \(\mathbb E[\langle g,v_a\rangle\langle g,v_b\rangle]
=\langle v_a,v_b\rangle\).
Dividing by \(\mathbb E\|g\|^{2k}=\prod_{i=0}^{k-1}(d+2i)\)
proves \eqref{eq:cubic-spherical-pairing}.

For each \(L\) in the expansion of \(\gamma_{4,d}\), the list
\(v_1,\ldots,v_{2k}\) contains one copy of \(x_\alpha\) if
\(\alpha\in L\) and three copies otherwise, so \(k=6-|L|\).
We represent the pairings in \eqref{eq:cubic-spherical-pairing} by
graphs on vertices \(1,\ldots,4\), with vertex \(\alpha\)
corresponding to \(x_\alpha\).
Pairing a factor \(\langle X,x_\alpha\rangle\) with a factor
\(\langle X,x_\beta\rangle\), where \(\alpha\ne\beta\), gives
\(\langle x_\alpha,x_\beta\rangle\) and an edge joining the two vertices.
Pairing two factors from the same vertex gives
\(\langle x_\alpha,x_\alpha\rangle=1\), so no edge is retained for
this pair.  The degree \(r_\alpha\) of vertex \(\alpha\) is therefore
\[
 r_\alpha=
 \begin{cases}
  1,&\alpha\in L,\\
  3-2=1,&\alpha\notin L\text{ and two of its factors are paired together},\\
  3,&\alpha\notin L\text{ and none of its factors are paired together}.
 \end{cases}
\]
The pairs joining distinct vertices define a loopless multigraph,
which may have multiple edges between two vertices but no edge from
a vertex to itself.
For \(0\le\ell\le4\), let \(\mathfrak M_\ell\) be the set
of these graphs on \(\{1,2,3,4\}\) with exactly \(\ell\) vertices of
degree one and all other vertices of degree three.
For \(G\in\mathfrak M_\ell\), let
\(m_{\alpha\beta}(G)=m_{\beta\alpha}(G)\) be the number of edges
joining distinct vertices \(\alpha\) and \(\beta\), counting parallel
edges separately.  Then
\[
 r_\alpha=\sum_{\beta\ne\alpha}m_{\alpha\beta}(G),\qquad
 \sum_{\alpha<\beta}m_{\alpha\beta}(G)
 =\frac12\sum_{\alpha=1}^4r_\alpha=6-\ell.
\]
To count the pairings giving \(G\), distinguish the \(r_\alpha\)
remaining factor positions at each vertex and temporarily label
the \(m_{\alpha\beta}(G)\) parallel edges between each pair of vertices.
At vertex \(\alpha\), assign its factor positions bijectively to
the \(r_\alpha\) incident edge ends in \(r_\alpha!\) ways.
The two positions assigned to the ends of each edge are then paired,
so there are \(\prod_\alpha r_\alpha!\) choices with labeled edges.
Each pairing of factor positions is counted
\(\prod_{\alpha<\beta}m_{\alpha\beta}(G)!\) times, since permuting
the labels of parallel edges leaves the paired positions unchanged.
Dividing by this multiplicity gives
\[
 \kappa_G
 =\frac{\prod_{\alpha=1}^4r_\alpha!}
       {\prod_{\alpha<\beta}m_{\alpha\beta}(G)!}
 =\frac{6^{4-\ell}}{\prod_{\alpha<\beta}m_{\alpha\beta}(G)!}
 \le1296.
\]

Fix \(G\in\mathfrak M_\ell\), and let \(j\) of its degree-one vertices
come from cubic terms with a pair at the same vertex.
Then \(|L|=\ell-j\), and the expectation has
\(|L|+3(4-|L|)=2(6-\ell+j)\) factors, so
\eqref{eq:cubic-spherical-pairing} is used with \(k=6-\ell+j\).
For this fixed \(G\) and \(j\), collect all choices of \(L\) and
pairings in \eqref{eq:cubic-spherical-pairing} whose pairs between
distinct vertices give \(G\) and which have exactly \(j\) pairs
within vertices.
There are \(\binom\ell j\) choices of the vertices with an internal
pair, \(3^j\) choices of those pairs, and \(\kappa_G\) choices of
the pairs between vertices.  Summing their contributions gives
\begin{equation}
\begin{aligned}
 \gamma_{4,d}(x_1,\ldots,x_4)
 &=a_{3,d}^4\sum_{\ell=0}^4\sum_{G\in\mathfrak M_\ell}\sum_{j=0}^\ell
   \frac{\binom\ell j3^j\kappa_G(-3/(d+2))^{\ell-j}}
        {\prod_{i=0}^{5-\ell+j}(d+2i)}
   \prod_{\alpha<\beta}
   \langle x_\alpha,x_\beta\rangle^{m_{\alpha\beta}(G)}\\
 &=\sum_{\ell=0}^4 \omega_{\ell,d}
   \sum_{G\in\mathfrak M_\ell}\kappa_G
   \prod_{\alpha<\beta}
   \langle x_\alpha,x_\beta\rangle^{m_{\alpha\beta}(G)},
\end{aligned}
 \label{eq:cubic-fourth-diagram}
\end{equation}
where
\[
 \omega_{\ell,d}
 :=3^\ell a_{3,d}^4\sum_{j=0}^\ell\binom\ell j
       \frac{\{-1/(d+2)\}^{\ell-j}}{\prod_{i=0}^{5-\ell+j}(d+2i)}.
\]
Evaluating the sum over \(j\) gives
\[
\begin{gathered}
 \omega_{0,d}=\frac{(d+2)^3}
 {d(d-1)^4(d+4)(d+6)(d+8)(d+10)},\\
 \omega_{1,d}=-\frac{24}{d+2}\,\omega_{0,d},\qquad
 \omega_{2,d}=-\frac{18(d-22)}{(d+2)^2}\,\omega_{0,d},\\
 \omega_{3,d}=\frac{216(5d-14)}{(d+2)^3}\,\omega_{0,d},\qquad
 \omega_{4,d}=\frac{324(3d^2-80d-76)}{(d+2)^4}\,\omega_{0,d}.
\end{gathered}
\]
In particular, \(|\omega_{\ell,d}|\le Cd^{-6-\lceil\ell/2\rceil}\).
Since \(\sum_{\alpha<\beta}m_{\alpha\beta}(G)=6-\ell\)
and \(\kappa_G\le1296\), the assumed bound on the pairwise inner
products gives
\[
 |\omega_{\ell,d}|\kappa_G
 \prod_{\alpha<\beta}|\langle x_\alpha,x_\beta\rangle|^{m_{\alpha\beta}(G)}
 \le C_A d^{-6-\lceil\ell/2\rceil}
       \left(\frac{\log d}{\sqrt d}\right)^{6-\ell}
 \le C_A d^{-9}\log^6d.
\]
There are only finitely many graphs, independently of \(d\), so
\(N\asymp d^3\) gives
\begin{equation}
 |\gamma_{4,d}(x_1,x_2,x_3,x_4)|
 \le C_A\frac{\log^6 d}{N^3}.
 \label{eq:cubic-fourth-local-kernel}
\end{equation}

\emph{Step 2: construction of a bounded, completely degenerate kernel.}
Let \(\chi(\xi_1,\ldots,\xi_4)\) be the indicator of
\[
 \max_{\alpha<\beta}
 \max\{|\langle X_\alpha,X_\beta\rangle|,
        |\langle X_\alpha',X_\beta'\rangle|\}
 \le C_{A,M_\star}\frac{\log d}{\sqrt d},
\]
with the constant from \eqref{eq:cubic-raw-pair-separation}, and put
\(\widetilde K_{4,d}=K_{4,d}\chi\).
On \(\{\chi=1\}\), equation~\eqref{eq:cubic-fourth-local-kernel}
applies to both \((X_1,\ldots,X_4)\) and
\((X_1',\ldots,X_4')\), whereas \(\widetilde K_{4,d}=0\) on
\(\{\chi=0\}\).  Hence, for every choice of the four paired arguments,
\[
 |\widetilde K_{4,d}(\xi_1,\ldots,\xi_4)|
 =\chi\,|\gamma_{4,d}(X_1,\ldots,X_4)|
          |\gamma_{4,d}(X_1',\ldots,X_4')|
 \le C_A\frac{\log^{12}d}{N^6}.
\]
Both \(K_{4,d}\) and \(\chi\) are symmetric in their four paired
arguments, so \(\widetilde K_{4,d}\) is symmetric as well.
On \(\mathcal E_n^{\rm reg}\),
\eqref{eq:cubic-raw-pair-separation} makes \(\chi=1\) for every
four-subset, giving
\[
 Z_{4,n}
 =\sum_{i<j<k<\ell}\widetilde K_{4,d}(\xi_i,\xi_j,\xi_k,\xi_\ell)
 =\sum_{i<j<k<\ell}K_{4,d}(\xi_i,\xi_j,\xi_k,\xi_\ell)
 =T_{4,n}.
\]
This proves \eqref{eq:cubic-fourth-truncated-identity}.
For \(\xi_i^-=(-X_i,X_i')\), oddness of \(Z_{3,d}\) and invariance
of \(\chi\) under this sign change imply
\[
 \widetilde K_{4,d}(\xi_1,\ldots,\xi_i^-,\ldots,\xi_4)
 =-\widetilde K_{4,d}(\xi_1,\ldots,\xi_i,\ldots,\xi_4).
\]
Since \(P_0\) is also invariant under \(\xi_i\mapsto\xi_i^-\),
\[
 \int\widetilde K_{4,d}(\xi_1,\ldots,\xi_4)\,dP_0(\xi_i)
 =-\int\widetilde K_{4,d}(\xi_1,\ldots,\xi_4)\,dP_0(\xi_i)=0
\]
for each \(i\), with the other arguments fixed.
Thus \(\widetilde K_{4,d}\) is completely degenerate under \(P_0\).

\emph{Step 3: moment and probability tail bounds.}
All probabilities and norms in this step are under \(P_0^{\otimes n}\),
with additional independent signs when indicated.
We first prove \eqref{eq:cubic-fourth-U-moment} for \(p\ge2\).
Let \(\varepsilon_1,\ldots,\varepsilon_n\) be independent Rademacher
variables, independent of the paired sample.
The pairs \((\varepsilon_iX_i,X_i')\) still have joint law
\(P_0^{\otimes n}\).  By the sign identity from Step 2,
\[
 Z_{4,n}\ \stackrel{\rm law}{=}\
 \sum_{i<j<k<\ell}\varepsilon_i\varepsilon_j\varepsilon_k\varepsilon_\ell
       \widetilde K_{4,d}(\xi_i,\xi_j,\xi_k,\xi_\ell).
\]
Conditional on all pairs \(\xi_i=(X_i,X_i')\), \(1\le i\le n\),
the kernel values are fixed coefficients and the \(\varepsilon_i\)
remain independent Rademacher variables.
Write \(\mathbb E_\varepsilon\) and \(\|\cdot\|_{L^p(\varepsilon)}\)
for expectation and norm over these signs, with \(\xi_1,\ldots,\xi_n\)
held fixed.  For four-subsets \(A,B\subseteq[n]\), independence,
\(\mathbb E_\varepsilon\varepsilon_i=0\), and \(\varepsilon_i^2=1\) give
\(\mathbb E_\varepsilon[(\prod_{i\in A}\varepsilon_i)
(\prod_{j\in B}\varepsilon_j)]=\1_{\{A=B\}}\).
Thus the squared \(L^2(\varepsilon)\)-norm of the degree-four polynomial
is the sum of its squared coefficients.  The hypercontractive inequality
\cite[Theorem~9.21]{ODonnell2014} therefore gives
\[
\begin{aligned}
 \left\|\sum_{i<j<k<\ell}
 \varepsilon_i\varepsilon_j\varepsilon_k\varepsilon_\ell
 \widetilde K_{4,d}(\xi_i,\xi_j,\xi_k,\xi_\ell)\right\|_{L^p(\varepsilon)}
 &\le(p-1)^2
 \left(\sum_{i<j<k<\ell}|\widetilde K_{4,d}(\xi_i,\xi_j,\xi_k,\xi_\ell)|^2
       \right)^{1/2}\\
 &\le C_Ap^2n^2\frac{\log^{12}d}{N^6}.
\end{aligned}
\]
Taking the \(L^p\)-norm over the paired sample proves
\eqref{eq:cubic-fourth-U-moment}, with a constant depending only on \(A\).

Choose \(p=d^2\).  Condition
\eqref{eq:cubic-local-epsilon-conditions} implies
\(\epsilon_d\ge d^{-1/4}\sqrt{\log d}\) eventually.
Since \(n\beta^2\asymp Nd\), \(|\beta|\le1\), and \(N\asymp d^3\),
\[
 \frac{\beta^8n\|Z_{4,n}\|_p}{\epsilon_d}
 \le C_A\frac{p^2\beta^8n^3\log^{12}d}{\epsilon_dN^6}
 \le C_A\frac{\log^{12}d}{\epsilon_dd^2}
 \le C_A\frac{\log^{23/2}d}{d^{7/4}}
 \le d^{-1}
\]
for sufficiently large \(d\).  Therefore
\eqref{eq:cubic-fourth-truncated-identity} and Markov's inequality give
\[
 \mathbb P\{\mathcal E_n^{\rm reg},\ \beta^8n|T_{4,n}|>\epsilon_d\}
 \le\mathbb P\{\beta^8n|Z_{4,n}|>\epsilon_d\}
 \le\left(\frac{\beta^8n\|Z_{4,n}\|_p}{\epsilon_d}\right)^p
 \le e^{-d^2\log d},
\]
which proves \eqref{eq:cubic-fourth-U-tail}.
\end{proof}

\subsection{Conditional entropy and row expansions}
\label{subsec:cubic-local-row-tools}

Under \eqref{eq:cubic-local-standing-assumptions}, let
\(\mathcal E\subseteq\mathcal E_n^{\rm reg}\) be invariant under
permutations of the paired sample.
We reduce \(\log\mathbb E_{P_0^{\otimes n}}[\1_{\mathcal E}W_{n,\beta}]\)
to conditional exponential moments, then expand them through fourth
order.  Section~\ref{subsec:cubic-local-summed-estimates} sums the
expansions using the sample bounds from
Section~\ref{subsec:cubic-local-fourth-order}.

On \(\mathcal G_n^\dagger(X)\cap\mathcal G_n^\dagger(X')\),
\eqref{eq:cubic-pair-separation} and Taylor's formula give, uniformly,
\begin{equation}
  \log W_{n,\beta}
  =F_n^{(2)}(\beta)+o(1),\qquad
  F_n^{(2)}(\beta)
  :=\beta^2E_n-\frac{\beta^4}{2}\sum_{i<j}h_{ij}^2,
  \label{eq:cubic-restricted-quadratic-energy}
\end{equation}
because
\[
  \beta^6\sum_{i<j}|h_{ij}|^3
  \le C\frac{\beta^6n^2\log^9 d}{N^3}
  =O\!\left(\frac{\beta^2d^2\log^9 d}N\right)=o(1).
\]
The uniform error and the Gibbs formula
\eqref{eq:cubic-restricted-variational} give
\begin{equation}
\begin{aligned}
 \log\mathbb E_{P_0^{\otimes n}}[\1_{\mathcal E}W_{n,\beta}]
 &\le\log\mathbb E_{P_0^{\otimes n}}
       [\1_{\mathcal E}e^{F_n^{(2)}(\beta)}]+o(1)\\
 &=\sup_{Q:Q(\mathcal E)=1}
 \{\mathbb E_QF_n^{(2)}(\beta)-D(Q\Vert P_0^{\otimes n})\}+o(1).
\end{aligned}
 \label{eq:cubic-shared-restricted-gibbs}
\end{equation}
The supremum is over laws \(Q\) of the paired sample with
\(Q(\mathcal E)=1\) and \(D(Q\Vert P_0^{\otimes n})<\infty\).
Averaging \(Q\) over permutations of the pairs preserves
\(Q(\mathcal E)=1\) and \(\mathbb E_QF_n^{(2)}(\beta)\), and cannot
increase \(D(Q\Vert P_0^{\otimes n})\).  It therefore suffices to
consider exchangeable \(Q\).

Recall the functions \(f,g\) in the proof of
Lemma~\ref{lem:cubic-compensated-row}.  With the first \(m\) features
fixed, specialize them to \(r=m\), \(v_j=w_j\), and \(S=S_m\), and write
\[
 f_m(w)=\langle S_m,w\rangle_{\HS},\qquad
 g_m(w)=\sum_{j\le m}\langle w_j,w\rangle_{\HS}^2,\qquad
 Z_m(u,c)=\mathbb E_w e^{uf_m(w)-cu^2g_m(w)/2}.
\]
Here \(\mathbb E_w\) integrates an independent pair with law \(P_0\).
The row function used in \eqref{eq:cubic-main-conditional-entropy}
is \(H_{m+1}(w)=\beta^2f_m(w)-\beta^4g_m(w)/2\), and
\(F_n^{(2)}(\beta)=\sum_{m<n}H_{m+1}(w_{m+1})\).
Let \(Q_{m+1}\) be the conditional law under \(Q\) of
\((X_{m+1},X_{m+1}')\) given the first \(m\) pairs.

For \(0\le\eta<1\), apply \eqref{eq:cubic-linear-entropy-duality}
conditionally to \(Q_{m+1}\) and \(H_{m+1}/(1-\eta)\).
Multiplication by \(1-\eta\) and expectation under \(Q\) give
\begin{equation}
 \mathbb E_QH_{m+1}(w_{m+1})
 -(1-\eta)\mathbb E_QD(Q_{m+1}\Vert P_0)
 \le(1-\eta)\mathbb E_Q\log Z_m(\beta^2/(1-\eta),1-\eta).
 \label{eq:cubic-conditional-entropy-row}
\end{equation}
The chain rule gives
\(\sum_{m<n}\mathbb E_QD(Q_{m+1}\Vert P_0)=D(Q\Vert P_0^{\otimes n})\).
Summing \eqref{eq:cubic-conditional-entropy-row} therefore yields
\begin{equation}
 \mathbb E_QF_n^{(2)}(\beta)-D(Q\Vert P_0^{\otimes n})
 \le (1-\eta)\sum_{m<n}\mathbb E_Q\log Z_m(\beta^2/(1-\eta),1-\eta)
      -\eta D(Q\Vert P_0^{\otimes n}).
 \label{eq:cubic-shared-entropy-sum}
\end{equation}

The expansion requires moment bounds for \(\|S_m\|_{\HS}\) and a tail
bound for \(\ell(S_m)\).
To use exchangeability, let \(\Pi\) be an independent uniform permutation
of \([n]\), and write \(\mathbb P_{\rm ord},\mathbb E_{\rm ord}\) for its
probability and expectation with the sample fixed.
For \(I\subseteq[n]\) with \(|I|=m\),
\(\mathbb P_{\rm ord}\{\{\Pi(1),\ldots,\Pi(m)\}=I\}
=m!(n-m)!/n!=\binom nm^{-1}\).
Thus each prefix is a uniform sample without replacement, and
exchangeability gives, for every bounded measurable \(F\),
\[
 \mathbb E_QF(w_1,\ldots,w_m)
 =\mathbb E_Q\mathbb E_{\rm ord}
    F(w_{\Pi(1)},\ldots,w_{\Pi(m)}).
\]
Inside \(\mathbb P_{\rm ord}\) and \(\mathbb E_{\rm ord}\),
\(S_m\) denotes the sum \(\sum_{i=1}^m w_{\Pi(i)}\) of the first \(m\)
features in this order.  Lemma~\ref{lem:cubic-regular-prefix-estimates}
gives moment and tail bounds for these sums in terms of \(S_n\).

\begin{lemma}[Moment and tail bounds for random prefixes]
\label{lem:cubic-regular-prefix-estimates}
Assume \eqref{eq:cubic-local-standing-assumptions},
and fix a full sample in
\(\mathcal G_n^\dagger(X)\cap\mathcal G_n^\dagger(X')\).
Probabilities and expectations with subscript \({\rm ord}\) are over
the independent uniform reveal order.  For fixed \(0<\epsilon<c_1\),
uniformly in \(m\le n\),
\begin{equation}
 \1_{\{\beta^2\ell(S_n)\le\epsilon d\}}\,
 \mathbb P_{\rm ord}\{\beta^2\ell(S_m)>c_1d\}
 \le \exp\{-c_{\epsilon,c_1}d/\beta^2\}.
 \label{eq:cubic-orbit-source-bridge-tail}
\end{equation}
If \(\|S_n\|_{\HS}^2\le(1+R)n\), where \(R\ge1\), then
\begin{equation}
 \mathbb E_{\rm ord}\|S_m\|_{\HS}^{2k}\le C_kR^km^k,
 \qquad k\in\{2,3\}.
 \label{eq:cubic-parameterized-Hilbert-bridge}
\end{equation}
\end{lemma}

\begin{proof}
We first prove \eqref{eq:cubic-orbit-source-bridge-tail}, then bound
the fourth and sixth moments.  Using the matrix representation
\(w_i=q_{3,d}(X_i)q_{3,d}(X_i')^{\mathsf T}\), set
\(B_i=w_i-S_n/n\).
Then \(\sum_iB_i=0\), \(\|B_i\|_{\rm op}\le\|B_i\|_{\HS}\le2\),
and \eqref{eq:cubic-hereditary-covariance-profile} gives
\[
 \sum_iB_iB_i^{\mathsf T}
 =\sum_iw_iw_i^{\mathsf T}-\frac1nS_nS_n^{\mathsf T}
 \preceq\sum_iq_{3,d}(X_i)q_{3,d}(X_i)^{\mathsf T}
 \preceq\frac{Cn}{N}I_N.
\]
The analogous identity for \(\sum_iB_i^{\mathsf T}B_i\) uses the
\(X'\)-sample.  In particular,
\begin{equation}
  \max\left\{
  \left\|\sum_iB_iB_i^{\mathsf T}\right\|_{\rm op},
  \left\|\sum_iB_i^{\mathsf T}B_i\right\|_{\rm op}
  \right\}
  \le \frac{Cn}{N}.
  \label{eq:cubic-matrix-bridge-variance}
\end{equation}
Let \(I_m=\{\Pi(1),\ldots,\Pi(m)\}\), and let
\(\widetilde B_1,\ldots,\widetilde B_m\) be independent uniform draws
from \(B_1,\ldots,B_n\).  The matrix Bernstein inequality in
\cite[Theorem~6.1]{Tropp2012} applies to self-adjoint matrices.
To use it for the possibly nonsymmetric \(B_i\), form their
self-adjoint dilations \cite[Section~2.6]{Tropp2012}
\[
 \mathcal D(B)=\begin{pmatrix}0&B\\B^{\mathsf T}&0\end{pmatrix},
 \qquad
 \mathcal D(B)^2=
 \begin{pmatrix}BB^{\mathsf T}&0\\0&B^{\mathsf T}B\end{pmatrix}.
\]
Since \(\sum_iB_i=0\) and
\(\|\mathcal D(B_i)\|_{\rm op}=\|B_i\|_{\rm op}\le2\), the
independent matrices \(\mathcal D(\widetilde B_j)\) are centered,
with norm at most two.
The comparison of expected trace exponentials for sampling without
replacement \cite[Section~D]{GrossNesme2010}, together with the
matrix Bernstein bound \cite[Theorem~6.1 and Remark~6.3]{Tropp2012},
therefore gives, for \(t>0\),
\begin{equation}
 \begin{aligned}
  \mathbb P_{\rm ord}\left\{
    \left\|\sum_{i\in I_m}B_i\right\|_{\rm op}>t\right\}
  &=\mathbb P_{\rm ord}\left\{
     \lambda_{\max}\left(\sum_{i\in I_m}\mathcal D(B_i)\right)>t\right\}\\
  &\le2N\exp\left(-\frac{ct^2}
      {\left\|\sum_{j=1}^m\mathbb E\mathcal D(\widetilde B_j)^2\right\|_{\rm op}+t}\right)\\
  &=2N\exp\left(-\frac{ct^2}
      {\frac mn\left\|\sum_{i=1}^n\mathcal D(B_i)^2\right\|_{\rm op}+t}\right)
  \le2N\exp\left(-\frac{ct^2}{Cm/N+t}\right),
 \end{aligned}
  \label{eq:cubic-centered-matrix-bridge-tail}
\end{equation}
where the last inequality uses \eqref{eq:cubic-matrix-bridge-variance}.
Equation~\eqref{eq:cubic-centered-matrix-bridge-tail} also controls
\(\ell(S_m)\).  Indeed, \(\ell(H)\le\|H\|_{\rm op}\) and
\[
 \ell(S_m)
 \le\left\|\sum_{i\in I_m}B_i\right\|_{\rm op}
    +\frac mn\ell(S_n).
\]
On \(\{\beta^2\ell(S_n)\le\epsilon d\}\), take
\(t=(c_1-\epsilon)d/\beta^2\) in
\eqref{eq:cubic-centered-matrix-bridge-tail}.  Since
\(m/N\le n/N=O(d/\beta^2)\), we obtain
\[
 \mathbb P_{\rm ord}\{\beta^2\ell(S_m)>c_1d\}
 \le2N\exp\left(-\frac{c(c_1-\epsilon)^2}{C+c_1-\epsilon}
                         \frac d{\beta^2}\right)
 \le e^{-c_{\epsilon,c_1}d/\beta^2},
\]
proving \eqref{eq:cubic-orbit-source-bridge-tail}.

For the moment bounds, assume \(m\ge1\), since \(m=0\) is immediate.
Fix \(k\in\{2,3\}\).
Independence and centering imply that any term in which an index occurs
exactly once has zero expectation.  Thus a term with nonzero expectation
contains \(r\le k\) distinct indices.  There are \(\binom mr\) choices
of these indices and at most \(r^{2k}\) assignments to the \(2k\)
positions.  Cauchy--Schwarz, \(\|\widetilde B_j\|_{\HS}\le2\), and
\(\binom mr\le m^r/r!\) therefore give
\begin{align*}
 \mathbb E\Big\|\sum_{j=1}^m\widetilde B_j\Big\|_{\HS}^{2k}
 &=\sum_{i_1,\ldots,i_{2k}=1}^m
     \mathbb E\prod_{a=1}^k
        \langle\widetilde B_{i_{2a-1}},\widetilde B_{i_{2a}}\rangle_{\HS}\\
 &\le\sum_{\substack{1\le i_1,\ldots,i_{2k}\le m\\
                      |\{i_1,\ldots,i_{2k}\}|\le k}}
      \mathbb E\prod_{b=1}^{2k}\|\widetilde B_{i_b}\|_{\HS}
 \le2^{2k}\sum_{r=1}^k\binom mr r^{2k}\\
 &\le2^{2k}\sum_{r=1}^k\frac{r^{2k}}{r!}m^r
 \le2^{2k}m^k\sum_{r=1}^k\frac{r^{2k}}{r!}\le C_km^k.
\end{align*}
Applying Hoeffding's comparison for sampling without replacement
\cite[Theorem~4]{Hoeffding1963}, in its vector form
\cite[Section~D]{GrossNesme2010}, to the convex function
\(H\mapsto\|H\|_{\HS}^{2k}\) gives
\[
 \mathbb E_{\rm ord}\left\|\sum_{i\in I_m}B_i\right\|_{\HS}^{2k}
 \le\mathbb E\left\|\sum_{j=1}^m\widetilde B_j\right\|_{\HS}^{2k}
 \le C_km^k.
\]
The mean term \(\mathbb E_{\rm ord}S_m=(m/n)S_n\) satisfies, by
\(\|S_n\|_{\HS}^2\le(1+R)n\le2Rn\) and \(m\le n\),
\[
 \|(m/n)S_n\|_{\HS}^{2k}
 \le2^k(m/n)^{2k}(Rn)^k\le2^kR^km^k,\qquad k=2,3.
\]
Using \(S_m=\sum_{i\in I_m}B_i+(m/n)S_n\), the triangle inequality
and convexity of \(x\mapsto x^{2k}\) on \([0,\infty)\) give
\[
 \mathbb E_{\rm ord}\|S_m\|_{\HS}^{2k}
 \le2^{2k-1}\left(
   \mathbb E_{\rm ord}\left\|\sum_{i\in I_m}B_i\right\|_{\HS}^{2k}
   +\|(m/n)S_n\|_{\HS}^{2k}\right)
 \le C_kR^km^k.
\]
This proves \eqref{eq:cubic-parameterized-Hilbert-bridge}.
\end{proof}

On \(\{\|S_n\|_{\rm op}\le\epsilon_dd/\beta^2\}\), the condition
\(\epsilon_d\to0\) in \eqref{eq:cubic-local-epsilon-conditions} gives,
for fixed \(0<\epsilon<c_1\) and sufficiently large \(d\),
\[
 \beta^2\ell(S_n)\le\beta^2\|S_n\|_{\rm op}
 \le\epsilon_dd\le\epsilon d.
\]
This verifies the condition \(\beta^2\ell(S_n)\le\epsilon d\) needed to
apply \eqref{eq:cubic-orbit-source-bridge-tail} in
Lemma~\ref{lem:cubic-regular-prefix-estimates}.

To use the moment estimate \eqref{eq:cubic-parameterized-Hilbert-bridge}
for \(\log Z_m(u,c)\), we next establish a fourth-order expansion whose
remainder is bounded by
\(C|u|^6(\|S_m\|_{\HS}^6+m^3)/N^6\).
The sixth-moment bound in Lemma~\ref{lem:cubic-regular-prefix-estimates}
then controls its average over permutations.
Lemma~\ref{lem:cubic-sixth-row-remainder} proves this expansion for fixed
features under explicit exponential-moment and smallness conditions.

\begin{lemma}[Fourth-order expansion with a sixth-order remainder]
\label{lem:cubic-sixth-row-remainder}
Fix the feature matrices \(w_1,\ldots,w_m\), and let
\(w=q_{3,d}(\widetilde X)q_{3,d}(\widetilde X')^{\mathsf T}\), where
\((\widetilde X,\widetilde X')\sim P_0=\sigma_d\otimes\sigma_d\)
is independent of the sample.  Recall \(S_m=\sum_{j\le m}w_j\), and write
\[
 f(w)=f_m(w)=\langle S_m,w\rangle_{\HS},
 \qquad g(w)=g_m(w)=\sum_{j\le m}\langle w_j,w\rangle_{\HS}^2.
\]
Fix \(C_0<\infty\).  Suppose that
\(u\in\mathbb R\), \(0\le c\le1\), and
\begin{equation}
 \mathbb E_we^{2uf(w)}\le C_0,
 \qquad
 u^2(m+\|S_m\|_{\HS}^2)/N^2=o(1).
 \label{eq:cubic-sixth-row-hypotheses}
\end{equation}
Then, for sufficiently large \(d\),
\begin{equation}
\begin{aligned}
 \log Z_m(u,c)
 &=\frac{u^2}{2}A_2(c)+\frac{u^4}{24}\Gamma_{4,m}(c)+R_{6,m}(u,c),\\
 |R_{6,m}(u,c)|
 &\le C_{C_0}|u|^6
       \frac{\|S_m\|_{\HS}^6+m^3}{N^6},
\end{aligned}
 \label{eq:cubic-sixth-row-remainder}
\end{equation}
where
\[
 A_2(c)=\mathbb E_w(f^2-cg),\qquad
 \Gamma_{4,m}(c)=\mathbb E_w(f^4-6cf^2g+3c^2g^2)-3A_2(c)^2.
\]
The constant depends only on \(C_0\).
\end{lemma}

\begin{proof}
We first expand \(Z_m(u,c)\) through fourth order and bound its
sixth-order remainder, then expand its logarithm to prove
\eqref{eq:cubic-sixth-row-remainder}.
All expectations and \(L^p(P_0)\) norms below integrate only the independent
pair defining \(w\), and \(C\) denotes an absolute constant.

\emph{Step 1: Taylor expansion of \(Z_m\).}
The bounds \(|f|\le m\) and \(0\le g\le m\) justify differentiation
under \(\mathbb E_w\).  Replacing \(\widetilde X\) by
\(-\widetilde X\) changes \(f\) to \(-f\) and leaves \(g\) unchanged.
Thus \(Z_m(-u,c)=Z_m(u,c)\), so its odd derivatives at zero vanish.
Direct differentiation gives
\[
 Z_m(0,c)=1,\qquad
 \partial_u^2Z_m(0,c)=A_2(c),\qquad
 \partial_u^4Z_m(0,c)=\Gamma_{4,m}(c)+3A_2(c)^2.
\]
Taylor's theorem therefore gives
\begin{equation}
 \left|Z_m(u,c)-1-\frac{u^2}{2}A_2(c)
       -\frac{u^4}{24}\big\{\Gamma_{4,m}(c)+3A_2(c)^2\big\}\right|
 \le\frac{|u|^6}{6!}\sup_{|t|\le|u|}|\partial_t^6Z_m(t,c)|.
 \label{eq:cubic-sixth-row-taylor-bound}
\end{equation}
To bound the derivative in \eqref{eq:cubic-sixth-row-taylor-bound},
use \(e^{tf-ct^2g/2}\le e^{|u||f|}\) for \(|t|\le|u|\), together
with Young's inequalities \(|f|^4g\le2|f|^6/3+g^3/3\) and
\(|f|^2g^2\le|f|^6/3+2g^3/3\), to obtain
\[
\begin{aligned}
 |\partial_t^6Z_m(t,c)|
 &=\left|\mathbb E_we^{tf-ct^2g/2}
 \bigl[(f-ctg)^6-15cg(f-ctg)^4+45c^2g^2(f-ctg)^2-15c^3g^3\bigr]\right|\\
 &\le C\mathbb E_we^{|u||f|}
       \bigl(|f|^6+g^3+u^2g^4+u^4g^5+u^6g^6\bigr).
\end{aligned}
\]
Apply \eqref{eq:cubic-product-hypercontractivity}
with \((H,p)=(S_m,12)\) and \((H,p)=(w_j,24)\).
Using \(\|w_j\|_{\HS}=1\) and the triangle inequality in
\(L^{12}(P_0)\), we obtain
\begin{equation}
 \begin{aligned}
 \|f\|_{L^{12}(P_0)}
 &\le C\|f\|_{L^2(P_0)}=C\frac{\|S_m\|_{\HS}}{N},\\
 \|g\|_{L^{12}(P_0)}
 &\le\sum_{j\le m}\|\langle w_j,w\rangle_{\HS}\|_{L^{24}(P_0)}^2
 \le\frac{C}{N^2}\sum_{j\le m}\|w_j\|_{\HS}^2
 =\frac{Cm}{N^2}.
 \end{aligned}
 \label{eq:cubic-sixth-row-moments}
\end{equation}
Since \(f\) and \(-f\) have the same law, the first condition in
\eqref{eq:cubic-sixth-row-hypotheses} gives
\[
 \mathbb E_we^{2|u||f|}
 \le\mathbb E_we^{2uf}+\mathbb E_we^{-2uf}
 =2\mathbb E_we^{2uf}\le2C_0.
\]
Cauchy--Schwarz, \(\mathbb E_we^{2|u||f|}\le2C_0\), and
\eqref{eq:cubic-sixth-row-moments} now give
\[
 \sup_{|t|\le|u|}|\partial_t^6Z_m(t,c)|
 \le C_{C_0}\left(\frac{\|S_m\|_{\HS}^6+m^3}{N^6}
       +\frac{m^3}{N^6}\sum_{j=1}^3\left(\frac{u^2m}{N^2}\right)^j\right)
 \le C_{C_0}\frac{\|S_m\|_{\HS}^6+m^3}{N^6},
\]
where the last inequality uses \(u^2m/N^2=o(1)\) from
\eqref{eq:cubic-sixth-row-hypotheses}.

Substituting this bound for
\(\sup_{|t|\le|u|}|\partial_t^6Z_m(t,c)|\) into
\eqref{eq:cubic-sixth-row-taylor-bound} gives
\begin{equation}
 Z_m(u,c)=1+\frac{u^2}{2}A_2(c)
   +\frac{u^4}{24}\big\{\Gamma_{4,m}(c)+3A_2(c)^2\big\}
   +O\!\left(|u|^6\frac{\|S_m\|_{\HS}^6+m^3}{N^6}\right).
 \label{eq:cubic-sixth-row-partition-taylor}
\end{equation}

\emph{Step 2: expansion of \(\log Z_m\).}
By \eqref{eq:cubic-sixth-row-moments} and H\"older's inequality,
\[
 |A_2(c)|\le C\frac{\|S_m\|_{\HS}^2+m}{N^2},\qquad
 |\Gamma_{4,m}(c)+3A_2(c)^2|
 \le C\frac{(\|S_m\|_{\HS}^2+m)^2}{N^4}.
\]
Writing \(x=Z_m(u,c)-1\), the expansion
\eqref{eq:cubic-sixth-row-partition-taylor} therefore gives
\(x=O\{u^2(m+\|S_m\|_{\HS}^2)/N^2\}=o(1)\) and
\(x-u^2A_2(c)/2=O\{u^4(m+\|S_m\|_{\HS}^2)^2/N^4\}\).
Consequently,
\[
 \left|x^2-\left\{\frac{u^2}{2}A_2(c)\right\}^2\right|+|x|^3
 \le C|u|^6\frac{(m+\|S_m\|_{\HS}^2)^3}{N^6}
 \le4C|u|^6\frac{m^3+\|S_m\|_{\HS}^6}{N^6}.
\]
Substituting \eqref{eq:cubic-sixth-row-partition-taylor} into
\(\log(1+x)=x-x^2/2+O(|x|^3)\) now proves
\eqref{eq:cubic-sixth-row-remainder}, since
\[
 \frac{u^4}{24}\big\{\Gamma_{4,m}(c)+3A_2(c)^2\big\}
 -\frac12\left\{\frac{u^2}{2}A_2(c)\right\}^2
 =\frac{u^4}{24}\Gamma_{4,m}(c).
\]
\end{proof}

We now apply Lemma~\ref{lem:cubic-sixth-row-remainder} to the sample
prefixes.  Lemma~\ref{lem:cubic-regular-row-expansion} verifies
\eqref{eq:cubic-sixth-row-hypotheses} when
\(\beta^2\ell(S_m)\le c_0d\), evaluates the expansion coefficients,
and, when \(\beta^2\ell(S_n)\le\epsilon d<c_0d\), uses the tail bound in
Lemma~\ref{lem:cubic-regular-prefix-estimates} to control the contribution
from the remaining prefixes.

\begin{lemma}[Expansions for sample prefixes]
\label{lem:cubic-regular-row-expansion}
Assume \eqref{eq:cubic-local-standing-assumptions}
and fix a full sample in
\(\mathcal G_n^\dagger(X)\cap\mathcal G_n^\dagger(X')\).
Retain \(f,g,A_2(c),\Gamma_{4,m}(c)\), and \(R_{6,m}(u,c)\) from
Lemma~\ref{lem:cubic-sixth-row-remainder}.
For each fixed \(0<C_\star<\infty\), choose \(c_0>0\) sufficiently small.
The assertions below are uniform for \(|u|\le C_\star\beta^2\) and
\(0\le c\le1\).

\smallskip
\noindent\textup{(i)} If \(\beta^2\ell(S_m)\le c_0d\), the expansion and
remainder bound \eqref{eq:cubic-sixth-row-remainder} hold.

\smallskip
\noindent\textup{(ii)} For every prefix, \(A_2(c)=(\|S_m\|_{\HS}^2-cm)/N^2\), and
\begin{equation}
 \Gamma_{4,m}(1)
 =24\mathbb E_we_{4,m}(w)
  +\lambda_d^2(-2m-16E_m)-3A_2(1)^2,
 \label{eq:cubic-complete-fourth-cumulant}
\end{equation}
where \(\lambda_d=\int Z_{3,d}(x,z)^4\,d\sigma_d(z)=\{93+o(1)\}N^{-2}\)
and \(e_{4,m}\) is defined in
Section~\ref{subsec:cubic-local-fourth-order}.
Moreover,
\begin{equation}
\begin{split}
 \Gamma_{4,m}(c)-\Gamma_{4,m}(1)
 &=6(1-c)\operatorname{Cov}_w(f^2,g)
   -3(1-c^2)\operatorname{Var}_w(g)\\
 &\le C(1-c)\frac{m\|S_m\|_{\HS}^2}{N^4}.
\end{split}
 \label{eq:cubic-compensator-cumulant}
\end{equation}
\smallskip
\noindent\textup{(iii)} If \(\beta^2\ell(S_n)\le\epsilon d\) for a fixed
\(0<\epsilon<c_0\), then
\begin{equation}
\begin{split}
 &\sum_{m<n}\mathbb E_{\rm ord}\!\left[
 \1_{\{\beta^2\ell(S_m)>c_0d\}}
 \big\{|\log Z_m(u,c)|+u^2|A_2(c)|+u^4|\Gamma_{4,m}(c)|\big\}\right]\\
 &\qquad\le
 C(\beta^2n^2+\beta^4n^3+\beta^8n^5)e^{-c'd/\beta^2}=o(1).
\end{split}
 \label{eq:cubic-exceptional-row-restoration}
\end{equation}
\end{lemma}

\begin{proof}
We prove the three assertions in order.  Step 1 verifies
\eqref{eq:cubic-sixth-row-hypotheses} to apply
Lemma~\ref{lem:cubic-sixth-row-remainder}.
Step 2 evaluates the coefficients in
\eqref{eq:cubic-complete-fourth-cumulant}--\eqref{eq:cubic-compensator-cumulant}.
Step 3 combines deterministic bounds with
Lemma~\ref{lem:cubic-regular-prefix-estimates} to prove
\eqref{eq:cubic-exceptional-row-restoration}.

\emph{Step 1: verifying the conditions for the expansion.}
Apply Lemma~\ref{lem:cubic-subset-tensor-bounds} to the full sample of
\(n\) pairs with \(I=\{1,\ldots,m\}\).
Its bound \eqref{eq:cubic-prefix-source} gives
\(\|S_m\|_{\HS}^2\le C_A\{m+(m^2/N)\log^3(en/m)\}\) for \(m\ge1\).
Taking the supremum over \(m\le n\) and using \(S_0=0\) and
\(n\beta^2=O(Nd)\) yields
\begin{equation}
 \sup_{m\le n}\frac{\beta^4(m+\|S_m\|_{\HS}^2)}{N^2}
 \le C\left(\frac{\beta^4n}{N^2}+\frac{\beta^4n^2}{N^3}\right)
 \le C\left(\frac{\beta^2d}{N}+\frac{d^2}{N}\right)=o(1).
 \label{eq:cubic-uniform-row-smallness}
\end{equation}
For a prefix satisfying \(\beta^2\ell(S_m)\le c_0d\), choose \(c_0>0\) so small
that \(2|u|\ell(S_m)\le c_{\rm src}d\) for \(|u|\le C_\star\beta^2\).
Equation~\eqref{eq:cubic-bounded-source} then gives
\[
 \mathbb E_we^{2uf_m(w)}
 \le\exp\left(\frac{4Cu^2\|S_m\|_{\HS}^2}{N^2}\right)
 =1+o(1),
\]
using \eqref{eq:cubic-uniform-row-smallness}.
Together these verify the hypotheses of
Lemma~\ref{lem:cubic-sixth-row-remainder}, with \(C_0=2\) for
sufficiently large \(d\), so \eqref{eq:cubic-sixth-row-remainder}
holds uniformly in the stated parameter range.

\emph{Step 2: evaluating the expansion coefficients.}
Isotropy gives
\(A_2(c)=\mathbb E_w(f^2-cg)=(\|S_m\|_{\HS}^2-cm)/N^2\).
To compute \(\Gamma_{4,m}(1)\), we first rewrite
\(f^4-6f^2g+3g^2\) using Newton's identities
\cite[Chapter~I, Section~2, equation~(2.11$'$)]{Macdonald1995}, which give
\begin{equation}
  f^4-6f^2g+3g^2=24e_{4,m}(w)-8p_1p_3+6p_4,
  \label{eq:cubic-newton-fourth-identity}
\end{equation}
where \(p_k=\sum_{j\le m}x_j^k\) and
\(x_j=\langle w_j,w\rangle_{\HS}\), so \(p_1=f\) and \(p_2=g\).
Since \(p_1p_3=\sum_{j,k\le m}x_jx_k^3\) and
\(p_4=\sum_{j\le m}x_j^4\), we compute
\(\mathbb E_w(x_jx_k^3)\) for all \(j,k\le m\) to evaluate
\(\mathbb E_w(-8p_1p_3+6p_4)\).

For fixed \(x\), the function \(z\mapsto Z_{3,d}(x,z)\) belongs to
\(\cH_{3,d}\), whereas \(Z_{3,d}(y,z)^3\) depends on \(z\) only through
\(\langle y,z\rangle\).  The Funk--Hecke formula
\cite[Section~1.2]{DaiXu2013} therefore gives
\[
 \int Z_{3,d}(x,z)Z_{3,d}(y,z)^3\,d\sigma_d(z)
 =\lambda_d Z_{3,d}(x,y).
\]
Setting \(x=y\) determines the coefficient, since \(Z_{3,d}(y,y)=1\).
For \(x_j=Z_{3,d}(X_j,\widetilde X)Z_{3,d}(X_j',\widetilde X')\),
independence of \(\widetilde X\) and \(\widetilde X'\) and this integral
identity give
\[
 \mathbb E_w(x_jx_k^3)
 =\lambda_d^2 Z_{3,d}(X_j,X_k)Z_{3,d}(X_j',X_k')
 =\lambda_d^2h_{jk}.
\]
In particular, \(\mathbb E_wx_j^4=\lambda_d^2\) because \(h_{jj}=1\).
Recall \(E_m=\sum_{1\le j<k\le m}h_{jk}\), so
\(\sum_{j,k\le m}h_{jk}=m+2E_m\).  Substituting these moments gives
\begin{equation}
  \mathbb E_w(-8p_1p_3+6p_4)
  =-8\lambda_d^2\sum_{j,k\le m}h_{jk}+6m\lambda_d^2
  =\lambda_d^2(-2m-16E_m).
  \label{eq:cubic-repeated-index-cancellation}
\end{equation}
Taking expectations in \eqref{eq:cubic-newton-fourth-identity},
substituting \eqref{eq:cubic-repeated-index-cancellation}, and subtracting
\(3A_2(1)^2\) gives \eqref{eq:cubic-complete-fourth-cumulant}.
Finally, substituting \eqref{eq:cubic-zonal-polynomial} and using
\eqref{eq:cubic-spherical-even-moments} yields
\[
  \lambda_d=
  \frac{108(31d^3-51d^2-12d+140)}
  {d(d-1)^3(d+2)(d+4)(d+6)(d+8)(d+10)}
  =\{93+o(1)\}N^{-2}.
\]

Since \(A_2(c)=A_2(1)+(1-c)\mathbb E_wg\), direct subtraction gives
\begin{align*}
 \Gamma_{4,m}(c)-\Gamma_{4,m}(1)
 &=6(1-c)(\mathbb E_wf^2g-\mathbb E_wf^2\mathbb E_wg)
   -3(1-c^2)\{\mathbb E_wg^2-(\mathbb E_wg)^2\}\\
 &\le6(1-c)|\operatorname{Cov}_w(f^2,g)|.
\end{align*}
The hypercontractive inequality \eqref{eq:cubic-product-hypercontractivity},
applied with \((H,p)=(S_m,4)\), gives
\begin{equation}
  \mathbb E_wf^4=\|f\|_{L^4(P_0)}^4
  \le C\frac{\|S_m\|_{\HS}^4}{N^4}.
  \label{eq:cubic-local-fourth-moment}
\end{equation}
Together with
\(\|g\|_{L^2(P_0)}\le\|g\|_{L^{12}(P_0)}\le Cm/N^2\) from
\eqref{eq:cubic-sixth-row-moments}, Cauchy--Schwarz gives
\[
 |\operatorname{Cov}_w(f^2,g)|
 \le(\mathbb E_wf^4)^{1/2}(\mathbb E_wg^2)^{1/2}
 \le C\frac{m\|S_m\|_{\HS}^2}{N^4}.
\]
This proves \eqref{eq:cubic-compensator-cumulant}.

\emph{Step 3: prefixes with \(\beta^2\ell(S_m)>c_0d\).}
Assume now that \(\beta^2\ell(S_n)\le\epsilon d\), where
\(0<\epsilon<c_0\).  To prove
\eqref{eq:cubic-exceptional-row-restoration}, we first bound its summand
for every fixed prefix.  We have \(|f_m|\le m\) and \(0\le g_m\le m\).
Centering and isotropy of \(w\) give \(\mathbb E_wf_m=0\) and
\(\mathbb E_wg_m=\sum_{j\le m}\|w_j\|_{\HS}^2/N^2=m/N^2\).
Jensen's inequality therefore yields
\[
 \log Z_m(u,c)
 =\log\mathbb E_we^{uf_m-cu^2g_m/2}
 \ge u\mathbb E_wf_m-\frac{cu^2}{2}\mathbb E_wg_m
 =-\frac{cu^2m}{2N^2}.
\]
For the upper bound, \(uf_m-cu^2g_m/2\le |u|m\) implies
\(Z_m(u,c)\le e^{|u|m}\), hence \(\log Z_m(u,c)\le |u|m\).
Together with the definitions of \(A_2(c)\) and \(\Gamma_{4,m}(c)\)
in Lemma~\ref{lem:cubic-sixth-row-remainder}, these bounds give, for
\(m\ge1\), \(|u|\le C_\star\beta^2\), and \(0\le c\le1\),
\[
\begin{aligned}
 |\log Z_m(u,c)|
 &\le |u|m+\frac{cu^2m}{2N^2}\le C\beta^2m,\\
 u^2|A_2(c)|
 &\le u^2(m^2+m)\le C\beta^4m^2,\\
 u^4|\Gamma_{4,m}(c)|
 &\le u^4\big\{m^4+6m^3+3m^2+3(m^2+m)^2\big\}
 \le C\beta^8m^4.
\end{aligned}
\]
The summand is zero when \(m=0\).  Under
\(\beta^2\ell(S_n)\le\epsilon d\),
Lemma~\ref{lem:cubic-regular-prefix-estimates}, specifically
\eqref{eq:cubic-orbit-source-bridge-tail} with \(c_1=c_0\), gives
\(\mathbb P_{\rm ord}\{\beta^2\ell(S_m)>c_0d\}
\le e^{-c'd/\beta^2}\), uniformly in \(m\le n\).
Multiplying the three deterministic bounds by this probability and
summing over \(m<n\) yields
\begin{align*}
 &\sum_{m<n}\mathbb E_{\rm ord}\!\left[
 \1_{\{\beta^2\ell(S_m)>c_0d\}}
 \{|\log Z_m(u,c)|+u^2|A_2(c)|+u^4|\Gamma_{4,m}(c)|\}\right]\\
 &\quad\le C\sum_{m<n}(\beta^2m+\beta^4m^2+\beta^8m^4)e^{-c'd/\beta^2}
 \le C(\beta^2n^2+\beta^4n^3+\beta^8n^5)e^{-c'd/\beta^2}=o(1),
\end{align*}
since \(n\le d^A\) and \(0<|\beta|\le1\).
This proves \eqref{eq:cubic-exceptional-row-restoration}.
The bound is uniform over full samples satisfying the stated conditions,
so it also holds after averaging over any exchangeable law \(Q\)
supported on such samples.
\end{proof}

\subsection{Summed estimates and entropy bounds}
\label{subsec:cubic-local-summed-estimates}

To bound the right-hand side of \eqref{eq:cubic-shared-entropy-sum},
we estimate
\[
 (1-\eta)\sum_{m<n}\mathbb E_Q
 \log Z_m\!\left(\frac{\beta^2}{1-\eta},1-\eta\right)
\]
by summing \eqref{eq:cubic-sixth-row-remainder}, as justified by
Lemma~\ref{lem:cubic-regular-row-expansion}.
The fourth-order coefficient in \eqref{eq:cubic-complete-fourth-cumulant}
contains \(24\mathbb E_we_{4,m}(w)\).
For any exchangeable law \(Q\) of the paired sample, the definitions of
\(e_{4,m}\) and \(T_{4,n}\) give
\begin{equation}
\begin{aligned}
 \sum_{m<n}\mathbb E_Q\mathbb E_we_{4,m}(w)
 &=\sum_{m<n}\frac{\binom m4}{\binom n4}\mathbb E_QT_{4,n}
 =\frac{n-4}{5}\mathbb E_QT_{4,n},\\
 \beta^8\left|\sum_{m<n}\mathbb E_Q\mathbb E_we_{4,m}(w)\right|
 &\le\frac{\beta^8n}{5}\mathbb E_Q|T_{4,n}|,
\end{aligned}
 \label{eq:cubic-fourth-prefix-identity}
\end{equation}
where we used \(\sum_{m=4}^{n-1}\binom m4=\binom n5\).

Lemma~\ref{lem:cubic-parameterized-row-ledger} combines
\eqref{eq:cubic-fourth-prefix-identity} with the quadratic terms and
the remainder in \eqref{eq:cubic-sixth-row-remainder} to bound the
sum in \eqref{eq:cubic-shared-entropy-sum}.
It also gives an auxiliary bound on
\(\log\mathbb E[\1_{\mathcal E}e^{\kappa\beta^2E_n}]\) by applying
the same expansion with \(c=0\).
Both estimates retain their dependence on \(R\) in
\(\|S_n\|_{\HS}^2\le(1+R)n\).

\begin{lemma}[Summed row and exponential-moment bounds]
\label{lem:cubic-parameterized-row-ledger}
Assume \eqref{eq:cubic-local-standing-assumptions}.
Let \(\mathcal E\subseteq\mathcal E_n^{\rm reg}\) be invariant under
permutations of the paired sample.  For some \(1\le R\le d^A\), suppose
that on \(\mathcal E\)
\begin{equation}
 \|S_n\|_{\HS}^2\le(1+R)n,\qquad E_n\ge0,
 \qquad \beta^2\ell(S_n)\le\epsilon d,
 \label{eq:cubic-parameterized-row-hypotheses}
\end{equation}
where \(\epsilon>0\) is sufficiently small for a fixed
\(0<\eta_\star<1\).

\smallskip
\noindent\textup{(i)} For every exchangeable law \(Q\) with
\(Q(\mathcal E)=1\) and every \(0<\eta\le\eta_\star\),
\begin{align}
 &\sum_{m<n}\mathbb E_Q\biggl[
  (1-\eta)\log Z_m(\beta^2/(1-\eta),1-\eta)
  -\frac{\beta^4}{2N^2}(\|S_m\|_{\HS}^2-m)
  -\frac{\eta\beta^4}{2(1-\eta)N^2}\|S_m\|_{\HS}^2
 \biggr]\notag\\
 &\qquad\le C\left(\beta^2R^2+R^3\frac{\beta^4d^4}{N^2}\right)
     +C_4\beta^8n\mathbb E_Q|T_{4,n}|+o(1).
\label{eq:cubic-parameterized-reserved-row-ledger}
\end{align}
\smallskip
\noindent\textup{(ii)} Under \(P_0^{\otimes n}\), for every sufficiently
small fixed \(\kappa>0\),
\begin{equation}
 \log\mathbb E\left[\1_{\cal E}e^{\kappa \beta^2E_n}\right]
 \le C_\kappa\left(Rd^2+\beta^2R^2+R^3\frac{\beta^4d^4}{N^2}\right)+o(1).
 \label{eq:cubic-parameterized-linear-ledger}
\end{equation}
The choice of \(\epsilon\) and the constants may depend on
\(\eta_\star\); they are uniform in \(R,\beta,n,d,\eta\) in the
displayed ranges.
\end{lemma}

\begin{proof}
We apply the expansion in Lemma~\ref{lem:cubic-regular-row-expansion}
at \((u,c)=(\beta^2/(1-\eta),1-\eta)\) to prove
\eqref{eq:cubic-parameterized-reserved-row-ledger}.
For \eqref{eq:cubic-parameterized-linear-ledger}, we first apply
conditional entropy duality and then use the same expansion at
\((u,c)=(\kappa\beta^2,0)\).

\emph{Step 1: the summed row bound.}
Put \(u=\beta^2/(1-\eta)\) and \(c=1-\eta\).
Apply Lemma~\ref{lem:cubic-regular-row-expansion}(i), (iii) with
\(C_\star=(1-\eta_\star)^{-1}\) and choose \(\epsilon<c_0\).
Equation~\eqref{eq:cubic-sixth-row-remainder} holds whenever
\(\beta^2\ell(S_m)\le c_0d\).
By exchangeability of \(Q\), averaging
\eqref{eq:cubic-exceptional-row-restoration} under \(Q\) bounds the
sums of \(\log Z_m(u,c)\), \(u^2A_2(c)\), and \(u^4\Gamma_{4,m}(c)\)
over the complementary prefixes by \(o(1)\).
We may therefore sum the two coefficients over all prefixes, obtaining
\begin{equation}
\begin{aligned}
 (1-\eta)\sum_{m<n}\mathbb E_Q\log Z_m(u,c)
 &=\frac{(1-\eta)u^2}{2}\sum_{m<n}\mathbb E_QA_2(c)
   +\frac{(1-\eta)u^4}{24}\sum_{m<n}\mathbb E_Q\Gamma_{4,m}(c)\\
 &\quad+(1-\eta)\sum_{m<n}\mathbb E_Q\!\left[
   \1_{\{\beta^2\ell(S_m)\le c_0d\}}R_{6,m}(u,c)\right]+o(1).
\end{aligned}
 \label{eq:cubic-summed-row-decomposition}
\end{equation}

The quadratic coefficient gives the two subtractions in
\eqref{eq:cubic-parameterized-reserved-row-ledger}, since
\[
 \frac{1-\eta}{2}\left(\frac{\beta^2}{1-\eta}\right)^2A_2(1-\eta)
 =\frac{\beta^4}{2N^2}(\|S_m\|_{\HS}^2-m)
    +\frac{\eta\beta^4}{2(1-\eta)N^2}\|S_m\|_{\HS}^2.
\]
To bound the fourth-order coefficient and the remainder in
\eqref{eq:cubic-summed-row-decomposition}, we use exchangeability and
Lemma~\ref{lem:cubic-regular-prefix-estimates} to obtain
\[
 \mathbb E_Q\|S_m\|_{\HS}^{2k}
 =\mathbb E_Q\mathbb E_{\rm ord}\|S_m\|_{\HS}^{2k}
 \le C_kR^km^k,\qquad k\in\{1,2,3\}.
\]
Here \eqref{eq:cubic-parameterized-Hilbert-bridge} gives \(k=2,3\),
and Cauchy--Schwarz gives \(k=1\).
For the fourth-order coefficient, exchangeability and \(E_n\ge0\) also give
\[
 \mathbb E_QE_m=\frac{m(m-1)}{n(n-1)}\mathbb E_QE_n\ge0.
\]
Thus the terms \(\lambda_d^2(-2m-16E_m)\) and \(-3A_2(1)^2\) in
\eqref{eq:cubic-complete-fourth-cumulant} have nonpositive
\(Q\)-expectations.  Combining that identity with
\eqref{eq:cubic-compensator-cumulant} and
\eqref{eq:cubic-fourth-prefix-identity} yields
\begin{align*}
 \frac{\beta^8}{24(1-\eta)^3}\sum_{m<n}\mathbb E_Q\Gamma_{4,m}(1-\eta)
 &\le\frac{\beta^8(n-4)}{5(1-\eta)^3}\mathbb E_QT_{4,n}
    +\frac{C\eta\beta^8}{N^4}\sum_{m<n}m\mathbb E_Q\|S_m\|_{\HS}^2\\
 &\le C_4\beta^8n\mathbb E_Q|T_{4,n}|
       +\frac{CR\beta^8n^3}{N^4}
 \le C_4\beta^8n\mathbb E_Q|T_{4,n}|+C\beta^2R^2.
\end{align*}
The last step uses
\(R\beta^8n^3/N^4=O(R\beta^2d^3/N)=O(\beta^2R^2)\),
which follows from \(n\beta^2\asymp Nd\), \(N\asymp d^3\), and \(R\ge1\).
For the remainder \(R_{6,m}\), equations
\eqref{eq:cubic-sixth-row-remainder} and
\eqref{eq:cubic-parameterized-Hilbert-bridge} with \(k=3\) give
\begin{equation}
\begin{aligned}
 \sum_{m<n}\mathbb E_Q\!\left[
 \1_{\{\beta^2\ell(S_m)\le c_0d\}}
 |R_{6,m}(\beta^2/(1-\eta),1-\eta)|\right]
 &\le\frac{C\beta^{12}}{N^6}\sum_{m<n}(\mathbb E_Q\|S_m\|_{\HS}^6+m^3)\\
 &\le C R^3\frac{\beta^{12}n^4}{N^6}
 \le C R^3\frac{\beta^4d^4}{N^2}.
\end{aligned}
 \label{eq:cubic-parameterized-sixth-error}
\end{equation}
Substituting the identity for \(A_2(1-\eta)\), the bound on
\(\sum_{m<n}\mathbb E_Q\Gamma_{4,m}(1-\eta)\), and
\eqref{eq:cubic-parameterized-sixth-error} into
\eqref{eq:cubic-summed-row-decomposition} proves
\eqref{eq:cubic-parameterized-reserved-row-ledger}.

\emph{Step 2: the linear exponential-moment bound.}
If \(P_0^{\otimes n}(\mathcal E)=0\), then
\eqref{eq:cubic-parameterized-linear-ledger} holds with \(\log0=-\infty\).
Otherwise, the restricted Gibbs formula
\eqref{eq:cubic-restricted-variational} and permutation symmetrization
reduce the proof to bounding
\(\kappa\beta^2\mathbb E_QE_n-D(Q\Vert P_0^{\otimes n})\)
over exchangeable \(Q\) with \(Q(\mathcal E)=1\) and finite relative
entropy.  Since \(E_n=\sum_{m<n}f_m(w_{m+1})\), apply
\eqref{eq:cubic-linear-entropy-duality} to the conditional law
\(Q_{m+1}\) with \(H(w)=\kappa\beta^2f_m(w)\).
Taking \(Q\)-expectations, summing over \(m<n\), and applying the
entropy chain rule, we obtain
\[
 \kappa\beta^2\mathbb E_QE_n-D(Q\Vert P_0^{\otimes n})
 \le\sum_{m<n}\mathbb E_Q\log Z_m(\kappa\beta^2,0).
\]
For fixed sufficiently small \(0<\kappa\le1\), the parameters
\((u,c)=(\kappa\beta^2,0)\) lie in the range verified in Step 1.
Here \(A_2(0)=\|S_m\|_{\HS}^2/N^2\), and
\eqref{eq:cubic-local-fourth-moment} gives
\[
 |\Gamma_{4,m}(0)|
 =\left|\mathbb E_wf_m^4-3(\mathbb E_wf_m^2)^2\right|
 \le4\mathbb E_wf_m^4
 \le C\frac{\|S_m\|_{\HS}^4}{N^4}.
\]
Using \eqref{eq:cubic-sixth-row-remainder}, controlling the prefixes
with \(\beta^2\ell(S_m)>c_0d\) by
\eqref{eq:cubic-exceptional-row-restoration}, and applying the moment
bounds for \(\|S_m\|_{\HS}^{2k}\), \(k=1,2,3\), from Step 1, we obtain
\begin{align*}
 \sum_{m<n}\mathbb E_Q\log Z_m(\kappa\beta^2,0)
 &\le\frac{\kappa^2\beta^4}{2N^2}\sum_{m<n}\mathbb E_Q\|S_m\|_{\HS}^2
      +\frac{C_\kappa\beta^8}{N^4}\sum_{m<n}\mathbb E_Q\|S_m\|_{\HS}^4\\
 &\quad+\frac{C_\kappa\beta^{12}}{N^6}\sum_{m<n}(\mathbb E_Q\|S_m\|_{\HS}^6+m^3)+o(1)\\
 &\le C_\kappa\left(\frac{R\beta^4n^2}{N^2}
              +\frac{R^2\beta^8n^3}{N^4}
              +\frac{R^3\beta^{12}n^4}{N^6}\right)+o(1)\\
 &\le C_\kappa\left(Rd^2+\beta^2R^2
              +\frac{R^3\beta^4d^4}{N^2}\right)+o(1).
\end{align*}
Taking the supremum over \(Q\) proves
\eqref{eq:cubic-parameterized-linear-ledger}.
\end{proof}

To substitute \eqref{eq:cubic-parameterized-reserved-row-ledger}
into \eqref{eq:cubic-shared-entropy-sum}, we first evaluate its two
quadratic sums.
For any exchangeable law \(Q\) of the paired sample,
\(\mathbb E_Qh_{ij}\) is the same for all \(i\ne j\).
Since \(\|w_i\|_{\HS}=1\), we obtain
\begin{equation}
  \mathbb E_Q\|S_m\|_{\HS}^2
  =m+\frac{m(m-1)}{n(n-1)}\mathbb E_Q(\|S_n\|_{\HS}^2-n).
  \label{eq:cubic-exact-prefix-bridge}
\end{equation}
Set \(b_n=\beta^4(n-2)/(3N^2)\).  Summing
\eqref{eq:cubic-exact-prefix-bridge} over \(0\le m<n\) and using
\(\sum_{m<n}m(m-1)=n(n-1)(n-2)/3\) and
\(\|S_n\|_{\HS}^2-n=2E_n\) yields the two required identities
\begin{gather}
  \frac{\beta^4}{2N^2}\sum_{m<n}
     \mathbb E_Q(\|S_m\|_{\HS}^2-m)
  =b_n\mathbb E_QE_n,
  \label{eq:cubic-prefix-linear-identity}\\
  \sum_{m<n}\mathbb E_Q\|S_m\|_{\HS}^2
  =\frac{n(n-1)}2
    +\frac{n-2}{3}\mathbb E_Q(\|S_n\|_{\HS}^2-n).
  \label{eq:cubic-prefix-diagonal-identity}
\end{gather}

For \(Q\) satisfying the hypotheses of
Lemma~\ref{lem:cubic-parameterized-row-ledger} and
\(D(Q\Vert P_0^{\otimes n})<\infty\), substituting
\eqref{eq:cubic-parameterized-reserved-row-ledger},
\eqref{eq:cubic-prefix-linear-identity}, and
\eqref{eq:cubic-prefix-diagonal-identity} into
\eqref{eq:cubic-shared-entropy-sum} gives
\begin{equation}
\begin{aligned}
 \mathbb E_QF_n^{(2)}(\beta)-D(Q\Vert P_0^{\otimes n})
 &\le \frac{b_n}{1-\eta}\mathbb E_QE_n
      +\frac{\eta\beta^4n(n-1)}{4(1-\eta)N^2}
      -\eta D(Q\Vert P_0^{\otimes n})\\
 &\quad+C_4\beta^8n\mathbb E_Q|T_{4,n}|
      +C\left(\beta^2R^2+\frac{R^3\beta^4d^4}{N^2}\right)+o(1).
\end{aligned}
 \label{eq:cubic-summed-variational-bound}
\end{equation}
The bound \eqref{eq:cubic-parameterized-linear-ledger} will be used
in the tail estimates to control the linear term
\(b_n\mathbb E_QE_n/(1-\eta)\).
For fixed \(\eta>0\), Lemma~\ref{lem:cubic-fourth-U-exponential}
gives \(C_4\beta^8n\mathbb E_Q|T_{4,n}|
\le(\eta/2)D(Q\Vert P_0^{\otimes n})+o(1)\).
Its proof uses the moment estimate in
Lemma~\ref{lem:cubic-fourth-U-moments} and the Gibbs variational
formula \eqref{eq:cubic-restricted-variational}.

\begin{lemma}[Exponential moments and entropy bound for the fourth-order statistic]
\label{lem:cubic-fourth-U-exponential}
Under \eqref{eq:cubic-local-standing-assumptions}, for every fixed \(C_0>0\),
\begin{equation}
 \mathbb E_{P_0^{\otimes n}}\left[
  \1_{\mathcal E_n^{\rm reg}}
  \exp\{C_0\beta^8n|T_{4,n}|\}\right]\le1+o(1).
 \label{eq:cubic-fourth-U-exponential}
\end{equation}
Consequently, for every fixed \(c,\eta>0\) and every law \(Q\) with
\(Q(\mathcal E_n^{\rm reg})=1\) and \(D(Q\Vert P_0^{\otimes n})<\infty\),
\begin{equation}
 c\beta^8n\mathbb E_Q|T_{4,n}|
 \le\eta D(Q\Vert P_0^{\otimes n})+o(1).
 \label{eq:cubic-fourth-entropy-control}
\end{equation}
The error terms are uniform when \(n\beta^2/(Nd)\) stays in a fixed compact
subset of \((0,\infty)\), and the error in
\eqref{eq:cubic-fourth-entropy-control} is also uniform over these laws \(Q\).
\end{lemma}

\begin{proof}
By \eqref{eq:cubic-fourth-truncated-identity},
\(T_{4,n}=Z_{4,n}\) on \(\mathcal E_n^{\rm reg}\), so
\eqref{eq:cubic-fourth-U-exponential} follows once we prove
\(\mathbb E e^{C_0\beta^8n|Z_{4,n}|}\le1+o(1)\) under
\(P_0^{\otimes n}\).  We obtain this bound by integrating a tail
probability, then deduce \eqref{eq:cubic-fourth-entropy-control} from
\eqref{eq:cubic-restricted-variational}.

Set
\(\upsilon_d:=C_A\beta^8n^3\log^{12}d/N^6\).
Equation \eqref{eq:cubic-fourth-U-moment} and \(n\beta^2\asymp Nd\)
give, for every \(p\ge2\),
\begin{equation}
 \|\beta^8nZ_{4,n}\|_p\le \upsilon_dp^2,
 \qquad
 \upsilon_d\le C\frac{\beta^2 d^3\log^{12} d}{N^3}=o(1).
 \label{eq:cubic-fourth-scaled-moments}
\end{equation}
By \eqref{eq:cubic-fourth-scaled-moments}, monotonicity of
\(L^p\)-norms for \(1\le p<2\), and the sub-Weibull moment--tail
equivalence \cite[Theorem~2.1]{VladimirovaGirardNguyenArbel2020},
there is an absolute \(c_*>0\) such that
\[
 \mathbb P\{\beta^8n|Z_{4,n}|>t\}
 \le2\exp\left(-c_*\sqrt{\frac{t}{\upsilon_d}}\right),
 \qquad t\ge0.
\]
The kernel bound in Lemma~\ref{lem:cubic-fourth-U-moments} also gives
\[
 \beta^8n|Z_{4,n}|
 \le\beta^8n\binom n4\|\widetilde K_{4,d}\|_\infty
 \le R_d:=C_A\frac{\beta^8n^5\log^{12}d}{N^6}.
\]
Using \(n\beta^2\asymp Nd\) and \(N\asymp d^3\), we obtain
\begin{equation}
 R_d\upsilon_d
 \le C\frac{d^8\log^{24}d}{N^4}
 \le C\frac{\log^{24} d}{d^4}=o(1).
 \label{eq:cubic-fourth-cap-scale-product}
\end{equation}

For fixed \(C_0>0\), \eqref{eq:cubic-fourth-cap-scale-product} implies,
for sufficiently large \(d\) and \(0\le t\le R_d\),
\[
 C_0t
 \le C_0\sqrt{R_d\upsilon_d}\sqrt{\frac{t}{\upsilon_d}}
 \le\frac{c_*}{2}\sqrt{\frac{t}{\upsilon_d}}.
\]
The tail-probability integral therefore satisfies
\[
\begin{aligned}
 \mathbb E e^{C_0\beta^8n|Z_{4,n}|}-1
 &=C_0\int_0^{R_d}
    e^{C_0t}\mathbb P\{\beta^8n|Z_{4,n}|>t\}\,dt\\
 &\le2C_0\int_0^{R_d}
    \exp\left(C_0t-c_*\sqrt{\frac{t}{\upsilon_d}}\right)\,dt\\
 &\le2C_0\int_0^\infty
    \exp\left(-\frac{c_*}{2}\sqrt{\frac{t}{\upsilon_d}}\right)\,dt
 =O_{C_0}(\upsilon_d).
\end{aligned}
\]
Together with \eqref{eq:cubic-fourth-truncated-identity}, this proves
\eqref{eq:cubic-fourth-U-exponential}.

Finally, let \(Q(\mathcal E_n^{\rm reg})=1\) and
\(D(Q\Vert P_0^{\otimes n})<\infty\).
Applying \eqref{eq:cubic-restricted-variational} with
\(H=(c/\eta)\beta^8n|T_{4,n}|\) gives
\[
\begin{aligned}
 c\beta^8n\mathbb E_Q|T_{4,n}|-\eta D(Q\Vert P_0^{\otimes n})
 &\le\eta\log\mathbb E_{P_0^{\otimes n}}\!\left[
 \1_{\mathcal E_n^{\rm reg}}e^{(c/\eta)\beta^8n|T_{4,n}|}\right]\\
 &\le\eta\log(1+C_{c/\eta}\upsilon_d)
 \le C_{c,\eta}\upsilon_d=o(1).
\end{aligned}
\]
The last bound is independent of \(Q\), which proves
\eqref{eq:cubic-fourth-entropy-control} with the stated uniformity.
\end{proof}

\subsection{Preliminary estimates for the first inequality}
\label{subsec:cubic-local-deletion}

This subsection and Section~\ref{subsec:cubic-direct-orbit-closure}
establish the first inequality in \eqref{eq:cubic-local-roadmap}.
Figure~\ref{fig:cubic-first-inequality-regions}
shows the successive restrictions leading to \(\mathcal R_n\), followed
by its division into the negative tail, central interval, and positive tail.
Fix \(\epsilon_\star>0\)
sufficiently small for
Lemma~\ref{lem:cubic-parameterized-row-ledger} with \(\eta_\star=1/2\).

\begin{figure}[!htbp]
\centering
\begin{tikzpicture}[x=1cm,y=1cm,>=stealth,
  every node/.style={font=\footnotesize},
  region/.style={draw,align=center,text width=3.15cm,
    minimum height=0.85cm,inner sep=3pt}]
\node[region] (start) at (1.75,0)
  {$\mathcal E_n^{\rm reg}\cap\{s_n<\delta\}$};
\node[region] (small) at (5.6,0) {$s_n<\underline s_d$};
\node[region] (energy) at (9.45,0) {$\Delta\le v_\star$};
\node[region] (direction) at (13.3,0)
  {$\beta^2\ell(S_n)\le\epsilon_\star d$};
\draw[->] (start) -- (small);
\draw[->] (small) -- (energy);
\draw[->] (energy) -- (direction);
\node[region] (meso) at (1.75,-1.4)
  {$\underline s_d\le s_n<\delta$\\Lemma~\ref{lem:cubic-submacroscopic-deletion}};
\node[region] (high) at (5.6,-1.4)
  {$\Delta>v_\star$\\Lemma~\ref{lem:cubic-high-energy-entropy}};
\node[region] (large) at (9.45,-1.4)
  {$\beta^2\ell(S_n)>\epsilon_\star d$\\Lemma~\ref{lem:cubic-clipped-orbit-tail}};
\draw[->] (start) -- (meso);
\draw[->] (small) -- (high);
\draw[->] (energy) -- (large);

\node[region,text width=14.4cm,minimum height=0.65cm] (regular) at (7.5,-2.6)
  {$\mathcal R_n=\mathcal E_n^{\rm reg}\cap
    \{s_n<\underline s_d,\ \Delta\le v_\star,\
      \beta^2\ell(S_n)\le\epsilon_\star d\}$};
\draw[->] (direction.south) -- (13.3,-2.0) -- (7.5,-2.0) -- (regular.north);
\node[region,text width=4.2cm,minimum height=1.05cm] (negative) at (2.5,-4.2)
  {$\Delta<-K_dd^2/\beta^2$\\Negative tail\\\eqref{eq:cubic-local-negative-tail}};
\node[region,text width=4.2cm,minimum height=1.05cm] (central) at (7.5,-4.2)
  {$|\Delta|\le K_dd^2/\beta^2$\\Central interval retained\\Section~\ref{subsec:cubic-local-centered}};
\node[region,text width=4.2cm,minimum height=1.05cm] (positive) at (12.5,-4.2)
  {$\Delta>K_dd^2/\beta^2$\\Positive tail\\Lemma~\ref{lem:cubic-positive-energy-tail}};
\draw (regular.south) -- (7.5,-3.3);
\draw[->] (7.5,-3.3) -- (2.5,-3.3) -- (negative.north);
\draw[->] (7.5,-3.3) -- (central.north);
\draw[->] (7.5,-3.3) -- (12.5,-3.3) -- (positive.north);
\end{tikzpicture}
\caption{Region reduction for the first inequality in
\eqref{eq:cubic-local-roadmap}.
Here $\Delta=\|S_n\|_{\HS}^2-n$, $\underline s_d=d^{-1}\log^4d$,
$v_\star=n/(\beta^2\log^4d)$, and
$K_d\uparrow\infty$ satisfies $K_d=o((\log d)^{1/8})$.
Conditions accumulate along each path.
The three side branches and the two tails represent regions $\mathcal A$ with
$\mathbb E_{P_0^{\otimes n}}[\1_{\mathcal A}W_{n,\beta}]=o(1)$.
The central interval is retained for the second inequality in
\eqref{eq:cubic-local-roadmap}.}
\label{fig:cubic-first-inequality-regions}
\end{figure}
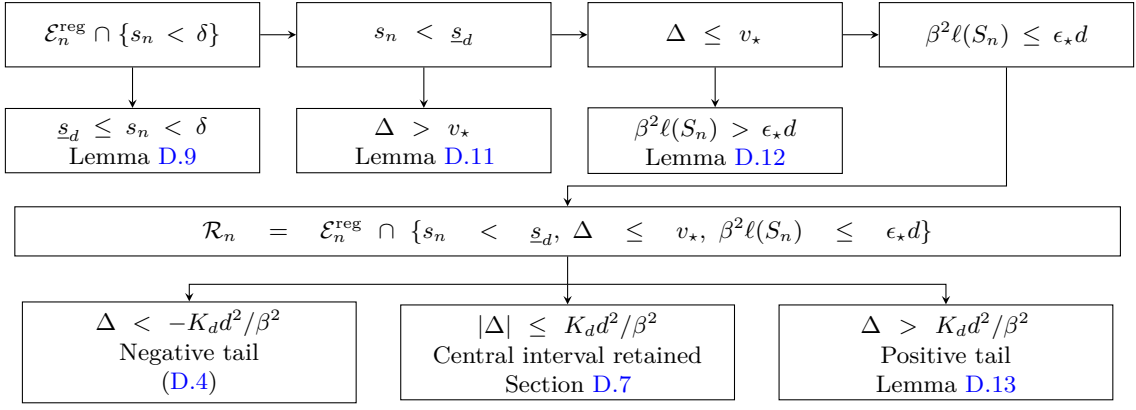


\begin{lemma}[Mesoscopic overlap deletion]
\label{lem:cubic-submacroscopic-deletion}
Under \eqref{eq:cubic-local-standing-assumptions}, there is a
sequence \(\underline s_d\downarrow0\), which we take to be
\(\underline s_d=d^{-1}\log^4 d\), such that
\[
  \frac{d\underline s_d}{\log d}\longrightarrow\infty,
\]
and, for every sufficiently small fixed \(\delta>0\),
\begin{equation}
  \mathbb E_{P_0^{\otimes n}}
  \left[\1_{{\cal G}_n^\dagger(X)}\1_{{\cal G}_n^\dagger(X')}
  W_{n,\beta}\1_{\{\underline s_d\le s_n\le\delta\}}\right]
  =o(1).
  \label{eq:cubic-mesoscopic-deletion}
\end{equation}
\end{lemma}

\begin{proof}
All probabilities and expectations in this proof are under
\(P_0^{\otimes n}\).  Step 1 uses a fixed partition and conditional
exchangeability to construct an event \(\mathcal A_s\) with
\(\mathbb P(\mathcal E_s^{\rm reg}\cap\mathcal A_s)
\ge(1-o(1))\mathbb P(\mathcal E_s^{\rm reg})\), where
\(\mathcal E_s^{\rm reg}=\{s\le s_n<2s,\ X\in\mathcal G_n^\dagger,
X'\in\mathcal G_n^\dagger\}\).
Step 2 applies Lemma~\ref{lem:cubic-directional-mgf} and Chernoff's
inequality after truncating the scalar-product weights to obtain
\(\mathbb P(\mathcal A_s)\le\exp(-cnd s^{2/3})\), which also bounds
\(\mathbb P(\mathcal E_s^{\rm reg})\) after decreasing \(c\).
Step 3 combines this bound with \(W_{n,\beta}\le e^{Cnds}\) on the same
region and sums over \(s=2^j\underline s_d\le\delta\), \(j\ge0\).

\emph{Step 1: a scalar-product lower bound and weight-moment bounds.}
Fix \(s\in[\underline s_d,\delta]\), and put
\(\mathcal E_s=\{s\le s_n<2s\}\).
When \(\mathbb P(\mathcal E_s^{\rm reg})=0\), the contribution
\(\mathbb E[\1_{\mathcal E_s^{\rm reg}}W_{n,\beta}]\) is zero.
It therefore suffices to consider \(s\) for which
\(\mathbb P(\mathcal E_s^{\rm reg})>0\), so that conditioning on
\(\mathcal E_s^{\rm reg}\) is defined.
The assumptions give \(n\ge\underline c Nd\) for some fixed
\(\underline c>0\), so Lemma~\ref{lem:cubic-hereditary-square-profile}
applies to both samples.  Recall
\(w_i=q_{3,d}(X_i)q_{3,d}(X_i')^{\mathsf T}\), and let
\(R=n^{-1}\sum_iw_i\), so that \(s_n=N\|R\|_{\HS}^2\).
Fix \(I=\{1,\ldots,\lfloor n/2\rfloor\}\) and \(J=[n]\setminus I\),
and define
\[
  R_I:=\frac1{|I|}\sum_{i\in I}w_i,
  \qquad C_I:=NR_I,
\]
with \(R_J\) defined analogously.  We write
\[
 \langle R_J,C_I\rangle_{\HS}
 =\frac1{|J|}\sum_{j\in J}
 \langle q_{3,d}(X_j),C_Iq_{3,d}(X_j')\rangle
\]
and put \(b_j=\|C_Iq_{3,d}(X_j')\|\) for \(j\in J\).
Let
\(\mathcal B_I(s)\) be the event
\[
 \frac1{|J|}\sum_{j\in J}b_j^2\le Cs,
 \qquad
 \frac1{|J|}\sum_{j\in J}b_j^4\le Cs^2,
\]
where \(C\) is a sufficiently large fixed constant.
For \(\mathcal A_s=\mathcal B_I(s)\cap
\{\langle R_J,C_I\rangle_{\HS}\ge c_*s\}\), we will prove
\begin{equation}
 \mathbb P(\mathcal A_s\mid\mathcal E_s^{\rm reg})
 \ge1-\varepsilon_{n,s},
 \qquad \varepsilon_{n,s}=C\sqrt{\frac{N}{ns}}=o(1),
 \label{eq:cubic-conditional-split-success}
\end{equation}
uniformly in \(s\), with fixed \(c_*>0\).
Step 1(a) bounds \(N\|R_I-R\|_{\HS}^2\).
On the event where this quantity is \(o(s)\), Step 1(b) proves
\(\langle R_J,C_I\rangle_{\HS}\ge c_*s\), and Step 1(c) verifies
the two moment bounds defining \(\mathcal B_I(s)\).

\emph{Step 1(a): approximate \(R\) by \(R_I\).}
The product law and
\(\mathcal E_s^{\rm reg}\) are invariant under simultaneous permutations
of the pairs \((X_i,X_i')\), so the \(w_i\) remain exchangeable
conditional on \(\mathcal E_s^{\rm reg}\).
Thus \(\mathbb E[\langle w_i,w_j\rangle_{\HS}\mid\mathcal E_s^{\rm reg}]\)
has the same value for every \(i\ne j\).  Expanding
\(\|R\|_{\HS}^2=n^{-2}\sum_{i,j}\langle w_i,w_j\rangle_{\HS}\)
and using \(\|w_i\|_{\HS}=1\) identifies this value as
\[
 \mathbb E[\langle w_i,w_j\rangle_{\HS}\mid\mathcal E_s^{\rm reg}]
 =\frac{n\mathbb E[\|R\|_{\HS}^2\mid\mathcal E_s^{\rm reg}]-1}{n-1},
 \qquad i\ne j.
\]
In the expansion
\(R_I-R=\sum_i(\1_{\{i\in I\}}/|I|-1/n)w_i\), the coefficients sum
to zero and their squares sum to \((n-|I|)/(n|I|)\).
Substituting the displayed conditional expectation therefore gives
\begin{equation}
 \mathbb E[\|R_I-R\|_{\HS}^2\mid\mathcal E_s^{\rm reg}]
 =\frac{n-|I|}{|I|(n-1)}
   \left(1-\mathbb E[\|R\|_{\HS}^2\mid\mathcal E_s^{\rm reg}]\right)
 \le\frac Cn.
 \label{eq:cubic-conditional-subsample-variance}
\end{equation}
Moreover, \(n\ge\underline c Nd\) and \(s\ge\underline s_d\) give
\(N/(ns)\le C/(ds)\le C(\log d)^{-4}=o(1)\).
Markov's inequality applied to
\eqref{eq:cubic-conditional-subsample-variance} therefore yields
\begin{equation}
 \mathbb P\!\left\{N\|R_I-R\|_{\HS}^2>\sqrt{Ns/n}
       \,\middle|\,\mathcal E_s^{\rm reg}\right\}
 \le\varepsilon_{n,s}.
 \label{eq:cubic-conditional-split-deviation}
\end{equation}
\emph{Step 1(b): bound \(\langle R_J,C_I\rangle_{\HS}\) from below.}
Work on
\(\mathcal E_s^{\rm reg}\cap
\{N\|R_I-R\|_{\HS}^2\le\sqrt{Ns/n}\}\).
Since \(\sqrt{Ns/n}/s=\sqrt{N/(ns)}=o(1)\), this gives
\(N\|R_I-R\|_{\HS}^2=o(s)\).
Using \(R_J=R-(|I|/|J|)(R_I-R)\) and expanding the scalar product gives
\begin{equation}
 \langle R_J,C_I\rangle_{\HS}
 =s_n-\frac{N|I|}{|J|}\|R_I-R\|_{\HS}^2
   +\frac{N(|J|-|I|)}{|J|}\langle R,R_I-R\rangle_{\HS}
 =s_n+o(s)\ge c_*s.
 \label{eq:cubic-pilot-cross-identity}
\end{equation}
The final equality and inequality use \(s\le s_n<2s\),
\(N\|R_I-R\|_{\HS}^2=o(s)\),
and the bound
\[
 \left|N\frac{|J|-|I|}{|J|}\langle R,R_I-R\rangle_{\HS}\right|
 \le\frac Cn\sqrt{s_nN\|R_I-R\|_{\HS}^2}=o(s),
\]
which follows from Cauchy--Schwarz and \(|\,|J|-|I|\,|\le1\).

\emph{Step 1(c): bound the second and fourth moments of the weights.}
Continue on the event in Step 1(b).
The triangle inequality and \(N\|R_I-R\|_{\HS}^2=o(s)\) give
\[
 N^{-1}\|C_I\|_{\HS}^2=N\|R_I\|_{\HS}^2=s_n\{1+o(1)\}\le Cs.
\]
Also, \eqref{eq:cubic-hereditary-covariance-profile} gives
\[
 \|C_I\|_{\rm op}
 \le N\left\|\frac1{|I|}\sum_{i\in I}q_{3,d}(X_i)q_{3,d}(X_i)^{\mathsf T}\right\|_{\rm op}^{1/2}
       \left\|\frac1{|I|}\sum_{i\in I}q_{3,d}(X_i')q_{3,d}(X_i')^{\mathsf T}\right\|_{\rm op}^{1/2}
 \le C.
\]
We can therefore apply \eqref{eq:cubic-hereditary-fourth-profile} with
\(B=C_I^{\mathsf T}C_I/C^2\preceq I_N\).
Since \(\tr(C_I^{\mathsf T}C_I)/N\le Cs\), this gives
\[
 \frac1{|J|}\sum_{j\in J}b_j^4
 \le C\left(s^2+\frac{s}{d}
       +\frac{N\log d}{|J|}s\right)
 \le Cs^2
\]
because \(ds/\log d\to\infty\) and \(N/|J|\le C/d\).
Cauchy--Schwarz also gives
\(|J|^{-1}\sum_{j\in J}b_j^2\le
(|J|^{-1}\sum_{j\in J}b_j^4)^{1/2}\le Cs\).
Thus the two inequalities defining \(\mathcal B_I(s)\) hold.
Together with \eqref{eq:cubic-pilot-cross-identity}, they give
\[
 \mathcal E_s^{\rm reg}\cap
 \{N\|R_I-R\|_{\HS}^2\le\sqrt{Ns/n}\}
 \subseteq\mathcal A_s.
\]
Combining this inclusion with
\eqref{eq:cubic-conditional-split-deviation} proves
\eqref{eq:cubic-conditional-split-success}.

\emph{Step 2: truncation and conditional concentration.}
We now work under the unconditioned product law \(P_0^{\otimes n}\) and set
\(\mathcal F_I=\sigma((X_i,X_i')_{i\in I},(X_j')_{j\in J})\).
Then \(\mathcal B_I(s)\in\mathcal F_I\), and conditional on
\(\mathcal F_I\), the variables \((X_j)_{j\in J}\) are independent with
law \(\sigma_d\).
The weights \(b_j\) and the unit vectors
\(v_j=C_Iq_{3,d}(X_j')/b_j\) for \(b_j>0\) are fixed under this
conditioning.
To apply Lemma~\ref{lem:cubic-directional-mgf} to the \(j\)th term
at parameter \(\theta=\lambda b_j\), we need \(\lambda b_j\le c_0d\).
We therefore introduce a cutoff \(B_s>0\) and retain the indices
\(J_s=\{j\in J:0<b_j\le B_s\}\).
On \(\mathcal B_I(s)\), the fourth-moment bound gives
\[
 \left|\sum_{j\in J:\,b_j>B_s}
 b_j\langle q_{3,d}(X_j),v_j\rangle\right|
 \le\sum_{j\in J:\,b_j>B_s}b_j
 \le\frac{\sum_{j\in J}b_j^4}{B_s^3}
 \le\frac{C|J|s^2}{B_s^3}.
\]
Consequently, if \(\langle R_J,C_I\rangle_{\HS}\ge c_*s\), then
\begin{equation}
 \sum_{j\in J_s}b_j\langle q_{3,d}(X_j),v_j\rangle
 \ge |J|\left(c_*s-\frac{Cs^2}{B_s^3}\right).
 \label{eq:cubic-weight-truncation}
\end{equation}

The set \(J_s\) is \(\mathcal F_I\)-measurable.
Conditional independence and Lemma~\ref{lem:cubic-directional-mgf},
applied with \(\theta=\lambda b_j\), give on \(\mathcal B_I(s)\)
\[
 \log\mathbb E\!\left[
 e^{\lambda\sum_{j\in J_s}b_j\langle q_{3,d}(X_j),v_j\rangle}
 \,\middle|\,\mathcal F_I\right]
 \le\frac{C\lambda^2}{N}\sum_{j\in J_s}b_j^2
 \le\frac{C\lambda^2|J|s}{N},
 \qquad 0\le\lambda\le\frac{c_0d}{B_s}.
\]
Taking \(\lambda=c\min\{N,d/B_s\}\), with \(c>0\) sufficiently small,
in Chernoff's inequality yields
\begin{equation}
 \mathbb P\left\{\sum_{j\in J_s}b_j
  \langle q_{3,d}(X_j),v_j\rangle\ge\frac{c_*}{2}|J|s
  \,\middle|\,\mathcal F_I\right\}
 \le\exp\left(-c|J|\min\left\{Ns,\frac{ds}{B_s}\right\}\right).
 \label{eq:cubic-weighted-directional-Bernstein}
\end{equation}
To make the lower bound in \eqref{eq:cubic-weight-truncation} at least
\(c_*|J|s/2\), we require \(B_s^3\ge(2C/c_*)s\).
We therefore take
\(B_s=K_0s^{1/3}\), with \(K_0\) fixed and \(C/K_0^3\le c_*/2\).
Then \(ds/B_s=K_0^{-1}ds^{2/3}\), and
\(Ns^{1/3}/d\to\infty\) uniformly for
\(s\ge\underline s_d=d^{-1}\log^4d\), since \(N\asymp d^3\).
Thus \(|J|\asymp n\), \eqref{eq:cubic-weight-truncation}, and
\eqref{eq:cubic-weighted-directional-Bernstein} give, on
\(\mathcal B_I(s)\),
\begin{equation}
 \mathbb P\left\{
  \langle R_J,C_I\rangle_{\HS}\ge c_*s\,\middle|\,\mathcal F_I\right\}
 \le \exp(-cnd s^{2/3}).
 \label{eq:cubic-cross-fitted-overlap-tail}
\end{equation}
Since \(\mathcal B_I(s)\in\mathcal F_I\),
\eqref{eq:cubic-conditional-split-success} and
\eqref{eq:cubic-cross-fitted-overlap-tail} give
\begin{equation}
\begin{aligned}
 (1-\varepsilon_{n,s})\mathbb P(\mathcal E_s^{\rm reg})
 &\le\mathbb P(\mathcal E_s^{\rm reg}\cap\mathcal A_s)
 \le\mathbb P(\mathcal A_s)\\
 &=\mathbb E\!\left[
  \1_{\mathcal B_I(s)}
  \mathbb P\!\left\{\langle R_J,C_I\rangle_{\HS}\ge c_*s
       \,\middle|\,\mathcal F_I\right\}\right]
 \le\exp(-cnd s^{2/3}).
\end{aligned}
\label{eq:cubic-conditional-probability-transfer}
\end{equation}
Since \(\varepsilon_{n,s}=o(1)\) and \(nd s^{2/3}\to\infty\) uniformly
for \(\underline s_d\le s\le\delta\),
\eqref{eq:cubic-conditional-probability-transfer} gives, after decreasing \(c\),
\begin{equation}
  P_0^{\otimes n}\{s\le s_n<2s,
    {\cal G}_n^\dagger(X),{\cal G}_n^\dagger(X')\}
  \le \exp(-cnd s^{2/3}),
  \qquad \underline s_d\le s\le\delta.
  \label{eq:cubic-root-overlap-tail}
\end{equation}

\emph{Step 3: summation with the weight \(W_{n,\beta}\).}
Equation~\eqref{eq:cubic-local-linear-reward} and the upper bound on
\(r_n\) give
\(W_{n,\beta}\le e^{Cnds}\) on \(\mathcal E_s\).
Choose \(\delta>0\) sufficiently small that
\(Cs\le(c/2)s^{2/3}\) for every \(0<s\le\delta\), where
\(c\) is the constant in \eqref{eq:cubic-root-overlap-tail}.
Set \(s_j=2^j\underline s_d\) for integers \(j\ge0\).
The intervals \([s_j,2s_j)\) with \(s_j\le\delta\) cover
\([\underline s_d,\delta]\).
Their number is \(1+\lfloor\log_2(\delta/\underline s_d)\rfloor=O(\log d)\).
Thus \eqref{eq:cubic-root-overlap-tail} and
\eqref{eq:cubic-local-linear-reward} yield
\begin{align*}
 &\mathbb E_{P_0^{\otimes n}}\big[\1_{\mathcal G_n^\dagger(X)}
   \1_{\mathcal G_n^\dagger(X')}W_{n,\beta}
   \1_{\{\underline s_d\le s_n\le\delta\}}\big]\\
 &\quad\le\sum_{j\ge0:\,s_j\le\delta}
     e^{Cnd s_j}P_0^{\otimes n}(\mathcal E_{s_j}^{\rm reg})
 \le C\log d\,\exp(-cnd\,\underline s_d^{2/3})
 =o(1).
\end{align*}
The last step uses \(\underline s_d=d^{-1}\log^4 d\) and
\(n\ge\underline c Nd\), and proves \eqref{eq:cubic-mesoscopic-deletion}.

\end{proof}

Lemma~\ref{lem:cubic-coarse-terminal-deletions} applies the probability
bounds in Lemmas~\ref{lem:cubic-operator-norm-tail} and
\ref{lem:cubic-fourth-U-moments}, together with
\eqref{eq:cubic-local-linear-reward}, to the region
\(\bigl|\|S_n\|_{\HS}^2-n\bigr|\le b_\star d^2\log d/\beta^2\).

\begin{lemma}[Coarse full-sample deletions on the logarithmic band]
\label{lem:cubic-coarse-terminal-deletions}
Under \eqref{eq:cubic-local-standing-assumptions}, fix \(\epsilon_\star>0\)
as at the start of this subsection.  There exists a constant \(b_\star>0\)
such that
\begin{equation}
 \mathbb E\Bigl[
  \1_{\mathcal E_n^{\rm reg}}
  W_{n,\beta}\1_{\{\lvert\|S_n\|_{\HS}^2-n\rvert\le b_\star d^2\log d/\beta^2\}}
  \1_{\{
     \|S_n\|_{\rm op}>\epsilon_\star d/\beta^2
     \ \text{or}\ 
     \beta^8n|T_{4,n}|>1\}}
 \Bigr]=o(1).
\label{eq:cubic-coarse-terminal-deletions}
\end{equation}
\end{lemma}

\begin{proof}
For sufficiently large \(d\), \(\epsilon_d<\epsilon_\star\).
Lemma~\ref{lem:cubic-operator-norm-tail} therefore gives
\[
 \mathbb P\left\{\mathcal E_n^{\rm reg},\
       \|S_n\|_{\rm op}>\epsilon_\star d/\beta^2\right\}
 \le2\exp\left(-\frac{c\epsilon_d^2d^{5/2}}{\beta^2}\right)
 \le e^{-c_\star d^2\log d}.
\]
The last inequality follows from
\eqref{eq:cubic-local-epsilon-conditions} and \(|\beta|\le1\).

Since \(\epsilon_d<1\), Lemma~\ref{lem:cubic-fourth-U-moments} gives
\[
 \mathbb P\{\mathcal E_n^{\rm reg},\ \beta^8n|T_{4,n}|>1\}
 \le\mathbb P\{\mathcal E_n^{\rm reg},\
                     \beta^8n|T_{4,n}|>\epsilon_d\}
 \le e^{-d^2\log d}.
\]
Finally, throughout the full-sample band,
\[
 W_{n,\beta}\le\exp\{\beta^2(\|S_n\|_{\HS}^2-n)/2\}
 \le\exp\left(b_\star d^2\log d/2\right).
\]
Multiplying the two tail bounds by this upper bound for \(W_{n,\beta}\)
shows that the left-hand side of \eqref{eq:cubic-coarse-terminal-deletions}
is at most
\[
 e^{b_\star d^2\log d/2}
       \left(e^{-c_\star d^2\log d}+e^{-d^2\log d}\right)
 \le2e^{-\{\min(c_\star,1)-b_\star/2\}d^2\log d}=o(1),
\]
provided \(b_\star<2\min(c_\star,1)\).
\end{proof}

\subsection{Proof of the first inequality}
\label{subsec:cubic-direct-orbit-closure}

This subsection proves the first inequality in \eqref{eq:cubic-local-roadmap}, reducing
the integral over \(s_n<\delta\) to
\(\bigl|\|S_n\|_{\HS}^2-n\bigr|\le K_dd^2/\beta^2\) at an additive cost of \(o(1)\).
To apply \eqref{eq:cubic-summed-variational-bound} to the positive tail,
we first bound the weighted integral over
\(\|S_n\|_{\HS}^2-n>n/(\beta^2\log^4d)\) by \(o(1)\).
On the remaining region, imposing
\(\beta^2\ell(S_n)\le\epsilon_\star d\) changes the weighted integral by \(o(1)\),
with \(\epsilon_\star\) chosen at the start of
Section~\ref{subsec:cubic-local-deletion}.

Probabilities and expectations in this subsection are under \(P_0^{\otimes n}\).
Lemma~\ref{lem:cubic-high-energy-entropy} first bounds the weighted integral over
\(\{s_n<\delta,\ \|S_n\|_{\HS}^2-n>n/(\beta^2\log^4d)\}\).

\begin{lemma}[Weighted upper-tail bound for \(\|S_n\|_{\HS}^2-n>v_\star\)]
\label{lem:cubic-high-energy-entropy}
Set \(v_\star=n/(\beta^2\log^4d)\). There are \(\delta_0,c>0\), depending only on the fixed parameters,
such that, uniformly for \(0<\delta\le\delta_0\),
\begin{equation}
 \mathbb E\big[
  \1_{\mathcal E_n^{\rm reg}}W_{n,\beta}
  \1_{\{s_n<\delta,\ \|S_n\|_{\HS}^2-n>v_\star\}}\big]
 \le e^{-cn/\log^4 d}.
 \label{eq:cubic-high-energy-entropy}
\end{equation}
\end{lemma}

\begin{proof}
Write \(\Delta=\|S_n\|_{\HS}^2-n\) and
\(\mathcal A_v=\mathcal E_n^{\rm reg}\cap\{s_n<\delta,\ v<\Delta\le2v\}\).
We use dyadic peeling, first proving
\[
 \mathbb P(\mathcal A_v)\le2e^{-8\beta^2v},
 \qquad v\ge v_\star,
\]
and then summing over \(v=2^jv_\star\), \(j\ge0\).
Step 1 constructs a balanced split with a large cross inner product.
Step 2 bounds the conditional feature mean and fourth moment, and
Step 3 converts these bounds into a lower bound for relative entropy.
Step 4 deduces the displayed probability estimate and sums the weighted
integrals, using \(W_{n,\beta}\le e^{\beta^2v}\) on \(\mathcal A_v\)
from \eqref{eq:cubic-local-linear-reward}.

\emph{Step 1: a balanced split on a fixed interval.}
Fix \(v\ge v_\star\).
Equation~\eqref{eq:cubic-degree-three-feature-coherence}, applied to both
samples on \(\mathcal E_n^{\rm reg}\), gives
\[
 \sum_{i<j}h_{ij}^2\le Cn^2\log^{12} d/N^2.
\]
Choose a uniform balanced split \(I\cup J=[n]\), independently of the
paired sample, and write \(S_I=\sum_{i\in I}w_i\),
\(S_J=\sum_{j\in J}w_j\), and \(C_{I,J}=\langle S_I,S_J\rangle_{\HS}\).
Write \(\mathbb P_I\) for the law of this split and \(\mathbb E_I\) for
expectation with respect to it, with the sample fixed.
For even \(n\), put \(\varepsilon_i=\1_{\{i\in I\}}-\1_{\{i\in J\}}\).
Then
\[
 C_{I,J}=\frac14
   \left(\Delta-\sum_{i\ne j}\varepsilon_i\varepsilon_jh_{ij}\right),
 \qquad
 \mathbb E_I C_{I,J}=\frac{n\Delta}{4(n-1)}.
\]
For distinct indices,
\(\mathbb E_I\varepsilon_i\varepsilon_j=-1/(n-1)\) and
\(\mathbb E_I\varepsilon_i\varepsilon_j\varepsilon_k\varepsilon_l
 =3/\{(n-1)(n-3)\}\).
For an edge \(e=\{i,j\}\), write \(\zeta_e=\varepsilon_i\varepsilon_j\)
and \(h_e=h_{ij}\).  Then
\[
 \operatorname{Cov}_I(\zeta_e,\zeta_f)=
 \begin{cases}
  1-(n-1)^{-2},&e=f,\\
  -n/(n-1)^2,&|e\cap f|=1,\\
  2n/\{(n-1)^2(n-3)\},&e\cap f=\varnothing.
 \end{cases}
\]
Using \(2|h_eh_f|\le h_e^2+h_f^2\), with at most \(2n\) adjacent
and \(n^2\) disjoint partners for each edge, gives
\begin{align*}
 \operatorname{Var}_I(C_{I,J})
 &=\frac14\sum_{e,f}h_eh_f\operatorname{Cov}_I(\zeta_e,\zeta_f)\\
 &\le C\left(\sum_eh_e^2+\frac1n\sum_{|e\cap f|=1}|h_eh_f|
               +\frac1{n^2}\sum_{e\cap f=\varnothing}|h_eh_f|\right)
 \le C\sum_eh_e^2.
\end{align*}
For odd \(n\), adjoin a zero vector, split the resulting \(n+1\) vectors
equally, and then remove the zero vector.
The same calculation gives
\(\mathbb E_I C_{I,J}\ge\Delta/4\) and
\(\operatorname{Var}_I(C_{I,J})\le C\sum_eh_e^2\), with
\(|I|,|J|\ge n/3\).
Thus, for either parity of \(n\), Chebyshev's inequality gives,
uniformly on \(\mathcal A_v\),
\[
 \mathbb P_I\{C_{I,J}<v/8\}
 \le\frac{64\operatorname{Var}_I(C_{I,J})}{v^2}
 \le\frac{Cn^2\log^{12}d}{N^2v_\star^2}
 =\frac{C\beta^4\log^{20}d}{N^2}=o(1),
\]
where \(|\beta|\le1\) and \(N\asymp d^3\) give
\(\beta^4\log^{20}d/N^2=O(d^{-6}\log^{20}d)\to0\), uniformly in \(\beta\).
Hence at least \(1/2\) of the splits satisfy
\begin{equation}
 C_{I,J}\ge v/8,\qquad
 \|S_I\|_{\HS}^2+\|S_J\|_{\HS}^2\le n+2v.
 \label{eq:cubic-high-energy-good-split}
\end{equation}
The second inequality follows from
\(\|S_I\|_{\HS}^2+\|S_J\|_{\HS}^2=\|S_n\|_{\HS}^2-2C_{I,J}\).

\emph{Step 2: conditional feature means and fourth moments.}
The bound \(\mathbb P(\mathcal A_v)\le2e^{-8\beta^2v}\) holds when
\(\mathbb P(\mathcal A_v)=0\), so assume \(\mathbb P(\mathcal A_v)>0\).
The joint law of the paired sample and the independent split is
\(P_0^{\otimes n}\otimes\mathbb P_I\).
Let \(\mathcal D_v\) be the event that \(\mathcal A_v\) and both inequalities in
\eqref{eq:cubic-high-energy-good-split} hold. Step 1 gives
\((P_0^{\otimes n}\otimes\mathbb P_I)(\mathcal D_v)
\ge\mathbb P(\mathcal A_v)/2>0\), so we may define
\(\mathsf Q_v=(P_0^{\otimes n}\otimes\mathbb P_I)(\,\cdot\mid\mathcal D_v)\).

Under \(\mathsf Q_v\), fix \(I\) and \((X_i,X_i')_{i\in I}\), and write
\[
 \nu_{J\mid I}=\mathcal L_{\mathsf Q_v}\bigl(
     (X_j,X_j')_{j\in J}\mid I,(X_i,X_i')_{i\in I}\bigr).
\]
For \(m=|J|\), let \(\nu_{j\mid I}\) be its \(j\)-th pair marginal,
put \(\pi_I=m^{-1}\sum_{j\in J}\nu_{j\mid I}\), and denote the
\(X'\)-marginal of \(\pi_I\) by \(\pi_{I,X'}\).
With \(t=v/n\), define
\[
 M_I=\mathbb E_{\pi_I}q_{3,d}(X)q_{3,d}(X')^{\mathsf T}
     =\frac1m\mathbb E_{\nu_{J\mid I}}S_J,\qquad
 s=N\|M_I\|_{\HS}^2.
\]
Our goal is to bound \(s\) in terms of \(t\) and control
\(\mathbb E_{\pi_{I,X'}}\|NM_Iq_{3,d}(X')\|^4\).
Taking \(\nu_{J\mid I}\)-expectations in
\eqref{eq:cubic-high-energy-good-split}, with \(S_I\) fixed, and
applying Cauchy--Schwarz and Jensen's inequality give
\begin{gather*}
 \frac v8\le\mathbb E_{\nu_{J\mid I}}C_{I,J}
 =m\langle S_I,M_I\rangle_{\HS}
 \le m\|S_I\|_{\HS}\|M_I\|_{\HS}
 \le m\sqrt{n+2v}\,\|M_I\|_{\HS},\\
 m^2\|M_I\|_{\HS}^2
 =\bigl\|\mathbb E_{\nu_{J\mid I}}S_J\bigr\|_{\HS}^2
 \le\mathbb E_{\nu_{J\mid I}}\|S_J\|_{\HS}^2\le n+2v.
\end{gather*}
Since \(n/3\le m\le n\), these imply
\begin{equation}
 c\frac{Nv^2}{n^2(n+v)}
 \le s\le C\frac{N(n+v)}{n^2},
 \quad\text{hence}\quad
 c\frac{\beta^2t^2}{d(1+t)}
 \le s\le C\frac{\beta^2(1+t)}d,\qquad s\le C\delta.
 \label{eq:cubic-high-energy-conditional-scale}
\end{equation}
The bound \(s\le C\delta\) follows from
\(n+v<\|S_n\|_{\HS}^2<\delta n^2/N\) on \(\mathcal A_v\).
The constants are uniform over the conditioned coordinates and \(r_n\in K\).

To obtain a fourth-moment bound in terms of \(s\), set \(B_I=NM_I\).
Averaging \eqref{eq:cubic-hereditary-covariance-profile} on \(J\) under
\(\nu_{J\mid I}\) gives
\[
 \mathbb E_{\pi_I}q_{3,d}(X)q_{3,d}(X)^{\mathsf T}\preceq\frac{C}{N}I_N,
 \qquad
 \mathbb E_{\pi_{I,X'}}q_{3,d}(X')q_{3,d}(X')^{\mathsf T}\preceq\frac{C}{N}I_N.
\]
For unit vectors \(\mathbf e,\mathbf f\in\mathbb R^N\),
Cauchy--Schwarz now gives
\begin{align*}
 |\mathbf e^{\mathsf T}M_I\mathbf f|
 &=\left|\mathbb E_{\pi_I}\langle\mathbf e,q_{3,d}(X)\rangle
                 \langle\mathbf f,q_{3,d}(X')\rangle\right|\\
 &\le\big(\mathbb E_{\pi_I}\langle\mathbf e,q_{3,d}(X)\rangle^2
           \mathbb E_{\pi_I}\langle\mathbf f,q_{3,d}(X')\rangle^2\big)^{1/2}
 \le\frac{C}{N}.
\end{align*}
Taking the supremum over these unit vectors gives
\(\|M_I\|_{\rm op}\le C/N\) and \(\|B_I\|_{\rm op}\le C\).
Choose a fixed \(C_0\) with \(\|B_I\|_{\rm op}\le C_0\).
Since \eqref{eq:cubic-hereditary-fourth-profile} holds simultaneously for
all positive contractions, we may apply it to \(B_I^{\mathsf T}B_I/C_0^2\)
for the \(X'\)-sample and average under \(\nu_{J\mid I}\).
Using \(\operatorname{tr}(B_I^{\mathsf T}B_I)=Ns\) and
\(N/m\le C_K\beta^2/d\), we obtain, with \(q_d=(1+\beta^2\log d)/d\),
\begin{equation}
  \mathbb E_{\pi_{I,X'}}\|B_Iq_{3,d}(X')\|^4
  \le C\left(s^2+\frac{s}{d}+\frac{N\log d}{m}s\right)
  \le Cs(s+q_d).
  \label{eq:cubic-high-energy-fourth-profile}
\end{equation}

\emph{Step 3: a lower bound for \(D(\pi_I\Vert P_0)\).}
We prove \(D(\pi_I\Vert P_0)\ge24\beta^2v/n\) by first decomposing the
relative entropy with respect to \(X'\).  Since \(P_0=\sigma_d\otimes\sigma_d\),
the entropy chain rule gives
\[
 D(\pi_I\Vert P_0)
 =D(\pi_{I,X'}\Vert\sigma_d)
   +\mathbb E_{\pi_{I,X'}}D(\pi_I(\,\cdot\mid X')\Vert\sigma_d)
 \ge\mathbb E_{\pi_{I,X'}}D(\pi_I(\,\cdot\mid X')\Vert\sigma_d).
\]
Put \(r(X')=\|\mathbb E_{\pi_I}[q_{3,d}(X)\mid X']\|\le1\).
For fixed \(X'\) with \(r(X')>0\), we have
\begin{equation}
 D(\pi_I(\,\cdot\mid X')\Vert\sigma_d)
 \ge\theta r(X')-\frac{C\theta^2}{N}
 \ge\frac{c_1Nr(X')^2}{2\{1+(N/d)r(X')\}},
 \label{eq:cubic-high-energy-directional-entropy}
\end{equation}
where the first inequality follows from
\eqref{eq:cubic-linear-entropy-duality} applied to
\(\pi_I(\,\cdot\mid X')\otimes\sigma_d\) with exponent
\(\theta\langle\mathbb E_{\pi_I}[q_{3,d}(X)\mid X']/r(X'),q_{3,d}(x)\rangle\),
together with \eqref{eq:cubic-directional-mgf} in the same unit direction.
The second inequality follows by choosing
\(\theta=c_1Nr(X')/\{1+(N/d)r(X')\}\le c_1d\),
with \(c_1\le c_0\) and \(Cc_1\le1/2\).
When \(r(X')=0\), \eqref{eq:cubic-high-energy-directional-entropy}
holds with \(\theta=0\) by nonnegativity of relative entropy.

Write \(\mu_I=\mathbb E_{\pi_{I,X'}}r(X')\).
By the entropy decomposition above,
\eqref{eq:cubic-high-energy-directional-entropy}, and convexity of
\(x\mapsto x^2/\{1+(N/d)x\}\) on \([0,\infty)\), Jensen's inequality gives
\[
 D(\pi_I\Vert P_0)
 \ge cN\mathbb E_{\pi_{I,X'}}\frac{r(X')^2}{1+(N/d)r(X')}
 \ge c\frac{N\mu_I^2}{1+(N/d)\mu_I}.
\]
To bound \(\mu_I\) from below, use
\(s=\mathbb E_{\pi_I}\langle q_{3,d}(X),B_Iq_{3,d}(X')\rangle\).
H\"older's inequality, \(0\le r(X')\le1\), and
\eqref{eq:cubic-high-energy-fourth-profile} give
\begin{align*}
 s&\le\mathbb E_{\pi_{I,X'}}[r(X')\|B_Iq_{3,d}(X')\|]\\
  &\le(\mathbb E_{\pi_{I,X'}}r^{4/3})^{3/4}
       (\mathbb E_{\pi_{I,X'}}\|B_Iq_{3,d}(X')\|^4)^{1/4}
   \le\mu_I^{3/4}\{Cs(s+q_d)\}^{1/4}.
\end{align*}
Consequently, \(\mu_I\ge cs/(s+q_d)^{1/3}\).
Since \(t\ge1/(\beta^2\log^4d)\) and \(t^2/(1+t)\) is increasing,
\eqref{eq:cubic-high-energy-conditional-scale} gives
\[
 s\ge\frac{c\beta^2t^2}{d(1+t)}
 \ge\frac{c}{d(\beta^2\log^8d+\log^4d)}
 \ge\frac{c'}{d\log^8d}.
\]
Since \(s+q_d\le C\), we have \(\mu_I\ge cs\) and hence
\((N/d)\mu_I\ge cN/(d^2\log^8d)\to\infty\) uniformly.
Thus, for large \(d\),
\[
 D(\pi_I\Vert P_0)
 \ge c\frac{N\mu_I^2}{2(N/d)\mu_I}
 \ge c'd\mu_I\ge c''\frac{ds}{(s+q_d)^{1/3}}.
\]
For \(t\le1\), \eqref{eq:cubic-high-energy-conditional-scale} gives
\(ds\ge c\beta^2t^2\) and \(s+q_d\le C\log d/d\), so
\[
 \frac{D(\pi_I\Vert P_0)}{\beta^2t}
 \ge ct(d/\log d)^{1/3}
 \ge c\,d^{1/3}/\log^{13/3} d\longrightarrow\infty,
\]
where \(t\ge1/(\beta^2\log^4d)\) and \(\beta^2\le1\) give the last inequality.
For \(t\ge1\), \eqref{eq:cubic-high-energy-conditional-scale} gives
\(ds\ge c\beta^2t\) and \(s\le C\delta\), hence
\[
 \frac{D(\pi_I\Vert P_0)}{\beta^2t}
 \ge \frac{c}{(\delta+\log d/d)^{1/3}}.
\]
Choose \(\delta_0>0\) so that \(c/(2\delta_0)^{1/3}\ge24\),
with \(c\) as in the last display.  For \(0<\delta\le\delta_0\) and
all sufficiently large \(d\), we have \(\log d/d\le\delta_0\), so the
ratio is at least \(24\) when \(t\ge1\).
The lower bound for \(t\le1\) also exceeds \(24\) for large \(d\).
Thus, uniformly in both ranges,
\(D(\pi_I\Vert P_0)\ge24\beta^2v/n\).

\emph{Step 4: relative entropy and dyadic summation.}
We first pass from the bound for \(\pi_I\) to a bound for the full
conditioned law \(\mathsf Q_v\).  With \(I\) and
\((X_i,X_i')_{i\in I}\) fixed, the relative-entropy decomposition gives
\[
 D(\nu_{J\mid I}\Vert P_0^{\otimes m})
 =D\bigl(\nu_{J\mid I}\Vert
               {\textstyle\bigotimes_{j\in J}}\nu_{j\mid I}\bigr)
   +\sum_{j\in J}D(\nu_{j\mid I}\Vert P_0).
\]
The first term on the right is nonnegative.  Since \(m\ge n/3\),
convexity in the first argument, \(\pi_I=m^{-1}\sum_{j\in J}\nu_{j\mid I}\),
and Step 3 give
\[
 D(\nu_{J\mid I}\Vert P_0^{\otimes m})
 \ge\sum_{j\in J}D(\nu_{j\mid I}\Vert P_0)
 \ge m D(\pi_I\Vert P_0)\ge8\beta^2v.
\]

Under \(P_0^{\otimes n}\otimes\mathbb P_I\), conditional on \(I\), the
pairs indexed by \(I\) and \(J\) have independent laws
\(P_0^{\otimes |I|}\) and \(P_0^{\otimes m}\).
Applying the entropy chain rule first to \(I\) and then to
\((X_i,X_i')_{i\in I}\) gives
\begin{align*}
 &D(\mathsf Q_v\Vert P_0^{\otimes n}\otimes\mathbb P_I)
 =D(\mathcal L_{\mathsf Q_v}(I)\Vert\mathbb P_I)\\
 &\quad+\mathbb E_{\mathsf Q_v}D\left(
      \mathcal L_{\mathsf Q_v}((X_i,X_i')_{i\in I}\mid I)
          \,\middle\Vert\,P_0^{\otimes |I|}\right)
       +\mathbb E_{\mathsf Q_v}D(\nu_{J\mid I}\Vert P_0^{\otimes m})\\
 &\quad\ge\mathbb E_{\mathsf Q_v}D(\nu_{J\mid I}\Vert P_0^{\otimes m})
 \ge8\beta^2v,
\end{align*}
where the first inequality drops the two nonnegative relative entropies.

Since \(\mathsf Q_v=(P_0^{\otimes n}\otimes\mathbb P_I)
(\,\cdot\mid\mathcal D_v)\), its density and relative entropy are
\begin{gather*}
 \frac{d\mathsf Q_v}{d(P_0^{\otimes n}\otimes\mathbb P_I)}
 =\frac{\1_{\mathcal D_v}}
        {(P_0^{\otimes n}\otimes\mathbb P_I)(\mathcal D_v)},\\
 D(\mathsf Q_v\Vert P_0^{\otimes n}\otimes\mathbb P_I)
 =\mathbb E_{\mathsf Q_v}\log
       \frac{1}{(P_0^{\otimes n}\otimes\mathbb P_I)(\mathcal D_v)}
 =-\log(P_0^{\otimes n}\otimes\mathbb P_I)(\mathcal D_v).
\end{gather*}
To obtain a bound for \(\mathbb P(\mathcal A_v)\) from this identity,
recall that, for each sample in
\(\mathcal A_v\), Step 1 gives
\(\mathbb P_I\{C_{I,J}\ge v/8\}\ge1/2\), and the second inequality in
\eqref{eq:cubic-high-energy-good-split} then holds as well.  Consequently,
\[
 \frac12\mathbb P(\mathcal A_v)
 \le\mathbb E\big[\1_{\mathcal A_v}\mathbb P_I\{C_{I,J}\ge v/8\}\big]
 =(P_0^{\otimes n}\otimes\mathbb P_I)(\mathcal D_v)
 \le e^{-8\beta^2v}.
\]
This proves \(\mathbb P(\mathcal A_v)\le2e^{-8\beta^2v}\).
For the final summation, use the disjoint decomposition
\[
 \mathcal E_n^{\rm reg}\cap\{s_n<\delta,\ \Delta>v_\star\}
 =\bigsqcup_{j\ge0}\mathcal A_{2^jv_\star}.
\]
Using \(W_{n,\beta}\le e^{\beta^2v}\) on \(\mathcal A_v\) and
\(\beta^2v_\star=n/\log^4d\), we obtain
\begin{align*}
 \mathbb E\big[\1_{\mathcal E_n^{\rm reg}}W_{n,\beta}
       \1_{\{s_n<\delta,\,\Delta>v_\star\}}\big]
 &\le\sum_{j\ge0}
       e^{\beta^2 2^jv_\star}\mathbb P(\mathcal A_{2^jv_\star})\\
 &\le2\sum_{j\ge0}e^{-7\cdot2^j n/\log^4 d}
 \le4e^{-7n/\log^4 d}\le e^{-cn/\log^4 d}.
\end{align*}
The last line uses \(2^j\ge j+1\) and \(n/\log^4 d\to\infty\).
This proves \eqref{eq:cubic-high-energy-entropy}.
\end{proof}

On the remaining region \(\|S_n\|_{\HS}^2-n\le v_\star\),
Lemma~\ref{lem:cubic-clipped-orbit-tail}, applied with
\(\epsilon=\epsilon_\star\), bounds the weighted integral over
\(\{\beta^2\ell(S_n)>\epsilon_\star d\}\) by \(o(1)\).

\begin{lemma}[Tail bound for \(\ell(S_n)\)]
\label{lem:cubic-clipped-orbit-tail}
For every fixed \(\epsilon>0\), there is \(c>0\) such that, for all
sufficiently large \(d\),
\begin{equation}
 \mathbb P\{\mathcal E_n^{\rm reg},\,
                 \ell(S_n)>\epsilon d/\beta^2\}
 \le \exp(-cn/\log^3 d).
 \label{eq:cubic-clipped-orbit-tail}
\end{equation}
Consequently, after decreasing \(c\),
\begin{equation}
 \mathbb E\big[
  \1_{\mathcal E_n^{\rm reg}} W_{n,\beta}
  \1_{\{\|S_n\|_{\HS}^2-n\le v_\star\}}
  \1_{\{\ell(S_n)>\epsilon d/\beta^2\}}\big]
 \le \exp(-cn/\log^3 d).
 \label{eq:cubic-clipped-orbit-weighted}
\end{equation}
\end{lemma}

\begin{proof}
Put \(T=\epsilon d/\beta^2=\epsilon n/(r_nN)\).
Step 1 approximates \(q_{3,d}(x)^{\mathsf T}S_nq_{3,d}(y)\)
by a sum of bounded terms, uniformly in \(x,y\), with error \(o(T)\).
Step 2 uses a finite net and Bernstein's inequality to prove
\eqref{eq:cubic-clipped-orbit-tail}.  Step 3 then applies
\(W_{n,\beta}\le e^{n/(2\log^4d)}\) on
\(\{\|S_n\|_{\HS}^2-n\le v_\star\}\) to obtain
\eqref{eq:cubic-clipped-orbit-weighted}.

\emph{Step 1: a uniform truncation bound.}
For \(x,y\in\mathbb S^{d-1}\), write
\(\alpha_i(x)=Z_{3,d}(x,X_i)\) and
\(\alpha_i'(y)=Z_{3,d}(y,X_i')\).
To bound the sum of squares of the terms discarded by truncation, set
\(\tau_d^2=C_1\log^3d/N\),
\(J=\{i:|\alpha_i(x)|>\tau_d\}\), and \(k=|J|\).
On \(\mathcal E_n^{\rm reg}\), \eqref{eq:cubic-zonal-hereditary} holds
simultaneously for all subsets, so it applies to this sample-dependent
\(J\) whenever \(k\ge1\).  Choosing the fixed constant \(C_1\)
sufficiently large, we obtain
\[
 \sum_{i\in J}\alpha_i(x)^2
 \le C\left(1+\frac{k}{N}\log^3\frac{en}{k}\right)
 \le C+\frac12k\tau_d^2
 \le C+\frac12\sum_{i\in J}\alpha_i(x)^2,
\]
where the second inequality uses \(\log(en/k)\le1+A\log d\), and
the last uses \(\alpha_i(x)^2>\tau_d^2\) for \(i\in J\).
Subtracting half the sum from both sides gives
\(\sum_{i\in J}\alpha_i(x)^2\le2C\), also valid when \(J\) is empty.
Together with \eqref{eq:cubic-zonal-hereditary} for \(J=[n]\), this gives,
after enlarging \(C\) and uniformly in \(x\),
\begin{equation}
 \sum_i\alpha_i(x)^2\1_{\{|\alpha_i(x)|>\tau_d\}}\le C,
 \qquad \sum_i\alpha_i(x)^2\le Cn/N.
 \label{eq:cubic-orbit-clipping-profile}
\end{equation}
Here \(1+n/N\le Cn/N\), since \(n/N\ge(\min K)d\to\infty\).
The same bounds hold for \(\alpha_i'(y)\).  Define the truncated variables
\(a_i(x)=\alpha_i(x)\1_{\{|\alpha_i(x)|\le\tau_d\}}\) and
\(a_i'(y)=\alpha_i'(y)\1_{\{|\alpha_i'(y)|\le\tau_d\}}\).
Expanding
\(\sum_i(\alpha_i\alpha_i'-a_ia_i')
=\sum_i(\alpha_i-a_i)\alpha_i'+\sum_i a_i(\alpha_i'-a_i')\),
then applying Cauchy--Schwarz and
\eqref{eq:cubic-orbit-clipping-profile}, gives on
\(\mathcal E_n^{\rm reg}\)
\begin{equation}
 \sup_{x,y}\left|
 q_{3,d}(x)^{\mathsf T}S_nq_{3,d}(y)
       -\sum_i a_i(x)a_i'(y)\right|
 \le C\sqrt{n/N}=o(T).
 \label{eq:cubic-orbit-truncation-error}
\end{equation}
The last equality follows from
\(\sqrt{n/N}/T=(r_n/\epsilon)\sqrt{N/n}\to0\).

\emph{Step 2: a deterministic net and Bernstein's inequality.}
To control the supremum
\(\ell(S_n)=\sup_{x,y\in\mathbb S^{d-1}}
|q_{3,d}(x)^{\mathsf T}S_nq_{3,d}(y)|\),
let \(\mathcal N_d\) be a Euclidean \(h\)-net of
\(\mathbb S^{d-1}\), with \(h=T/(24n)\).
The sphere covering bound \cite[Corollary~4.2.13]{Vershynin2018}
gives
\[
 |\mathcal N_d|\le(1+2/h)^d,\qquad
 \log|\mathcal N_d|^2\le C_{\epsilon,K}d\log N,
\]
since \(h=\epsilon/(24r_nN)\).
For net points \(x_0,y_0\) with
\(\|x-x_0\|_2,\|y-y_0\|_2\le h\),
\eqref{eq:cubic-feature-Lipschitz}, the unit feature norms, and
\(\|S_n\|_{\rm op}\le n\) give
\[
 \left|q_{3,d}(x)^{\mathsf T}S_nq_{3,d}(y)
       -q_{3,d}(x_0)^{\mathsf T}S_nq_{3,d}(y_0)\right|
 \le3n(\|x-x_0\|_2+\|y-y_0\|_2)\le T/4.
\]
Thus \(\ell(S_n)>T\) implies
\(|q_{3,d}(x_0)^{\mathsf T}S_nq_{3,d}(y_0)|>3T/4\)
for some \(x_0,y_0\in\mathcal N_d\).
Applying \eqref{eq:cubic-orbit-truncation-error} at this pair, with
\(C\sqrt{n/N}\le T/4\) for large \(d\), gives
\begin{equation}
 \{\mathcal E_n^{\rm reg},\ \ell(S_n)>T\}
 \subseteq
 \bigcup_{x_0,y_0\in\mathcal N_d}
 \left\{\left|\sum_i a_i(x_0)a_i'(y_0)\right|>T/2\right\}.
 \label{eq:cubic-orbit-net-reduction}
\end{equation}

It remains to bound the probabilities on the right-hand side of
\eqref{eq:cubic-orbit-net-reduction} under \(P_0^{\otimes n}\).
For a fixed pair \(x_0,y_0\in\mathcal N_d\), the variables
\(\xi_i=a_i(x_0)a_i'(y_0)\) are independent under \(P_0^{\otimes n}\).
Oddness of the cubic kernel and the symmetric truncation give
\(\mathbb E a_i(x_0)=\mathbb E a_i'(y_0)=0\).
Moreover, \eqref{eq:cubic-addition-formula} and orthonormality give
\(\mathbb E\alpha_1(x_0)^2=\mathbb E\alpha_1'(y_0)^2=1/N\).
Independence of \(X_i\) and \(X_i'\) therefore yields
\[
 \mathbb E\xi_i=0,\qquad |\xi_i|\le\tau_d^2,\qquad
 \sum_i\mathbb E\xi_i^2
 =n\mathbb E a_1(x_0)^2\,\mathbb E a_1'(y_0)^2\le n/N^2.
\]
Bernstein's inequality \cite[Section~2.8]{BoucheronLugosiMassart2013}
yields
\[
 \mathbb P\left(\left|\sum_i\xi_i\right|>T/2\right)
 \le2\exp\left(-\frac{T^2}{8(n/N^2+\tau_d^2T/6)}\right)
 \le2\exp(-cn/\log^3 d).
\]
The last inequality uses \(T=\epsilon n/(r_nN)\) and
\(n/N^2+\tau_d^2T/6\le C_{\epsilon,K}n\log^3d/N^2\).
Combining this probability bound with
\eqref{eq:cubic-orbit-net-reduction} and the union bound gives
\[
 \mathbb P\{\mathcal E_n^{\rm reg},\ \ell(S_n)>T\}
 \le2|\mathcal N_d|^2e^{-cn/\log^3 d}
 \le\exp(C_{\epsilon,K}d\log N-cn/\log^3 d)
 \le e^{-c'n/\log^3 d}.
\]
Here \(d\log N=o(n/\log^3d)\), since
\(n\ge(\min K)Nd\) and \(N\asymp d^3\).
This proves \eqref{eq:cubic-clipped-orbit-tail}.

\emph{Step 3: restoring the likelihood weight.}
To deduce \eqref{eq:cubic-clipped-orbit-weighted},
\eqref{eq:cubic-local-linear-reward} gives
\(W_{n,\beta}\le e^{\beta^2v_\star/2}=e^{n/(2\log^4d)}\) on
\(\{\|S_n\|_{\HS}^2-n\le v_\star\}\).
By \eqref{eq:cubic-clipped-orbit-tail}, we have
\begin{align*}
 \mathbb E\big[\1_{\mathcal E_n^{\rm reg}}W_{n,\beta}
    \1_{\{\|S_n\|_{\HS}^2-n\le v_\star,\ \ell(S_n)>T\}}\big]
 &\le e^{n/(2\log^4 d)}
       \mathbb P\{\mathcal E_n^{\rm reg},\ \ell(S_n)>T\}\\
 &\le e^{n/(2\log^4 d)-c'n/\log^3 d}
 \le e^{-c''n/\log^3 d}.
\end{align*}
This proves \eqref{eq:cubic-clipped-orbit-weighted}.
\end{proof}


For sufficiently small fixed \(\delta>0\), recall the region in
Figure~\ref{fig:cubic-first-inequality-regions},
\[
 \mathcal R_n=\mathcal E_n^{\rm reg}\cap
  \{s_n<\underline s_d,\ \|S_n\|_{\HS}^2-n\le v_\star,
          \ \beta^2\ell(S_n)\le\epsilon_\star d\}.
\]
Lemma~\ref{lem:cubic-submacroscopic-deletion} bounds the integral over
\(\underline s_d\le s_n<\delta\) by \(o(1)\).
Lemmas~\ref{lem:cubic-high-energy-entropy} and
\ref{lem:cubic-clipped-orbit-tail}, the latter with \(\epsilon=\epsilon_\star\),
give the same bound on
\(\{s_n<\delta,\ \|S_n\|_{\HS}^2-n>v_\star\}\) and
\(\{\|S_n\|_{\HS}^2-n\le v_\star,\ \beta^2\ell(S_n)>\epsilon_\star d\}\),
respectively.  Since \(W_{n,\beta}\ge0\), we obtain
\begin{equation}
 \mathbb E[\1_{\mathcal E_n^{\rm reg}}W_{n,\beta}\1_{\{s_n<\delta\}}]
 \le\mathbb E[\1_{\mathcal R_n}W_{n,\beta}]+o(1).
 \label{eq:cubic-first-regular-reduction}
\end{equation}

It remains to control the tails of \(\|S_n\|_{\HS}^2-n\) on
\(\mathcal R_n\).
Lemma~\ref{lem:cubic-positive-energy-tail} combines
Lemmas~\ref{lem:cubic-parameterized-row-ledger}
and~\ref{lem:cubic-fourth-U-exponential} to treat the entire positive tail
\(\|S_n\|_{\HS}^2-n>K_dd^2/\beta^2\).
The negative tail follows directly from
\eqref{eq:cubic-local-linear-reward}.
Corollary~\ref{cor:cubic-local-central-reduction} combines these two bounds
with \eqref{eq:cubic-first-regular-reduction}.

\begin{lemma}[Weighted upper-tail bound]
\label{lem:cubic-positive-energy-tail}
Under \eqref{eq:cubic-local-standing-assumptions}, fix
\(\epsilon_\star>0\) as at the start of
Section~\ref{subsec:cubic-local-deletion}.
For every \(K_d\uparrow\infty\),
\begin{equation}
 \mathbb E_{P_0^{\otimes n}}\!\left[
  \1_{\mathcal E_n^{\rm reg}}W_{n,\beta}
  \1_{\{s_n<\underline s_d,\ \beta^2\ell(S_n)\le\epsilon_\star d\}}
  \1_{\{\|S_n\|_{\HS}^2-n>K_dd^2/\beta^2\}}
 \right]=o(1).
 \label{eq:cubic-positive-energy-tail}
\end{equation}
The error is uniform for \(r_n\) in a fixed compact subset
\(K\subset(0,\infty)\).
\end{lemma}

\begin{proof}
All expectations are under \(P_0^{\otimes n}\), unless another law is
specified.  We use dyadic peeling, bounding the weighted integral on
each interval \(t<\|S_n\|_{\HS}^2-n\le2t\), subject to the restrictions
in \eqref{eq:cubic-positive-energy-tail}, by \(e^{-c\beta^2t}\).
Step 1 verifies the conditions of
Lemma~\ref{lem:cubic-parameterized-row-ledger} on each interval and
controls its error terms.  Step 2 starts from
\eqref{eq:cubic-summed-variational-bound}, controls its term involving
\(T_{4,n}\) by \eqref{eq:cubic-fourth-entropy-control}, and uses
\eqref{eq:cubic-parameterized-linear-ledger} to convert
the remaining negative relative-entropy term into a negative multiple
of \(\beta^2t\).
These estimates yield \eqref{eq:cubic-positive-shell-bound} through
\eqref{eq:cubic-shared-restricted-gibbs}.  Step 3 sums the bounds over
\(t=2^jK_dd^2/\beta^2\).

\emph{Step 1: the parameter and error bounds on each interval.}
For \(t\ge K_dd^2/\beta^2\), set
\[
 \mathcal A_t=\mathcal E_n^{\rm reg}\cap
 \{s_n<\underline s_d,\ \beta^2\ell(S_n)\le\epsilon_\star d,\
              t<\|S_n\|_{\HS}^2-n\le2t\},
 \qquad R_t=\max\{1,2t/n\}.
\]
If \(\mathbb P(\mathcal A_t)=0\), its weighted integral is zero.
Otherwise, on \(\mathcal A_t\),
\[
 \|S_n\|_{\HS}^2\le n+2t\le(1+R_t)n,\qquad
 E_n=(\|S_n\|_{\HS}^2-n)/2>t/2.
\]
Recall that \(\underline s_d=d^{-1}\log^4d\).
Since \(s_n=N\|S_n\|_{\HS}^2/n^2<\underline s_d\) on
\(\mathcal A_t\), we have \(t<\underline s_dn^2/N\), and hence
\[
 1\le R_t\le\max\{1,2\underline s_dn/N\}\le n\le d^A,\qquad
 \beta^2R_t\le1+2r_nd\underline s_d\le C_K\log^4d
\]
for all sufficiently large \(d\).  The event \(\mathcal A_t\) is
invariant under permutations of the paired sample and satisfies
\eqref{eq:cubic-parameterized-row-hypotheses} with \(R=R_t\) and
\(\epsilon=\epsilon_\star\).
To control the error terms in
\eqref{eq:cubic-summed-variational-bound} and
\eqref{eq:cubic-parameterized-linear-ledger}, use \(N\asymp d^3\) to obtain
\begin{equation}
 \frac{\beta^2R_t^2+R_t^3\beta^4d^4/N^2}{R_td^2}
 =\frac{\beta^2R_t}{d^2}+\frac{(\beta^2R_t)^2d^2}{N^2}
 \le C_K\left(\frac{\log^4d}{d^2}+\frac{\log^8d}{d^4}\right)=o(1),
 \label{eq:cubic-positive-shell-errors}
\end{equation}
uniformly over these intervals.

\emph{Step 2: the weighted integral on each interval.}
By \eqref{eq:cubic-shared-restricted-gibbs}, it suffices to prove
\[
 \mathbb E_QF_n^{(2)}(\beta)-D(Q\Vert P_0^{\otimes n})
 \le-c\beta^2t
\]
for some fixed \(c>0\), uniformly over exchangeable laws \(Q\) with
\(Q(\mathcal A_t)=1\) and \(D(Q\Vert P_0^{\otimes n})<\infty\).
Fix \(\eta=1/4\) and recall \(b_n=\beta^4(n-2)/(3N^2)\).
Applying \eqref{eq:cubic-summed-variational-bound} with \(R=R_t\),
then using \eqref{eq:cubic-positive-shell-errors} and
\(\beta^4n^2/N^2=r_n^2d^2\), gives
\[
 \mathbb E_QF_n^{(2)}(\beta)-D(Q\Vert P_0^{\otimes n})
 \le\frac{b_n}{1-\eta}\mathbb E_QE_n
       -\eta D(Q\Vert P_0^{\otimes n})
       +C_4\beta^8n\mathbb E_Q|T_{4,n}|+CR_td^2+o(1).
\]
We first control \(C_4\beta^8n\mathbb E_Q|T_{4,n}|\).
Since \(Q(\mathcal E_n^{\rm reg})=1\),
\eqref{eq:cubic-fourth-entropy-control} in
Lemma~\ref{lem:cubic-fourth-U-exponential}, with \(c=C_4\) and entropy
coefficient \(\eta/2\), gives
\[
 C_4\beta^8n\mathbb E_Q|T_{4,n}|
 \le\frac{\eta}{2}D(Q\Vert P_0^{\otimes n})+o(1).
\]
After this substitution, the relative-entropy term is
\(-(\eta/2)D(Q\Vert P_0^{\otimes n})\).  To bound it, fix \(\kappa>0\)
sufficiently small for \eqref{eq:cubic-parameterized-linear-ledger}.
The Gibbs variational formula \eqref{eq:cubic-restricted-variational} gives
\[
 D(Q\Vert P_0^{\otimes n})
 \ge\kappa\beta^2\mathbb E_QE_n
       -\log\mathbb E[\1_{\mathcal A_t}e^{\kappa\beta^2E_n}]
 \ge\kappa\beta^2\mathbb E_QE_n-C_\kappa R_td^2-o(1),
\]
where the second inequality uses \eqref{eq:cubic-parameterized-linear-ledger}
with \(\mathcal E=\mathcal A_t\) and \(R=R_t\), and
\eqref{eq:cubic-positive-shell-errors} gives
\(\beta^2R_t^2+R_t^3\beta^4d^4/N^2=o(R_td^2)\).
Since \(b_n/\beta^2\le r_nd/(3N)=o(1)\), for large \(d\) we have
\(b_n/(1-\eta)\le\eta\kappa\beta^2/4\).
Substituting the entropy lower bound and using \(E_n>t/2\) on
\(\mathcal A_t\), we obtain
\[
\begin{aligned}
 \mathbb E_QF_n^{(2)}(\beta)-D(Q\Vert P_0^{\otimes n})
 &\le\left(\frac{b_n}{1-\eta}-\frac{\eta\kappa\beta^2}{2}\right)
       \mathbb E_QE_n+CR_td^2+o(1)\\
 &\le-\frac{\eta\kappa}{8}\beta^2t+CR_td^2+o(1).
\end{aligned}
\]
The constants and errors are uniform in \(t\) and \(Q\).
The term \(CR_td^2\) is negligible relative to \(\beta^2t\), since
\[
 \frac{R_td^2}{\beta^2t}
 \le\frac{d^2}{\beta^2t}+\frac{2d^2}{n\beta^2}
 \le\frac1{K_d}+\frac{2d}{r_nN}=o(1).
\]
Also, \(\beta^2t\ge K_dd^2\to\infty\), so
\(\mathbb E_QF_n^{(2)}(\beta)-D(Q\Vert P_0^{\otimes n})
\le-c\beta^2t\) uniformly over \(Q\), for some fixed \(c>0\).
Taking the supremum in
\eqref{eq:cubic-shared-restricted-gibbs} and exponentiating yields,
after decreasing \(c\),
\begin{equation}
 \mathbb E[\1_{\mathcal A_t}W_{n,\beta}]
 \le e^{-c\beta^2t},\qquad t\ge K_dd^2/\beta^2.
 \label{eq:cubic-positive-shell-bound}
\end{equation}

\emph{Step 3: summation over the positive tail.}
Let \(t_j=2^jK_dd^2/\beta^2\) for \(j\ge0\).
The events \(\mathcal A_{t_j}\) partition the region in
\eqref{eq:cubic-positive-energy-tail}.  Applying
\eqref{eq:cubic-positive-shell-bound} and \(2^j\ge j+1\) gives
\[
 \sum_{j\ge0}\mathbb E[\1_{\mathcal A_{t_j}}W_{n,\beta}]
 \le\sum_{j\ge0}e^{-c2^jK_dd^2}
 \le\frac{e^{-cK_dd^2}}{1-e^{-cK_dd^2}}=o(1).
\]
This proves \eqref{eq:cubic-positive-energy-tail}.
\end{proof}

Corollary~\ref{cor:cubic-local-central-reduction} combines
Lemma~\ref{lem:cubic-positive-energy-tail} with the pointwise bound
\eqref{eq:cubic-local-linear-reward} to obtain the first inequality
in \eqref{eq:cubic-local-roadmap}.

\begin{corollary}[Reduction to the central region]
\label{cor:cubic-local-central-reduction}
Under \eqref{eq:cubic-local-standing-assumptions},
let \(K_d\uparrow\infty\) satisfy \(K_d=o\{(\log d)^{1/8}\}\).
For sufficiently small fixed \(\delta>0\),
\begin{equation}
 \mathbb E\!\left[
 \1_{\mathcal E_n^{\rm reg}}W_{n,\beta}\1_{\{s_n<\delta\}}\right]
 \le
 \mathbb E\!\left[\1_{\mathcal E_n^{\rm reg}}W_{n,\beta}
             \1_{\{\lvert\|S_n\|_{\HS}^2-n\rvert\le K_dd^2/\beta^2\}}\right]+o(1).
 \label{eq:cubic-local-central-reduction}
\end{equation}
The \(o(1)\) is uniform for \(r_n\) in a fixed compact subset
\(K\subset(0,\infty)\).
\end{corollary}

\begin{proof}
Fix \(\delta>0\) within the range of
\eqref{eq:cubic-first-regular-reduction}, and recall \(\mathcal R_n\)
defined there.  By that inequality, it suffices to remove the
positive and negative tails of \(\|S_n\|_{\HS}^2-n\) on \(\mathcal R_n\).

Since \(\mathcal R_n\subseteq\mathcal E_n^{\rm reg}
\cap\{s_n<\underline s_d,\ \beta^2\ell(S_n)\le\epsilon_\star d\}\),
Lemma~\ref{lem:cubic-positive-energy-tail} gives
\[
 \mathbb E\big[\1_{\mathcal R_n}W_{n,\beta}
       \1_{\{\|S_n\|_{\HS}^2-n>K_dd^2/\beta^2\}}\big]=o(1).
\]
The negative tail satisfies
\[
 \mathbb E\big[\1_{\mathcal R_n}W_{n,\beta}
       \1_{\{\|S_n\|_{\HS}^2-n<-K_dd^2/\beta^2\}}\big]
 \le e^{-K_dd^2/2}=o(1),
\]
by \(W_{n,\beta}\le e^{\beta^2(\|S_n\|_{\HS}^2-n)/2}\).
Removing these two tails leaves
\(\mathcal R_n\cap\{\lvert\|S_n\|_{\HS}^2-n\rvert\le K_dd^2/\beta^2\}\).
Since \(\mathcal R_n\subseteq\mathcal E_n^{\rm reg}\) and
\(W_{n,\beta}\ge0\), enlarging this remaining set proves
\eqref{eq:cubic-local-central-reduction}.
\end{proof}

\subsection{Proof of the second inequality}
\label{subsec:cubic-local-centered}

Corollary~\ref{cor:cubic-local-central-reduction} leaves the integral over
\(\bigl|\|S_n\|_{\HS}^2-n\bigr|\le K_dd^2/\beta^2\).
Corollary~\ref{cor:cubic-central-weighted-deletions} first inserts the
operator-norm and fourth-order restrictions needed for the sharp row
estimate.  Lemma~\ref{lem:cubic-centered-simple-edge} then bounds the
retained integral by \(\exp\{\binom n3\beta^6/N^4+o(1)\}\), completing
the second inequality.

\begin{corollary}[Weighted restrictions on the central region]
\label{cor:cubic-central-weighted-deletions}
Under \eqref{eq:cubic-local-standing-assumptions} and
\eqref{eq:cubic-local-epsilon-conditions}, let \(K_d\uparrow\infty\)
satisfy \(K_d=o\{(\log d)^{1/8}\}\).  Define
\[
 \mathcal B_{n,d}=\{\bigl|\|S_n\|_{\HS}^2-n\bigr|\le K_dd^2/\beta^2\},\quad
 \mathcal O_{n,d}=\{\|S_n\|_{\rm op}\le\epsilon_dd/\beta^2\},\qquad
 \mathcal V_{n,d}=\{\beta^8n|T_{4,n}|\le\epsilon_d\}.
\]
Then
\begin{align}
 \mathbb E_{P_0^{\otimes n}}\!
 [\1_{\mathcal E_n^{\rm reg}}W_{n,\beta}
       \1_{\mathcal B_{n,d}}\1_{\mathcal O_{n,d}^c}]&=o(1),
 \label{eq:cubic-terminal-operator-deletion}\\
 \mathbb E_{P_0^{\otimes n}}\!
 [\1_{\mathcal E_n^{\rm reg}}W_{n,\beta}
       \1_{\mathcal B_{n,d}}\1_{\mathcal V_{n,d}^c}]&=o(1).
 \label{eq:cubic-fourth-U-weighted-deletion}
\end{align}
\end{corollary}

\begin{proof}
On \(\mathcal B_{n,d}\), \eqref{eq:cubic-local-linear-reward} gives
\begin{equation}
 W_{n,\beta}\le e^{\beta^2(\|S_n\|_{\HS}^2-n)/2}\le e^{K_dd^2/2}.
 \label{eq:cubic-terminal-operator-weight}
\end{equation}
Combining \eqref{eq:cubic-terminal-operator-weight} with
Lemmas~\ref{lem:cubic-operator-norm-tail} and
\ref{lem:cubic-fourth-U-moments}, respectively, yields
\[
\begin{aligned}
 \mathbb E[\1_{\mathcal E_n^{\rm reg}}W_{n,\beta}
       \1_{\mathcal B_{n,d}}\1_{\mathcal O_{n,d}^c}]
 &\le2\exp\left(\frac{K_dd^2}{2}
              -\frac{c\epsilon_d^2d^{5/2}}{\beta^2}\right)=o(1),\\
 \mathbb E[\1_{\mathcal E_n^{\rm reg}}W_{n,\beta}
       \1_{\mathcal B_{n,d}}\1_{\mathcal V_{n,d}^c}]
 &\le\exp\left(\frac{K_dd^2}{2}-d^2\log d\right)=o(1).
\end{aligned}
\]
The limits follow from \(|\beta|\le1\) and \(K_d=o(\log d)\), since
\eqref{eq:cubic-local-epsilon-conditions} gives
\[
 \frac{\epsilon_d^2\sqrt d}{K_d}
 =\frac{\epsilon_d^2\sqrt d}{\log d}\frac{\log d}{K_d}
 \longrightarrow\infty.
\]
\end{proof}

We now apply the shared row estimates on
\(\mathcal E_n^{\rm reg}\cap\mathcal B_{n,d}\cap
\mathcal O_{n,d}\cap\mathcal V_{n,d}\).
Lemma~\ref{lem:cubic-centered-simple-edge} identifies the leading
constant in the resulting upper bound.

\begin{lemma}[Partition-function bound near the diagonal scale]
\label{lem:cubic-centered-simple-edge}
Under \eqref{eq:cubic-local-standing-assumptions}, let \(K_d\uparrow\infty\)
satisfy \(K_d=o\{(\log d)^{1/8}\}\).
\begin{equation}
  \mathbb E_{P_0^{\otimes n}}\left[
    \1_{{\cal G}_n^\dagger(X)}\1_{{\cal G}_n^\dagger(X')}
    W_{n,\beta}\1_{\{\lvert\|S_n\|_{\HS}^2-n\rvert\le K_dd^2/\beta^2\}}
  \right]
  \le
  \exp\left\{\binom n3\frac{\beta^6}{N^4}+o(1)\right\}.
  \label{eq:cubic-centered-simple-edge-bound}
\end{equation}
The \(o(1)\) is uniform when \(n\beta^2/(Nd)\) stays in a compact subset of
\((0,\infty)\).
\end{lemma}

\begin{proof}
Step 1 uses Corollary~\ref{cor:cubic-central-weighted-deletions} and
\eqref{eq:cubic-shared-restricted-gibbs} to reduce the desired bound to
an upper bound on
\(\mathbb E_QF_n^{(2)}(\beta)-D(Q\Vert P_0^{\otimes n})\).
Step 2 shows that the fourth-order terms and higher-order remainders
contribute at most \(o(1)\), leaving the two quadratic sums and the
relative-entropy term in \eqref{eq:cubic-entropy-reserved-ledger}.
Step 3 evaluates these sums and bounds
\(b_n\mathbb E_QE_n-\eta D(Q\Vert P_0^{\otimes n})\) to obtain
\(\binom n3\beta^6/N^4+o(1)\).

\emph{Step 1: restricted variational reduction.}
Recall \(\mathcal B_{n,d}\), \(\mathcal O_{n,d}\), and
\(\mathcal V_{n,d}\) from Corollary~\ref{cor:cubic-central-weighted-deletions},
and let
\(\mathcal E^{\rm term}_{n,d}
=\mathcal E_n^{\rm reg}\cap\mathcal B_{n,d}\cap
\mathcal O_{n,d}\cap\mathcal V_{n,d}\).
Corollary~\ref{cor:cubic-central-weighted-deletions} gives
\begin{align*}
 0&\le\mathbb E[\1_{\mathcal G_n^\dagger(X)}\1_{\mathcal G_n^\dagger(X')}
                  W_{n,\beta}\1_{\mathcal B_{n,d}}]
       -\mathbb E[\1_{\mathcal E^{\rm term}_{n,d}}W_{n,\beta}]\\
 &\le\mathbb E[\1_{\mathcal G_n^\dagger(X)}\1_{\mathcal G_n^\dagger(X')}
        W_{n,\beta}\1_{\mathcal B_{n,d}}
        (\1_{\mathcal O_{n,d}^c}+\1_{\mathcal V_{n,d}^c})]=o(1).
\end{align*}

Applying \eqref{eq:cubic-shared-restricted-gibbs} to
\(\mathcal E^{\rm term}_{n,d}\) gives
\begin{equation}
 \log\mathbb E\!\left[\1_{\mathcal E^{\rm term}_{n,d}}W_{n,\beta}\right]
 \le \sup_{Q:Q(\mathcal E^{\rm term}_{n,d})=1}
 \{\mathbb E_QF_n^{(2)}(\beta)-D(Q\Vert P_0^{\otimes n})\}+o(1).
 \label{eq:cubic-restricted-gibbs}
\end{equation}
The event \(\mathcal E^{\rm term}_{n,d}\) and
\(F_n^{(2)}(\beta)\) are invariant under permutations of the paired
sample.  Thus the permutation averaging used in
\eqref{eq:cubic-shared-restricted-gibbs} allows us to restrict the
supremum to exchangeable \(Q\) with finite relative entropy.

\emph{Step 2: bounding the variational expression.}
Take \(\eta=\beta^2(n-2)/(3N^2)\).
For every exchangeable law \(Q\) in \eqref{eq:cubic-restricted-gibbs},
\eqref{eq:cubic-shared-entropy-sum} gives
\[
 \mathbb E_QF_n^{(2)}(\beta)-D(Q\Vert P_0^{\otimes n})
 \le(1-\eta)\sum_{m<n}\mathbb E_Q
 \log Z_m\!\left(\frac{\beta^2}{1-\eta},1-\eta\right)
 -\eta D(Q\Vert P_0^{\otimes n}).
\]
Here \(\eta=O(d/N)=o(1)\), so \(1-\eta\ge1/2\) eventually, and
\(b_n=\beta^2\eta\) by \eqref{eq:cubic-prefix-linear-identity}.
We estimate the logarithmic sum on the right using
\eqref{eq:cubic-sixth-row-remainder}.

On \(\mathcal E^{\rm term}_{n,d}\),
\(\|S_n\|_{\HS}^2\le n+K_dd^2/\beta^2\le2n\) eventually and
\(\beta^2\ell(S_n)\le\beta^2\|S_n\|_{\rm op}\le\epsilon_dd\).
Since \(\epsilon_d\to0\) by \eqref{eq:cubic-local-epsilon-conditions},
Lemmas~\ref{lem:cubic-regular-prefix-estimates}
and~\ref{lem:cubic-regular-row-expansion} apply with \(R=1\) and
\(C_\star=2\).  For prefixes satisfying
\(\beta^2\ell(S_m)\le c_0d\), applying
\eqref{eq:cubic-sixth-row-remainder} at
\((u,c)=(\beta^2/(1-\eta),1-\eta)\) gives
\[
\begin{aligned}
 (1-\eta)\log Z_m\!\left(\frac{\beta^2}{1-\eta},1-\eta\right)
 &\le\frac{\beta^4}{2(1-\eta)}A_2(1-\eta)
    +\frac{\beta^8}{24(1-\eta)^3}\Gamma_{4,m}(1-\eta)\\
 &\quad+\frac{C\beta^{12}}{N^6}
       (\|S_m\|_{\HS}^6+m^3).
\end{aligned}
\]
By exchangeability, \eqref{eq:cubic-exceptional-row-restoration} bounds
the total \(Q\)-expected absolute contribution of the logarithms and
their quadratic and fourth-order terms on
\(\{\beta^2\ell(S_m)>c_0d\}\) by \(o(1)\).
Taking \(Q\)-expectations, summing over \(m<n\), and using
\(A_2(1-\eta)=(\|S_m\|_{\HS}^2-(1-\eta)m)/N^2\) gives
\[
\begin{aligned}
 &(1-\eta)\sum_{m<n}\mathbb E_Q
       \log Z_m\!\left(\frac{\beta^2}{1-\eta},1-\eta\right)\\
 &\quad\le\frac{\beta^4}{2N^2}\sum_{m<n}
       \mathbb E_Q(\|S_m\|_{\HS}^2-m)
       +\frac{\eta\beta^4}{2(1-\eta)N^2}\sum_{m<n}
       \mathbb E_Q\|S_m\|_{\HS}^2\\
 &\qquad+\frac{\beta^8}{24(1-\eta)^3}\sum_{m<n}
       \mathbb E_Q\Gamma_{4,m}(1-\eta)
       +\frac{C\beta^{12}}{N^6}\sum_{m<n}
       \mathbb E_Q(\|S_m\|_{\HS}^6+m^3)+o(1).
\end{aligned}
\]
We retain the two quadratic sums and bound the last two sums by
\(o(1)\).  For the fourth-order sum, we estimate
\(\beta^8\sum_{m<n}\mathbb E_Q\Gamma_{4,m}(1-\eta)/\{24(1-\eta)^3\}\)
from above.
First consider \(\Gamma_{4,m}(1)\) in
\eqref{eq:cubic-complete-fourth-cumulant}.
The terms \(-2m\lambda_d^2\) and \(-3A_2(1)^2\) are nonpositive.
For the term \(-16\lambda_d^2E_m\),
\eqref{eq:cubic-exact-prefix-bridge} gives
\[
\begin{aligned}
 \left|\frac{\beta^8}{24(1-\eta)^3}
       \sum_{m<n}\mathbb E_Q(-16\lambda_d^2E_m)\right|
 &=\frac{2\beta^8\lambda_d^2(n-2)}{9(1-\eta)^3}
       |\mathbb E_QE_n|\\
 &\le CK_d\frac{\beta^6nd^2}{N^4}
 \le CK_d\frac{\beta^4d^3}{N^3}=o(1),
\end{aligned}
\]
where we used \(|\mathbb E_QE_n|\le K_dd^2/(2\beta^2)\) from
\(Q(\mathcal B_{n,d})=1\),
\(\lambda_d=O(N^{-2})\), \(1-\eta\ge1/2\), and
\(n\beta^2=r_nNd\) with \(r_n\in K\).
For the terms \(24\mathbb E_we_{4,m}(w)\), the constraint
\(Q({\cal V}_{n,d})=1\) gives
\(\beta^8n\mathbb E_Q|T_{4,n}|\le\epsilon_d\).
Applying \eqref{eq:cubic-fourth-prefix-identity}, we obtain
\begin{equation}
 \frac{\beta^8}{(1-\eta)^3}
 \left|\sum_{m<n}\mathbb E_Q\mathbb E_we_{4,m}(w)\right|
 \le \frac{\beta^8n}{5(1-\eta)^3}\mathbb E_Q|T_{4,n}|
 \le \frac{\epsilon_d}{5(1-\eta)^3}=o(1).
 \label{eq:cubic-terminal-fourth-error}
\end{equation}
To replace \(\Gamma_{4,m}(1)\) by \(\Gamma_{4,m}(1-\eta)\), use
\eqref{eq:cubic-compensator-cumulant} with \(c=1-\eta\).
Together with \eqref{eq:cubic-parameterized-Hilbert-bridge} for \(R=1\),
it gives
\[
 \frac{\beta^8}{24(1-\eta)^3}\sum_{m<n}
    \mathbb E_Q\Gamma_{4,m}(1-\eta)
 \le\frac{\beta^8}{24(1-\eta)^3}\sum_{m<n}
    \mathbb E_Q\Gamma_{4,m}(1)
    +C\eta\frac{\beta^8n^3}{N^4}\le o(1).
\]
For the sixth-order remainders in \eqref{eq:cubic-sixth-row-remainder},
\eqref{eq:cubic-parameterized-Hilbert-bridge} with \(R=1\) gives
\begin{equation}
  \frac{C\beta^{12}}{N^6}\sum_{m<n}
  \mathbb E_Q(\|S_m\|_{\HS}^6+m^3)
  \le C\frac{\beta^{12}n^4}{N^6}
  =O(\beta^4d^4/N^2)=o(1).
  \label{eq:cubic-higher-row-ledger}
\end{equation}

Substituting these estimates into the bound for
\((1-\eta)\sum_{m<n}\mathbb E_Q
\log Z_m(\beta^2/(1-\eta),1-\eta)\) gives
\begin{equation}
 \sum_{m<n}\mathbb E_Q\bigg[
  (1-\eta)\log Z_m\!\left(\frac{\beta^2}{1-\eta},1-\eta\right)
  -\frac{\beta^4}{2N^2}(\|S_m\|_{\HS}^2-m)
  -\frac{\eta\beta^4}{2(1-\eta)N^2}\|S_m\|_{\HS}^2
  \bigg]\le o(1).
  \label{eq:cubic-sharp-row-ledger}
\end{equation}
Substituting \eqref{eq:cubic-sharp-row-ledger} into
\eqref{eq:cubic-shared-entropy-sum} gives
\begin{align}
  \mathbb E_QF_n^{(2)}(\beta)-D(Q\Vert P_0^{\otimes n})
  \le{}&
  \frac{\beta^4}{2N^2}\sum_{m<n}
     \mathbb E_Q(\|S_m\|_{\HS}^2-m)
     -\eta D(Q\Vert P_0^{\otimes n}) \notag\\
  &+\frac{\eta \beta^4}{2(1-\eta) N^2}
     \sum_{m<n}\mathbb E_Q\|S_m\|_{\HS}^2+o(1).
  \label{eq:cubic-entropy-reserved-ledger}
\end{align}

\emph{Step 3: the linear term and the leading constant.}
By \eqref{eq:cubic-prefix-linear-identity}, the first sum in
\eqref{eq:cubic-entropy-reserved-ledger} equals
\(b_n\mathbb E_QE_n=\eta\beta^2\mathbb E_QE_n\).
To control it together with \(-\eta D(Q\Vert P_0^{\otimes n})\),
the Gibbs variational formula \eqref{eq:cubic-restricted-variational}
reduces the problem to estimating
\[
  \Xi_{n,d}(\beta^2):=
  \log\mathbb E[\1_{{\cal E}^{\rm term}_{n,d}}e^{\beta^2E_n}].
\]
As in \eqref{eq:cubic-restricted-gibbs}, it suffices to consider
exchangeable \(Q\) supported on \({\cal E}^{\rm term}_{n,d}\) with
finite relative entropy.  Since
\(E_n=\sum_{m<n}f_m(w_{m+1})\), applying
\eqref{eq:cubic-linear-entropy-duality} to each conditional law
with \(H(w)=\beta^2f_m(w)\) and using the entropy chain rule gives
\[
 \beta^2\mathbb E_QE_n-D(Q\Vert P_0^{\otimes n})
 \le\sum_{m<n}\mathbb E_Q\log Z_m(\beta^2,0).
\]
Apply Lemma~\ref{lem:cubic-regular-row-expansion} at
\((u,c)=(\beta^2,0)\), with the full-sample hypotheses verified in
Step 2.  Combining \eqref{eq:cubic-local-fourth-moment} with the remainder
bound \eqref{eq:cubic-sixth-row-remainder} and the exceptional-prefix
estimate \eqref{eq:cubic-exceptional-row-restoration} gives
\begin{align*}
 \sum_{m<n}\mathbb E_Q\log Z_m(\beta^2,0)
 &\le\frac{\beta^4}{2N^2}\sum_{m<n}\mathbb E_Q\|S_m\|_{\HS}^2
    +\frac{C\beta^8}{N^4}\sum_{m<n}\mathbb E_Q\|S_m\|_{\HS}^4\\
 &\quad+\frac{C\beta^{12}}{N^6}\sum_{m<n}
       \big(\mathbb E_Q\|S_m\|_{\HS}^6+m^3\big)
    +C(\beta^2n^2+\beta^4n^3+\beta^8n^5)e^{-cd/\beta^2}.
\end{align*}
We estimate the quadratic sum using
\eqref{eq:cubic-prefix-diagonal-identity} and
\(Q(\mathcal B_{n,d})=1\), which give
\[
\begin{aligned}
 \frac{\beta^4}{2N^2}\sum_{m<n}\mathbb E_Q\|S_m\|_{\HS}^2
 &=\frac{\beta^4n(n-1)}{4N^2}
   +\frac{\beta^4(n-2)}{6N^2}\mathbb E_Q(\|S_n\|_{\HS}^2-n)\\
 &\le\frac{\beta^4n(n-1)}{4N^2}
   +\frac{K_d\beta^2(n-2)d^2}{6N^2},
\end{aligned}
\]
where \(\beta^2(n-2)d^2/N^2\le r_nd^3/N=O(1)\).
For the fourth- and sixth-order sums,
\eqref{eq:cubic-parameterized-Hilbert-bridge} with \(R=1\) gives
\[
 \frac{C\beta^8}{N^4}\sum_{m<n}\mathbb E_Q\|S_m\|_{\HS}^4
 +\frac{C\beta^{12}}{N^6}\sum_{m<n}
     \big(\mathbb E_Q\|S_m\|_{\HS}^6+m^3\big)
 \le\frac{C\beta^8n^3}{N^4}+\frac{C\beta^{12}n^4}{N^6}
 =O(1).
\]
The term \(C(\beta^2n^2+\beta^4n^3+\beta^8n^5)e^{-cd/\beta^2}\)
is \(o(1)\) by \eqref{eq:cubic-exceptional-row-restoration}.
Taking the supremum over \(Q\) and using
\(K_d=o\{(\log d)^{1/8}\}=o(d^2)\) yields
\begin{equation}
 \Xi_{n,d}(\beta^2)
 \le\frac{\beta^4n(n-1)}{4N^2}+o(d^2).
 \label{eq:cubic-restricted-linear-partition}
\end{equation}

Returning to \(b_n\mathbb E_QE_n-\eta D(Q\Vert P_0^{\otimes n})\),
use \(b_n=\beta^2\eta\) and \eqref{eq:cubic-restricted-variational} to obtain
\begin{equation}
  b_n\mathbb E_QE_n-\eta D(Q\Vert P_0^{\otimes n})
  \le\eta\Xi_{n,d}(\beta^2)
  \le\frac{\eta \beta^4n(n-1)}{4N^2}+o(1).
  \label{eq:cubic-linear-entropy-absorption}
\end{equation}
Here \(\eta=O(d/N)=O(d^{-2})\), so the \(o(d^2)\) error in
\eqref{eq:cubic-restricted-linear-partition} becomes \(o(1)\).
It remains to evaluate the last sum in
\eqref{eq:cubic-entropy-reserved-ledger}.
By \eqref{eq:cubic-prefix-diagonal-identity},
\[
 \frac{\eta\beta^4}{2(1-\eta)N^2}
       \sum_{m<n}\mathbb E_Q\|S_m\|_{\HS}^2
 =\frac{\eta\beta^4n(n-1)}{4(1-\eta)N^2}
  +\frac{\eta\beta^4(n-2)}{6(1-\eta)N^2}
       \mathbb E_Q(\|S_n\|_{\HS}^2-n).
\]
The second term is \(o(1)\), since \(Q(\mathcal B_{n,d})=1\) gives
\[
 \left|\frac{\eta\beta^4(n-2)}{6(1-\eta) N^2}
       \mathbb E_Q(\|S_n\|_{\HS}^2-n)\right|
 \le\frac{C\eta\beta^2 nK_dd^2}{N^2}
 \le\frac{CK_dd^4}{N^2}=o(1).
\]
Substituting this evaluation and
\eqref{eq:cubic-linear-entropy-absorption} into
\eqref{eq:cubic-entropy-reserved-ledger} gives, uniformly over the laws
\(Q\) in \eqref{eq:cubic-restricted-gibbs},
\begin{align*}
  \mathbb E_QF_n^{(2)}(\beta)-D(Q\Vert P_0^{\otimes n})
  &\le
  \frac{\eta \beta^4n(n-1)}{4N^2}
  +\frac{\eta \beta^4n(n-1)}{4(1-\eta) N^2}+o(1)\\
  &=\frac{\beta^6n(n-1)(n-2)}{12N^4}(1+(1-\eta)^{-1})+o(1)\\
  &=\binom n3\frac{\beta^6}{N^4}+o(1).
\end{align*}
Here \(\eta=\beta^2(n-2)/(3N^2)=o(1)\) and
\(n^3\beta^6/N^4=r_n^3d^3/N=O(1)\).
Substituting this bound into \eqref{eq:cubic-restricted-gibbs} and
exponentiating bounds the restricted integral by
\(\exp\{\binom n3\beta^6/N^4+o(1)\}\).
The \(o(1)\) difference established in Step 1 can be absorbed into
this bound, since \(\exp\{\binom n3\beta^6/N^4\}\ge1\).
This proves \eqref{eq:cubic-centered-simple-edge-bound}.
\end{proof}

\begin{proof}[Proof of Proposition~\ref{prop:cubic-raw-local-partition}]
Set \(\underline c=\min K>0\).  The sampling condition in
Lemma~\ref{lem:cubic-hereditary-square-profile} holds because
\(n/(Nd)=r_n/\beta^2\ge\underline c\).
Take \(\mathcal G_n^\dagger=\mathcal G_{n,M_\star}^\dagger\) from that
lemma.  Its construction gives the inclusion
\(\mathcal G_n^\dagger\subseteq\mathcal G_{n,M_\star}\), the stated invariances, and
\(\rho_n(\mathcal G_n^\dagger)=1-o(d^{-M_\star})\), without using \(\beta\).
With \(J=[n]\), \eqref{eq:cubic-hereditary-covariance-profile} gives
\[
 \left\|\sum_{i=1}^nq_{3,d}(X_i)q_{3,d}(X_i)^{\mathsf T}\right\|_{\rm op}
 \le C\frac{n}{N}.
\]
The polynomial formula \eqref{eq:cubic-zonal-polynomial} and
\eqref{eq:cubic-raw-pair-separation} give
\[
 \max_{i\ne j}|Z_{3,d}(X_i,X_j)|
 \le C\left\{\left(\frac{\log d}{\sqrt d}\right)^3
             +\frac{\log d}{d^{3/2}}\right\}
 \le C\frac{\log^3 d}{\sqrt N}.
\]
It remains to prove \eqref{eq:cubic-raw-local-partition}.
Under the pairing of coordinates, \(P_0^{\otimes n}\) is
\(\rho_n\otimes\rho_n\), and
\(\1_{\mathcal E_n^{\rm reg}}
 =\1_{\mathcal G_n^\dagger}(X)\1_{\mathcal G_n^\dagger}(X')\).
Choose \(\delta>0\) in the range of
Corollary~\ref{cor:cubic-local-central-reduction} and
\(K_d\uparrow\infty\) with \(K_d=o\{(\log d)^{1/8}\}\).
Combining \eqref{eq:cubic-local-central-reduction} with
Lemma~\ref{lem:cubic-centered-simple-edge} gives
\begin{align*}
 \mathbb E\big[\1_{\mathcal E_n^{\rm reg}}W_{n,\beta}
                    \1_{\{s_n<\delta\}}\big]
 &\le\mathbb E\big[\1_{\mathcal E_n^{\rm reg}}W_{n,\beta}
                    \1_{\{\lvert\|S_n\|_{\HS}^2-n\rvert\le K_dd^2/\beta^2\}}\big]+o(1)\\
 &\le\exp\left\{\binom n3\frac{\beta^6}{N^4}+o(1)\right\}+o(1)
 \longrightarrow e^{r^3},
\end{align*}
because \(\binom n3\beta^6/N^4=(1+o(1))r_n^3d^3/(6N)\to r^3\).
This proves \eqref{eq:cubic-raw-local-partition}.
Uniform boundedness on \(K\) also follows by taking a convergent
subsequence of \(r_n\) if a uniform bound were to fail.
\end{proof}


\section*{Statements and Declarations}
\paragraph{Author contributions.}
All authors reviewed the manuscript.

\paragraph{Competing interests.}
The authors declare no competing interests.

\paragraph{Data availability.}
No empirical datasets were generated or analysed in this theoretical study.

\paragraph{Use of AI-assisted tools.}
During preparation of this manuscript, the authors used OpenAI tools for
assistance with mathematical reasoning and proof checking, exposition,
language editing, and LaTeX preparation. The authors are responsible for
the mathematical claims, proofs, references, and final text.

\bibliographystyle{alpha}
\begingroup\emergencystretch=2em\bibliography{references}\endgroup
\end{document}